\documentclass{article}
\usepackage{amsmath}
\usepackage{amssymb}
\usepackage{amsthm}
\usepackage{cite}
\usepackage{euscript}
\usepackage[arrow,curve,matrix,arc,2cell]{xy}
\UseAllTwocells
\usepackage{hyperref}
\DeclareFontFamily{U}{rsfs}{} \DeclareFontShape{U}{rsfs}{n}{it}{<->
rsfs10}{} \DeclareSymbolFont{mscr}{U}{rsfs}{n}{it}
\DeclareSymbolFontAlphabet{\scr}{mscr}
\def\mathscr{\scr}
\begin{document}
\def\e#1\e{\begin{equation}#1\end{equation}}
\def\ea#1\ea{\begin{align}#1\end{align}}
\def\eq#1{{\rm(\ref{#1})}}
\theoremstyle{plain}
\newtheorem{thm}{Theorem}[section]
\newtheorem{lem}[thm]{Lemma}
\newtheorem{prop}[thm]{Proposition}
\newtheorem{cor}[thm]{Corollary}
\theoremstyle{definition}
\newtheorem{dfn}[thm]{Definition}
\newtheorem{ex}[thm]{Example}
\newtheorem{rem}[thm]{Remark}
\numberwithin{figure}{section}
\numberwithin{equation}{section}
\def\dim{\mathop{\rm dim}\nolimits}
\def\codim{\mathop{\rm codim}\nolimits}
\def\vdim{\mathop{\rm vdim}\nolimits}
\def\depth{\mathop{\rm depth}\nolimits}
\def\Ker{\mathop{\rm Ker}}
\def\supp{\mathop{\rm supp}}
\def\Im{\mathop{\rm Im}}
\def\rank{\mathop{\rm rank}\nolimits}
\def\Hom{\mathop{\rm Hom}\nolimits}
\def\id{{\mathop{\rm id}\nolimits}}
\def\Id{{\mathop{\rm Id}\nolimits}}
\def\inc{{\rm inc}}
\def\ev{{\rm ev}}
\def\gp{{\rm gp}}
\def\Spec{\mathop{\rm Spec}\nolimits}
\def\Sets{{\mathop{\bf Sets}}}
\def\Mon{{\mathop{\bf Mon}}}
\def\Monfg{{\mathop{\bf Mon^{fg}}}}
\def\Monwt{{\mathop{\bf Mon^{wt}}}}
\def\Monto{{\mathop{\bf Mon^{to}}}}
\def\Man{{\mathop{\bf Man}}}
\def\Manb{{\mathop{\bf Man^b}}}
\def\Manc{{\mathop{\bf Man^c}}}
\def\Mangc{{\mathop{\bf Man^{gc}}}}
\def\Mancra{{\mathop{\bf Man^c_{ra}}}}
\def\Mangcra{{\mathop{\bf Man^{gc}_{ra}}}}
\def\Mancin{{\mathop{\bf Man^c_{in}}}}  
\def\Mancsi{{\mathop{\bf Man^c_{si}}}} 
\def\Mancst{{\mathop{\bf Man^c_{st}}}} 
\def\Mancis{{\mathop{\bf Man^c_{is}}}}
\def\cManc{{\mathop{\bf\check{M}an^c}}}
\def\cMancin{{\mathop{\bf\check{M}an^c_{in}}}} 
\def\cMancst{{\mathop{\bf\check{M}an^c_{st}}}}
\def\cMancis{{\mathop{\bf\check{M}an^c_{is}}}}
\def\cMancsi{{\mathop{\bf\check{M}an^c_{si}}}}
\def\cMangc{{\mathop{\bf\check{M}an^{gc}}}}
\def\Mangcin{{\mathop{\bf Man^{gc}_{in}}}}
\def\cMangcin{{\mathop{\bf\check{M}an^{gc}_{in}}}}
\def\Mangcsi{{\mathop{\bf Man^{gc}_{si}}}}
\def\cMangcsi{{\mathop{\bf\check{M}an^{gc}_{si}}}}
\def\dMan{{\mathop{\bf dMan}}}
\def\dOrb{{\mathop{\bf dOrb}}}
\def\MKur{{\mathop{\bf MKur}}}
\def\Kur{{\mathop{\bf Kur}}}
\def\MKurc{{\mathop{\bf MKur^c}}}
\def\Kurc{{\mathop{\bf Kur^c}}}
\def\MKurgc{{\mathop{\bf MKur^{gc}}}}
\def\Kurgc{{\mathop{\bf Kur^{gc}}}}
\def\Kurgcin{{\mathop{\bf Kur^{gc}_{in}}}}
\def\ul{\underline}
\def\bs{\boldsymbol}
\def\ge{\geqslant}
\def\le{\leqslant\nobreak}
\def\D{{\mathbin{\mathbb D}}}
\def\R{{\mathbin{\mathbb R}}}
\def\Z{{\mathbin{\mathbb Z}}}
\def\Q{{\mathbin{\mathbb Q}}}
\def\N{{\mathbin{\mathbb N}}}
\def\C{{\mathbin{\mathbb C}}}
\def\cC{{\mathbin{\cal C}}}
\def\cI{{\mathbin{\cal I}}}
\def\O{{\mathbin{\cal O}}}
\def\oM{{\mathbin{\smash{\,\,\overline{\!\!\mathcal M\!}\,}}}}
\def\bW{{\bs W}\kern -0.1em}
\def\bX{{\bs X}}
\def\bY{{\bs Y}\kern -0.1em}
\def\bZ{{\bs Z}}
\def\al{\alpha}
\def\be{\beta}
\def\ga{\gamma}
\def\de{\delta}
\def\io{\iota}
\def\ep{\epsilon}
\def\la{\lambda}
\def\ka{\kappa}
\def\th{\theta}
\def\ze{\zeta}
\def\up{\upsilon}
\def\vp{\varphi}
\def\si{\sigma}
\def\om{\omega}
\def\De{\Delta}
\def\La{\Lambda}
\def\Om{\Omega}
\def\Up{\Upsilon}
\def\Ga{\Gamma}
\def\Si{\Sigma}
\def\Th{\Theta}
\def\pd{\partial}
\def\ts{\textstyle}
\def\st{\scriptstyle}
\def\sst{\scriptscriptstyle}
\def\w{\wedge}
\def\sm{\setminus}
\def\bu{\bullet}
\def\sh{\sharp}
\def\op{\oplus}
\def\op{\oplus}
\def\ot{\otimes}
\def\ov{\overline}
\def\bigop{\bigoplus}
\def\bigot{\bigotimes}
\def\iy{\infty}
\def\es{\emptyset}
\def\ra{\rightarrow}
\def\Ra{\Rightarrow}
\def\Longra{\Longrightarrow}
\def\ab{\allowbreak}
\def\longra{\longrightarrow}
\def\hookra{\hookrightarrow}
\def\dashra{\dashrightarrow}
\def\t{\times}
\def\ci{\circ}
\def\ti{\tilde}
\def\d{{\rm d}}
\def\md#1{\vert #1 \vert}
\title{A generalization of manifolds with corners}
\author{Dominic Joyce}
\date{}
\maketitle

\begin{abstract} In conventional Differential Geometry one studies manifolds, locally modelled on $\R^n$, manifolds with boundary, locally modelled on $[0,\iy)\t\R^{n-1}$, and manifolds with corners, locally modelled on $[0,\iy)^k\t\R^{n-k}$. They form categories $\Man\subset\Manb\subset\Manc$. Manifolds with corners $X$ have boundaries $\pd X$, also manifolds with corners, with~$\dim\pd X\!=\!\dim X\!-\!1$.

We introduce a new notion of {\it manifolds with generalized corners}, or {\it manifolds with g-corners}, extending manifolds with corners, which form a category $\Mangc$ with $\Man\subset\Manb\subset\Manc\subset\Mangc$. Manifolds with g-corners are locally modelled on $X_P=\Hom_\Mon\bigl(P,[0,\iy)\bigr)$ for $P$ a weakly toric monoid, where $X_P\cong[0,\iy)^k\t\R^{n-k}$ for $P=\N^k\t\Z^{n-k}$.

Most differential geometry of manifolds with corners extends nicely to manifolds with g-corners, including well-behaved boundaries $\pd X$. In some ways manifolds with g-corners have better properties than manifolds with corners; in particular, transverse fibre products in $\Mangc$ exist under much weaker conditions than in~$\Manc$.

This paper was motivated by future applications in symplectic geometry, in which some moduli spaces of $J$-holomorphic curves can be manifolds or Kuranishi spaces with g-corners rather than ordinary corners. 

Our manifolds with g-corners are related to the `interior binomial varieties' of Kottke and Melrose \cite{KoMe}, and the `positive log differentiable spaces' of Gillam and Molcho~\cite{GiMo}.
\end{abstract}

\setcounter{tocdepth}{2}
\tableofcontents

\section{Introduction}
\label{gc1}

Manifolds with corners are differential-geometric spaces locally modelled on $\R^n_k=[0,\iy)^k\t\R^{n-k}$, just as manifolds are spaces locally modelled on $\R^n$. Manifolds with corners form a category $\Manc$, containing manifolds $\Man\subset\Manc$ as a full subcategory. Some references are Melrose \cite{Melr2,Melr3,Melr4} and the author~\cite{Joyc1}.

This paper introduces an extension of manifolds with corners, called {\it manifolds with generalized corners}, or {\it manifolds with g-corners}. They are differential-geometric spaces locally modelled on $X_P=\Hom_\Mon\bigl(P,[0,\iy)\bigr)$ for $P$ a weakly toric monoid, where $\Mon$ is the category of (commutative) monoids, and $[0,\iy)$ is a monoid under multiplication. When $P=\N^k\t\Z^{n-k}$ we have $X_P\cong\R^n_k=[0,\iy)^k\t\R^{n-k}$, so the local models include those for manifolds with corners. Manifolds with g-corners form a category $\Mangc$, which contains manifolds with corners $\Manc\subset\Mangc$ as a full subcategory.

To convey the idea, we start with an example:

\begin{ex} The simplest manifold with g-corners which is not a manifold with corners is $X=\bigl\{(x_1,x_2,x_3,x_4)\in[0,\iy)^4: x_1x_2=x_3x_4\bigr\}$. We have $X\cong X_P$, where $P$ is the monoid $P=\bigl\{(a,b,c)\in\N^3: c\le a+b\bigr\}$.

Then $X$ is 3-dimensional, and has four 2-dimensional boundary faces
\begin{align*}
X_{13}&\!=\!\bigl\{(x_1,0,x_3,0):x_1,x_3\!\in\![0,\iy)\bigr\}, &
X_{14}&\!=\!\bigl\{(x_1,0,0,x_4):x_1,x_4\!\in\![0,\iy)\bigr\}, \\
X_{23}&\!=\!\bigl\{(0,x_2,x_3,0):x_2,x_3\!\in\![0,\iy)\bigr\}, &
X_{24}&\!=\!\bigl\{(0,x_2,0,x_4):x_2,x_4\!\in\![0,\iy)\bigr\},
\end{align*}
and four 1-dimensional edges 
\begin{align*}
X_1&=\bigl\{(x_1,0,0,0):x_1\in[0,\iy)\bigr\}, &
X_2&=\bigl\{(0,x_2,0,0):x_2\in[0,\iy)\bigr\}, \\
X_3&=\bigl\{(0,0,x_3,0):x_3\in[0,\iy)\bigr\}, &
X_4&=\bigl\{(0,0,0,x_4):x_4\in[0,\iy)\bigr\},
\end{align*}
all meeting at the vertex $(0,0,0,0)\in X$. In a 3-manifold with (ordinary) corners such as $[0,\iy)^3$, three 2-dimensional boundary faces and three 1-dimensional edges meet at each vertex, so $X$ has an exotic corner structure at~$(0,0,0,0)$.
\label{gc1ex1}
\end{ex}

Most of the important differential geometry of manifolds with corners extends to manifolds with g-corners, and in some respects manifolds with g-corners are better behaved than manifolds with corners. In particular, for manifolds with corners, {\it transverse fibre products\/} $X\t_{g,Z,h}Y$ in $\Manc$ exist only under restrictive combinatorial conditions on the boundary strata $\pd^jX,\pd^kY,\pd^lZ$, but for manifolds with g-corners, transverse fibre products $X\t_{g,Z,h}Y$ in $\Mangc$ exist under much milder assumptions. One can in fact regard $\Mangc$ as being a kind of closure of $\Manc$ under a certain class of transverse fibre products.

The author's motivation for introducing manifolds with g-corners concerns eventual applications in symplectic geometry. As we explain in \S\ref{gc44}, {\it Kuranishi spaces\/} are a geometric structure on moduli spaces of $J$-holomorphic curves in symplectic geometry, introduced by Fukaya, Oh, Ohta and Ono \cite{FOOO,FuOn}. Finding a good definition of Kuranishi space has a problem from the outset. Recently the author gave a new definition \cite{Joyc5}, and explained that Kuranishi spaces should be interpreted as {\it derived smooth orbifolds}, where `derived' is in the sense of the Derived Algebraic Geometry of Jacob Lurie and To\"en--Vezzosi.

Given a suitable category of manifolds, such as manifolds without boundary $\Man$ or manifolds with corners $\Manc$, the author \cite{Joyc5} defines a 2-category of Kuranishi spaces $\Kur$ or Kuranishi spaces with corners $\Kurc$ containing $\Man\subset\Kur$ and $\Manc\subset\Kurc$ as full (2-)subcategories. Beginning with manifolds with g-corners, the same construction yields a 2-category $\Kurgc$ of {\it Kuranishi spaces with g-corners\/} $\Kurgc$ with full (2-)subcategories $\Kur\subset\Kurc\subset\Kurgc$ and~$\Man\subset\Manc\subset\Mangc\subset\Kurgc$.

For some applications the author is planning, it will be important to work in $\Kurgc$ rather than $\Kurc$. One reason is that fibre products in $\Kurgc$ exist under milder conditions than in $\Kurc$ (basically, some fibre products in $\Kurc$ ought to be Kuranishi spaces with g-corners rather than ordinary corners, and so exist in $\Kurgc$ but not in $\Kurc$) and this is needed in some constructions. 

A second reason is that some classes of moduli spaces of $J$-holomorphic curves will be Kuranishi spaces with g-corners rather than ordinary corners. Ma'u, Wehrheim and Woodward \cite{Mau,MaWo,WeWo1,WeWo2,WeWo3}, study moduli spaces of {\it pseudoholomorphic quilts}, which are used to define actions of Lagrangian correspondences on Lagrangian Floer cohomology and Fukaya categories. 

Ma'u and Woodward \cite{MaWo} define moduli spaces $\oM_{n,1}$ of `stable $n$-marked quilted discs'. As in \cite[\S 6]{MaWo}, for $n\ge 4$ these are not manifolds with corners, but have an exotic corner structure; in the language of this paper, the $\oM_{n,1}$ are manifolds with g-corners. More generally, one should expect moduli spaces of stable marked quilted $J$-holomorphic curves to be Kuranishi spaces with g-corners. Pardon \cite{Pard} uses moduli spaces of $J$-holomorphic curves with g-corners to define contact homology of Legendrian submanifolds.

Manifolds with g-corners may also occur in moduli problems elsewhere in geometry. Work of Chris Kottke (private communication) suggests that natural compactifications of $\mathop{\rm SU}(2)$ magnetic monopole spaces may have the structure of manifolds with g-corners.

In \cite{Joyc6} the author defines `M-homology', a new homology theory $MH_*(Y;R)$ of a manifold $Y$ and a commutative ring $R$, canonically isomorphic to ordinary homology $H_*(Y;R)$. The chains $MC_k(Y;R)$ for $MH_*(Y;R)$ are $R$-modules generated by quadruples $[V,n,s,t]$ for $V$ an oriented manifold with corners (or something similar) with $\dim V=n+k$ and $s:V\ra\R^n$, $t:V\ra Y$ smooth maps with $s$ proper near 0 in $\R^n$. In future work the author will define virtual chains for Kuranishi spaces in M-homology, for applications in symplectic geometry. The set-up of \cite{Joyc6} allows $V$ to be a manifold with g-corners.

The inspiration for this paper came from two main sources. Firstly, Kottke and Melrose \cite[\S 9]{KoMe} define {\it interior binomial varieties} $X\subset Y$, which in our language are a manifold with g-corners $X$ embedded as a submanifold of a manifold with corners $Y$. They study transverse fibre products $W=X\t_{g,Z,h}Y$ in $\Manc$, and observe that often the fibre product may not exist as a manifold with corners, but still makes sense as an interior binomial variety~$W\subset X\t Y$. 

For Kottke and Melrose, the exotic corners of interior binomial varieties are a problem to be eliminated, and one of their main results \cite[\S 10]{KoMe} in our language is essentially an algorithm to repeatedly blow up a manifold with g-corners (interior binomial variety) $X$ at its corner strata to obtain a manifold with corners $\ti X$. In contrast, we embrace manifolds with g-corners as an attractive new idea, which are just as good as manifolds with corners for many purposes. It seems clear from \cite{KoMe} that Kottke and Melrose could have written a paper similar to this one, had they wanted to.

Kottke \cite{Kott} translates the results of \cite{KoMe} into our language of manifolds with g-corners and extends them, explaining how (after making some discrete choices) to blow up a manifold with g-corners $X$ to get a manifold with corners $\ti X$ with a proper, surjective blow-down map $\pi:\ti X\ra X$ satisfying a universal property, and that such blow-ups pull back by interior maps $f:X_1\ra X_2$ in~$\Mangc$.

Secondly, as part of a project to generalize logarithmic geometry in algebraic geometry, Gillam and Molcho \cite[\S 6]{GiMo} define a category of {\it positive log differentiable spaces}, singular differential-geometric spaces with good notions of boundary and corners. In their setting, manifolds with g-corners (or manifolds with corners) correspond to positive log differentiable spaces which are log smooth (or log smooth with free log structure). Their morphisms correspond to our interior maps. Motivated by \cite{GiMo}, the author learnt a lot of useful material on monoids and log smoothness from the literature on logarithmic geometry, in particular Ogus \cite{Ogus}, Gillam \cite{Gill}, Kazuya Kato \cite{Kato3,Kato4} and Fumiharo Kato~\cite{Kato1,Kato2}.
 
We begin in \S\ref{gc2} with background material on manifolds with corners. The category $\Mangc$ of manifolds with g-corners is defined in \S\ref{gc3}. Section \ref{gc4} studies the differential geometry of manifolds with g-corners, including immersions, embeddings, submanifolds, and existence of fibre products under suitable transversality conditions. Longer proofs of theorems in \S\ref{gc4} are postponed to~\S\ref{gc5}.
\medskip

\noindent{\it Acknowledgements.} I would like to thank Lino Amorim and Chris Kottke for helpful conversations, Paul Seidel for pointing out the references \cite{Mau,MaWo,WeWo1,WeWo2,WeWo3}, and a referee for helpful comments. This research was supported by EPSRC grants EP/H035303/1 and EP/J016950/1.

\section{Manifolds with corners}
\label{gc2}

We discuss the category of {\it manifolds with corners}, spaces locally modelled on $\R^n_k=[0,\iy)^k\t\R^{n-k}$ for $0\le k\le n$. Some references are Melrose \cite{Melr2,Melr3,Melr4} and the author \cite{Joyc1}, \cite[\S 5]{Joyc2}, \cite[\S 3.1--\S 3.3]{Joyc5}.

\subsection{The definition of manifolds with corners}
\label{gc21}

We now define the category $\Manc$ of manifolds with corners. The relation of our definitions to other definitions in the literature is explained in Remark~\ref{gc2rem1}.

\begin{dfn} Use the notation $\R^m_k=[0,\iy)^k\t\R^{m-k}$
for $0\le k\le m$, and write points of $\R^m_k$ as $u=(u_1,\ldots,u_m)$ for $u_1,\ldots,u_k\in[0,\iy)$, $u_{k+1},\ldots,u_m\in\R$. Let $U\subseteq\R^m_k$ and $V\subseteq \R^n_l$ be open, and $f=(f_1,\ldots,f_n):U\ra V$ be a continuous map, so that $f_j=f_j(u_1,\ldots,u_m)$ maps $U\ra[0,\iy)$ for $j=1,\ldots,l$ and $U\ra\R$ for $j=l+1,\ldots,n$. Then we say:
\begin{itemize}
\setlength{\itemsep}{0pt}
\setlength{\parsep}{0pt}
\item[(a)] $f$ is {\it weakly smooth\/} if all derivatives $\frac{\pd^{a_1+\cdots+a_m}}{\pd u_1^{a_1}\cdots\pd u_m^{a_m}}f_j(u_1,\ldots,u_m):U\ra\R$ exist and are continuous in for all $j=1,\ldots,m$ and $a_1,\ldots,a_m\ge 0$, including one-sided derivatives where $u_i=0$ for $i=1,\ldots,k$.

By Seeley's Extension Theorem, this is equivalent to requiring $f_j$ to extend to a smooth function $f_j':U'\ra\R$ on open neighbourhood $U'$ of $U$ in~$\R^m$.
\item[(b)] $f$ is {\it smooth\/} if it is weakly smooth and every $u=(u_1,\ldots,u_m)\in U$ has an open neighbourhood $\ti U$ in $U$ such that for each $j=1,\ldots,l$, either:
\begin{itemize}
\setlength{\itemsep}{0pt}
\setlength{\parsep}{0pt}
\item[(i)] we may uniquely write $f_j(\ti u_1,\ldots,\ti u_m)=F_j(\ti u_1,\ldots,\ti u_m)\cdot\ti u_1^{a_{1,j}}\cdots\ti u_k^{a_{k,j}}$ for all $(\ti u_1,\ldots,\ti u_m)\in\ti U$, where $F_j:\ti U\ra(0,\iy)$ is weakly smooth and $a_{1,j},\ldots,a_{k,j}\in\N=\{0,1,2,\ldots\}$, with $a_{i,j}=0$ if $u_i\ne 0$; or 
\item[(ii)] $f_j\vert_{\smash{\ti U}}=0$.
\end{itemize}
\item[(c)] $f$ is {\it interior\/} if it is smooth, and case (b)(ii) does not occur.
\item[(d)] $f$ is {\it b-normal\/} if it is interior, and in case (b)(i), for each $i=1,\ldots,k$ we have $a_{i,j}>0$ for at most one $j=1,\ldots,l$.
\item[(e)] $f$ is {\it strongly smooth\/} if it is smooth, and in case (b)(i), for each $j=1,\ldots,l$ we have $a_{i,j}=1$ for at most one $i=1,\ldots,k$, and $a_{i,j}=0$ otherwise. 
\item[(f)] $f$ is {\it simple\/} if it is interior, and in case (b)(i), for each $i=1,\ldots,k$ with $u_i=0$ we have $a_{i,j}=1$ for exactly one $j=1,\ldots,l$ and $a_{i,j}=0$ otherwise, and for all $j=1,\ldots,l$ we have $a_{i,j}=1$ for at most one $i=1,\ldots,k$. 

Simple maps are strongly smooth and b-normal. 
\item[(g)] $f$ is a {\it diffeomorphism\/} if it is a bijection, and both $f:U\ra V$ and $f^{-1}:V\ra U$ are weakly smooth.

This implies that $f,f^{-1}$ are also smooth, interior, b-normal, strongly smooth, and simple. Hence, all the different definitions of smooth maps of manifolds with corners we discuss yield the same notion of diffeomorphism. 

\end{itemize}

All seven of these classes of maps $f:U\ra V$ include identities, and are closed under compositions from $f:U\ra V$, $g:V\ra W$ to $g\ci f:U\ra W$. Thus, each of them makes the open subsets $U\subseteq\R^m_k$ for all $m,k$ into a category.

\label{gc2def1}
\end{dfn}

\begin{dfn} Let $X$ be a second countable Hausdorff topological space. An {\it $m$-dimensional chart on\/} $X$ is a pair $(U,\phi)$, where
$U\subseteq\R^m_k$ is open for some $0\le k\le m$, and $\phi:U\ra X$ is a
homeomorphism with an open set~$\phi(U)\subseteq X$.

Let $(U,\phi),(V,\psi)$ be $m$-dimensional charts on $X$. We call
$(U,\phi)$ and $(V,\psi)$ {\it compatible\/} if
$\psi^{-1}\ci\phi:\phi^{-1}\bigl(\phi(U)\cap\psi(V)\bigr)\ra
\psi^{-1}\bigl(\phi(U)\cap\psi(V)\bigr)$ is a diffeomorphism between open subsets of $\R^m_k,\R^m_l$, in the sense of Definition \ref{gc2def1}(g).

An $m$-{\it dimensional atlas\/} for $X$ is a system
$\{(U_a,\phi_a):a\in A\}$ of pairwise compatible $m$-dimensional
charts on $X$ with $X=\bigcup_{a\in A}\phi_a(U_a)$. We call such an
atlas {\it maximal\/} if it is not a proper subset of any other
atlas. Any atlas $\{(U_a,\phi_a):a\in A\}$ is contained in a unique
maximal atlas, the set of all charts $(U,\phi)$ of this type on $X$
which are compatible with $(U_a,\phi_a)$ for all~$a\in A$.

An $m$-{\it dimensional manifold with corners\/} is a second
countable Hausdorff topological space $X$ equipped with a maximal
$m$-dimensional atlas. Usually we refer to $X$ as the manifold,
leaving the atlas implicit, and by a {\it chart\/ $(U,\phi)$ on\/}
$X$, we mean an element of the maximal atlas.

Now let $X,Y$ be manifolds with corners of dimensions $m,n$, and $f:X\ra Y$ a continuous map. We call $f$ {\it weakly smooth}, or {\it smooth}, or {\it interior}, or {\it b-normal}, or {\it strongly smooth}, or {\it simple}, if whenever $(U,\phi),(V,\psi)$ are charts on $X,Y$ with $U\subseteq\R^m_k$, $V\subseteq\R^n_l$ open, then
\e
\psi^{-1}\ci f\ci\phi:(f\ci\phi)^{-1}(\psi(V))\longra V
\label{gc2eq1}
\e
is weakly smooth, or smooth, or interior, or b-normal, or strongly smooth, or simple, respectively, as maps between open subsets of $\R^m_k,\R^n_l$ in the sense of Definition \ref{gc2def1}. It is sufficient to check this on any collections of charts $(U_a,\phi_a)_{a\in A}$ covering $X$ and $(V_b,\psi_b)_{b\in B}$ covering~$Y$.

We call $f:X\ra Y$ a {\it diffeomorphism\/} if $f$ is a bijection and $f:X\ra Y$, $f^{-1}:Y\ra X$ are weakly smooth. This implies that $f,f^{-1}$ are also smooth, interior, strongly smooth, and simple.

These seven classes of (a) weakly smooth maps, (b) smooth maps, (c) interior maps, (d) b-normal maps, (e) strongly smooth maps, (f) simple maps, and (g) diffeomorphisms, of manifolds with corners, all contain identities and are closed under composition, so each makes manifolds with corners into a category. 

In this paper, we work with smooth maps of manifolds with corners (as we have defined them), and we write $\Manc$ for the category with objects manifolds with corners $X,Y,$ and morphisms smooth maps $f:X\ra Y$ in the sense above. 

We will also write $\Mancin,\Mancst,\Mancis,\Mancsi$ for the subcategories of $\Manc$ with morphisms interior maps, and strongly smooth maps, and interior strongly smooth maps, and simple maps, respectively.

Write $\cManc$ for the category whose objects are disjoint unions $\coprod_{m=0}^\iy X_m$, where $X_m$ is a manifold with corners of dimension $m$, allowing $X_m=\es$, and whose morphisms are continuous maps $f:\coprod_{m=0}^\iy X_m\ra\coprod_{n=0}^\iy Y_n$, such that
$f\vert_{X_m\cap f^{-1}(Y_n)}:X_m\cap f^{-1}(Y_n)\ra
Y_n$ is a smooth map of manifolds with corners for all $m,n\ge 0$.
Objects of $\cManc$ will be called {\it manifolds with corners of
mixed dimension}. We regard $\Manc$ as a full subcategory of~$\cManc$.

Alternatively, we can regard $\cManc$ as the category defined exactly as for $\Manc$ above, except that in defining atlases $\{(U_a,\phi_a):a\in A\}$ on $X$, we omit the condition that all charts $(U_a,\phi_a)$ in the atlas must have the same dimension~$\dim U_a=m$.

We will also write $\cMancin,\cMancst,\cMancis,\cMancsi$ for the subcategories of $\cManc$ with the same objects, and morphisms interior, or strongly smooth, or interior strongly smooth, or simple maps, respectively.
\label{gc2def2}
\end{dfn}

\begin{ex}{\bf(i)} $f:\R\ra[0,\iy)$, $f(x)=x^2$ is weakly
smooth but not smooth.

\smallskip

\noindent{\bf(ii)} $f:\R\ra[0,\iy)$, $f(x)=x^2+1$ is strongly smooth and
interior.
\smallskip

\noindent{\bf(iii)} $f:[0,\iy)\ra[0,\iy)$, $f(x)=x^2$ is interior,
but not strongly smooth.
\smallskip

\noindent{\bf(iv)} $f:*\ra[0,\iy)$, $f(*)=0$ is strongly smooth but not
interior.
\smallskip

\noindent{\bf(v)} $f:*\ra[0,\iy)$, $f(*)=1$ is strongly smooth and interior.
\smallskip

\noindent{\bf(vi)} $f:[0,\iy)^2\ra[0,\iy)$, $f(x,y)=x+y$ is weakly
smooth, but not smooth.
\smallskip

\noindent{\bf(vii)} $f:[0,\iy)^2\ra[0,\iy)$, $f(x,y)=xy$ is
interior, but not strongly smooth.
\label{gc2ex1}
\end{ex}

\begin{rem} Some references on manifolds with corners are Cerf \cite{Cerf}, Douady \cite{Doua}, Gillam and Molcho \cite[\S 6.7]{GiMo}, Kottke and Melrose \cite{KoMe}, Margalef-Roig and Outerelo Dominguez \cite{MaOu}, Melrose \cite{Melr2,Melr3,Melr4}, Monthubert \cite{Mont}, and the author \cite{Joyc1}, \cite[\S 5]{Joyc2}. Just as objects, without considering morphisms, most authors define manifolds with corners as in Definition \ref{gc2def2}. However, Melrose \cite{KoMe,Melr1,Melr2,Melr3,Melr4} and authors who follow him impose an extra condition: in \S\ref{gc22} we will define the boundary $\pd X$ of a manifold with corners $X$, with an immersion $i_X:\pd X\ra X$. Melrose requires that $i_X\vert_C:C\ra X$ should be injective for each connected component $C$ of $\pd X$ (such $X$ are sometimes called {\it manifolds with faces\/}).

There is no general agreement in the literature on how to define smooth maps, or morphisms, of manifolds with corners: 
\begin{itemize}
\setlength{\itemsep}{0pt}
\setlength{\parsep}{0pt}
\item[(i)] Our notion of `smooth map' in Definitions \ref{gc2def1} and \ref{gc2def2} is due to Melrose \cite[\S 1.12]{Melr3}, \cite[\S 1]{Melr1}, \cite[\S 1]{KoMe}, who calls them {\it b-maps}. 

Our notation of `interior maps' and `b-normal maps' is also due to Melrose.
\item[(ii)] Monthubert's {\it morphisms of manifolds with corners\/} \cite[Def.~2.8]{Mont} coincide with our strongly smooth b-normal maps. 
\item[(iii)] The author \cite{Joyc1} defined and studied `strongly smooth maps' above (which were just called `smooth maps' in \cite{Joyc1}). 

Strongly smooth maps were also used to define {\it d-manifolds with corners\/} in the 2012 version of \cite{Joyc2}. However, the final version of \cite{Joyc2} will have a different definition using smooth maps (i.e.~Melrose's b-maps).
\item[(iv)] Gillam and Molcho's {\it morphisms of manifolds with corners\/} \cite[\S 6.7]{GiMo} coincide with our `interior maps'.
\item[(v)] Most other authors, such as Cerf \cite[\S I.1.2]{Cerf}, define smooth maps of manifolds with corners to be weakly smooth maps, in our notation.
\end{itemize}
\label{gc2rem1}
\end{rem}

\begin{rem} We can also define {\it real analytic\/} manifolds with corners, and {\it real analytic\/} maps between them. To do this, if $U\subseteq\R^m_k$ and $V\subseteq \R^n_l$ are open, we define a smooth map $f=(f_1,\ldots,f_n):U\ra V$ in Definition \ref{gc2rem1} to be {\it real analytic\/} if each map $f_i:U\ra\R$ for $i=1,\ldots,n$ is of the form $f_i=f_i'\vert_U$, for $U'$ an open neighbourhood of $U$ in $\R^m$ and $f_i':U'\ra\R$ real analytic in the usual sense (i.e.\ the Taylor series of $f_i'$ at $x$ converges to $f_i'$ near $x$ for each $x\in U'$).

Then we define $\{(U_a,\phi_a):a\in A\}$ to be a {\it real analytic atlas\/} on a topological space $X$ as in Definition \ref{gc2def2}, except that the transition functions $\phi_b^{-1}\ci\phi_a$ are required to be real analytic rather than just smooth. We define a {\it real analytic manifold with corners\/} to be a Hausdorff, second countable topological space $X$ equipped with a maximal real analytic atlas. 

Given real analytic manifolds with corners $X,Y$, we define a continuous map $f:X\ra Y$ to be {\it real analytic\/} if whenever $(U,\phi),(V,\psi)$ are real analytic charts on $X,Y$ (that is, charts in the maximal real analytic atlases), the transition map $\psi^{-1}\ci f\ci\phi$ in \eq{gc2eq1} is a real analytic map between open subsets of $\R^m_k,\R^n_l$ in the sense above. Then real analytic manifolds with corners and real analytic maps between them form a category~$\Mancra$.

There is an obvious faithful functor $F_\Mancra^\Manc:\Mancra\ra\Manc$, which on objects replaces the maximal real analytic atlas by the (larger) corresponding maximal smooth atlas containing it. Note that given a smooth manifold with corners $X$, making $X$ into a real analytic manifold with corners is an additional structure on $X$, a refinement of the maximal smooth atlas on $X$, which can be done in many ways. So $F_\Mancra^\Manc$ is far from injective on objects. Essentially all the material we discuss for manifolds with corners also works for real analytic manifolds with corners, except for constructions requiring partitions of unity.
\label{gc2rem2}
\end{rem}

\subsection{Boundaries and corners of manifolds with corners}
\label{gc22}

The material of this section broadly follows the author \cite{Joyc1}, \cite[\S 5]{Joyc2}.

\begin{dfn} Let $U\subseteq\R^n_l$ be open. For each
$u=(u_1,\ldots,u_n)$ in $U$, define the {\it depth\/} $\depth_Uu$ of
$u$ in $U$ to be the number of $u_1,\ldots,u_l$ which are zero. That
is, $\depth_Uu$ is the number of boundary faces of $U$
containing~$u$.

Let $X$ be an $n$-manifold with corners. For $x\in X$, choose a
chart $(U,\phi)$ on the manifold $X$ with $\phi(u)=x$ for $u\in U$,
and define the {\it depth\/} $\depth_Xx$ of $x$ in $X$ by
$\depth_Xx=\depth_Uu$. This is independent of the choice of
$(U,\phi)$. For each $l=0,\ldots,n$, define the {\it depth\/ $l$
stratum\/} of $X$ to be
\begin{equation*}
S^l(X)=\bigl\{x\in X:\depth_Xx=l\bigr\}.
\end{equation*}
Then $X=\coprod_{l=0}^nS^l(X)$ and $\overline{S^l(X)}=
\bigcup_{k=l}^n S^k(X)$. The {\it interior\/} of $X$ is
$X^\ci=S^0(X)$. Each $S^l(X)$ has the structure of an
$(n-l)$-manifold without boundary.

\label{gc2def3}
\end{dfn}

\begin{dfn} Let $X$ be an $n$-manifold with corners, $x\in X$, and $k=0,1,\ldots,n$. A {\it local $k$-corner component\/ $\ga$ of\/ $X$ at\/} $x$ is a local choice of connected component of $S^k(X)$ near $x$. That is, for
each sufficiently small open neighbourhood $V$ of $x$ in $X$, $\ga$
gives a choice of connected component $W$ of $V\cap S^k(X)$ with
$x\in\overline W$, and any two such choices $V,W$ and $V',W'$ must
be compatible in that~$x\in\overline{(W\cap W')}$.

Let $\depth_Xx=l$. Choose a chart $(U,\phi)$ on $X$ with $(0,\ldots,0)\in U\subseteq\R^n_l$ open and $\phi(0,\ldots,0)=x$. Then we have
\e
\begin{split}
S^k(U)=\coprod_{1\le a_1<a_2<\cdots<a_k\le l}\begin{aligned}[t]
\bigl\{(u_1&,\ldots,u_n)\in U:u_{a_i}=0,\;\> i=1,\ldots,k,\\
&u_j\ne 0,\;\> j\in \{1,\ldots,l\}\sm\{a_1,\ldots,a_k\}\bigr\}.
\end{aligned}
\end{split}
\label{gc2eq2}
\e
For each choice of $a_1,\ldots,a_k$, the subset on the right hand of \eq{gc2eq2} contains $(0,\ldots,0)$ in its closure in $U$, and its intersection with a small ball about $(0,\ldots,0)$ is connected. Thus this subset determines a local $k$-corner component of $U$ at $(0,\ldots,0)$, and hence a local $k$-corner component of $X$ at~$x$. 

Equation \eq{gc2eq2} implies that all local $k$-corner components of $U$ at $(0,\ldots,0)$ and $X$ at $x$ are of this form. Therefore, local $k$-corner components of $U\subseteq\R^n_l$ at $(0,\ldots,0)$ are in 1-1 correspondence with subsets $\{a_1,\ldots,a_k\}\subseteq\{1,\ldots,l\}$ of size $k$, and there are $\binom{\depth_Xx}{k}$ distinct local $k$-corner components of $X$ at~$x$.

When $k=1$, we also call local 1-corner components {\it local boundary components of\/ $X$ at\/} $x$. There are $\depth_Xx$ distinct local boundary components of $X$ at $x$. By considering the local model $\R^n_l$, it is easy to see that there is a natural 1-1 correspondence between local $k$-corner components $\ga$ of $X$ at $x$, and (unordered) sets $\{\be_1,\ldots,\be_k\}$ of $k$ distinct local boundary components $\be_1,\ldots,\be_k$ of $X$ at $x$, such that if $V$ is a sufficiently small open neighbourhood of $x$ in $X$ and $\be_1,\ldots,\be_k$ and $\ga$ give connected components $W_1,\ldots,W_k$ of $V\cap S^1(X)$ and $W'$ of $V\cap S^k(X)$, then $W'\subseteq\bigcap_{i=1}^k\kern .1em\overline{\kern -.1em W}_{\!i}$.

As sets, define the {\it boundary\/} $\pd X$ and {\it k-corners\/} $C_k(X)$ for $k=0,1,\ldots,n$ by
\ea
\pd X&=\bigl\{(x,\be):\text{$x\in X$, $\be$ is a local boundary
component of $X$ at $x$}\bigr\},
\label{gc2eq3}
\\
C_k(X)&=\bigl\{(x,\ga):\text{$x\in X$, $\ga$ is a local $k$-corner 
component of $X$ at $x$}\bigr\},
\label{gc2eq4}
\ea
so that $\pd X=C_1(X)$. The 1-1 correspondence above shows that
\e
\begin{aligned}
C_k(X)\cong\bigl\{(x,\,&\{\be_1,\ldots,\be_k\}):\text{$x\in X,$
$\be_1,\ldots,\be_k$ are distinct}\\
&\text{local boundary components for $X$ at $x$}\bigr\}.
\end{aligned}
\label{gc2eq5}
\e
Since each $x\in X$ has a unique 0-boundary component, we have~$C_0(X)\cong X$.

If $(U,\phi)$ is a chart on $X$ with $U\subseteq\R^n_l$ open, then
for each $i=1,\ldots,l$ we can define a chart $(U_i,\phi_i)$ on $\pd
X$ by
\begin{align*}
&U_i=\bigl\{(v_1,\ldots,v_{n-1})\in\R^{n-1}_{l-1}:
(v_1,\ldots,v_{i-1},0,v_i,\ldots,v_{n-1})\in
U\subseteq\R^n_l\bigr\},\\
&\phi_i:(v_1,\ldots,v_{n-1})\longmapsto\bigl(\phi
(v_1,\ldots,v_{i-1},0,v_i,\ldots,v_{n-1}),\phi_*(\{u_i=0\})\bigr).
\end{align*}
Similarly, if $0\le k\le l$, then for each $1\le a_1<\cdots<a_k\le l$ we can define a chart $(U_{\{a_1,\ldots,a_k\}},\phi_{\{a_1,\ldots,a_k\}})$ on $C_k(X)$ by
\e
\begin{split}
&U_{\{a_1,\ldots,a_k\}}\!=\!\bigl\{(v_1,\ldots,v_{n-k})\!\in\!\R^{n-k}_{l-k}:
(v_1,\ldots,v_{a_1-1},0,v_{a_1},\ldots,v_{a_2-2},0,\\
&\quad v_{a_2-1},\ldots,v_{a_3-3},0,v_{a_3-2},\ldots,v_{a_k-k},0,v_{a_k-k+1},\ldots,v_{n-k})\!\in\!
U\subseteq\R^n_l\bigr\},\\
&\phi_{\{a_1,\ldots,a_k\}}:(v_1,\ldots,v_{n-k})\longmapsto\bigl( 
\phi(v_1,\ldots,v_{a_1-1},0,v_{a_1},\ldots,v_{a_2-2},0,\\
&\quad v_{a_2-1},\ldots,v_{a_3-3},0,v_{a_3-2},\ldots,v_{a_k-k},0,v_{a_k-k+1},\ldots,v_{n-k}),\\
&\qquad\qquad\qquad\qquad \phi_*(\{u_{a_1}=\cdots=u_{a_k}=0\})\bigr).
\end{split}
\label{gc2eq6}
\e
The families of all such charts on $\pd X$ and $C_k(X)$ are pairwise compatible, and define atlases on $\pd X$ and $C_k(X)$. The corresponding maximal atlases make $\pd X$ into an $(n-1)$-manifold with corners and $C_k(X)$ into an $(n-k)$-manifold with corners, with $\pd X=C_1(X)$ and $C_0(X)\cong X$ as manifolds with corners.

We call $X$ a {\it manifold without boundary\/} if $\pd X=\es$, and
a {\it manifold with boundary\/} if $\pd^2X=\es$. We write $\Man$ and $\Manb$ for the full subcategories of $\Manc$ with objects manifolds without boundary, and manifolds with boundary, so that $\Man\subset\Manb\subset\Manc$. This definition of $\Man$ is equivalent to the usual definition of the category of manifolds.

Define maps $i_X:\pd X\ra X$, $\Pi:C_k(X)\ra X$ and $\io:X\ra C_0(X)$ by $i_X:(x,\be)\mapsto x$, $\Pi:(x,\ga)\mapsto x$ and $\io:x\mapsto(x,[X^\ci])$. Considering local models, we see that $i_X,\Pi,\io$ are are (strongly) smooth, but $i_X,\Pi$ are not interior. Note that these maps $i_X,\Pi$ {\it may not be injective}, since the preimage of $x\in X$ is $\depth_Xx$ points in $\pd X$ and $\binom{\depth_Xx}{k}$ points in $C_k(X)$. So we cannot regard $\pd X$ and $C_k(X)$ as subsets of $X$.
\label{gc2def4}
\end{dfn}

\begin{ex} The {\it teardrop\/} $T=\bigl\{(x,y)\in\R^2:x\ge 0$, $y^2\le x^2-x^4\bigr\}$, shown in Figure \ref{gc2fig1}, is a manifold with corners of dimension 2. The boundary $\pd T$ is diffeomorphic to $[0,1]$, and so is connected, but $i_T:\pd T\ra T$ is not injective. Thus $T$ is not a manifold with faces, in the sense of Remark \ref{gc2rem1}.
\begin{figure}[htb]
\begin{xy}
,(-1.5,0)*{}
,<6cm,-1.5cm>;<6.7cm,-1.5cm>:
,(3,.3)*{x}
,(-1.2,2)*{y}
,(-1.5,0)*{\bullet}
,(-1.5,0); (1.5,0) **\crv{(-.5,1)&(.1,1.4)&(1.5,1.2)}
?(.06)="a"
?(.12)="b"
?(.2)="c"
?(.29)="d"
?(.4)="e"
?(.5)="f"
?(.6)="g"
?(.7)="h"
?(.83)="i"
,(-1.5,0); (1.5,0) **\crv{(-.5,-1)&(.1,-1.4)&(1.5,-1.2)}
?(.06)="j"
?(.12)="k"
?(.2)="l"
?(.29)="m"
?(.4)="n"
?(.5)="o"
?(.6)="p"
?(.7)="q"
?(.83)="r"
,"a";"j"**@{.}
,"b";"k"**@{.}
,"c";"l"**@{.}
,"d";"m"**@{.}
,"e";"n"**@{.}
,"f";"o"**@{.}
,"g";"p"**@{.}
,"h";"q"**@{.}
,"i";"r"**@{.}
\ar (-1.5,0);(3,0)
\ar (-1.5,0);(-3,0)
\ar (-1.5,0);(-1.5,2)
\ar (-1.5,0);(-1.5,-2)
\end{xy}
\caption{The teardrop, a 2-manifold with corners}
\label{gc2fig1}
\end{figure}
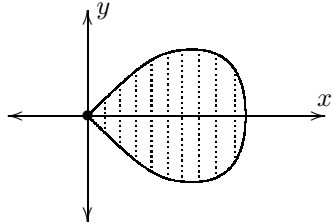

\label{gc2ex2}
\end{ex}

If $X$ is an $n$-manifold with corners, we can take boundaries repeatedly to get manifolds with corners $\pd X,\pd^2X=\pd(\pd X),\pd^3X,\ldots,\pd^nX$. To relate these to the corners $C_k(X)$, note that by considering local models $U\subseteq\R^n_l$, it is easy to see that there is a natural 1-1 correspondence
\begin{align*}
&\bigl\{\text{local boundary components of $\pd X$ at $(x,\be)$}\bigr\}\cong\\
&\bigl\{\text{local boundary components $\be'$ of $X$ at $x$ with $\be'\ne\be$}\bigr\}.
\end{align*}
Using this and induction, we can show that there is a natural identification
\e
\begin{split}
\pd^kX\cong\bigl\{(x,\be_1,\ldots,\be_k):\,&\text{$x\in X,$
$\be_1,\ldots,\be_k$ are distinct}\\
&\text{local boundary components for $X$ at $x$}\bigr\},
\end{split}
\label{gc2eq7}
\e
where under the identifications \eq{gc2eq7}, the map $i_{\pd^{k-1}X}:\pd^kX\ra\pd^{k-1}X$ maps $(x,\be_1,\ldots,\be_k)\mapsto(x,\be_1,\ldots,\be_{k-1})$. From \eq{gc2eq7}, we see that there is a natural, free action of the symmetric group $S_k$ on $\pd^kX$, by permutation of $\be_1,\ldots,\be_k$. The action is by diffeomorphisms, so the quotient $\pd^kX/S_k$ is also a manifold with corners. Dividing by $S_k$ turns the ordered $k$-tuple $\be_1,\ldots,\be_k$ into an unordered set $\{\be_1,\ldots,\be_k\}$. So from \eq{gc2eq5}, we see that there is a natural diffeomorphism
\e
C_k(X)\cong \pd^kX/S_k.
\label{gc2eq8}
\e

Corners commute with boundaries: there are natural isomorphisms
\e
\begin{aligned}
\pd C_k(X)\cong C_k&(\pd X)\cong \bigl\{(x,\{\be_1,\ldots,
\be_k\},\be_{k+1}):x\in X,\; \be_1,\ldots,\be_{k+1}\\
&\text{are distinct local boundary components for $X$ at
$x$}\bigr\}.
\end{aligned}
\label{gc2eq9}
\e

Products $X\t Y$ of manifolds with corners are defined in the
obvious way. Boundaries and corners of products $X\t Y$ behave well.
It is easy to see that there is a natural identification
\begin{align*}
&\bigl\{\text{local boundary components for $X\t Y$ at
$(x,y)$}\bigr\} \cong \\
&\qquad\qquad\bigl\{\text{local boundary components for $X$ at
$x$}\bigr\}\amalg{}\\
&\qquad\qquad \bigl\{\text{local boundary components for $Y$ at
$y$}\bigr\}.
\end{align*}
Using this, from \eq{gc2eq3} and \eq{gc2eq5} we get natural
isomorphisms
\ea
\pd(X\t Y)&\cong (\pd X\t Y)\amalg (X\t\pd Y),
\label{gc2eq10}\\
C_k(X\t Y)&\cong \ts\coprod_{i,j\ge 0,\; i+j=k}C_i(X)\t C_j(Y).
\label{gc2eq11}
\ea

Next we consider how smooth maps $f:X\ra Y$ of manifolds with corners act on boundaries $\pd X,\pd Y$ and corners $C_k(X),C_l(Y)$. The following lemma is easy to prove from Definition \ref{gc2def1}(b). The analogue is {\it false\/} for weakly smooth maps (e.g. consider $f:\R\ra[0,\iy)$, $f(x)=x^2$, which is weakly smooth but not smooth), so the rest of the section does not work in the weakly smooth case.

\begin{lem} Let\/ $f:X\ra Y$ be a smooth map of manifolds with corners. Then $f$ \begin{bfseries}is compatible with the depth stratifications\end{bfseries} $X=\coprod_{k\ge 0}S^k(X),$ $Y=\coprod_{l\ge 0}S^l(Y)$ in Definition\/ {\rm\ref{gc2def3},} in the sense that if\/ $\es\ne W\subseteq S^k(X)$ is a connected subset for some $k\ge 0,$ then $f(W)\subseteq S^l(Y)$ for some unique $l\ge 0$.
\label{gc2lem1}
\end{lem}

It is {\it not\/} true that general smooth $f:X\ra Y$ induce maps $\pd f:\pd X\ra\pd Y$ or $C_k(f):C_k(Y)\ra C_k(Y)$ (although this does hold for {\it simple\/} maps, as in Proposition \ref{gc2prop1}(d)). For example, if $f:X\ra Y$ is the inclusion $[0,\iy)\hookra\R$ then no map $\pd f:\pd X\ra\pd Y$ exists, as $\pd X\ne\es$ and $\pd Y=\es$. So boundaries and $k$-corners do not give functors on $\Manc$. However, if we work in the enlarged category $\cManc$ of Definition \ref{gc2def2} and consider the full corners $C(X)=\coprod_{k\ge 0}C_k(X)$, we can define a functor.

\begin{dfn} Define the {\it corners\/} $C(X)$ of a manifold with corners $X$ by
\begin{align*}
&C(X)=\ts\coprod_{k=0}^{\dim X}C_k(X)\\
&=\bigl\{(x,\ga):\text{$x\in X$, $\ga$ is a local $k$-corner 
component of $X$ at $x$, $k\ge 0$}\bigr\},
\end{align*}
considered as an object of $\cManc$ in Definition \ref{gc2def2}, a manifold with corners of mixed dimension. Define $\Pi:C(X)\ra X$ by $\Pi:(x,\ga)\mapsto x$. This is smooth (i.e. a morphism in $\cManc$) as the maps $\Pi:C_k(X)\ra X$ are smooth for~$k\ge 0$.

Equations \eq{gc2eq9} and \eq{gc2eq11} imply that if $X,Y$ are manifolds with corners, we have natural isomorphisms
\e
\pd C(X)\cong C(\pd X),\quad C(X\t Y)\cong C(X)\t C(Y).
\label{gc2eq12}
\e

Let $f:X\ra Y$ be a smooth map of manifolds with corners, and suppose $\ga$ is a local $k$-corner component of $X$ at $x\in X$. For each sufficiently small open neighbourhood $V$ of $x$ in $X$, $\ga$ gives a choice of connected component $W$ of $V\cap S^k(X)$ with $x\in\overline W$, so by Lemma \ref{gc2lem1} $f(W)\subseteq S^l(Y)$ for some $l\ge 0$. As $f$ is continuous, $f(W)$ is connected, and $f(x)\in\ov{f(W)}$. Thus there is a unique local $l$-corner component $f_*(\ga)$ of $Y$ at $f(x)$, such that if $\ti V$ is a sufficiently small open neighbourhood of $f(x)$ in $Y$, then the connected component $\ti W$ of $\ti V\cap S^l(Y)$ given by $f_*(\ga)$ has $\ti W\cap f(W)\ne\es$. This $f_*(\ga)$ is independent of the choice of sufficiently small $V,\ti V$, so is well-defined.

Define a map $C(f):C(X)\ra C(Y)$ by $C(f):(x,\ga)\mapsto (f(x),f_*(\ga))$. Given charts $(U,\phi)$ on $X$ and $(V,\psi)$ on $Y$, so that \eq{gc2eq1} gives a smooth map $\psi^{-1}\ci f\ci\phi$, then in the charts $(U_{\{a_1,\ldots,a_k\}},\phi_{\{a_1,\ldots,a_k\}})$ on $C_k(X)$ and $(V_{\{b_1,\ldots,b_l\}},\psi_{\{b_1,\ldots,b_l\}})$ on $C_l(Y)$ defined from $(U,\phi)$ and $(V,\psi)$ in \eq{gc2eq6}, we see that 
\begin{align*}
\psi_{\{b_1,\ldots,b_l\}}^{-1}\ci C(f)\ci\phi_{\{a_1,\ldots,a_k\}}:(C(f)\ci\phi_{\{a_1,\ldots,a_k\}})^{-1}(\psi_{\{b_1,\ldots,b_l\}}(V_{\{b_1,\ldots,b_l\}}))&\\
\longra V_{\{b_1,\ldots,b_l\}}&
\end{align*}
is just the restriction of $\psi^{-1}\ci f\ci\phi$ to a map from a codimension $k$ boundary face of $U$ to a codimension $l$ boundary face of $V$, and so is clearly smooth in the sense of Definition \ref{gc2def1}. Since such charts $(U_{\{a_1,\ldots,a_k\}},\phi_{\{a_1,\ldots,a_k\}})$ and $(V_{\{b_1,\ldots,b_l\}},\psi_{\{b_1,\ldots,b_l\}})$ cover $C_k(X)$ and $C_l(Y)$, it follows that $C(f)$ is smooth (that is, $C(f)$ is a morphism in $\cManc$).

If $g:Y\ra Z$ is another smooth map of manifolds with corners, and $\ga$ is a local $k$-corner component of $X$ at $x$, it is easy to see that $(g\ci f)_*(\ga)=g_*\ci f_*(\ga)$ in local $m$-corner components of $Z$ at $g\ci f(x)$. Therefore $C(g\ci f)=C(g)\ci C(f):C(X)\ra C(Z)$. Clearly $C(\id_X)=\id_{C(X)}:C(X)\ra C(X)$. Hence $C:\Manc\ra\cManc$ is a functor, which we call the {\it corner functor}. We extend $C$ to
$C:\cManc\ra\cManc$ by $C(\coprod_{m\ge 0}X_m)=\coprod_{m\ge
0}C(X_m)$. 
\label{gc2def5}
\end{dfn}

The following properties of the corner functor are easy to check using the local models in Definition~\ref{gc2def1}.

\begin{prop} Let\/ $f:X\ra Y$ be a smooth map of manifolds with corners. 

\smallskip

\noindent{\bf(a)} $C(f):C(X)\ra C(Y)$ is an interior map of manifolds with corners of mixed dimension, so $C$ is a functor $C:\Manc\ra\cMancin$.
\smallskip

\noindent{\bf(b)} $f$ is interior if and only if\/ $C(f)$ maps $C_0(X)\ra C_0(Y),$ if and only if the following commutes:
\begin{equation*}
\xymatrix@C=95pt@R=13pt{ *+[r]{X} \ar[d]^\io \ar[r]_{f} &
*+[l]{Y} \ar[d]_\io \\ *+[r]{C(X)} \ar[r]^{C(f)} & *+[l]{C(Y).\!\!} }
\end{equation*}
Thus $\io:\Id\!\Ra\! C$ is a natural transformation on
$\Id,C\vert_{\Mancin}:\Mancin\!\ra\!\cMancin$.

\noindent{\bf(c)} $f$ is b-normal if and only if\/ $C(f)$ maps $C_k(X)\ra \coprod_{l=0}^kC_l(Y)$ for all\/~$k$.
\smallskip

\noindent{\bf(d)} If\/ $f$ is simple then $C(f)$ maps $C_k(X)\ra C_k(Y)$ for all\/ $k\ge 0,$ and\/ $C_k(f):=C(f)\vert_{C_k(X)}:C_k(X)\ra C_k(Y)$ is also a simple map.

Thus we have a \begin{bfseries}boundary functor\end{bfseries} $\pd:\Mancsi\ra\Mancsi$ mapping $X\mapsto\pd X$ on objects and\/ $f\mapsto \pd f:=C(f)\vert_{C_1(X)}:\pd X\ra\pd Y$ on (simple) morphisms $f:X\ra Y,$ and for all\/ $k\ge 0$ a \begin{bfseries}$k$-corner functor\end{bfseries} $C_k:\Mancsi\ra\Mancsi$ mapping $X\mapsto C_k(X)$ on objects and\/ $f\mapsto C_k(f):=C(f)\vert_{C_k(X)}:C_k(X)\ra C_k(Y)$ on (simple) morphisms.
\smallskip

\noindent{\bf(e)} The following commutes:
\begin{equation*}
\xymatrix@C=95pt@R=13pt{ *+[r]{C(X)} \ar[d]^\Pi \ar[r]_{C(f)} &
*+[l]{C(Y)} \ar[d]_\Pi \\ *+[r]{X} \ar[r]^f & *+[l]{Y.\!\!} }
\end{equation*}
Thus $\Pi:C\Ra\Id$ is a natural transformation.
\smallskip

\noindent{\bf(f)} The functor $C$ \begin{bfseries}preserves products
and direct products\end{bfseries}. That is, if\/ $f:W\ra Y,$ $g:X\ra
Y,$ $h:X\ra Z$ are smooth then the following commute
\begin{equation*}
\xymatrix@C=60pt@R=20pt{ *+[r]{C(W\t X)} \ar[d]^\cong \ar[r]_{C(f\t
h)} & *+[l]{C(Y\t Z)} \ar[d]_\cong \\ *+[r]{C(W)\!\t\! C(X)}
\ar[r]^{\raisebox{8pt}{$\st C(f) \t C(h)$}} &
*+[l]{C(Y)\!\t\! C(Z),} }\;
\xymatrix@C=65pt@R=3pt{ & C(Y\t Z) \ar[dd]^\cong \\
C(X) \ar[ur]^(0.4){C((g,h))} \ar[dr]_(0.4){(C(g),C(h))} \\
& C(Y)\!\t\! C(Z), }
\end{equation*}
where the columns are the isomorphisms \eq{gc2eq12}.
\label{gc2prop1}
\end{prop}

\begin{ex}{\bf(a)} Let $X=[0,\iy)$, $Y=[0,\iy)^2$, and define $f:X\ra Y$
by $f(x)=(x,x)$. We have
\begin{align*}
C_0(X)&\cong[0,\iy), \qquad\quad C_1(X)\cong\{0\}, & C_0(Y)&\cong[0,\iy)^2,\\
C_1(Y)&\cong\bigl(\{0\}\t [0,\iy)\bigr)\amalg \bigl([0,\iy)\t\{0\}\bigr),&
C_2(Y)&\cong\{(0,0)\}.
\end{align*}
Then $C(f)$ maps $C_0(X)\ra C_0(Y)$, $x\mapsto (x,x)$, and
$C_1(X)\ra C_2(Y)$, $0\mapsto(0,0)$.

\smallskip

\noindent{\bf(b)} Let $X=*$, $Y=[0,\iy)$ and define $f:X\ra Y$ by
$f(*)=0$. Then $C_0(X)\cong *$, $C_0(Y)\cong [0,\iy)$,
$C_1(Y)\cong\{0\}$, and $C(f)$ maps $C_0(X)\ra C_1(Y)$, $*\mapsto
0$.
\label{gc2ex3}
\end{ex}

Note that $C(f)$ need not map $C_k(X)\ra C_k(Y)$.

\subsection{Tangent bundles and b-tangent bundles}
\label{gc23}

Manifolds with corners $X$ have two notions of tangent bundle with
functorial properties, the ({\it ordinary\/}) {\it tangent bundle\/}
$TX$, the obvious generalization of tangent bundles of manifolds
without boundary, and the {\it b-tangent bundle\/} ${}^bTX$
introduced by Melrose \cite[\S 2.2]{Melr2}, \cite[\S I.10]{Melr3},
\cite[\S 2]{Melr1}. Taking duals gives two notions of cotangent
bundle $T^*X,{}^bT^*X$. First we discuss vector bundles:

\begin{dfn} Let $X$ be an $n$-manifold with corners. A {\it vector
bundle\/ $E\ra X$ of rank\/} $k$ is a manifold with corners $E$ and
a smooth (in fact strongly smooth and simple) map $\pi:E\ra X$, such that each fibre $E_x:=\pi^{-1}(x)$ for $x\in X$ is given the structure of a $k$-dimensional real vector space, and $X$ may be covered by open
subsets $U\subseteq X$ with diffeomorphisms $\pi^{-1}(U)\cong
U\t\R^k$ identifying $\pi\vert_{\pi^{-1}(U)}:\pi^{-1}(U)\ra U$ with
the projection $U\t\R^k\ra\R^k$, and the vector space structure on
$E_x$ with that on $\{x\}\t\R^k\cong\R^k$, for each $x\in U$.

A {\it section\/} of $E$ is a smooth map $s:X\ra E$ with $\pi\ci
s=\id_X$. As a map of manifolds with corners, $s:X\ra E$ is
automatically strongly smooth.

Morphisms of vector bundles, dual vector bundles, tensor products of
vector bundles, exterior products, and so on, all work as usual.

Write $C^\iy(X)$ for the $\R$-algebra of smooth functions
$f:X\ra\R$. Write $C^\iy(E)$ for
the $\R$-vector space of smooth sections $s:X\ra E$. Then $C^\iy(E)$
is a module over~$C^\iy(X)$.

Sometimes we also consider {\it vector bundles of mixed rank\/}
$E\ra X$, in which we allow the rank $k$ to vary over $X$, so that
$E$ can have different ranks on different connected components of
$X$. This happens often when working with objects
$X=\coprod_{m=0}^\iy X_m$ in the category $\cManc$ from Definition
\ref{gc2def2}, for instance, the tangent bundle $TX$ has rank $m$
over $X_m$ for each~$m$.
\label{gc2def6}
\end{dfn}

\begin{dfn} Let $X$ be an $m$-manifold with corners. The {\it tangent
bundle\/} $\pi:TX\ra X$ of $X$ is a natural (unique up to canonical
isomorphism) rank $m$ vector bundle on $X$. Here are two equivalent ways to characterize~$TX$:
\begin{itemize}
\setlength{\itemsep}{0pt}
\setlength{\parsep}{0pt}
\item[(a)] {\bf In coordinate charts:} let $(U,\phi)$ be a chart 
on $X$, with $U\subseteq\R^m_k$ open. Then over $\phi(U)$, $TX$ is the trivial vector bundle with basis of sections $\frac{\pd}{\pd u_1},\ldots,\frac{\pd}{\pd u_m}$, for $(u_1,\ldots,u_m)$ the coordinates on $U$. There is a corresponding chart $(TU,T\phi)$ on $TX$, where $TU=U\t\R^m\subseteq\R^{2m}_k$, such that $(u_1,\ldots,u_m,q_1,\ldots,q_m)\in TU$ represents the vector $q_1\frac{\pd}{\pd u_1}+\cdots+q_m\frac{\pd}{\pd u_m}$ over $(u_1,\ldots,u_m)\in U$ or $\phi(u_1,\ldots,u_m)\in X$. Under change of coordinates $(u_1,\ldots,u_m)\rightsquigarrow(\ti u_1,\ldots,\ti u_m)$ from $(U,\phi)$ to $(\ti U,\ti\phi)$, the corresponding change $(u_1,\ldots,u_m,q_1,\ldots,q_m)\rightsquigarrow(\ti u_1,\ldots,\ti u_m,\ti q_1,\ldots,\ti q_m)$ from $(TU,T\phi)$ to $(T\ti U,T\ti\phi)$ is determined by $\frac{\pd}{\pd u_i}=\sum_{j=1}^m\frac{\pd\ti u_j}{\pd u_i}(u_1,\ldots,u_m)\cdot\frac{\pd}{\pd\ti u_j}$, so that~$\ti q_j=\sum_{i=1}^m\frac{\pd\ti u_j}{\pd u_i}(u_1,\ldots,u_m)q_i$.
\item[(b)] {\bf Intrinsically in terms of germs:} For $x\in X$, write $C^\iy_x(X)$ for the set of {\it germs\/ $[a]$ at\/ $x$ of smooth functions\/ $a:X\ra\R$ defined near\/} $x\in X$. That is, elements of $C^\iy_x(X)$ are equivalence classes $[a]$ of smooth functions $a:U\ra\R$ in the sense of \S\ref{gc21}, where $U$ is an open neighbourhood of $x$ in $X$, and $a:U\ra\R,$ $a':U'\ra\R$ are equivalent if there exists an open neighbourhood $U''$ of $x$ in $U\cap U'$ with $a\vert_{U''}=a'\vert_{U''}$. Then $C^\iy_x(X)$ is a commutative $\R$-algebra, with operations $\la[a]+\mu[b]=[\la a+\mu b]$ and $[a]\cdot[b]=[a\cdot b]$ for $[a],[b]\in C^\iy_x(X)$ and $\la,\mu\in\R$. It has an {\it evaluation map\/} $\ev:C^\iy_x(X)\ra\R$ mapping $\ev:[a]\mapsto a(x)$, an $\R$-algebra morphism.

Then there is a natural isomorphism
\e
\begin{split}
&T_xX\cong\bigl\{v:\text{$v$ is a linear map $C^\iy_x(X)\ra\R$
satisfying}\\
&\;\>\text{$v([a]\!\cdot\! [b])\!=\!v([a])\ev([b])\!+\!\ev([a])v([b])$, all $[a],[b]\!\in\!
C^\iy_x(X)$}\bigr\}.
\end{split}
\label{gc2eq13}
\e
This also holds with $C^\iy(X)$ in place of $C^\iy_x(X)$.

Also there is a natural isomorphism of $C^\iy(X)$-modules
\begin{align*}
C^\iy(TX)\cong\bigl\{v:\,&\text{$v$ is a linear map $C^\iy(X)\ra
C^\iy(X)$
satisfying}\\
&\text{$v(ab)=v(a)\cdot b+a\cdot v(b)$ for all $a,b\in
C^\iy(X)$}\bigr\}.
\end{align*}
Elements of $C^\iy(TX)$ are called {\it vector fields}.
\end{itemize}

Now suppose $f:X\ra Y$ is a smooth map of manifolds with corners. We
will define a natural smooth map $Tf:TX\ra TY$ so that the following
commutes:
\begin{equation*}
\xymatrix@C=80pt@R=15pt{ *+[r]{TX} \ar[d]^\pi \ar[r]_{Tf} &
*+[l]{TY} \ar[d]_\pi \\ *+[r]{X} \ar[r]^f & *+[l]{Y.} }
\end{equation*}
For definition (a) of $TX,TY$, let $(U,\phi)$ and $(V,\psi)$ be coordinate charts on $X,Y$ with $U\subseteq\R^m_k$, $V\subseteq\R^n_l$, with coordinates $(u_1,\ldots,u_m)\in U$ and $(v_1,\ldots,v_n)\in V$, and let $(TU,T\phi)$, $(TV,T\psi)$ be the corresponding charts on $TX,TY$, with coordinates $(u_1,\ab\ldots,\ab u_m,\ab q_1,\ab\ldots,\ab q_m)\in TU$ and $(v_1,\ldots,v_n,r_1,\ldots,r_n)\in TV$. Equation \eq{gc2eq1} defines a map $\psi^{-1}\ci f\ci\phi$ between open subsets of $U,V$. Write $\psi^{-1}\ci f\ci\phi=(f_1,\ldots,f_n)$, for $f_j=f_j(u_1,\ldots,u_m)$. Then the corresponding $T\psi^{-1}\ci Tf\ci T\phi$ maps
\begin{align*}
&T\psi^{-1}\ci Tf\ci T\phi:(u_1,\ldots,u_m,q_1,\ldots,q_m)\longmapsto
\bigl(f_1(u_1,\ldots,u_m),\ldots,\\
&f_n(u_1,\ldots,u_m),\ts\sum_{i=1}^m\frac{\pd f_1}{\pd u_i}(u_1,\ldots,u_m)q_i,\ldots,\ts\sum_{i=1}^m\frac{\pd f_n}{\pd u_i}(u_1,\ldots,u_m)q_i\bigr).
\end{align*}

For definition (b) of $TX,TY$, $Tf$ acts as $Tf:(x,v)\mapsto (y,w)$ for $y=f(x)\in Y$ and $w=v\ci f^*$, where $f^*:C^\iy_y(Y)\ra C^\iy_x(X)$ maps $f^*:[a]\mapsto[a\ci f]$.

If $g:Y\ra Z$ is smooth then $T(g\ci f)=Tg\ci Tf:TX\ra
TZ$, and $T(\id_X)=\id_{TX}:TX\ra TX$. Thus, the assignment
$X\mapsto TX$, $f\mapsto Tf$ is a functor, the {\it tangent
functor\/} $T:\Manc\ra\Manc$. It restricts to $T:\Mancin\ra\Mancin$. 

If $f:X\ra Y$ is only weakly smooth, the same
definition gives a weakly smooth map $Tf:TX\ra TY$. We can also
regard $Tf$ as a vector bundle morphism $\d f:TX\ra f^*(TY)$ on $X$.

The {\it cotangent bundle\/} $T^*X$ of a manifold with corners $X$
is the dual vector bundle of $TX$. Cotangent bundles $T^*X$ are not
functorial in the same way, though we do have vector bundle morphisms $(\d f)^*:f^*(T^*Y)\ra T^*X$ on~$X$.
\label{gc2def7}
\end{dfn}

Here is the parallel definition for b-(co)tangent bundles:

\begin{dfn} Let $X$ be an $m$-manifold with corners. The {\it b-tangent
bundle\/} ${}^bTX\ra X$ of $X$ is a natural (unique up to canonical
isomorphism) rank $m$ vector bundle on $X$. It has a natural {\it inclusion morphism\/} $I_X:{}^bTX\ra TX$, which is an isomorphism over the interior $X^\ci$, but not over the boundary strata $S^k(X)$ for $k\ge 1$. Here are three equivalent ways to characterize~${}^bTX,I_X$:
\begin{itemize}
\setlength{\itemsep}{0pt}
\setlength{\parsep}{0pt}
\item[(a)] {\bf In coordinate charts:} let $(U,\phi)$ be a chart 
on $X$, with $U\subseteq\R^m_k$ open. Then over $\phi(U)$, ${}^bTX$ is the trivial vector bundle with basis of sections $u_1\frac{\pd}{\pd u_1},\ldots,u_k\frac{\pd}{\pd u_k},\frac{\pd}{\pd u_{k+1}},\ldots,
\frac{\pd}{\pd u_m}$, for $(u_1,\ldots,u_m)$ the coordinates on $U$. There is a corresponding chart $({}^bTU,{}^bT\phi)$ on ${}^bTX$, where ${}^bTU=U\t\R^m\subseteq\R^{2m}_k$, such that $(u_1,\ldots,u_m,s_1,\ldots,s_m)\in {}^bTU$ represents the vector $s_1u_1\frac{\pd}{\pd u_1}+\cdots+s_ku_k\frac{\pd}{\pd u_k}+s_{k+1}\frac{\pd}{\pd u_{k+1}}+\cdots+s_m\frac{\pd}{\pd u_m}$ over $(u_1,\ldots,u_m)$ in $U$ or $\phi(u_1,\ldots,u_m)$ in $X$. Under change of coordinates $(u_1,\ldots,u_m)\rightsquigarrow(\ti u_1,\ldots,\ti u_m)$ from $(U,\phi)$ to $(\ti U,\ti\phi)$, the corresponding change $(u_1,\ab\ldots,\ab u_m,\ab s_1,\ab\ldots,\ab s_m)\ab\rightsquigarrow(\ti u_1,\ldots,\ti u_m,\ti s_1,\ldots,\ti s_m)$ from $({}^bTU,{}^bT\phi)$ to $({}^bT\ti U,{}^bT\ti\phi)$ is 
\begin{equation*}
\ti s_j=\begin{cases} \sum_{i=1}^k\ti u_j^{-1}u_i\frac{\pd\ti u_j}{\pd u_i}\,s_i+\sum_{i=k+1}^m\ti u_j^{-1}\frac{\pd\ti u_j}{\pd u_i}\,s_i, & j\le k, \\[4pt]
\sum_{i=1}^ku_i\frac{\pd\ti u_j}{\pd u_i}\,s_i+\sum_{i=k+1}^m\frac{\pd\ti u_j}{\pd u_i}\,s_i, & j>k. \end{cases}
\end{equation*}

The morphism $I_X:{}^bTX\ra TX$ acts in coordinate charts $({}^bTU,{}^bT\phi)$, $(TU,T\phi)$ by
\begin{align*}
&(u_1,\ldots,u_m,s_1,\ldots,s_m)\longmapsto (u_1,\ldots,u_m,q_1,\ldots,q_m)\\
&\qquad\qquad =(u_1,\ldots,u_m,u_1s_1,\ldots,u_ks_k,s_{k+1},\ldots,s_m).
\end{align*}
\item[(b)] {\bf Intrinsically in terms of germs:} Let $x\in X$. As in Definition \ref{gc2def7}(b), write $C^\iy_x(X)$ for the set of germs $[a]$ at $x$ of smooth functions $a:X\ra\R$. Then $C^\iy_x(X)$ is an $\R$-algebra, with evaluation map $\ev:C^\iy_x(X)\ra\R$, $\ev:[a]\mapsto a(x)$. Also write $\cI_x(X)$ for the subset of germs $[b]$ at $x\in X$ of interior maps $b:X\ra[0,\iy)$. Then $\cI_x(X)$ is a monoid with operation multiplication $[b]\cdot[c]=[b\cdot c]$ and identity $[1]$. It has an {\it evaluation map\/} $\ev:\cI_x(X)\ra[0,\iy)$, $\ev:[b]\mapsto b(x)$, a monoid morphism. There is also an {\it exponential map\/} $\exp:C^\iy_x(X)\ra \cI_x(X)$ mapping $\exp:[a]\mapsto[\exp a]$, which is a monoid morphism, regarding $C^\iy_x(X)$ as a monoid under addition, and an {\it inclusion map\/} $\inc:\cI_x(X)\ra C^\iy_x(X)$ mapping $\inc:[b]\mapsto[b]$, which is a monoid morphism, regarding $C^\iy_x(X)$ as a monoid under multiplication.

Then there is a natural isomorphism
\ea
{}^bT_xX\cong\bigl\{&(v,v'):\text{$v$ is a linear map $C^\iy_x(X)\ra\R$,}
\nonumber\\
&\text{$v'$ is a monoid morphism $\cI_x(X)\ra\R$,}
\nonumber\\
&\text{$v([a]\!\cdot\! [b])\!=\!v([a])\ev([b])\!+\!\ev([a])v([b])$, all $[a],[b]\!\in\!
C^\iy_x(X)$},
\nonumber\\
&\text{$v'\ci\exp([a])=v([a])$, all $[a]\in C^\iy_x(X)$, and}
\nonumber\\
&\text{$v\ci\inc([b])=\ev([b])v'\bigl([b]\bigr)$, all $[b]\in\cI_x(X)$}\bigr\}.
\label{gc2eq14}
\ea
Here in pairs $(v,v')$ in \eq{gc2eq14}, $v$ is as in \eq{gc2eq13}. If $[b]\in\cI_x(X)$ with $\ev([b])>0$, then $[\log b]\in C^\iy_x(X)$ with $v'([b])=v([\log b])$. So the extra data in $v'$ is $v'([b])$ for $[b]\in\cI_x(X)$ with $\ev([b])=0$.

The morphism $I_X:{}^bTX\ra TX$ acts by $I_X:(v,v')\mapsto v$.

If $X$ is a manifold with faces, as in Remark \ref{gc2rem1}, then we can replace $C^\iy_x(X),\cI_x(X)$ by $C^\iy(X),\cI(X)$, where $\cI(X)$ is the monoid of interior maps $X\ra[0,\iy)$. But if $X$ does not have faces, in general there are too few interior maps $X\ra[0,\iy)$ for the definition to work. This is why we use germs $C^\iy_x(X),\cI_x(X)$ in~\eq{gc2eq14}.
\item[(c)] {\bf In terms of $TX$:} there is a natural isomorphism
of $C^\iy(X)$-modules
\e
{}\!\!\!\!\!\!\!\! C^\iy({}^bTX)\!\cong\!\bigl\{v\!\in\! C^\iy(TX):\text{$v\vert_{S^k(X)}$ is tangent to $S^k(X)$ for all $k$}\bigr\}.
\label{gc2eq15}
\e
Elements of $C^\iy({}^bTX)$ are called {\it b-vector fields}.

The morphism $I_X:{}^bTX\ra TX$ induces
$(I_X)_*:C^\iy({}^bTX)\ra C^\iy(TX)$, which under the isomorphism
\eq{gc2eq15} corresponds to the inclusion of the right hand side
of \eq{gc2eq15} in~$C^\iy(TX)$.
\end{itemize}

In Definition \ref{gc2def7}, we defined $Tf:TX\ra TY$ for any smooth
(or even weakly smooth) map $f:X\ra Y$. As in \cite[\S 2]{Melr1},
\cite[\S 1]{KoMe} the analogue for b-tangent bundles works only for
{\it interior\/} maps $f:X\ra Y$. So let $f:X\ra Y$ be an interior
map of manifolds with corners. We will define a natural smooth map
${}^bTf:{}^bTX\ra{}^bTY$ so that the following commutes:
\begin{equation*}
\xymatrix@C=30pt@R=15pt{ {{}^bTX} \ar[ddr]_(0.6)\pi \ar[dr]^{I_X}
\ar[rrr]_{{}^bTf} &&&
{{}^bTY} \ar[dr]^{I_Y} \ar[ddr]_(0.6)\pi \\
& {TX} \ar[d]^\pi \ar[rrr]^{Tf} &&&
{TY} \ar[d]^\pi \\
& {X} \ar[rrr]^f &&& {Y.\!} }
\end{equation*}

For definition (a) of ${}^bTX,{}^bTY$ above, let $(U,\phi)$ and $(V,\psi)$ be coordinate charts on $X,Y$ with $U\subseteq\R^m_k$, $V\subseteq\R^n_l$, with coordinates $(u_1,\ldots,u_m)\in U$ and $(v_1,\ldots,v_n)\in V$, and let $({}^bTU,{}^bT\phi)$, $({}^bTV,{}^bT\psi)$ be the corresponding charts on ${}^bTX,{}^bTY$, with coordinates $(u_1,\ab\ldots,\ab u_m,\ab s_1,\ab\ldots,\ab s_m)\in {}^bTU$ and $(v_1,\ldots,v_n,t_1,\ldots,t_n)\in {}^bTV$. Then \eq{gc2eq1} defines a map $\psi^{-1}\ci f\ci\phi$ between open subsets of $U,V$. Write $\psi^{-1}\ci f\ci\phi=(f_1,\ldots,f_n)$, for $f_j=f_j(u_1,\ldots,u_m)$. Then the corresponding ${}^bT\psi^{-1}\ci {}^bTf\ci {}^bT\phi$ maps
\e
\begin{split}
&{}^bT\psi^{-1}\!\ci\! {}^bTf\!\ci\! {}^bT\phi:(u_1,\ldots,u_m,s_1,\ldots,s_m)\!\longmapsto\! (v_1,\ldots,v_n,t_1,\ldots,t_n),\!\!\!\!\!\!\!\!{}
\\
&\text{where}\quad v_j=f_j(u_1,\ldots,u_m),\quad j=1\ldots,n,\\
&\text{and}\quad t_j=
\begin{cases} \sum_{i=1}^kf_j^{-1}u_i\frac{\pd f_j}{\pd u_i}\,s_i+\sum_{i=k+1}^mf_j^{-1}\frac{\pd f_j}{\pd u_i}\,s_i, & j\le l, \\[4pt]
\sum_{i=1}^ku_i\frac{\pd f_j}{\pd u_i}\,s_i+\sum_{i=k+1}^m\frac{\pd f_j}{\pd u_i}\,s_i, & j>l. \end{cases}
\end{split}
\label{gc2eq16}
\e

Since $f$ is interior, the functions $f_j^{-1}u_i\frac{\pd f_j}{\pd u_i}$ for $i\le k$, $j\le l$ and $f_j^{-1}\frac{\pd f_j}{\pd u_i}$ for $i>k$, $j\le l$ occurring in \eq{gc2eq16} extend uniquely to
smooth functions of $(u_1,\ldots,u_m)$ where $f_j=0$, which by Definition \ref{gc2def1}(b)(i) is only where $u_i=0$ for certain $i=1,\ldots,k$. If $f$ were not interior, we could have $f_j(u_1,\ldots,u_m)=0$ for all $(u_1,\ldots,u_m)$, and then there are no natural values for $f_j^{-1}u_i\frac{\pd f_j}{\pd u_i}$, $f_j^{-1}\frac{\pd f_j}{\pd u_i}$ (just setting them zero is not functorial under change of coordinates), so we could not define~${}^bTf$.

For definition (b) of ${}^bTX,{}^bTY$, ${}^bTf$ acts by ${}^bTf:(x,v,v')\mapsto (y,w,w')$ for $y=f(x)$, $w=v\ci f^*$ and $w'=v'\ci f^*$, where composition with $f$ maps $f^*:C^\iy_y(Y)\ra C^\iy_x(X)$, $f^*:\cI_y(Y)\ra\cI_x(X)$, as $f$ is interior.

If $g:Y\ra Z$ is another interior map then ${}^bT(g\ci f)={}^bTg\ci
{}^bTf:{}^bTX\ra {}^bTZ$, and ${}^bT(\id_X)=\id_{{}^bTX}:{}^bTX\ra
{}^bTX$. Thus, writing $\Mancin$ for the subcategory of $\Manc$ with morphisms interior maps, the assignment $X\mapsto {}^bTX$, $f\mapsto {}^bTf$ is a functor, the {\it b-tangent functor\/} ${}^bT:\Mancin\ra\Mancin$. The maps $I_X:{}^bTX\ra TX$ give a natural transformation $I:{}^bT\ra T$ of functors on~$\Mancin$.

We can also regard ${}^bTf$ as a vector bundle morphism ${}^b\d
f:{}^bTX\ra f^*({}^bTY)$ on $X$. The {\it b-cotangent bundle\/} ${}^bT^*X$ of $X$ is the dual vector bundle of ${}^bTX$. B-cotangent bundles ${}^bT^*X$ are not functorial in the same way, though we do have vector bundle morphisms $({}^b\d f)^*:f^*({}^bT^*Y)\ra {}^bT^*X$ for interior~$f$.
\label{gc2def8}
\end{dfn}

The next proposition describes the functorial properties of
$TX,{}^bTX$. The proof is straightforward.

\begin{prop}{\bf(a)} As in Definitions\/
{\rm\ref{gc2def7}--\ref{gc2def8},} we have \begin{bfseries}tangent
functors\end{bfseries} $T:\Manc\ra\Manc,$ $T:\cManc\ra\cManc$
preserving the subcategories $\Mancst,$ $\Mancin,$ $\Mancis,$
$\cMancst,$ $\cMancin,$ $\cMancis,$ and \begin{bfseries}b-tangent
functors\end{bfseries} ${}^bT:\ab\Mancin\ra\Mancin,$
${}^bT:\cMancin\ra\cMancin$ preserving $\Mancis,\cMancis$.
\smallskip

\noindent{\bf(b)} The projections $\pi:TX\ra X,$ $\pi:{}^bTX\ra X,$
zero sections $0:X\ra TX,$ $0:X\ra{}^bTX,$ and inclusion
$I_X:{}^bTX\ra TX$ induce natural transformations
\e
\pi:T\Longra \Id, \;\> \pi:{}^bT\Longra \Id, \;\> 0:\Id\Longra T,
\;\> 0:\Id\Longra {}^bT,\;\> I:{}^bT\Longra T
\label{gc2eq17}
\e
on the categories on which both sides are defined.
\smallskip

\noindent{\bf(c)} The functors $T,{}^bT$ \begin{bfseries}preserve
products and direct products\end{bfseries} in each category. That
is, there are natural isomorphisms $T(W\t X)\cong TW\t TX,$
${}^bT(W\t X)\cong {}^bTW\t {}^bTX,$ such that if\/ $f:W\ra Y$ and\/
$g:X\ra Z$ are smooth or interior then the following commute
\begin{equation*}
\xymatrix@C=80pt@R=15pt{ *+[r]{T(W\t X)} \ar[d]^\cong \ar[r]_{T(f\t
g)} & *+[l]{T(Y\t Z)} \ar[d]_\cong
 \\ *+[r]{TW\t TX} \ar[r]^{Tf \t Tg} &
*+[l]{TY\t TZ,} }\;
\xymatrix@C=80pt@R=15pt{ *+[r]{{}^bT(W\t X)} \ar[d]^\cong
\ar[r]_{{}^bT(f\t g)} & *+[l]{{}^bT(Y\t Z)} \ar[d]_\cong \\
 *+[r]{{}^bTW\t {}^bTX} \ar[r]^{{}^bTf \t {}^bTg}
& *+[l]{{}^bTY\t {}^bTZ,} }
\end{equation*}
and if\/ $f:X\ra Y,$ $g:X\ra Z$ are smooth or interior then the
following commute
\begin{equation*}
\xymatrix@C=95pt@R=15pt{ *+[r]{TX} \ar[d]^\id \ar[r]_(0.4){T(f,g)} &
*+[l]{T(Y\t Z)} \ar[d]_\cong \\ *+[r]{TX} \ar[r]^(0.4){(Tf,Tg)} &
*+[l]{TY\t TZ,} }\;
\xymatrix@C=95pt@R=15pt{ *+[r]{{}^bTX} \ar[d]^\id
\ar[r]_(0.4){{}^bT(f,g)} & *+[l]{{}^bT(Y\t Z)} \ar[d]_\cong \\
 *+[r]{{}^bTX} \ar[r]^(0.4){({}^bTf,{}^bTg)}
& *+[l]{{}^bTY\t {}^bTZ.} }
\end{equation*}
These isomorphisms $T(W\t X)\cong TW\t TX,$ ${}^bT(W\t X)\cong
{}^bTW\t {}^bTX$ are also compatible with the natural
transformations~\eq{gc2eq17}.
\label{gc2prop2}
\end{prop}

\begin{rem}{\bf(i)} It is part of the philosophy of this paper, following Melrose \cite{KoMe,Melr2,Melr3,Melr4}, that we prefer to work with b-tangent bundles ${}^bTX$ rather than tangent bundles $TX$ when we can. One reason for this, explained in \S\ref{gc35}, is that for manifolds with g-corners in \S\ref{gc3}, the analogue of ${}^bTX$ behaves better than the analogue of $TX$ (which is not a vector bundle). 
\smallskip

\noindent{\bf(ii)} If $f:X\ra Y$ is a smooth map of manifolds with corners, we can define ${}^bTf:{}^bTX\ra {}^bTY$ only if $f$ is interior. But $C(f):C(X)\ra C(Y)$ is interior for any smooth $f:X\ra Y$ by Proposition \ref{gc2prop1}(a). Hence ${}^bT\ci C(f):{}^bTC(X)\ra {}^bTC(Y)$ is defined for all smooth $f:X\ra Y$, and we can use it as a substitute for ${}^bTf:{}^bTX\ra {}^bTY$ when this is not defined.
\label{gc2rem3}
\end{rem}

\begin{dfn} A smooth map $f:X\ra Y$ of manifolds with corners is called {\it \'etale\/} if it is a local diffeomorphism. That is, $f$ is \'etale if and only if for all $x\in X$ there are open neighbourhoods $U$ of $x$ in $X$ and $V=f(U)$ of $f(x)$ in $Y$ such that $f\vert_U:U\ra V$ is a diffeomorphism (invertible with smooth inverse).
\label{gc2def9}
\end{dfn}

Here are two alternative characterizations of \'etale maps:

\begin{prop} Let\/ $f:X\ra Y$ be a smooth map of manifolds with corners. Then the following are equivalent:
\begin{itemize}
\setlength{\itemsep}{0pt}
\setlength{\parsep}{0pt}
\item[{\bf(i)}] $f$ is \'etale;
\item[{\bf(ii)}] $f$ is simple (hence interior) and\/ ${}^b\d f:{}^bTX\ra f^*({}^bTY)$ is an isomorphism of vector bundles on $X;$ and
\item[{\bf(iii)}] $f$ is simple and\/ $\d f:TX\ra f^*(TY)$ is an isomorphism on $X$.
\end{itemize}

If\/ $f$ is \'etale, then $f$ is a diffeomorphism if and only if it is a bijection.
\label{gc2prop3}
\end{prop}

\subsection{\texorpdfstring{(B-)normal bundles of $\pd^kX,C_k(X)$}{(B-)normal bundles of boundaries and corners}}
\label{gc24}

Next we study normal bundles of $\pd X,\pd^kX$ and $C_k(X)$ in $X$
using (b-)tangent bundles $TX,{}^bTX$. For tangent bundles the
picture is straightforward:

\begin{dfn} Let $X$ be a manifold with corners. From \S\ref{gc23},
the map $i_X:\pd X\ra X$ induces $Ti_X:T(\pd X)\ra TX$, which we may
regard as a morphism $\d i_X:T(\pd X)\ra i_X^*(TX)$ of vector
bundles on $\pd X$. This fits into a natural exact sequence of
vector bundles on~$\pd X$:
\e
\xymatrix@C=25pt{ 0 \ar[r] & T(\pd X) \ar[rr]^{\d i_X} && i_X^*(TX)
\ar[rr]^{\pi_N} && N_{\pd X} \ar[r] & 0, }
\label{gc2eq18}
\e
where $N_{\pd X}\ra\pd X$ is the {\it normal bundle\/} of $\pd X$ in
$X$. While $N_{\pd X}$ is not naturally trivial, it does have a
natural orientation by `outward-pointing' normal vectors, and so
$N_{\pd X}$ is trivializable. The dual vector bundle $N_{\pd X}^*$ of $N_{\pd X}$ is called the {\it conormal bundle\/} of $\pd X$ in $X$.

Similarly, we have projections $\Pi:\pd^kX\ra X$ and $\pi_1,\ldots,\pi_k:\pd^kX\ra\pd X$ mapping $\Pi:(x,\be_1,\ldots,\be_k)\mapsto x$ and $\pi_i:(x,\be_1,\ldots,\be_k)\mapsto(x,\be_i)$ under the identification \eq{gc2eq7}. As for \eq{gc2eq18}, we have a natural exact sequence
\begin{equation*}
\xymatrix@C=25pt{ 0 \ar[r] & T(\pd^k X) \ar[rr]^{\d\Pi} && \Pi^*(TX)
\ar[rr]^{\pi_N} && N_{\pd^kX} \ar[r] & 0 }
\end{equation*}
of vector bundles on $\pd^kX$, where $N_{\pd^kX}$ is the normal bundle of $\pd^kX$ in $X$, a vector bundle of rank $k$. Clearly, there is a natural isomorphism
\e
N_{\pd^kX}\cong\ts\bigop_{i=1}^k\pi_i^*(N_{\pd X}),
\label{gc2eq19}
\e
so that $N_{\pd^kX}$ is the direct sum of $k$ trivializable line
bundles, and is trivializable. It has dual bundle~$N_{\pd^kX}^*$.

As in \eq{gc2eq8} the symmetric group $S_k$ acts freely on $\pd^kX$, with $C_k(X)\cong\pd^kX/S_k$. The action of $S_k$ lifts naturally to $N_{\pd^kX}$, with $N_{\pd^kX}/S_k\cong
N_{C_k(X)}$, the normal bundle of $C_k(X)$ in $X$, in the exact sequence
\e
\xymatrix@C=12pt{ 0 \ar[r] & T(C_k(X)) \ar[rr]^{\d\Pi} && \Pi^*(TX)
\ar[rr]^(0.4){\pi_N} && N_{C_k(X)}\cong N_{\pd^kX}/S_k \ar[r] & 0. }
\label{gc2eq20}
\e
The action of $S_k$ on $N_{\pd^kX}\cong\bigop_{i=1}^k\pi_i^*(N_{\pd X})$ permutes the $k$ line bundles $\pi_i^*(N_{\pd X})$ for $i=1,\ldots,k$. Thus, $N_{C_k(X)}$ does not have a natural decomposition like \eq{gc2eq19} for $k\ge 2$. Similarly,~$N_{C_k(X)}^*\cong N_{\pd^kX}^*/S_k$.

For the corners $C(X)=\coprod_{k=0}^{\dim X}C_k(X)$, we define
vector bundles of mixed rank $N_{C(X)},N_{C(X)}^*$ on $C(X)$ by
$N_{C(X)}\vert_{C_k(X)}\!=\!N_{C_k(X)}$, $N_{C(X)}^*\vert_{C_k(X)}\!=\!N_{C_k(X)}^*$. As $\dim N_{C_k(X)}=\dim X$, these are objects of $\Manc$ rather than~$\cManc$.

Now let $f:X\ra Y$ be a smooth map of manifolds with corners. Form
the diagram of vector bundles of mixed rank on $C(X)$, with exact
rows:
\e
\begin{gathered}
\xymatrix@C=17pt@R=19pt{ 0 \ar[r] & T(C(X)) \ar[d]^(0.4){\d C(f)}
\ar[rr]^{\d\Pi} && \Pi^*(TX) \ar[d]^(0.4){\Pi^*(\d f)}
\ar[rr]^(0.5){\pi_N} && N_{C(X)} \ar@{.>}[d]^(0.4){N_{C(f)}}
\ar[r] & 0 \\
0 \ar[r] & {\begin{subarray}{l}\ts \quad C(f)^*\\
\ts (T(C(Y)))\end{subarray}} \ar[rr]^(0.4){C(f)^*(\d\Pi)} &&
{\begin{subarray}{l}\ts C(f)^*(\Pi^*(TY)) \\
\ts =\Pi^*(f^*(TY))\end{subarray}}
\ar[rr]^(0.6){C(f)^*(\pi_N)} && {\begin{subarray}{l}\ts \quad C(f)^* \\
\ts (N_{C(Y)})\end{subarray}} \ar[r] & 0. }\!\!\!\!\!\!{}
\end{gathered}
\label{gc2eq21}
\e
As the left hand square commutes, by exactness there is a unique
morphism $N_{C(f)}$ as shown making the diagram commute.

Suppose $g:Y\ra Z$ is another smooth map of manifolds with corners.
By considering the diagram
\begin{equation*}
\xymatrix@C=19pt@R=20pt{ 0 \ar[r] & T(C(X)) \ar[d]^(0.4){\d C(f)}
\ar@/_.4pc/@<-4.5ex>[dd]_(0.2){\d C(g\ci f)}
\ar[rr]^{\d\Pi} && \Pi^*(TX) \ar[d]^(0.4){\Pi^*(\d f)}
\ar@/_.4pc/@<-7.3ex>[dd]_(0.2){\Pi^*(\d(g\ci f))}
\ar[rr]^(0.5){\pi_N} && N_{C(X)} \ar[d]^(0.4){N_{C(f)}}
\ar@/_.4pc/@<-3.4ex>[dd]_(0.2){N_{C(g\ci f)}} \ar[r] & 0
\\
0 \ar[r] & {\begin{subarray}{l}\ts \quad C(f)^*\\
\ts (T(C(Y)))\end{subarray}} \ar[rr]^(0.4){C(f)^*(\d\Pi)}
\ar[d]^{C(f)^*(\d C(g))} &&
{\begin{subarray}{l}\ts C(f)^*(\Pi^*(TY)) \\
\ts =\Pi^*(f^*(TY))\end{subarray}} \ar[rr]^(0.6){C(f)^*(\pi_N)}
\ar[d]^{\Pi^*(\Pi^*(\d g))}
&& {\begin{subarray}{l}\ts \quad C(f)^*\\
\ts (N_{C(Y)})\end{subarray}} \ar[d]^{C(f)^*(N_{C(g)})} \ar[r] & 0
\\
0 \ar[r] & {\begin{subarray}{l}\ts C(g\!\ci\! f)^*\\
\ts (T(C(Z)))\end{subarray}} \ar[rr]^(0.4){C(g\ci f)^*(\d\Pi)} &&
{\begin{subarray}{l}\ts C(g\!\ci\! f)^*(\Pi^*(TZ)) \\
\ts =\Pi^*((g\!\ci\! f)^*(TZ))\end{subarray}}
\ar[rr]^(0.6){C(g\ci f)^*(\pi_N)} && {\begin{subarray}{l}\ts C(g\!\ci\! f)^*\\
\ts (N_{C(Z)})\end{subarray}} \ar[r] & 0, }
\end{equation*}
and using uniqueness of $N_{C(f)}$ in \eq{gc2eq21}, we see that
\e
N_{C(g\ci f)}=C(f)^*(N_{C(g)})\ci N_{C(f)}.
\label{gc2eq22}
\e

We can also regard $N_{C(f)}$ as a morphism $N_{C(f)}:N_{C(X)}\ra
N_{C(Y)}$. Then \eq{gc2eq22} implies that $N_{C(g\ci f)}=N_{C(g)}\ci
N_{C(f)}:N_{C(X)}\ra N_{C(Z)}$, so $X\mapsto N_{C(X)}$, $f\mapsto
N_{C(f)}$ is a functor $N_C:\Manc\ra\Manc$ and $N_C:\cManc\ra\cManc$. The zero section $z:C(X)\ra N_{C(X)}$ and projection $\pi:N_{C(X)}\ra C(X)$ give natural transformations $z:C\Ra N_C$ and $\pi:N_C\Ra C$. As in Propositions \ref{gc2prop1}(f) and \ref{gc2prop2}(c), one can show that $N_C$ preserves products and direct products.
\label{gc2def10}
\end{dfn}

Next we consider the analogue of the above for b-tangent bundles
${}^bTX$, which is more subtle. As $i_X:\pd X\ra X$ is not interior,
we do {\it not\/} have an induced map ${}^b\d i_X:{}^bT(\pd X)\ra
i_X^*({}^bTX)$, so we cannot form the analogue of \eq{gc2eq18} for
${}^bT(\pd X)$. We begin with analogues of $N_{C(X)},N_{C(f)}$ above:

\begin{dfn} Let $X$ be an $n$-manifold with corners, and $k=0,\ldots,n$. As in Definition \ref{gc2def4}, points of $C_k(X)$ are pairs $(x,\ga)$ for $x\in X$ and $\ga$ a local $k$-corner component of $X$ at $x$, and there is a natural 1-1 correspondence between such $\ga$ and (unordered) sets $\{\be_1,\ldots,\be_k\}$ of $k$ distinct local boundary components $\be_1,\ldots,\be_k$ of $X$ at $x$. Define a rank $k$ vector bundle $\pi: {}^bN_{C_k(X)}\ra C_k(X)$ over $C_k(X)$ to have fibre ${}^bN_{C_k(X)}\vert_{(x,\ga)}$ the vector
space with basis $\be_1,\ldots,\be_k$ for each $(x,\ga)\in C_k(X)$ with $\ga$ corresponding to $\{\be_1,\ldots,\be_k\}$. Considering local models, we see that the total space of
${}^bN_{C_k(X)}$ is naturally an $n$-manifold with corners.

Points of ${}^bN_{C_k(X)}$ will be written $(x,\ga,b_1\be_1+\cdots+ b_k\be_k)$ for
$(x,\ga)\in C_k(X)$ and $b_1,\ldots,b_k\in\R$, where $\ga$ corresponds to $\{\be_1,\ldots,\be_k\}$. Since $C_k(X)\cong\pd^kX/S_k$ by \eq{gc2eq8} there is an isomorphism ${}^bN_{C_k(X)}\cong (\pd^kX\t\R^k)/S_k$, where the symmetric group $S_k$ acts on $\R^k$ by permuting the coordinates.

For reasons that will become clear in Proposition \ref{gc2prop4}, we
call ${}^bN_{C_k(X)}$ the {\it b-normal bundle of\/ $C_k(X)$
in\/} $X$. The dual bundle ${}^bN_{C_k(X)}^*$ is called the {\it b-conormal bundle of\/ $C_k(X)$ in\/}~$X$.

Define the {\it monoid bundle\/} $M_{C_k(X)}$ as a subset in ${}^bN_{C_k(X)}$ by
\begin{equation*}
M_{C_k(X)}=\bigl\{(x,\ga,b_1\be_1+
\cdots+b_k\be_k)\in{}^bN_{C_k(X)}:b_i\in\N\bigr\},
\end{equation*}
where $\ga$ corresponds to $\{\be_1,\ldots,\be_k\}$ and $\N=\{0,1,2,\ldots\}$. It fibres over $C_k(X)$ with fibres $\N^k$, and is a submanifold of ${}^bN_{C_k(X)}$ of dimension $n-k$. The `$M$' in $M_{C_k(X)}$ stands for {\it monoid}, as we will
regard $\pi:M_{C_k(X)}\ra C_k(X)$ as a locally constant family
of commutative monoids $\N^k$ over $C_k(X)$, that is, each fibre
$\pi^{-1}(p)$ has a commutative, associative addition operation $+$
with identity~$0$. 

Define the {\it dual monoid bundle\/} $M_{C_k(X)}^\vee$ to be
\begin{equation*}
M_{C_k(X)}^\vee=\bigl\{(x',\bs b)\in{}^bN_{C_k(X)}^*:x'\in C_k(X),\; \bs b\bigl(M_{C_k(X)}\vert_{x'}\bigr)\subseteq\N\bigr\}.
\end{equation*}
It is a subbundle of ${}^bN_{C_k(X)}^*$ with fibre $\N^k$.

For more about monoids, see \S\ref{gc31}. The importance of the monoids $M_{C_k(X)}$ in understanding fibre products and blow-ups of manifolds with corners was emphasized
by Kottke and Melrose \cite[\S 6]{KoMe}, in their {\it basic smooth monoidal complexes}. Gillam and Molcho \cite{GiMo} work with the dual monoids~$M_{C_k(X)}^\vee$. 

Define morphisms ${}^bi_T:{}^bN_{C_k(X)}\ra\Pi^*({}^bTX)$ of vector
bundles and ${}^bi_T:M_{C_k(X)}\ra\Pi^*({}^bTX)$ of monoids on
$C_k(X)$ as follows, where $\Pi:C_k(X)\ra X$ is the projection.
Given $(x,\ga)\in C_k(X)$ where $\ga$ corresponds to $\{\be_1,\ldots,\be_k\}$, choose local
coordinates $(x_1,\ldots,x_n)\in\R^n_l$ on $X$ near $x,$ where $k\le
l\le n$ with $x=(0,\ldots,0)$ and $\be_i=\{x_i=0\}$ for
$i=1,\ldots,k$. Then define
\begin{equation*}
{}^bi_T\vert_{(x,\ga)}:b_1\be_1+\cdots+b_k\be_k
\longmapsto\ts\sum_{i=1}^kb_i\cdot \Pi^*(x_i\frac{\pd}{\pd x_i}).
\end{equation*}
One can show this is independent of the choice of coordinates. We
can also think of these as smooth maps
${}^bi_T:{}^bN_{C_k(X)}\ra{}^bTX$, ${}^bi_T:M_{C_k(X)}\ra{}^bTX$
of manifolds with corners. There is a dual morphism~${}^bi_T^*:\Pi^*({}^bTX^*)\ra\smash{{}^bN_{C_k(X)}^*}$.

\label{gc2def11}
\end{dfn}

In the next proposition, the local existence and uniqueness of
${}^b\pi_T$ is easy to check using a local model $\R^n_l$ for $X$.
The bottom row of \eq{gc2eq23} is \eq{gc2eq20}. The top row of
\eq{gc2eq23} is the analogue of \eq{gc2eq20} for
${}^bTX,{}^bT(C_k(X))$ (note the reversal of directions), and
justifies calling ${}^bN_{C_k(X)}$ the b-normal bundle of $C_k(X)$
in~$X$.

\begin{prop} Let\/ $X$ be a manifold with corners, and\/
$k=0,\ldots,\dim X$. Then there is a unique morphism
${}^b\pi_T:\Pi^*({}^bTX)\ra {}^bT(C_k(X))$ which makes the following
diagram of vector bundles on $C_k(X)$ commute, with exact rows:
\e
\begin{gathered}
\xymatrix@C=18pt@R=15pt{ 0 \ar[r] & {}^bN_{C_k(X)} \ar[d]^0
\ar[rr]_(0.5){{}^bi_T} && \Pi^*({}^bTX) \ar[d]^{\Pi^*(I_X)}
\ar@{.>}[rr]_{{}^b\pi_T} && {}^bT(C_k(X)) \ar[d]^{I_{C_k(X)}} \ar[r]
& 0 \\ 0 & N_{C_k(X)} \ar[l] && \Pi^*(TX) \ar[ll]_(0.5){\pi_N} &&
T(C_k(X)) \ar[ll]_{\d\Pi} & \ar[l] 0. }
\end{gathered}
\label{gc2eq23}
\e
\label{gc2prop4}
\end{prop}

When $k=1$, we have $C_1(X)\cong\pd X$ and
${}^bN_{C_1(X)}\cong\O_{\pd X}$, the trivial line bundle on $\pd X$.
So the top line of \eq{gc2eq23} becomes the exact sequence
\begin{equation*}
\xymatrix@C=22pt{ 0 \ar[r] & \O_{\pd X} \ar[rr]^(0.4){{}^bi_T} &&
i_X^*({}^bTX) \ar[rr]^{{}^b\pi_T} && {}^bT(\pd X) \ar[r] & 0 }
\end{equation*}
of vector bundles on $\pd X$, the analogue of \eq{gc2eq18} for
b-tangent spaces.

As for the $N_{C_k(X)}$, the ${}^bN_{C_k(X)}$ are functorial, but
only for interior maps:

\begin{dfn} In Definition \ref{gc2def11}, set
${}^bN_{C(X)}=\coprod_{k=0}^n{}^bN_{C_k(X)}$ and
$M_{C(X)}\ab=\coprod_{k=0}^nM_{C_k(X)}$. Then ${}^bN_{C(X)}$
is an $n$-manifold with corners, and $M_{C(X)}$ an object of
$\cManc$. We have projections $\pi:{}^bN_{C(X)},\ab M_{C(X)}\ra
C(X)$, making ${}^bN_{C(X)}$ into a vector bundle of mixed rank over
$C(X)$, and $M_{C(X)}$ into a locally constant family of
commutative monoids over~$C(X)$.

Now let $f:X\ra Y$ be an interior map of manifolds with corners.
From \S\ref{gc22} $C(f):C(X)\ra C(Y)$ is also interior, so from
\S\ref{gc23} we have smooth maps ${}^bTf:{}^bTX\ra{}^bTY$ and
${}^bTC(f):{}^bTC(X)\ra{}^bTC(Y)$, which we may write as vector
bundle morphisms ${}^b\d f:{}^bTX\ra f^*({}^bTY)$ on $X$ and ${}^b\d
C(f):{}^bTC(X)\ra C(f)^*({}^bTC(Y))$ on $C(X)$. Consider the diagram
\e
\begin{gathered}
{}\!\!\!\!\xymatrix@C=18pt@R=17pt{ 0 \ar[r] & {}^bN_{C(X)}
\ar@{.>}[d]^(0.4){{}^bN_{C(f)}} \ar[rr]_(0.4){{}^bi_T} &&
\Pi^*({}^bTX) \ar[d]^(0.4){\Pi^*({}^b\d f)} \ar[rr]_(0.6){{}^b\pi_T}
&& {}^bT(C(X))
\ar[d]^(0.4){{}^b\d C(f)} \ar[r] & 0 \\
0 \ar[r] & {\begin{subarray}{l}\ts\;\> C(f)^*\\
\ts ({}^bN_{C(Y)})\end{subarray}} \ar[rr]^(0.4){C(f)^*({}^bi_T)} &&
{\begin{subarray}{l}\ts C(f)^*\!\ci\!\Pi^*({}^bTY)\\
\ts =\Pi^*\!\ci\! f^*({}^bTY)\end{subarray}}
\ar[rr]^(0.6){C(f)^*({}^b\pi_T)} &&
{\begin{subarray}{l}\ts \quad C(f)^*\\
\ts({}^bT(C(Y)))\end{subarray}} \ar[r] & {0.\!\!} }\!\!\!\!\!\!\!{}
\end{gathered}
\label{gc2eq24}
\e
The rows come from the top row of \eq{gc2eq23} for $X,Y$, and are
exact. One can check using formulae in coordinates that the right
hand square commutes. Thus by exactness there is a unique map
${}^bN_{C(f)}$ as shown making the diagram commute.

We can give a formula for ${}^bN_{C(f)}$ as follows. Suppose $x\in
S^{k'}(X)\subseteq X$ with $f(x)=y\in S^{l'}(Y)\subseteq Y$. Then we
may choose local coordinates $(x_1,\ldots,x_m)\in\R^m_{k'}$ on $X$
with $x=(0,\ldots,0)$ and $(y_1,\ldots,y_n)\in\R^n_{l'}$ on $Y$ with
$y=(0,\ldots,0)$, so that $x_1,\ldots,x_{k'},y_1,\ldots,y_{l'}
\in[0,\iy)$ and $x_{k'+1},\ldots,x_m,y_{l'+1},\ldots,y_n\in\R$.
Write $f$ in coordinates as $\bigl(f_1(x_1,\ldots,x_m),
\ldots,f_n(x_1,\ldots,x_m)\bigr)$. As $f$ is interior, Definition
\ref{gc2def1} shows that for $j=1,\ldots,l'$, near $x=(0,\ldots,0)$
we have
\begin{equation*}
f_j(x_1,\ldots,x_m)=F_j(x_1,\ldots,x_m)\cdot\ts\prod_{i=1}^{k'}
x_i^{a_{i,j}},
\end{equation*}
where $F_j$ is smooth and positive and $a_{i,j}\in \N$. Since
$f_j(0,\ldots,0)=0$, we see that for each $j=1,\ldots,l'$ we have
$a_{i,j}>0$ for some $i=1,\ldots,k'$.

The local boundary components of $X$ at $x$ are $\be_i:=\{x_i=0\}$ for $i=1,\ldots,k'$, and of $Y$ at $y$ are $\ti\be_j:=\{y_j=0\}$ for $j=1,\ldots,l'$. Let $\ga$ be a local $k$-corner component of $X$ at $x$ corresponding to $\{\be_{i_1}, \ldots,\be_{i_k}\}$ for $1\le i_1<\cdots<i_k\le k'$, and $\ti\ga$ a local $l$-corner component of $Y$ at $y$ corresponding to $\{\ti\be_{j_1},\ldots,\ti\be_{j_l}\}$ for $1\le j_1<\cdots<j_l\le l'$, so that $(x,\ga)\in C_k(X)$ and $(y,\ti\ga)\in C_l(Y)$, and suppose $f_*(\ga)=\ti\ga$, so that $C(f):(x,\ga)\mapsto(y,\ti\ga)$. Then we can check from the definitions that
\begin{equation*}
\{j_1,\ldots,j_l\}=\bigl\{j\in \{1,\ldots,l'\}:\text{$a_{i_c,j}>0$, some
$c=1,\ldots,k$}\bigr\},
\end{equation*}
and ${}^bN_{C(f)}$ acts by
\e
{}^bN_{C(f)}:\bigl(x,\ga,
b_{i_1}\be_{i_1}+\cdots+b_{i_k}\be_{i_k}\bigr)\longmapsto
\bigl(y,\ti\ga,\ts\sum_{d=1}^l
\bigl[\sum_{c=1}^ka_{i_c,j_d}b_{i_c}\bigr]\ti\be_{j_d}\bigr).
\label{gc2eq25}
\e
Since $a_{i,j}\in\N$, ${}^bN_{C(f)}$ maps $M_{C(X)}\ra
C(f)^*(M_{C(Y)})$. So write
\begin{equation*}
M_{C(f)}:={}^bN_{C(f)} \vert_{M_{C(X)}}:M_{C(X)}\ra
C(f)^*(M_{C(Y)}).
\end{equation*}
Note that $f$ is {\it simple\/} if and only if $M_{C(f)}$ and ${}^bN_{C(f)}$ are isomorphisms.

Suppose $g:Y\ra Z$ is another interior map. From the diagram
\begin{equation*}
\text{\begin{small}$\displaystyle
\xymatrix@C=10pt@R=20pt{ 0 \ar[r] & {}^bN_{C(X)}
\ar[d]^{{}^bN_{C(f)}} \ar@<-4ex>@/_.5pc/[dd]_(0.25){{}^bN_{C(g\ci f)}}
\ar[rr]_(0.4){{}^bi_T} && \Pi^*({}^bTX) \ar[d]^{\Pi^*({}^b\d f)}
\ar@<-3ex>@/_.5pc/[dd]_(0.25){\Pi^*({}^b\d(g\ci f))}
\ar[rr]_(0.6){{}^b\pi_T} && {}^bT(C(X)) \ar[d]^{{}^b\d C(f)}
\ar@<-2ex>@/_.5pc/[dd]_(0.25){{}^b\d C(g\ci f)} \ar[r] & 0
\\
0 \ar[r] & C(f)^*({}^bN_{C(Y)}) \ar[d]^(0.4){C(f)^*({}^bN_{C(g)})}
\ar[rr]^{\raisebox{4pt}{$\st C(f)^*({}^bi_T)$}} &&
C(f)^*\!\ci\!\Pi^*({}^bTY) \ar[d]^(0.4){C(f)^*(\Pi^*({}^b\d g))}
\ar[rr]^{\raisebox{4pt}{$\st C(f)^*({}^b\pi_T)$}} &&
C(f)^*({}^bT(C(Y))) \ar[d]^(0.4){C(f)^*({}^b\d C(g))} \ar[r] & 0
\\
0 \ar[r] & C(g\!\ci\! f)^*({}^bN_{C(Z)})
\ar[rr]^{\raisebox{5pt}{$\st C(g\ci f)^*({}^bi_T)$}} &&
C(g\!\ci\! f)^*\!\ci\!\Pi^*({}^bTZ)
\ar[rr]^{\raisebox{5pt}{$\st C(g\ci f)^*({}^b\pi_T)$}} &&
C(g\!\ci\! f)^*({}^bT(C(Z))) \ar[r] & 0,\!\!{}
}$\end{small}}
\end{equation*}
using the functoriality of $C,{}^bT$, we find that ${}^bN_{C(g\ci
f)}=C(f)^*({}^bN_{C(g)})\ci {}^bN_{C(f)}$, and hence~$M_{C(g\ci
f)}=C(f)^*(M_{C(g)})\ci M_{C(f)}$.

We can also interpret ${}^bN_{C(f)}$ as a smooth map of manifolds
with corners ${}^bN_{C(f)}:{}^bN_{C(X)}\ra {}^bN_{C(Y)}$, and
$M_{C(f)}$ as a morphism $M_{C(f)}:M_{C(X)}\ra
M_{C(Y)}$ in $\cManc$, both of which are interior as $C(f)$ is.
Then for interior $f:X\ra Y$, $g:Y\ra Z$ we have ${}^bN_{C(g\ci
f)}={}^bN_{C(g)}\ci {}^bN_{C(f)}:{}^bN_{C(X)}\ra {}^bN_{C(Z)}$. Thus
$X\mapsto {}^bN_{C(X)}$, $f\mapsto{}^bN_{C(f)}$ defines functors
${}^bN_C:\Mancin\ra\Mancin$ and ${}^bN_C:\cMancin\ra\cMancin$, which we
call the {\it b-normal corner functors}. Similarly $X\mapsto
M_{C(X)}$, $f\mapsto M_{C(f)}$ defines functors
$M_C:\Mancin,\cMancin\ra\cMancin$, which we call the {\it
monoid corner functors}.

The dual bundles ${}^bN_{C(X)}^*,M_{C(X)}^*$ are not functorial in the same way.
\label{gc2def12}
\end{dfn}

The next proposition is easy to check:

\begin{prop} Definition\/ {\rm\ref{gc2def12}} defines functors
${}^bN_C:\Mancin\ra\Mancin,$ ${}^bN_C:\cMancin\ra\cMancin$ and\/
$M_C:\Mancin,\cMancin\ra\cMancin,$ preserving (direct) products, with a commutative diagram of natural transformations:
\begin{equation*}
\xymatrix@C=60pt@R=2pt{ & M_C \ar@{=>}[dd]^{\text{inclusion}}
\ar@{=>}[dr]^\Pi \\ C \ar@{=>}[ur]^{\raisebox{2pt}{$\st\text{zero section $0$\qquad}$}}
\ar@{=>}[dr]_{\text{zero section $0$\qquad}} && C. \\ & {}^bN_C
\ar@{=>}[ur]_\Pi }
\end{equation*}
\label{gc2prop5}
\end{prop}

Here is some similar notation to ${}^bN_{C(X)},M_{C(X)}$, but working over $X$ rather than~$C(X)$.

\begin{dfn} Let $X$ be a manifold with corners. For $x\in S^k(X)\subseteq X$, let $\be_1,\ldots,\be_k$ be the local boundary components of $X$ at $x$, and define
\begin{align*}
{}^b\ti N_xX&=\bigl\{b_1\be_1+\cdots+b_k\be_k:b_1,\ldots,b_k\in\R\bigl\},\\
{}^b\ti N_x^{\sst\ge 0}X&=\bigl\{b_1\be_1+\cdots+b_k\be_k:b_1,\ldots,b_k\in[0,\iy)\bigr\},\\
\ti M_xX&=\bigl\{b_1\be_1+\cdots+b_k\be_k:b_1,\ldots,b_k\in\N\bigr\},
\end{align*}
so that $\ti M_xX\subseteq{}^b\ti N_x^{\sst\ge 0}X\subseteq {}^b\ti N_xX$. That is, ${}^b\ti N_xX\cong\R^k$ is the vector space with basis the local boundary components $\be_1,\ldots,\be_k$ at $x$, with $\dim {}^b\ti N_xX=\depth_Xx$. We will think of $\ti M_xX\cong\N^k$ as a {\it toric monoid}, as in \S\ref{gc311} below, with ${}^b\ti N_xX=\ti M_xX\ot_\N\R$ the corresponding real vector space, and ${}^b\ti N_x^{\sst\ge 0}X\cong[0,\iy)^k$ as the corresponding {\it rational polyhedral cone\/} in ${}^b\ti N_xX$, as in~\S\ref{gc314}.

Now let $f:X\ra Y$ be an interior map of manifolds with corners, and $x\in S^k(X)\subseteq X$ with $f(x)=y\in S^l(Y)\subseteq Y$. Write $\be_1,\ldots,\be_k$ for the local boundary components of $X$ at $x$, and $\be_1',\ldots,\be_l'$ for the local boundary components of $Y$ at $y$. We can choose local coordinates $(x_1,\ldots,x_m)\in\R^m_k$ near $x$ in $X$ with $x=(0,\ldots,0)$, such that $\be_i=\{x_i=0\}$ for $i=1,\ldots,k$, and local coordinates $(y_1,\ldots,y_n)\in\R^n_l$ near $y$ in $Y$ with $y=(0,\ldots,0)$, such that $\be_j'=\{y_j=0\}$ for $j=1,\ldots,l$. Then as in \S\ref{gc21}, near $x$ we may write $f$ in coordinates as $f=(f_1,\ldots,f_n)$, where for $j=1,\ldots,l$ we have $f_j(x_1,\ldots,x_m)= F_j(x_1,\ldots,x_m)\cdot x_1^{a_{1,j}}\cdots x_k^{a_{k,j}}$ for some $a_{i,j}\in\N$ and positive smooth functions $F_j$. Define a linear map ${}^b\ti N_xf:{}^b\ti N_xX\ra {}^b\ti N_yY$ by
\begin{align*}
{}^b&\ti N_xf:b_1\be_1+\cdots+b_k\be_k\\
&\longmapsto (a_{1,1}b_1+\cdots+a_{k,1}b_k)\be_1'+\cdots+(a_{1,l}b_1+\cdots+a_{k,l}b_k)\be_l',
\end{align*}
as for $N_{C(f)}$ in \eq{gc2eq25}. Define ${}^b\ti N^{\sst\ge 0}_xf:{}^b\ti N_x^{\sst\ge 0}X\ra {}^b\ti N_y^{\sst\ge 0}Y$ and $\ti M_xf:\ti M_xX\ra\ti M_yY$ to be the restrictions of ${}^b\ti N_xf$ to ${}^b\ti N_x^{\sst\ge 0}X$ and $\ti M_xX$. Note that $f$ is {\it simple\/} if and only if $\ti M_xf:\ti M_xX\ra\ti M_yY$ is an isomorphism for all~$x\in X$.

If $g:Y\ra Z$ is another interior map of manifolds with corners then
\begin{equation*}
{}^b\ti N_x(g\ci f)\!=\!{}^b\ti N_yg\ci {}^b\ti N_xf,\;\>
{}^b\ti N^{\sst\ge 0}_x(g\ci f)\!=\!{}^b\ti N^{\sst\ge 0}_yg\ci {}^b\ti N^{\sst\ge 0}_xf,\;\>\ti M_x(g\ci f)\!=\!\ti M_yg\ci \ti M_xf,
\end{equation*}
and ${}^b\ti N_x\id_X,{}^b\ti N^{\sst\ge 0}_x\id_X,\ti M_x\id_X$ are identities. So the ${}^b\ti N_xX,\ab {}^b\ti N^{\sst\ge 0}_xX,\ab\ti M_xX,\ab{}^b\ti N_xf,\ab{}^b\ti N^{\sst\ge 0}_xf,\ab\ti M_xf$ are functorial.

We could define ${}^b\ti NX=\bigl\{(x,v):x\in X$, $v\in {}^b\ti N_xX\bigr\}$ and ${}^b\ti Nf:{}^b\ti NX\ra{}^b\ti NY$ by ${}^b\ti Nf:(x,v)\mapsto (f(x),{}^b\ti N_xf(v))$, and similarly for ${}^b\ti N^{\sst\ge 0}X,{}^b\ti N^{\sst\ge 0}f$ and $\ti MX,\ti Mf$, and these would also be functorial. However, in contrast to ${}^bN_{C(X)}$ above,
these ${}^b\ti NX,{}^b\ti N^{\sst\ge 0}X$ {\it would not be manifolds with corners}, even of mixed dimension, since the dimensions of the fibres ${}^b\ti N_xX,{}^b\ti N^{\sst\ge 0}_xX$ vary discontinuously with $x$ in $X$. They are useful for stating conditions on interior~$f:X\ra Y$.

\label{gc2def13}
\end{dfn}

\section{Manifolds with generalized corners}
\label{gc3}

We will now define a category $\Mangc$ of {\it manifolds with generalized corners}, or {\it manifolds with g-corners\/} for short, which contains the manifolds with corners $\Manc$ of \S\ref{gc2} as a full subcategory. We extend \S\ref{gc2} to manifolds with g-corners, with the exception of the ordinary tangent bundle $TX$ and normal bundle $N_{C(X)}$, which do not generalize well.
 
\subsection{Monoids}
\label{gc31}

We now discuss monoids, from the point of view usual in the theory of logarithmic geometry, in which they are basic objects. Some good references are Ogus \cite[\S I]{Ogus}, Gillam \cite[\S 1--\S 2]{Gill}, and Gillam and Molcho~\cite[\S 1]{GiMo}.

\subsubsection{The basic definitions}
\label{gc311}

Here are the basic definitions we will need in the theory of monoids. 

\begin{dfn} A ({\it commutative\/}) {\it monoid\/} $(P,+,0)$ is a set $P$ with a binary operation $+:P\t P\ra P$ and a distinguished element $0\in P$ satisfying $p+p'=p'+p$, $p+(p'+p'')=(p+p')+p''$ and $p+0=0+p=p$ for all $p,p',p''\in P$. All monoids in this paper will be commutative. Usually we write $P$ for the monoid, leaving $+,0$ implicit. 

A {\it morphism of monoids\/} $\mu:(P,+,0)\ra(Q,+,0)$ is a map $\mu:P\ra Q$ satisfying $\mu(p+p')=\mu(p)+\mu(p')$ for all $p,p'\in P$ and~$\mu(0)=0$.

If $p\in P$ and $n\in\N=\{0,1,\ldots\}$, we write $n\cdot p={\buildrel
{\ulcorner\,\,\,\text{$n$ copies } \,\,\,\urcorner} \over
{\vphantom{i}\smash{p+\cdots+p}}}$, with $0\cdot p=0$.

A {\it submonoid\/} of a monoid $P$ is a subset $Q\subseteq P$ such that $0\in Q$ and $q+q'\in Q$ for all $q,q'\in Q$. Then $Q$ is also a monoid.

If $Q\subset P$ is a submonoid, there is a natural {\it quotient monoid\/} $Q/P$ and surjective morphism $\pi:P\ra P/Q$, with the universal property that $\pi(Q)=\{0\}$, and if $\mu:P\ra R$ is a monoid morphism with $\mu(Q)=\{0\}$ then $\mu=\nu\ci\pi$ for a unique morphism $\nu:P/Q\ra R$. Explicitly, we may take $P/Q$ to be the set of $\sim$-equivalence classes $[p]$ of $p\in P$, where $p\sim p'$ if there exist $q,q'\in Q$ with $p+q=p'+q'$ in $P$, and $\pi:p\mapsto[p]$.

A {\it unit\/} $u$ in a monoid $P$ is an element $u\in P$ for which there exists $v\in P$ with $u+v=0$. This $v$ is unique, and we write it as $-u$. Write $P^\t$ for the set of all units in $P$. It is a submonoid of~$P$.

Any abelian group $G$ is a monoid. If $P$ is a monoid, then $P^\t$ is an abelian group, and $P$ is an abelian group if and only if $P^\t=P$.

If $P$ is a monoid, there is a natural morphism of monoids $\pi:P\ra P^\gp$ with $P^\gp$ an abelian group, with the universal property that if $\mu:P\ra G$ is a morphism with $G$ an abelian group, then $\mu=\nu\ci\pi$ for a unique morphism of abelian groups $\nu:P^\gp\ra G$. This determines $P^\gp,\pi$ up to canonical isomorphism. Explicitly, we may take $P^\gp$ to be the quotient monoid $(P\t P)/\De_P$, where $\De_P=\{(p,p):p\in P\}$ is the diagonal submonoid of $P\t P$, and $\pi:p\mapsto [p,0]$.

Let $P$ be a monoid. Then:
\begin{itemize}
\setlength{\itemsep}{0pt}
\setlength{\parsep}{0pt}
\item[(i)] We call $P$ {\it finitely generated\/} if there exists a surjective morphism $\pi:\N^k\ra P$ for some $k\ge 0$. Any such $\pi$ may be uniquely written $\pi(n_1,\ldots,n_k)=n_1\cdot p_1+\cdots+n_k\cdot p_k$ for $p_1,\ldots,p_k\in P$, which we call {\it generators\/} of $P$.

If $P$ is finitely generated then $P^\gp$ is a finitely generated abelian group.
\item[(ii)] A finitely generated monoid $P$ is called {\it free\/} if $P\cong\N^k$ for some $k\ge 0$.
\item[(iii)] We call $P$ {\it integral}, or {\it cancellative}, if $\pi:P\ra P^\gp$ is injective. Equivalently, $P$ is integral if $p+p''=p'+p''$ implies $p=p'$ for $p,p',p''\in P$. For integral $P$, we can regard $P$ as a subset of~$P^\gp$.
\item[(iv)] We call $P$ {\it saturated\/} if it is integral, and $p\in P^\gp$ with $n\cdot p\in P\subseteq P^\gp$ for $n\ge 1$ implies that $p\in P\subseteq P^\gp$.
\item[(v)] We call $P$ {\it torsion-free\/} if $P^\gp$ is torsion-free, that is, $n\cdot p=0$ for $n\ge 1$ and $p\in P^\gp$ implies~$p=0$.
\item[(vi)] We call $P$ {\it sharp\/} if $P^\t=\{0\}$. The {\it sharpening\/} $P^\sh$ of $P$ is $P^\sh=P/P^\t$, a sharp monoid with surjective projection $\pi:P\ra P^\sh$.
\item[(vii)] We call $P$ a {\it weakly toric monoid\/} if it is finitely-generated, integral, saturated, and torsion-free.
\item[(viii)] We call $P$ a {\it toric monoid\/} if it is finitely-generated, integral, saturated, torsion-free, and sharp. (Saturated and sharp together imply torsion-free.)
\end{itemize}
Note that definitions of toric monoids in the literature differ: some authors, including Ogus \cite{Ogus}, refer to our weakly toric monoids as {\it toric monoids}, and to our toric monoids as {\it sharp toric monoids}. 

Write $\Mon$ for the category of monoids, and $\Monfg,\Monwt,\Monto$ for the full subcategories of finitely generated, weakly toric, and toric monoids, respectively, so that~$\Monto\subset\Monwt\subset\Monfg\subset\Mon$.

If $P$ is a toric monoid then $P^\gp$ is a finitely generated, torsion-free abelian group, so $P^\gp\cong\Z^k$ for $k\ge 0$. We define the {\it rank\/} of $P$ to be $\rank P=k$.

If $P$ is weakly toric then $P^\t\cong\Z^l$ and $P^\sh$ is a toric monoid, and the exact sequence $0\ra P^\t\ra P\ra P^\sh\ra 0$ splits, so that $P\cong P^\sh\t\Z^l$ for $P^\sh$ a toric monoid. We define $\rank P=\rank P^\gp=\rank P^\sh+l$.
\label{gc3def1}
\end{dfn}

Here are some examples:

\begin{ex}{\bf(a)} $(\Q,+,0)$ is a non-finitely generated monoid. It is integral, saturated, and torsion-free, but not sharp, as $\Q^\t=\Q$.
\smallskip

\noindent{\bf(b)} $\bigl([0,\iy),\cdot,1\bigr)$ is a non-finitely generated monoid. (Note here that the monoid operation is multiplication `$\,\cdot\,$' rather than addition, and the identity is 1 not 0.) We have $[0,\iy)^\gp=\{0\}$, so $[0,\iy)$ is not integral, and $[0,\iy)^\t=(0,\iy)$, so $[0,\iy)$ is not sharp.
\smallskip

\noindent{\bf(c)} $\N^k$ is a toric monoid for $k=0,1,\ldots,$ with $(\N^k)^\gp\cong\Z^k$.
\smallskip

\noindent{\bf(d)} $\Z^k$ is a finitely generated monoid. For instance, as generators take the $k+1$ vectors $(1,0,\ldots,0),(0,1,0,\ldots,0),\ldots,(0,\ldots,1),(-1,-1,\ldots,-1)$. Also $\Z^k$ is integral, saturated, and torsion-free. But $\Z^k$ is not sharp, as $(\Z^k)^\t=\Z^k\ne 0$, so $\Z^k$ is weakly toric, but not toric.
\smallskip

\noindent{\bf(e)} Set $P=\N\amalg\{1'\}$, with `+' as usual on $\N$, and $n+1'=1'+n=n+1$ for $n>0$ in $\N$, and $0+1'=1'+0=1'$. Then $P$ is a finitely generated monoid, with generators $1,1'$, and is torsion-free and sharp. We have $P^\gp=\Z$, with $\pi:P\ra P^\gp$ mapping $\pi:n\mapsto n$ for $n\in\N$ and $\pi:1'\mapsto 1$. Then $\pi(1)=\pi(1')$, so $\pi:P\ra P^\gp$ is not injective, and $P$ is not integral, or saturated, or toric.
\smallskip

\noindent{\bf(f)} Set $P=\{0,1\}$ with $0+0=0$ and $1+0=0+1=1+1=1$. Then $P$ is a finitely generated monoid with generator 1, torsion-free, and sharp. But $P^\gp=\{0\}$, so $P$ is not integral, saturated, or toric.
\smallskip

\noindent{\bf(g)} $P=\{0,2,3,\ldots\}$ is a submonoid of $\N$, with $P^\gp=\Z\supset P$. It is finitely generated, with generators 2,3, and is integral, torsion-free, and sharp. But it is not saturated, since $1\in P^\gp$ with $2\cdot 1\in P$ but $1\not\in P$, so $P$ is not toric.
\smallskip

\noindent{\bf(h)} Set $P=\N\amalg\{1',2',3',\ldots\}$, with $m+n=(m+n)$, $m'+n=(m+n)'$, $m+n'=(m+n)'$, $m'+n'=(m+n)$ for all $m,n>0$ in $\N$, and $0+p=p+0=p$ for $p\in P$. Then $P$ is a finitely generated monoid, with generators $1,1'$, and is integral, saturated, and sharp. We have $P^\gp=\Z\t\Z_2$, where $\pi:P\ra P^\gp$ is $\pi(n)=(n,0)$ and $\pi(n')=(n,\al)$, writing $\Z_2=\{0,\al\}$ with $\al+\al=0$. Thus $P$ is not torsion-free, as $0\ne (0,\al)\in P^\gp$ with $2\cdot(0,\al)=0$, so $P$ is not toric.
\label{gc3ex1}
\end{ex}

\subsubsection{Duality}
\label{gc312}

We discuss dual monoids, following Ogus~\cite[\S 2.2]{Ogus}.

\begin{dfn} Let $P$ be a monoid. The {\it dual monoid}, written $P^\vee$ or $\D(P)$, is the monoid $\Hom(P,\N)$ of morphisms $\mu:P\ra\N$ in $\Mon$, with the obvious addition $(\mu+\nu)(p)=\mu(p)+\nu(p)$ and identity $0(p)=0$.

If $\al:P\ra Q$ is a morphism of monoids, the {\it dual morphism}, written $\al^\vee:Q^\vee\ra P^\vee$ or $\D(\al):\D(Q)\ra\D(P)$, is $\al^\vee:\mu\mapsto \mu\ci\al$ for all~$\mu:Q\ra\N$.

Then $\D:\Mon\ra\Mon^{\bf op}$ mapping $P\mapsto\D(P)$, $\al\mapsto\D(\al)$ is a functor, where $\Mon^{\bf op}$ is the opposite category to~$\Mon$.

Define a morphism $\eta(P):P\ra(P^\vee)^\vee$ by $\eta(P):p\mapsto\bigl(\mu\mapsto\mu(p)\bigr)$ for $p\in P$ and $\mu\in P^\vee$. Then $\eta:\Id_\Mon\Ra\D\ci\D$ is a natural transformation of functors $\Mon\ra\Mon$, where $\Id_\Mon:\Mon\ra\Mon$ is the identity functor.
\label{gc3def2}
\end{dfn}

From Ogus \cite[Th.~2.2.3]{Ogus} we may deduce:

\begin{thm} If\/ $P$ is a finitely generated monoid, then $P^\vee=\D(P)$ is toric. Hence $\D:\Mon\ra\Mon^{\bf op}$ restricts to $\D^{\bf fg}:\Monfg\ra(\Monto)^{\bf op}$ and\/ $\D^{\bf to}:\Monto\ra(\Monto)^{\bf op}$. Also, the natural morphism $\eta(P):P\ra(P^\vee)^\vee$ is an isomorphism if and only if\/ $P$ is a toric monoid. Thus $\eta^{\bf to}:\Id_\Monto\Ra\D^{\bf to}\ci\D^{\bf to}$ is a natural isomorphism of functors $\Mon^{\bf to}\ra\Mon^{\bf to},$ and\/ $\D^{\bf to}:\Monto\ra(\Monto)^{\bf op}$ is an equivalence of categories.
\label{gc3thm1}
\end{thm}

\begin{ex}{\bf(a)} $(\N^k)^\vee\cong\N^k$.
\smallskip

\noindent{\bf(b)} $(\Z^k)^\vee=\{0\}$, and more generally $G^\vee=\{0\}$ for any abelian group $G$.
\smallskip

\noindent{\bf(c)} $[0,\iy)^\vee=\{0\}$.
\label{gc3ex2}
\end{ex}

Write $R_{\bf fg}^{\bf to}=\D^{\bf to}\ci\D^{\bf fg}:\Monfg\ra\Monto$, and $I_{\bf to}^{\bf fg}:\Monto\hookra\Monfg$ for the inclusion functor. Then for each $P\in\Monfg$ and $Q\in\Monto$ we have 
\begin{align*}
\Hom\bigl(R_{\bf fg}^{\bf to}(P),Q\bigr)&=\Hom\bigl((P^\vee)^\vee,Q\bigr)\cong\Hom(P,Q)=\Hom\bigl(P,I_{\bf to}^{\bf fg}(Q)\bigr),
\end{align*}
where in the second step we use that as $Q$ is toric, any morphism $P\ra Q$ factors uniquely through the projection $P\ra (P^\vee)^\vee$. Thus $R_{\bf fg}^{\bf to}$ is a {\it left adjoint\/} for~$I_{\bf to}^{\bf fg}$.

\subsubsection{Pushouts and fibre products of monoids}
\label{gc313}

Next we discuss pushouts and fibre products of monoids. Some references are Gillam \cite[\S 1.2--\S 1.3]{Gill} and Ogus \cite[\S 1.1]{Ogus}.

\begin{thm}{\bf(a)} All direct and inverse limits exist in the category $\Mon,$ so in particular pushouts and fibre products exist. The construction of inverse limits, including fibre products, commutes with the forgetful functor $\Mon\ra\Sets$. Finite products and coproducts coincide in~$\Mon$.
\smallskip

\noindent{\bf(b)} The category $\Monfg$ is closed under pushouts in $\Mon$. Hence pushouts exist in $\Monfg$.
\smallskip

\noindent{\bf(c)} The category $\Monto$ is not closed under pushouts in $\Monfg$. Nonetheless, pushouts exist in the category $\Monto,$ though they may not agree with the same pushout in $\Monfg$. If\/ $\al:P\ra Q$ and\/ $\be:P\ra R$ are morphisms in $\Monto$ then $Q\amalg_P^{\bf to}R\cong R_{\bf fg}^{\bf to}(Q\amalg_P^{\bf fg}R),$ where $Q\amalg_P^{\bf to}R,Q\amalg_P^{\bf fg}R$ are the pushouts in $\Monto,\Monfg$ respectively, and\/ $R_{\bf fg}^{\bf to}$ is as in\/~{\rm\S\ref{gc312}.} 
\smallskip

\noindent{\bf(d)} The categories $\Monfg$ and\/ $\Monto$ are closed under fibre products in $\Mon$. Thus, fibre products exist in both\/ $\Monfg$ and\/ $\Monto,$ and can be computed as fibre products of the underlying sets.
\label{gc3thm2}
\end{thm}

\begin{proof} Part (a) can be found in Ogus \cite[\S 1.1]{Ogus} or Gillam \cite[\S 1.1--\S 1.2]{Gill}. If $\al:P\ra Q$ and $\be:P\ra R$ are morphisms in $\Mon$, then as in \cite[\S 1.3]{Gill} the pushout $S=Q\amalg_{\al,P,\be}R$ is $S=Q\op R/\sim$, where $\sim$ is the smallest monoidal equivalence relation on $Q\op R$ with $(\al(p),0)\sim(0,\be(p))$ for all $p\in P$. Actually computing $\sim$ or $Q\amalg_PR$ explicitly can be tricky.

For (b), if $S=Q\amalg_PR$ is as above with $Q,R\in\Monfg$, and $q_1,\ldots,q_k$, $r_1,\ldots,r_l$ are generators for $Q,R$, then $[q_1,0],\ldots,[q_k,0],[0,r_1],\ldots,[0,r_l]$ are generators for $S$, so $S\in\Monfg$, and $\Monfg$ is closed under pushouts in $\Mon$.

For (c), as $R_{\bf fg}^{\bf to}:\Monfg\ra\Monto$ has a right adjoint $I_{\bf to}^{\bf fg}$ from \S\ref{gc312}, it takes pushouts in $\Monfg$ to pushouts in $\Monto$. Thus, if $P,Q,R\in\Monto$ then
\begin{equation*}
R_{\bf fg}^{\bf to}(Q\amalg_P^{\bf fg}R)\cong R_{\bf fg}^{\bf to}(Q)\amalg_{R_{\bf fg}^{\bf to}(P)}^{\bf to}R_{\bf fg}^{\bf to}(R)\cong
Q\amalg_P^{\bf to}R.
\end{equation*}

For (d), Gillam \cite[Cor.~1.9.8]{Gill} shows $\Monfg$ is closed under fibre products in $\Mon$. If $\mu:P\ra R$ and $\nu:Q\ra R$ are morphisms in $\Monto$ then the fibre product $P\t_RQ$ in $\Mon$ is finitely generated, integral, and saturated by Ogus \cite[Th.~2.1.16(6)]{Ogus}, and it is torsion-free and sharp as $P\t_RQ$ is a submonoid of $P\op Q$, which is torsion-free and sharp since $P,Q$ are toric. Hence $P\t_RQ$ is toric, and $\Monto$ is closed under fibre products in~$\Mon$.
\end{proof}

\subsubsection{Toric monoids and rational polyhedral cones}
\label{gc314}

\begin{dfn} Let $\La$ be a {\it lattice\/} (that is, an abelian group isomorphic to $\Z^k$ for $k\ge 0$), so that $\La_\R:=\La\ot_\Z\R$ is a real vector space isomorphic to $\R^k$, with a natural inclusion $\La\hookra\La_\R$. We identify $\La$ with its image in $\La_\R$, so that $\La\subset\La_\R$. We also have the dual lattice $\La^*:=\Hom(\La,\Z)$ and dual vector space $\La_\R^*=\Hom(\La_\R,\R)$, and we identify $\La^*$ with a subset of~$\La_\R^*$.

A {\it rational polyhedral cone\/} in $\La_\R$ is a subset $C\subseteq\La_\R$ of the form
\e
C=\bigl\{\la\in\La_\R:\al_i(\la)\ge 0,\; i=1,\ldots,k\bigr\},
\label{gc3eq1}
\e
for some finite collection of elements $\al_1,\ldots,\al_k\in \La^*$. An {\it integral polyhedral cone\/} $C_\Z\subseteq\La$ is a subset of the form $C_\Z=C\cap\La$ for some rational polyhedral cone $C\subseteq\La_\R$. We call $C$ or $C_\Z$ {\it pointed\/} if $C\cap -C=\{0\}$ or $C_\Z\cap -C_\Z=\{0\}$. Note that an integral polyhedral cone $C_\Z$ is a monoid, as it is a submonoid of~$\La$.

For $C$ as in \eq{gc3eq1}, a {\it face\/} of $C$ is a subset $D\subseteq C$ of the form 
\begin{equation*}
D=\bigl\{\la\in\La_\R:\al_i(\la)=0,\;\> i\in J,\;\> \al_i(\la)\ge 0,\; i\in\{1,\ldots,k\}\sm J\bigr\},
\end{equation*}
for some $J\subseteq \{1,\ldots,k\}$. That is, we require equality in some of the inequalities in \eq{gc3eq1}. Each face $D$ of $C$ is also a rational polyhedral cone, and the collection of faces $D\subseteq C$ is independent of the choice of $\al_1,\ldots,\al_k$, for $C$ fixed.
\label{gc3def3}
\end{dfn}

The next proposition is well known (see for instance Gillam \cite[Proof of Th.~1.12.3]{Gill}). {\it Gordan's Lemma\/} says that an integral polyhedral cone $C_\Z$ is finitely generated, and the rest of the proof that $C_\Z$ is (weakly) toric is easy.

\begin{prop} A monoid\/ $P$ is weakly toric if and only if it is isomorphic to an integral polyhedral cone $C_\Z\subset\La,$ and toric if and only if it is isomorphic to a pointed integral polyhedral cone $C_\Z\subset\La$. In both cases, we may take the lattice $\La$ to be $P^\gp,$ and\/ $\al_1,\ldots,\al_k$ in \eq{gc3eq1} to be generators of the dual monoid\/~$P^\vee$.
\label{gc3prop1}
\end{prop}

Rational and integral polyhedral cones give us a geometric, visual way to think about (weakly) toric monoids, as corresponding to a class of polyhedra in $\R^n$, and are particularly helpful for studying {\it faces\/} of (weakly) toric monoids.

\subsubsection{Ideals, prime ideals, faces, and spectra of monoids}
\label{gc315}

The next definition is taken from Ogus \cite[\S 1.4]{Ogus} and Gillam~\cite[\S 2.1]{Gill}.

\begin{dfn} An {\it ideal\/} $I$ of a monoid $P$ is a subset $I\subsetneq P$ such that for all $i\in I$ and $p\in P$ we have $p+i\in I$. Then $0\notin I$, as otherwise $p=p+0\in I$ for all $p\in P$, contradicting $I\ne P$. An ideal $I$ is called {\it prime\/} if $p,q\in P$ and $p+q\in I$ imply that $p\in I$ or~$q\in I$. 

A submonoid $F\subseteq P$ is called a {\it face\/} of $P$ if $p,q\in P$ and $p+q\in F$ imply that $p\in F$ and $q\in F$. If is easy to see that $F\subseteq P$ is a face of $P$ if and only if $I=P\sm F$ is a prime ideal in $P$. This gives a bijection $F\longleftrightarrow I=P\sm F$ between faces $F$ of $P$ and prime ideals $I$ in~$P$.

The {\it codimension\/} $\codim F$ of a face $F\subseteq P$ is the rank of the abelian group $(P/F)^\gp$, which is defined when $(P/F)^\gp$ is finitely generated. If $P$ is toric then~$\rank F+\codim F=\rank P$.

The union $\bigcup_{\al\in A}I_\al$ of any family $I_\al:\al\in A$ of prime ideals in $P$ is a prime ideal in $P$. Dually, the intersection $\bigcap_{\al\in A}F_\al$ of any family $F_\al:\al\in A$ of faces of $P$ is a face of $P$.

The minimal ideal in $P$ is $\es$, and the maximal ideal is $P\sm P^\t$. Both are prime. Dually, the maximal face in $P$ is $P$, and the minimal face is $P^\t$.

The {\it spectrum\/} $\Spec P$ is the set of all prime ideals of $P$, which under $I\mapsto F=P\sm I$ is bijective to the set of faces of $P$.

There is a natural topology on $\Spec P$ called the {\it Zariski topology}, generated by the open sets $S_J=\{I\in \Spec P:J\subseteq I\}$ for all ideals $J\subset P$.

If $\mu:P\ra Q$ is a morphism of monoids, and $I$ is a (prime) ideal in $Q$, then $\mu^{-1}(I)$ is a (prime) ideal in $P$. Dually, if $F$ is a face of $Q$, then $\mu^{-1}(F)$ is a face of $P$. Defining $\Spec\mu:\Spec Q\ra\Spec P$ by $\Spec\mu:I\mapsto \mu^{-1}(I)$, then $\Spec\mu$ is continuous in the Zariski topologies. The natural projection $\pi:P\ra P^\sh$ induces a homeomorphism~$\Spec\pi:\Spec P^\sh\ra \Spec P$.
\label{gc3def4}
\end{dfn}

The parts of the next lemma are proved in Gillam and Molcho \cite[Lem.s 1.2.4 \& 1.4.1]{GiMo}, or are obvious.

\begin{lem}{\bf(i)} Suppose $F$ is a face of a monoid\/ $P$. If\/ $P$ is finitely generated, or integral, or saturated, or torsion-free, or sharp, or weakly toric, or toric, then $F$ is also finitely generated, \ldots, toric, respectively.
\smallskip

\noindent{\bf(ii)} Suppose $F$ is a face of a finitely generated monoid\/ $P$. If\/ $p_1,\ldots,p_n$ generate $P,$ then $\{p_i:p_i\in F,$ $i=1,\ldots,n\}$ generate $F$. 
\smallskip

\noindent{\bf(iii)} If\/ $P$ is a finitely generated monoid, then $\Spec P$ is finite.
\label{gc3lem1}
\end{lem}

The next proposition summarizes some facts about (weakly) toric monoids, which are well understood in toric geometry.

\begin{prop} Let\/ $P$ be a weakly toric monoid. Then:
\begin{itemize}
\setlength{\itemsep}{0pt}
\setlength{\parsep}{0pt}
\item[{\bf(a)}] By Proposition\/ {\rm\ref{gc3prop1}} we may identify $P\cong C\cap\La,$ where $\La=P^\gp$ is a lattice and\/ $C\subseteq \La_\R=\La\ot_\Z\R$ is a rational polyhedral cone. This identifies faces $F\subseteq P$ of the monoid\/ $P$ with subsets $D\cap\La\subset C\cap\La$ where $D\subseteq C$ is a face of the rational polyhedral cone $C$ as in Definition\/ {\rm\ref{gc3def3},} and this induces a $1$-$1$ correspondence between faces $F$ of\/ $P$ and faces $D$ of\/~$C$.
\item[{\bf(b)}] The faces $F$ of\/ $P$ are exactly the subsets $\al^{-1}(0)=\bigl\{p\in P:\al(p)=0\bigr\}$ for all\/ $\al$ in $P^\vee=\Hom(P,\N),$ the dual monoid of\/~$P$.
\item[{\bf(c)}] Let\/ $F$ be a face of $P,$ and write $F^\w=\bigl\{\al\in P^\vee:\al\vert_F=0\bigr\}$. Then $F^\w$ is a face of\/ $P^\vee,$ with\/ $\rank F^\w=\rank P-\rank F=\codim F,$ and the map $F\mapsto F^\w$ gives a $1$-$1$ correspondence between faces of\/ $P$ and faces of\/~$P^\vee$.
\end{itemize}
Now suppose $P$ is toric. Then:
\begin{itemize}
\setlength{\itemsep}{0pt}
\setlength{\parsep}{0pt}
\item[{\bf(d)}] Let\/ $\La,\La_\R,C$ be as in part\/ {\bf(a)}. Write $\La^*=\Hom(\La,\Z)$ for the dual lattice and\/ $\La_\R^*=\La^*\ot_\Z\R=(\La_\R)^*$ for the dual vector space, and define
\begin{equation*}
C^\vee=\bigl\{\al\in\La_\R^*:\text{$\al(c)\ge 0$ for all\/ $c\in C$}\bigr\}.
\end{equation*}
Then $C^\vee$ is a rational polyhedral cone in $\La_\R^*,$ and there is a natural isomorphism $P^\vee\cong C^\vee\cap\La^*,$ where $P^\vee$ is the dual monoid of $P$.
\item[{\bf(e)}] For each face $F$ of\/ $P$ we have $\rank F\!=\!\codim F^\w,$ $\codim F\!=\!\rank F^\w$.
\item[{\bf(f)}] The isomorphism $\eta(P):P\ra(P^\vee)^\vee$ from Theorem\/ {\rm\ref{gc3thm1}} induces an isomorphism $\eta(P)\vert_F:F\ra(F^\w)^\w$ for all faces $F\subseteq P$.
\end{itemize}
\label{gc3prop2}
\end{prop}

\subsubsection{Monoids and toric geometry}
\label{gc316}

We now explain the connection between monoids and toric geometry over $\C$. This material will not be used later, but explains the term `toric monoid', and may be helpful to those already familiar with toric geometry. It also helps motivate the definition of manifolds with g-corners in~\S\ref{gc32}.

Let $P$ be a weakly toric monoid. Define a commutative $\C$-algebra $\C[P]$ to be the $\C$-vector space with basis formal symbols $e^p$ for $p\in P$, with multiplication $e^p\cdot e^{p'}=e^{p+p'}$ and identity $1=e^0$. Write $Z_P=\Spec\C[P]$, as an affine $\C$-scheme, which is 
of finite type, reduced, and irreducible, as $P$ is weakly toric. 

There is a natural 1-1 correspondence between $\C$-points of $Z_P$ (that is, algebra morphisms $x:\C[P]\ra\C$), and monoid morphisms $\mu:P\ra(\C,\cdot)$ (where $(\C,\cdot)$ is $\C$ regarded as a monoid under multiplication, with identity 1), defined by $\mu(p)=x(e^p)\in\C$ for all~$p\in P$.

Define an algebraic $\C$-torus $T_P$ to be $T_P=\Hom(P,\C^\t)$, where $\C^\t=\C\sm\{0\}$, as an abelian group under multiplication. If $P^\gp\cong\Z^k$ then $T_P\cong(\C^\t)^k$. There is a natural action of $T_P$ on $Z_P$, which on $\C$-points acts by $(t\cdot\mu)(p)=t(p)\cdot\mu(p)$ for $p\in P$, where $t\in T_P=\Hom(P,\C^\t)$ and $\mu\in\Hom\bigl(P,(\C,\cdot)\bigr)$ corresponds to a $\C$-point $x$ of $Z_P$. This $T_P$-action makes $Z_P$ into an {\it affine toric $\C$-variety}. 

Every affine toric $\C$-variety $Z$ is isomorphic to some $Z_P$, for a weakly toric monoid $P$ unique up to isomorphism, where $P$ is toric if and only if $T_P$ has a fixed point (necessarily unique) in~$Z_P$.

\subsection{\texorpdfstring{The model spaces $X_P,$ for $P$ a weakly toric monoid}{The model spaces X(P), for P a weakly toric monoid}}
\label{gc32}

As in \S\ref{gc2}, manifolds with corners are locally modelled on $[0,\iy)^k\t\R^{n-k}$ for $0\le k\le n$. We will define manifolds with generalized corners in \S\ref{gc33} to be locally modelled on spaces $X_P$ depending on a weakly toric monoid $P$. This section defines and studies these spaces $X_P$, and `smooth maps' between them.

\begin{dfn} Let $P$ be a weakly toric monoid. Then as in \S\ref{gc31}, $P$ is isomorphic to a submonoid of $\Z^k$ for some $k\ge 0$. In \S\ref{gc33} we will suppose that $P$ is {\it equal to\/} a submonoid of some $\Z^k$. This is for set theory reasons: if $X$ is a manifold with g-corners, then without some such restriction on the monoids $P^i$, the maximal g-atlas $\{(P^i,U^i,\phi^i):i\in I\}$ on $X$ in Definition \ref{gc3def6} would not be a set, but only a proper class.

Define $X_P$ to be the set of monoid morphisms $x:P\ra[0,\iy)$, where $\bigl([0,\iy),\cdot\bigr)$ is the monoid $[0,\iy)$ with operation multiplication and identity 1. Define the {\it interior\/} $X_P^\ci\subset X_P$ of $X_P$ to be the subset of $x$ with~$x(P)\subseteq(0,\iy)\subset[0,\iy)$. For each $p\in P$, define a function $\la_p:X_P\ra[0,\iy)$ by $\la_p(x)=x(p)$. Then $\la_{p+q}=\la_p\cdot\la_q$ for $p,q\in P$, and~$\la_0=1$.

Define a topology on $X_P$ to be the weakest topology such that $\la_p:X_P\ra[0,\iy)$ is continuous for all $p\in P$. This makes $X_P$ into a locally compact, Hausdorff topological space, and $X_P^\ci$ is open in $X_P$. If $U\subseteq X_P$ is an open set, define the {\it interior\/} $U^\ci$ of $U$ to be~$U^\ci=U\cap X_P^\ci$.

Note that $X_P$ and $U$ are not manifolds, in general, so smooth functions on $X_P,U$ are not yet defined. Let $f:U\ra\R$ be a continuous function. We say that $f$ is a {\it smooth function\/} $U\ra\R$ if there exist $r_1,\ldots,r_n\in P$, an open subset $W\subseteq[0,\iy)^n$, and a smooth map $g:W\ra\R$ (in the usual sense, as in \S\ref{gc21}), such that for all $x\in U$ we have $\bigl(x(r_1),\ldots,x(r_n)\bigr)\in W$ and
\e
f(x)=g\bigl(x(r_1),\ldots,x(r_n)\bigr)=g\bigl(\la_{r_1}(x),\ldots,\la_{r_n}(x)\bigr).
\label{gc3eq2}
\e
We say that a continuous function $f:U\ra(0,\iy)$ is {\it smooth\/} if $f$ is smooth as a map~$U\ra\R$.

We say that a continuous function $f:U\ra[0,\iy)$ is {\it smooth\/} if on each connected component $U'$ of $U$, we either have $f\vert_{U'}=\la_p\vert_{U'}\cdot h$, where $p\in P$ and $h:U'\ra(0,\iy)$ is smooth, or $f\vert_{U'}=0$. Note that (as for manifolds of corners), $f$ is smooth as a map $U\ra[0,\iy)$ implies that $f$ is smooth as a map $f:U\ra\R$, but not vice versa. 

Now let $Q$ be another weakly toric monoid, and $V\subseteq X_Q$ an open set. We say that a continuous map $f:U\ra V$ is {\it smooth\/} if $\la_q\ci f:U\ra[0,\iy)$ is smooth for all $q\in Q$. We say that $f$ is a {\it diffeomorphism\/} if $f$ is invertible and $f,f^{-1}$ are smooth. We say that $f$ is {\it interior\/} if $f$ is smooth and~$f(U^\ci)\subseteq V^\ci$. The identity map $\id_U:U\ra U$ is smooth and interior.

Suppose $R$ is a third weakly toric monoid, and $W\subseteq X_R$ an open set, and $g:V\ra W$ is smooth. It is easy to show that $g\ci f:U\ra W$ is smooth, that is, compositions of smooth maps are smooth. Also compositions of diffeomorphisms (or interior maps) are diffeomorphisms (or interior maps). 
\label{gc3def5}
\end{dfn}

\begin{rem} In \S\ref{gc316}, given a weakly toric monoid $P$, we defined an affine toric $\C$-variety $Z_P=\Hom\bigl(P,(\C,\cdot)\bigr)$, acted on by an algebraic $\C$-torus $T_P=\Hom(P,\C^\t)$. This is related to $X_P$ above as follows. Write ${\rm U}(1)=\bigl\{z\in\C:\md{z}=1\bigr\}\subset\C^\t$ and $T_P^\R=\Hom(P,{\rm U}(1))\subset T_P$, so that $T_P^\R$ is a real torus, the maximal compact subgroup of $T_P$. Using $\C/{\rm U(1)}\cong[0,\iy)$, we can show there is a natural identification~$X_P=\Hom\bigl(P,([0,\iy),\cdot)\bigr)\cong Z_P/T_P^\R$. 

Thus, the spaces $X_P$ appear in the background of complex toric geometry, and several topics treated below --- for instance, the boundary and corners of $X_P$ --- are related to well known facts in toric geometry.
\label{gc3rem1}
\end{rem}

The next proposition gives an alternative description of the material of Definition \ref{gc3def5} in terms of choices of generators and relations for the monoids $P,Q$. The presentation of Proposition \ref{gc3prop3} is often easier to work with, but that of Definition \ref{gc3def5} has the advantage of being intrinsic to the monoids $P,Q$, and independent of choices of generators and relations.

\begin{prop} Suppose\/ $P$ is a weakly toric monoid. Choose generators $p_1,\ldots,p_m$ for $P,$ and a generating set of relations for $p_1,\ldots,p_m$ of the form
\e
a_1^jp_1+\cdots+a_m^jp_m=b_1^jp_1+\cdots+b_m^jp_m\quad\text{in $P$ for $j=1,\ldots,k,$}
\label{gc3eq3}
\e
where\/ $a_i^j,b_i^j\in\N$ for\/ $i=1,\ldots,m$ and\/ $j=1,\ldots,k$. Then:
\begin{itemize}
\setlength{\itemsep}{0pt}
\setlength{\parsep}{0pt}
\item[{\bf(a)}] $\la_{p_1}\t\cdots\t\la_{p_m}:X_P\ra[0,\iy)^m$ is a homeomorphism from $X_P$ to
\e
{}\!\!\!\!\!\!\!\!\!\!\!\!\!\!\!\! X_P'=\bigl\{(x_1,\ldots,x_m)\in[0,\iy)^m:x_1^{a_1^j}\cdots x_m^{a_m^j}=x_1^{b_1^j}\cdots x_m^{b_m^j},\; j=1,\ldots,k\bigr\},
\label{gc3eq4}
\e
regarding $X_P'$ as a closed subset of\/ $[0,\iy)^m$ with the induced topology.
\item[{\bf(b)}] Let\/ $U\subseteq X_P$ be open, and write\/ $U'=(\la_{p_1}\t\cdots\t\la_{p_m})(U)$ for the corresponding open subset of\/ $X_P'$. Then a function $f:U\ra\R$ is smooth in the sense of Definition\/ {\rm\ref{gc3def5}} if and only if there exists an open neighbourhood\/ $W$ of\/ $U'$ in $[0,\iy)^m$ and a smooth map $g:W\ra\R$ in the sense of\/ {\rm\S\ref{gc21},} regarding $W$ as a manifold with corners, such that\/ $f=g\ci(\la_{p_1}\t\cdots\t\la_{p_m}):U\ra\R$. The analogues hold for $f:U\ra(0,\iy),$ $f:U\ra[0,\iy)$ and\/ $g:W\ra(0,\iy),$ $g:W\ra[0,\iy)$.

\item[{\bf(c)}] Now let\/ $Q$ be another weakly toric monoid. Choose generators $q_1,\ldots,q_n$ for $Q$. Let\/ $V\subseteq X_Q$ be open. Then a map $f:U\ra V$ is smooth in the sense of Definition\/ {\rm\ref{gc3def5}} if and only if there exists an open neighbourhood\/ $W$ of\/ $U'$ in $[0,\iy)^m$ and a smooth map $g:W\ra[0,\iy)^n$ in the sense of\/ {\rm\S\ref{gc21},} such that\/ $(\la_{q_1}\t\cdots\t\la_{q_n})\ci f=g\ci(\la_{p_1}\t\cdots\t\la_{p_m}):U\ra[0,\iy)^n$.
\end{itemize}
\label{gc3prop3}
\end{prop}

\begin{proof} For (a), let $x\in X_P$, so that $x:P\ra\bigl([0,\iy),\cdot\bigr)$ is a monoid morphism, and set $x_i=x(p_i)=\la_{p_i}(x)\in[0,\iy)$ for $i=1,\ldots,m$. Since $x$ is a monoid morphism, applying $x$ to \eq{gc3eq3} gives $x_1^{a_1^j}\cdots x_m^{a_m^j}=x_1^{b_1^j}\cdots x_m^{b_m^j}$, as in \eq{gc3eq4}. As $p_1,\ldots,p_m$ generate $P$, and \eq{gc3eq3} is a generating set of relations, we see that $\la_{p_1}\t\cdots\t\la_{p_m}$ maps $x\mapsto(x_1,\ldots,x_m)$, and gives a bijection~$X_P\ra X_P'$.

Let $p\in P$. Then we may write $p=c_1p_1+\cdots+c_mp_m$ for $c_1,\ldots,c_m\in\N$, and $\la_p=\bigl(x_1^{c_1}\cdots x_m^{c_m}\bigr)\ci (\la_{p_1}\t\cdots\t\la_{p_m})$. The topology on $X_P$ is the weakest for which $\la_p:X_P\ra[0,\iy)$ is continuous for all $p\in P$. This is identified by $\la_{p_1}\t\cdots\t\la_{p_m}:X_P\ra X_P'$ with the weakest topology on $X_P'\subseteq[0,\iy)^m$ such that $x_1^{c_1}\cdots x_m^{c_m}:X_P'\ra[0,\iy)$ is continuous for all $c_1,\ldots,c_m\in\N$. But by taking $c_i=\de_{ij}$ for $j=1,\ldots,m$, we see this is just the topology on $X_P'$ induced by the inclusion $X_P'\subseteq[0,\iy)^m$, which proves~(a). 

For functions $f:U\ra\R$ in (b), the `if' part is trivial, taking $r_1,\ldots,r_n$ in Definition \ref{gc3def5} to be $p_1,\ldots,p_m$, with $n=m$. For the `only if' part, let $f:U\ra\R$ be smooth in the sense of Definition \ref{gc3def5}. Then $f(x)=g'\bigl(x(r_1),\ldots,x(r_n)\bigr)$ for all $x\in U$, where $r_1,\ldots,r_n\in P$ and $g':W'\ra\R$ is smooth for $W'$ an open neighbourhood of $(\la_{r_1}\t\cdots\t\la_{r_n})(U)$ in $[0,\iy)^n$. Since $p_1,\ldots,p_m$ generate $P$ we may write $r_j=\sum_{i=1}^mc_{ij}p_i$ for $c_{ij}\in\N$, $i=1,\ldots,m$, $j=1,\ldots,n$. Define $W\subseteq[0,\iy)^m$ and $g:W\ra\R$ by
\begin{align*}
&W=\bigl\{(x_1,\ldots,x_m)\in[0,\iy)^m:(x_1^{c_{11}}\cdots x_m^{c_{m1}},
\ldots,x_1^{c_{1n}}\cdots x_m^{c_{mn}})\in W'\bigr\},\\
&g(x_1,\ldots,x_m)=g'\bigl(x_1^{c_{11}}\cdots x_m^{c_{m1}},
\ldots,x_1^{c_{1n}}\cdots x_m^{c_{mn}}\bigr).
\end{align*}
Then $W$ is an open neighbourhood of $U'$ in $[0,\iy)^m$ and $g$ is smooth, and $f=g\ci(\la_{p_1}\t\cdots\t\la_{p_m})$. This proves part (b) for $f:U\ra\R$, and (b) for $f:U\ra(0,\iy)$ follows.

For functions $f:U\ra[0,\iy)$ in (b), observe that if $W\subseteq[0,\iy)^m$ is open and connected and $g:W\ra[0,\iy)$ is smooth in the sense of \S\ref{gc21} then either we may write $g(x_1,\ldots,x_m)=x_1^{c_1}\cdots x_m^{c_m}\cdot h(x_1,\ldots,x_m)$, where $c_1,\ldots,c_m\in\N$ and $h:W\ra(0,\iy)$ is smooth, or $g=0$. Using this and the argument of the first part of (b), we can prove (b) for~$f:U\ra[0,\iy)$.

For (c), first suppose $f:U\ra V$ is a map, $W$ is an open neighbourhood of $U'$ in $[0,\iy)^m$, and $g:W\ra[0,\iy)^n$ is smooth in the sense of \S\ref{gc21}, with $(\la_{q_1}\t\cdots\t\la_{q_n})\ci f=g\ci(\la_{p_1}\t\cdots\t\la_{p_m}):U\ra[0,\iy)^n$. Write $g=(g_1,\ldots,g_n)$, so that $g_i:W\ra[0,\iy)$ is smooth. Then $\la_{q_i}\ci f=g\ci(\la_{p_1}\t\cdots\t\la_{p_m}):U\ra[0,\iy)$, so part (b) shows that $\la_{q_i}\ci f:U\ra[0,\iy)$ is smooth in the sense of Definition \ref{gc3def5}, for $i=1,\ldots,n$. Let $q\in Q$. Then we may write $q=c_1q_1+\cdots+c_nq_n$ for $c_1,\ldots,c_n\in\N$, as $q_1,\ldots,q_n$ generate $Q$. Then 
\begin{equation*}
\la_q\ci f=(\la_{q_1}\ci f)^{c_1}\cdots(\la_{q_n}\ci f)^{c_n}:U\longra[0,\iy),
\end{equation*}
so $\la_q\ci f:U\ra[0,\iy)$ is smooth as in Definition \ref{gc3def5} as the $\la_{q_i}\ci f$ are, and $f:U\ra V$ is smooth as in Definition \ref{gc3def5}. This proves the `if' part of~(c).

Next suppose $f:U\ra V$ is smooth in the sense of Definition \ref{gc3def5}. Then $\la_{q_i}\ci f:U\ra[0,\iy)$ is smooth as in Definition \ref{gc3def5} for each $i=1,\ldots,n$, so by (b) there exists $W_i\subseteq[0,\iy)^m$ open and $g_i:W_i\ra[0,\iy)$ smooth as in \S\ref{gc21} such that $\la_{q_i}\ci f=g_i\ci(\la_{p_1}\t\cdots\t\la_{p_m}):U\ra[0,\iy)$. Set $W=W_1\cap\cdots\cap W_n$ and $g=g_1\vert_W\t\cdots\t g_n\vert_W:W\ra[0,\iy)^n$. Then $g$ is smooth and $(\la_{q_1}\t\cdots\t\la_{q_n})\ci f=g\ci(\la_{p_1}\t\cdots\t\la_{p_m}):U\ra[0,\iy)^n$, proving the `only if' part of (c), and completing the proof of the proposition.
\end{proof}

\begin{ex}{\bf(i)} When $P=\N$, points of $X_\N$ are monoid morphisms $x:\N\ra\bigl([0,\iy),\cdot\bigr)$, which may be written uniquely in the form $x(m)=y^m$, $m\in\N$, for $y\in[0,\iy)$. This gives an identification $X_\N\cong[0,\iy)$
mapping $x\mapsto y=x(1)$.

In Proposition \ref{gc3prop3}, we may take $P=\N$ to be generated by $p_1=1$, with no relations. Then part (a) shows that $\la_1:X_\N\ra X_\N'=[0,\iy)$ is a homeomorphism, the same identification $X_\N\cong[0,\iy)$ as above.
\smallskip

\noindent{\bf(ii)} When $P=\Z$, points of $X_\Z$ are monoid morphisms $x:\Z\ra\bigl([0,\iy),\cdot\bigr)$, which may be written uniquely in the form $x(m)=e^{my}$ for $y\in\R$. This gives an identification $X_\Z\cong\R$ mapping $x\mapsto y=\log x(1)$.

In Proposition \ref{gc3prop3}, we may take $P=\Z$ to be generated by $p_1=1$ and $p_2=-1$, with one relation $p_1+p_2=0$. Then part (a) shows that $\la_1\t\la_{-1}$ is a homeomorphism from $X_\Z$ to
\begin{equation*}
X_\Z'=\bigl\{(x_1,x_2)\in[0,\iy)^2:x_1x_2=1\bigr\}.
\end{equation*}
In terms of the identification $X_\Z\cong\R\in y$ above, we have
\begin{equation*}
X_\Z'=\bigl\{(e^y,e^{-y}):y\in\R\bigr\}\cong\R.
\end{equation*}

\noindent{\bf(iii)} When $P=\N^k\t\Z^{n-k}$, combining {\bf(i)\rm,\bf(ii)}, points of $X_P$ are monoid morphisms $x:P\ra\bigl([0,\iy),\cdot\bigr)$, which may be written uniquely in the form
\begin{equation*}
x(m_1,\ldots,m_n)=y_1^{m_1}\cdots y_k^{m_k}e^{m_{k+1}y_{k+1}+\cdots+m_ny_n}
\end{equation*}
for $(y_1,\ldots,y_n)\in[0,\iy)^k\t\R^{n-k}$. This identifies~$X_{\N^k\t\Z^{n-k}}\cong[0,\iy)^k\t\R^{n-k}$.
\smallskip

We will often use the identifications $X_\N\cong[0,\iy)$, $X_\Z\cong\R$ and $X_{\N^k\t\Z^{n-k}}\cong[0,\iy)^k\t\R^{n-k}=\R^n_k$ in {\bf(i)}--{\bf(iii)}. Using Proposition \ref{gc3prop3} we see that in each of {\bf(i)}--{\bf(iii)}, the topology on $X_P$, and the notions of smooth functions $U\ra\R$, $U\ra(0,\iy)$, $U\ra[0,\iy)$, agree with the usual topology and smooth functions (in the sense of \S\ref{gc21}) on $[0,\iy),\R,[0,\iy)^k\t\R^{n-k}$. Thus, the $X_P$ for general weakly toric monoids $P$ are a class of smooth spaces generalizing the spaces $[0,\iy)^k\t\R^{n-k}$ used as local models for manifolds with corners.
\label{gc3ex3}
\end{ex}

If $P,Q$ are weakly toric monoids then so is $P\t Q$, and monoid morphisms $P\t Q\ra\bigl([0,\iy),\cdot\bigr)$ are of the form $(p,q)\mapsto x(p)y(q)$, where $x:P\ra\bigl([0,\iy),\cdot\bigr)$ and $y:Q\ra\bigl([0,\iy),\cdot\bigr)$ are monoid morphisms. This gives a natural identification $X_{P\t Q}\cong X_P\t X_Q$. Using this and Example \ref{gc3ex3} we deduce:

\begin{lem} Let\/ $P$ be a weakly toric monoid. Then $P\cong P^\sh\t P^\t,$ where $P^\sh$ is a toric monoid and\/ $P^\t\cong\Z^l$ for $l\ge 0$. Hence $X_P\cong X_{P^\sh}\t X_{\Z^l}\cong X_{P^\sh}\t\R^l$. 

\label{gc3lem2}
\end{lem}

Thus, we can reduce from weakly toric to toric monoids $P$ by including products with $\R^l$ in the spaces $X_P$. A different way to reduce from weakly toric to toric monoids is to note that $\R^l$ is diffeomorphic to $(0,\iy)^l\subset[0,\iy)^l\cong X_{\N^l}$, so $X_P\cong X_{P^\sh}\t\R^l$ is diffeomorphic to an open subset in $X_{P_\sh}\t[0,\iy)^l\cong X_Q$, where $Q=P^\sh\t\N^l$ is toric, giving:

\begin{cor} Let $P$ be a weakly toric monoid. Then there exists a toric monoid\/ $Q$ and an open subset\/ $U_Q\subset X_Q$ such that\/ $X_P$ is diffeomorphic to\/~$U_Q$.

\label{gc3cor1}
\end{cor}

The next proposition describes the interior $X_P^\ci$ of $X_P$.

\begin{prop} Let\/ $P$ be a weakly toric monoid, so that the interior\/ $X_P^\ci$ of\/ $X_P$ is an open subset of\/ $X_P$. Set\/ $n=\rank P$. Then: 
\begin{itemize}
\setlength{\itemsep}{0pt}
\setlength{\parsep}{0pt}
\item[{\bf(a)}] $X_P^\ci$ is diffeomorphic in the sense of Definition\/ {\rm\ref{gc3def5}} to $\R^n\cong X_{\Z^n}$. 
\item[{\bf(b)}] $X_P^\ci$ is the subset of points\/ $x\in X_P$ which have an open neighbourhood in $X_P$ homeomorphic to an open ball in\/~$\R^n$.
\end{itemize}
\label{gc3prop4}
\end{prop}

\begin{proof} For (a), points of $X_P^\ci$ are monoid morphisms $x:P\ra\bigl((0,\iy),\cdot\bigr)$. As $\bigl((0,\iy),\cdot\bigr)$ is a group, any such morphism factorizes through the projection $P\ra P^\gp$, so points of $X_P^\ci$ correspond to group morphisms $P^\gp\ra\bigl((0,\iy),\cdot\bigr)$. But $P^\gp\cong\Z^n$, as $P$ is weakly toric of rank $n$, and monoid morphisms $\Z^n\ra\bigl((0,\iy),\cdot\bigr)$ are points of $X_{\Z^n}\cong\R^n$. Thus, a choice of isomorphism $P^\gp\cong\Z^n$ induces an identification $X_P^\ci\cong\R^n\cong X_{\Z^n}$, and it is easy to see that this is a diffeomorphism in the sense of Definition~\ref{gc3def5}.

For (b), if $x\in X_P^\ci$, part (a) implies that $X_P$ is locally homeomorphic to $\R^n$ near $x$. And if $x\in X_P\sm X_P^\ci$ then using Proposition \ref{gc3prop3}(a) we can show that $X_P$ is not locally homeomorphic to $\R^n$ near $x$.
\end{proof}

\subsection{\texorpdfstring{The category $\Mangc$ of manifolds with g-corners}{The category of manifolds with g-corners}}
\label{gc33}

We can now define the category $\Mangc$ of {\it manifolds with generalized corners}, or {\it g-corners}, extending Definition \ref{gc2def2} for the case of ordinary corners.

\begin{dfn} Let $X$ be a second countable Hausdorff topological space. An {\it $n$-dimensional generalized chart}, or {\it g-chart}, on $X$ is a triple $(P,U,\phi)$, where $P$ is a weakly toric monoid with $\rank P=n$, and $P$ is a submonoid of $\Z^k$ for some $k\ge 0$, and $U\subseteq X_P$ is open, for $X_P$ as in \S\ref{gc32}, and $\phi:U\ra X$ is a homeomorphism with an open set $\phi(U)$ in~$X$.

Let $(P,U,\phi),(Q,V,\psi)$ be $n$-dimensional g-charts on $X$. We call
$(P,U,\phi)$ and $(Q,V,\psi)$ {\it compatible\/} if
$\psi^{-1}\ci\phi:\phi^{-1}\bigl(\phi(U)\cap\psi(V)\bigr)\ra
\psi^{-1}\bigl(\phi(U)\cap\psi(V)\bigr)$ is a diffeomorphism between
open subsets of $X_P,X_Q$, in the sense of Definition~\ref{gc3def5}.

An $n$-{\it dimensional generalized atlas}, or {\it g-atlas}, for $X$ is a family $\{(P^i,U^i,\phi^i):i\!\in\! I\}$ of pairwise compatible $n$-dimensional g-charts on $X$ with $X\!=\!\bigcup_{i\in I}\phi^i(U^i)$. We call such a g-atlas {\it maximal\/} if it is not a proper subset of any other g-atlas. Any g-atlas $\{(P^i,U^i,\phi^i):i\in I\}$ is contained in a unique maximal g-atlas, the family of all g-charts $(P,U,\phi)$ on $X$ compatible with $(P^i,U^i,\phi^i)$ for all~$i\in I$.

An $n$-{\it dimensional manifold with generalized corners}, or {\it g-corners}, is a second countable Hausdorff topological space $X$ with a maximal $n$-dimensional g-atlas. Usually we refer to $X$ as the manifold, leaving the g-atlas implicit. By a {\it g-chart\/ $(P,U,\phi)$ on\/} $X$, we mean an element of the maximal g-atlas. Write~$\dim X=n$.

Motivated by Proposition \ref{gc3prop4}(b), define the {\it interior\/} $X^\ci$ of an $n$-manifold with g-corners $X$ to be the dense open subset $X^\ci\subset X$ of points $x\in X$ which have an open neighbourhood in $X$ homeomorphic to an open ball in $\R^n$. Then Proposition \ref{gc3prop4} implies that if $(P,U,\phi)$ is a g-chart on $X$ then $\phi^{-1}(X^\ci)=U^\ci$, where $U^\ci\subseteq U\subseteq X_P$ is as in Definition \ref{gc3def5}, so $(P,U^\ci,\phi)$ is a g-chart on~$X^\ci$.

Let $X,Y$ be manifolds with g-corners, and $f:X\ra Y$ a continuous map of the underlying topological spaces. We say that $f:X\ra Y$ is {\it smooth\/} if for all g-charts $(P,U,\phi)$ on $X$ and $(Q,V,\psi)$ on $Y$, the map
\e
\psi^{-1}\ci f\ci\phi: (f\ci\phi)^{-1}(\psi(V))\longra V
\label{gc3eq5}
\e
is a smooth map between the open subsets $(f\ci\phi)^{-1}(\psi(V))\subseteq U\subseteq X_P$ and $V\subseteq X_Q$, in the sense of Definition~\ref{gc3def5}. 

This condition is local in $X$ and $Y$, and it holds locally in some charts $(P,U,\phi)$ on $X$ and $(Q,V,\psi)$ on $Y$ if and only if it holds on compatible charts $(P',U',\phi')$, $(Q',V',\psi')$ covering the same open sets in $X,Y$. Thus, to show $f:X\ra Y$ is smooth, it suffices to check \eq{gc3eq5} is smooth only for $(P,U,\phi)$ in some choice of g-atlas $\{(P^i,U^i,\phi^i):i\in I\}$ for $X$ and for $(Q,V,\psi)$ in some choice of g-atlas $\{(Q^j,V^j,\psi^j):j\in J\}$ for $Y$, rather than for all~$(P,U,\phi),(Q,V,\psi)$.

We say that $f:X\ra Y$ is a {\it diffeomorphism\/} if it is a bijection, and both $f:X\ra Y$, $f^{-1}:Y\ra X$ are smooth.

We say that a smooth map $f:X\ra Y$ is {\it interior\/} if $f(X^\ci)\subseteq Y^\ci$. Equivalently, $f$ is interior if the maps \eq{gc3eq5} are interior in the sense of Definition \ref{gc3def5} for all $(P,U,\phi),(Q,V,\psi)$.

In Definition \ref{gc3def5} we saw that for open $U\subseteq X_P$, $V\subseteq X_Q$, $W\subseteq X_R$, compositions $g\ci f$ of smooth (or interior) maps $f:U\ra V$, $g:V\ra W$ are smooth (or interior), and identity maps $\id_U:U\ra U$ are smooth (and interior). It easily follows that compositions $g\ci f:X\ra Z$ of smooth (or interior) maps $f:X\ra Y$, $g:Y\ra Z$ of manifolds with g-corners are smooth (or interior), and identity maps $\id_X:X\ra X$ are smooth (and interior).

Thus, manifolds with g-corners and smooth maps, or interior maps, form a category. Write $\Mangc$ for the category with objects manifolds with g-corners $X,Y$ and morphisms smooth maps $f:X\ra Y$, and $\Mangcin\subset\Mangc$ for the (non-full) subcategory with objects manifolds with g-corners $X,Y$ and morphisms interior maps~$f:X\ra Y$.

Write $\cMangc$ for the category whose objects are disjoint unions $\coprod_{m=0}^\iy X_m$, where $X_m$ is a manifold with g-corners of dimension $m$, allowing $X_m=\es$, and whose morphisms are continuous maps $f:\coprod_{m=0}^\iy X_m\ra\coprod_{n=0}^\iy Y_n$, such that $f\vert_{X_m\cap f^{-1}(Y_n)}:X_m\cap f^{-1}(Y_n)\ra Y_n$ is a smooth map of manifolds with g-corners for all $m,n\ge 0$. Objects of $\cMangc$ will be called {\it manifolds with g-corners of mixed dimension}. We regard $\Mangc$ as a full subcategory of $\cMangc$ in the obvious way. Write $\cMangcin$ for the (non-full) subcategory of $\cMangc$ with the same objects, and morphisms $f:\coprod_{m=0}^\iy X_m\ra\coprod_{n=0}^\iy Y_n$ with $f\vert_{X_m\cap f^{-1}(Y_n)}$ an interior map for all~$m,n$.

Alternatively, we can regard $\cMangc,\cMangcin$ as the categories defined exactly as for $\Mangc,\Mangcin$ above, except that in defining g-atlases $\{(P^i,U^i,\phi^i):i\in I\}$ on $X$, we omit the condition that all charts $(P^i,U^i,\phi^i)$ in the g-atlas must have the same dimension~$\rank P^i=n$.
\label{gc3def6}
\end{dfn}

\begin{rem}{\bf(a)} Section \ref{gc32} and Definition \ref{gc3def6} were motivated by Kottke and Melrose's {\it interior binomial varieties\/} \cite[\S 9]{KoMe}.

In fact Kottke and Melrose do rather less than we do: they define interior binomial subvarieties $X$ only as subsets $X\subset Y$ of an ambient manifold with corners $Y$, rather than as geometric spaces in their own right. Their local models for the inclusion $X\subset Y$ are essentially the same as our inclusion $X_P'\subset[0,\iy)^m$ in Proposition \ref{gc3prop3}(a), and they do not highlight the fact that $X_P$ really depends only on the monoid $P$, and not on the embedding $X_P\hookra[0,\iy)^m$. Nonetheless, it seems clear that Kottke and Melrose could have written down a definition equivalent to Definition \ref{gc3def6}, if they had wanted to.

Our $\Mangcin$ is equivalent to a full subcategory of Gillam and Molcho's category of {\it positive log differentiable spaces},~\cite[\S 6]{GiMo}.
\smallskip

\noindent{\bf(b)} In the definition of g-charts $(P,U,\phi)$ above, we require that the weakly toric monoid $P$ {\it is a submonoid of\/} $\Z^k$ for some $k\ge 0$. As in \S\ref{gc31}, every weakly toric monoid $P$ is isomorphic to a submonoid of some $\Z^k$, so this does not restrict $P$ up to isomorphism. We assume it for set theory reasons, as if we did not then the maximal g-atlas $\{(P^i,U^i,\phi^i):i\in I\}$ of all g-charts $(P,U,\phi)$ on a manifold with g-corners $X$ would not be a set, but only a proper class. We will generally ignore this issue.
\smallskip

\noindent{\bf(c)} As in Remark \ref{gc2rem2} for manifolds with (ordinary) corners, we can also define {\it real analytic\/} manifolds with g-corners, and {\it real analytic\/} maps between them. To do this, in Definition \ref{gc3def5}, if $P$ is a weakly toric monoid and $U\subseteq X_P$ is open, we call a continuous function $f:U\ra\R$ {\it real analytic\/} if there exist $r_1,\ldots,r_n\in P$, an open subset $W\subseteq\R^n$, and a real analytic map $g:W\ra\R$ (i.e. the Taylor series of $g$ at $w$ converges to $g$ near $w$ for all $w\in W$), such that for all $x\in U$ we have $\bigl(x(r_1),\ldots,x(r_n)\bigr)\in W$ and \eq{gc3eq2} holds. 

If $Q$ is another weakly toric monoid, $V\subseteq X_Q$ is open, and $f:U\ra V$ is smooth in the sense of Definition \ref{gc3def5}, we say that $f$ is {\it real analytic\/} if $\la_q\ci f:U\ra\R$ is real analytic in the sense above for all $q\in Q$.

Then we define $\{(P^i,U^i,\phi^i):i\in I\}$ to be a {\it real analytic g-atlas\/} on a topological space $X$ as in Definition \ref{gc3def6}, except that we require the transition functions $(\phi^j)^{-1}\ci\phi^i$ for $i,j\in I$ to be real analytic rather than smooth. We define a {\it real analytic manifold with g-corners\/} to be a Hausdorff, second countable topological space $X$ equipped with a maximal real analytic g-atlas. 

Given real analytic manifolds with g-corners $X,Y$, we define a continuous map $f:X\ra Y$ to be {\it real analytic\/} if whenever $(P,U,\phi),(Q,V,\psi)$ are real analytic g-charts on $X,Y$ (that is, g-charts in the maximal real analytic g-atlases), the transition map $\psi^{-1}\ci f\ci\phi$ in \eq{gc3eq5} is a real analytic map between open subsets of $X_P,X_Q$ in the sense above. Then real analytic manifolds with g-corners and real analytic maps between them form a category~$\Mangcra$. 

There is an obvious faithful functor $F_\Mangcra^\Mangc:\Mangcra\ra\Mangc$, which on objects replaces the maximal real analytic g-atlas by the (larger) corresponding maximal smooth g-atlas containing it. Essentially all the material we discuss for manifolds with g-corners also works for real analytic manifolds with g-corners, except for constructions requiring partitions of unity.
\label{gc3rem2}
\end{rem}

\begin{ex} Let $P$ be a weakly toric monoid. Then $X_P$ is a manifold with g-corners, of dimension $\rank P$, covered by the single g-chart~$(P,X_P,\id_{X_P})$.

Let $\mu:Q\ra P$ be a morphism of weakly toric monoids. Define $X_\mu:X_P\ra X_Q$ by $X_\mu(x)=x\ci\mu$, noting that points $x\in X_P$ are monoid morphisms $x:P\ra \bigl([0,\iy),\cdot\bigr)$. It is easy to show that $X_\mu:X_P\ra X_Q$ is a smooth, interior map of manifolds with g-corners, and we have a functor $X:(\Monwt)^{\bf op}\ra\Mangc$ mapping $P\mapsto X_P$ on objects and $\mu\mapsto X_\mu$ on morphisms.
\label{gc3ex4}
\end{ex}

We relate manifolds with g-corners to manifolds with corners in~\S\ref{gc2}.

\begin{dfn} Let $X$ be an $n$-manifold with (ordinary) corners, in the sense of \S\ref{gc21}. Then $X$ has a maximal atlas of charts $(U,\phi)$, where $U\subseteq \R^n_k=[0,\iy)^k\t\R^{n-k}$ is open and $\phi:U\ra X$ is a homeomorphism with an open set $\phi(U)\subseteq X$, as in Definition \ref{gc2def2}. We can turn $X$ into a manifold with g-corners as follows. Let $(U,\phi)$ be a chart on $X$ with $U\subseteq[0,\iy)^k\t\R^{n-k}$ open. As in Example \ref{gc3ex3}(iii) we identify $X_{\N^k\t\Z^{n-k}}\cong[0,\iy)^k\t\R^{n-k}$, so we may regard $U$ as an open set in $X_{\N^k\t\Z^{n-k}}$, and thus $(\N^k\t\Z^{n-k},U,\phi)$ is a g-chart on~$X$. 

If $(V,\psi)$ is another chart on $X$ and $(\N^l\t\Z^{n-l},V,\psi)$ the corresponding g-chart, then $(U,\phi),(V,\psi)$ compatible in the sense of Definition \ref{gc2def2} implies that $(\N^k\t\Z^{n-k},U,\phi)$, $(\N^l\t\Z^{n-l},V,\psi)$ are compatible g-charts. Hence the maximal atlas of charts $(U,\phi)$ on $X$ induces a g-atlas of g-charts $(\N^k\t\Z^{n-k},U,\phi)$ on $X$, which is a subatlas of a unique maximal g-atlas of g-charts $(P,U,\phi)$ on $X$, making $X$ into a manifold with g-corners, which we temporarily write as~$\hat X$.

Thus, every manifold with corners $X$ may be given the structure of a manifold with g-corners $\hat X$. If $X,Y$ are manifolds with corners and $\hat X,\hat Y$ the corresponding manifolds with g-corners, then Proposition \ref{gc3prop3}(c) implies that a map $f:X\ra Y$ is a smooth map of manifolds with corners, as in \S\ref{gc21}, if and only $f:\hat X\ra\hat Y$ is a smooth map of manifolds with g-corners, in the sense above.

Define $F_\Manc^\Mangc:\Manc\hookra\Mangc$ by $F_\Manc^\Mangc:X\mapsto\hat X$ on objects and $F_\Manc^\Mangc:f\mapsto f$ on morphisms. Then $F_\Manc^\Mangc$ is full and faithful, and embeds the category $\Manc$ from \S\ref{gc2} as a full subcategory of the category $\Mangc$ above. Also $F_\Manc^\Mangc$ takes interior maps in $\Manc$ to interior maps in $\Mangc$, and so restricts to a full and faithful embedding~$F_\Mancin^\Mangcin:\Mancin\hookra\Mangcin$.

Similarly, we regard $\cManc$ in \S\ref{gc21} as a full subcategory of $\cMangc$ above.

Let $\hat X$ be an $n$-manifold with g-corners. Then $\hat X=F_\Manc^\Mangc(X)$ for some $n$-manifold with corners $X$ if and only if $\hat X$ admits a cover by g-charts of the form $(\N^k\t\Z^{n-k},U,\phi)$, and then the maximal atlas for $X$ is the family of all $(U,\phi)$ with $(\N^k\t\Z^{n-k},U,\phi)$ a g-chart on~$X$.

From this we see that the subcategory $F_\Manc^\Mangc(\Manc)$ in $\Mangc$ is closed under isomorphisms in $\Mangc$ (it is {\it strictly full\/}), and is strictly isomorphic (not just equivalent) to $\Manc$. We will often identify $\Manc$ with its image $F_\Manc^\Mangc(\Manc)$ in $\Mangc$, and regard $\Manc$ as a subcategory of $\Mangc$ (and similarly $\Mancin$ as a subcategory of $\Mangcin\subset\Mangc$), and manifolds with corners as special examples of manifolds with g-corners. Since the only difference between a manifold with corners $X$ and the corresponding manifold with g-corners $\hat X$ is the maximal atlas $\{(U^i,\phi^i):i\in I\}$ on $X$ or g-atlas $\{(P^i,U^i,\phi^i):i\in\hat I\}$ on $\hat X$, and we rarely write these (g-)atlases down, this identification should not cause confusion.

As in \S\ref{gc2}, we have full subcategories $\Man,\Manb\subset\Manc$ of {\it manifolds without boundary\/} and {\it manifolds with boundary}, and non-full subcategories $\Mancst,\ab\Mancis\subset\Manc$ of {\it strongly smooth\/} and {\it interior strongly smooth\/} morphisms in $\Manc$. We consider all of these as subcategories of $\Mangc$. If $X$ is any manifold with g-corners then $X^\ci$ is a manifold without boundary, that is, $X^\ci\in\Man\subset\Mangc$.
\label{gc3def7}
\end{dfn}

\begin{ex} The $X_P$ we now describe is the simplest example of a manifold with g-corners which is not a manifold with corners. We will return to this example several times to illustrate parts of the theory. Define 
\begin{equation*}
P=\bigl\{(a,b,c)\in \Z^3:a\ge 0,\; b\ge 0,\; a+b\ge c\ge 0\bigr\}.
\end{equation*}
Then $P$ is a toric monoid with rank 3, with $P^\gp=\Z^3\supset P$. Write
\e
p_1=(1,0,0),\;\> p_2=(0,1,1),\;\> p_3=(0,1,0),\;\> p_4=(1,0,1).
\label{gc3eq6}
\e
Then $p_1,p_2,p_3,p_4$ are generators for $P$, subject to the single relation
\begin{equation*}
p_1+p_2=p_3+p_4.
\end{equation*}
Thus Proposition \ref{gc3prop3}(a) shows that
\e
X_P\cong X_P'=\bigl\{(x_1,x_2,x_3,x_4)\in[0,\iy)^4:x_1x_2=x_3x_4\bigr\}.
\label{gc3eq7}
\e

\begin{figure}[htb]
\centerline{$\splinetolerance{.8pt}
\begin{xy}
0;<1mm,0mm>:
,(0,0)*{\bu}
,(14,0)*{(0,0,0,0)\cong\de_0}
,(-30,-20)*{\bu}
,(-39,-18)*{(x_1,0,0,0)}
,(30,-20)*{\bu}
,(39.5,-18)*{(0,0,x_3,0)}
,(-30,-30)*{\bu}
,(-39,-32)*{(0,0,0,x_4)}
,(30,-30)*{\bu}
,(39.5,-32)*{(0,x_2,0,0)}
,(0,-30)*{\bu}
,(0,-33)*{(0,x_2,0,x_4)}
,(0,-20)*{\bu}
,(0,-17)*{(x_1,0,x_3,0)}
,(-30,-25)*{\bu}
,(-40.5,-25)*{(x_1,0,0,x_4)}
,(30,-25)*{\bu}
,(40.5,-25)*{(0,x_2,x_3,0)}
,(0,0);(-45,-30)**\crv{}
?(.8444)="aaa"
?(.85)="bbb"
?(.75)="ccc"
?(.65)="ddd"
?(.55)="eee"
?(.45)="fff"
?(.35)="ggg"
?(.25)="hhh"
?(.15)="iii"
?(.05)="jjj"
,(0,0);(45,-30)**\crv{}
?(.8444)="aaaa"
?(.85)="bbbb"
?(.75)="cccc"
?(.65)="dddd"
?(.55)="eeee"
?(.45)="ffff"
?(.35)="gggg"
?(.25)="hhhh"
?(.15)="iiii"
?(.05)="jjjj"
,(0,0);(-40,-40)**\crv{}
?(.95)="a"
?(.85)="b"
?(.75)="c"
?(.65)="d"
?(.55)="e"
?(.45)="f"
?(.35)="g"
?(.25)="h"
?(.15)="i"
?(.05)="j"
,(0,0);(40,-40)**\crv{}
?(.95)="aa"
?(.85)="bb"
?(.75)="cc"
?(.65)="dd"
?(.55)="ee"
?(.45)="ff"
?(.35)="gg"
?(.25)="hh"
?(.15)="ii"
?(.05)="jj"
,"a";"aa"**@{.}
,"b";"bb"**@{.}
,"c";"cc"**@{.}
,"d";"dd"**@{.}
,"e";"ee"**@{.}
,"f";"ff"**@{.}
,"g";"gg"**@{.}
,"h";"hh"**@{.}
,"i";"ii"**@{.}
,"j";"jj"**@{.}
,"a";"aaa"**@{.}
,"b";"bbb"**@{.}
,"c";"ccc"**@{.}
,"d";"ddd"**@{.}
,"e";"eee"**@{.}
,"f";"fff"**@{.}
,"g";"ggg"**@{.}
,"h";"hhh"**@{.}
,"i";"iii"**@{.}
,"j";"jjj"**@{.}
,"aa";"aaaa"**@{.}
,"bb";"bbbb"**@{.}
,"cc";"cccc"**@{.}
,"dd";"dddd"**@{.}
,"ee";"eeee"**@{.}
,"ff";"ffff"**@{.}
,"gg";"gggg"**@{.}
,"hh";"hhhh"**@{.}
,"ii";"iiii"**@{.}
,"jj";"jjjj"**@{.}
,(-30,-20);(30,-20)**@{--}
,(-30,-20);(-30,-30)**\crv{}
,(-30,-30);(30,-30)**\crv{}
,(30,-30);(30,-20)**\crv{}
\end{xy}$}
\caption{3-manifold with g-corners $X_P'\cong X_P$ in \eq{gc3eq7}}
\label{gc3fig1}
\end{figure}
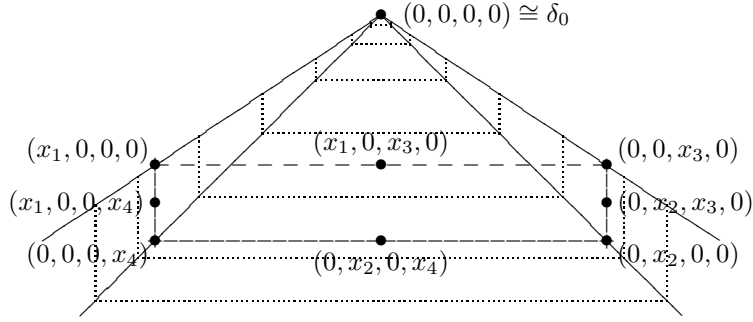

We sketch $X_P'$ in Figure \ref{gc3fig1}. We can visualize $X_P\cong X_P'$ as a 3-dimensional infinite pyramid on a square base. Using the ideas of \S\ref{gc34}, $X_P'$ has one vertex $(0,0,0,0)$ corresponding to $\de_0\in X_P$ mapping $\de_0:P\ra\bigl([0,\iy),\cdot\bigr)$ with $\de_0(p)=1$ if $p=(0,0,0)$ and $\de_0(p)=0$ otherwise, four 1-dimensional edges of points $(x_1,0,0,0),(0,x_2,0,0), (0,0,x_3,0),(0,0,0,x_4)$, four 2-dimensional faces of points $(x_1,0,x_3,0)$, $(x_1,0,0,x_4)$, $(0,x_2,x_3,0)$, $(0,x_2,0,x_4)$, and an interior $X_P^{\prime\ci}\ab\cong\R^3$ of points $(x_1,x_2,x_3,x_4)$. Then $X_P\sm\{\de_0\}$ is a 3-manifold with corners, but $X_P$ is not a manifold with corners near $\de_0$, as we can see from the non-simplicial face structure.
\label{gc3ex5}
\end{ex}

\begin{rem} Looking at Figure \ref{gc3fig1}, it is tempting to try and identify $X_P$ in Example \ref{gc3ex5} with a polyhedron in $\R^3$, with four linear faces, and one vertex like a corner of an octahedron. However, this is a mistake. Although the combinatorics of the edges, faces, etc.\ of $X_P$ are those of a polyhedron in $\R^3$, the smooth structure near $(0,0,0,0)$ is different to that of a polyhedron.
\label{gc3rem3}
\end{rem}

\begin{dfn} If $P,Q$ are weakly toric monoids then $P\t Q$ is a weakly toric monoid, and $X_{P\t Q}\cong X_P\t X_Q$. Thus, the class of local models for manifolds with g-corners is closed under products. Therefore, if $X,Y$ are manifolds with g-corners, we can give the product $X\t Y$ the structure of a manifold with g-corners, such that if $X,Y$ are locally modelled on $X_P,X_Q$ near $x,y$ then $X\t Y$ is locally modelled on $X_{P\t Q}$ near $(x,y)$. That is, if $(P,U,\phi)$ and $(Q,V,\psi)$ are g-charts on $X,Y$ then $(P\t Q,U\t V,\phi\t\psi)$ is a g-chart on $X\t Y$, identifying $U\t V\subseteq X_P\t X_Q$ with an open set in~$X_{P\t Q}\cong X_P\t X_Q$.

There are also two notions of product morphism in $\Mangc$: if $f:W\ra Y$ and $g:X\ra Z$ are smooth (or interior) maps of manifolds with g-corners then the {\it product\/} $f\t g:W\t X\ra Y\t Z$ mapping $f\t g:(w,x)\mapsto(f(w),g(x))$ is smooth (or interior), and if $f:X\ra Y,$ $g:X\ra Z$ are smooth (or interior) maps of manifolds with g-corners then the {\it direct product\/} $(f,g):X\ra Y\t Z$ mapping $(f,g):x\mapsto(f(x),g(x))$ is smooth (or interior).
\label{gc3def8}
\end{dfn}

\subsection{\texorpdfstring{Boundaries $\pd X$, corners $C_k(X)$, and the corner functor}{Boundaries, corners, and the corner functor}}
\label{gc34}

In Definition \ref{gc2def3} we defined the {\it depth stratification\/} $X=\coprod_{l=0}^{\dim X}S^l(X)$ of a manifold with corners $X$. We now generalize this to manifolds with g-corners.

\begin{dfn} Let $P$ be a weakly toric monoid, and $F$ a face of $P$, as in \S\ref{gc315}. For $X_F,X_P$ as in \S\ref{gc32}, define an inclusion map $i_F^P:X_F\hookra X_P$ by $i_F^P(y)=\bar y$, where $y\in X_F$ so that $y:F\ra[0,\iy)$ is a monoid morphism, and $\bar y:P\ra[0,\iy)$ is defined by
\begin{equation*}
\bar y(p)=\begin{cases} y(p), & p\in F, \\ 0, & p\in P\sm F. \end{cases}
\end{equation*}
The condition in Definition \ref{gc3def4} that if $p,q\in P$ with $p+q\in F$ then $p,q\in F$ implies that $\bar y$ is a monoid morphism, so $\bar y\in X_P$. Then $i_F^P:X_F\ra X_P$ is a smooth, injective map of manifolds with g-corners.

For each $x\in X_P$, define the {\it support\/} of $x$ to be
\begin{equation*}
\supp x=\bigl\{p\in P: x(p)\ne 0\bigr\}.
\end{equation*}
It is easy to see that $\supp x$ is a face of $P$. For each face $F$ of $P$, write
\begin{equation*}
X^P_F=\bigl\{x\in X_P:\supp x=F\bigr\}.
\end{equation*}
Then the interior $X_P^\ci$ is $X^P_P$, and we have a decomposition
\e
X_P=\coprod\nolimits_{\text{faces $F$ of $P$}}X^P_F.
\label{gc3eq8}
\e
From the definition of $i_F^P:X_F\hookra X_P$, it is easy to see that
\e
X^P_F=i_F^P(X_F^\ci)\quad\text{and}\quad \ov{X^P_F}=i_F^P(X_F)=\coprod\nolimits_{\text{faces $G$ of $P$ with $G\subseteq F$}}X^P_G,
\label{gc3eq9}
\e
where $\ov{X^P_F}$ is the closure of $X^P_F$ in $X_P$. By Proposition \ref{gc3prop4}(a) we have a diffeomorphism $X^P_F\cong X_F^\ci \cong \R^{\rank F}=\R^{\rank P-\codim F}$. Thus \eq{gc3eq8} is a locally closed stratification of $X_P$ into smooth manifolds without boundary.

For $x\in X_P$, define the {\it depth\/} $\depth_{X_P}x$ to be $\codim(\supp x)=\rank P-\rank(\supp x)$, so that $\depth_{X_P}x=0,\ldots,\dim X_P$. For each $l=0,\ldots,\dim X_P$, define the {\it depth\/ $l$ stratum\/} of $X_P$ to be
\begin{equation*}
S^l(X_P)=\bigl\{x\in X_P:\depth_{X_P}x=l\bigr\}.
\end{equation*}
Then the interior $X_P^\ci$ is $S^0(X_P)$, and
\e
S^l(X_P)=\coprod\nolimits_{\text{faces $F$ of $P$: $\codim F=l$}}X^P_F,
\label{gc3eq10}
\e
so that $S^l(X)$ is a smooth manifold without boundary of dimension $\dim X_P-l$, and \eq{gc3eq9} implies that $\overline{S^l(X_P)}=\bigcup_{k=l}^{\dim X_P}S^k(X_P)$. Hence
\begin{equation*}
X_P=\coprod\nolimits_{l=0}^{\dim X_P}S^l(X_P)
\end{equation*}
is a locally closed stratification of $X_P$ into smooth manifolds without boundary.

If $U\subseteq X_P$ is an open set, we define $S^l(U)=U\cap S^l(X_P)=\bigl\{u\in U:\depth_{X_P}u=l\bigr\}$ for $l=0,\ldots,\dim X_P=\dim U$. Then~$U=\coprod_{l=0}^{\dim U}S^l(U)$.

As in Definition \ref{gc2def8}(b), for $x\in X_P$ write $\cI_x(X_P)$ for the set of germs $[b]$ at $x$ of interior maps $b:X_P\ra[0,\iy)$. It is a monoid under multiplication. Using the notation of \S\ref{gc31}, the units $\cI_x(X_P)^\t$ are germs $[b]$ with $b(x)>0$, and $\cI_x(X_P)^\sh=\cI_x(X_P)/\cI_x(X_P)^\t$. Consider the monoid morphism
\begin{equation*}
\Pi_x:P\longra \cI_x(X_P)^\sh, \quad \Pi_x:p\longmapsto [\la_p]\cdot\cI_x(X_P)^\t.
\end{equation*}
Using Definition \ref{gc2def5}, we see that $\Pi$ is surjective, with kernel $\supp x$. Therefore
\begin{equation*}
P/\supp x\cong \cI_x(X_P)^\sh.
\end{equation*}
Thus $\cI_x(X_P)^\sh$ is a toric monoid, with
\begin{equation*}
\rank \bigl(\cI_x(X_P)^\sh\bigr)=\rank P-\rank(\supp x)=\depth_{X_P}x.
\end{equation*}
Hence if $U\subseteq X_P$ is open then for $l=0,\ldots,\dim U$ we have
\e
S^l(U)=\bigl\{u\in U:\rank\bigl(\cI_x(U)^\sh\bigr)=l\bigr\}.
\label{gc3eq11}
\e

Now \eq{gc3eq11} depends only on $U$ as a manifold with g-corners, rather than as an open subset of some $X_P$. It follows that {\it the depth stratification $U=\coprod_{l=0}^{\dim U}S^l(U)$ is invariant under diffeomorphisms}. That is, if $P,Q$ are weakly toric monoids with $\rank P=\rank Q$, and $U\subseteq X_P$, $V\subseteq X_Q$ are open, and $f:U\ra V$ is a diffeomorphism in the sense of \S\ref{gc32}, then $f\bigl(S^l(U)\bigr)=S^l(V)$ for $l=0,\ldots,\dim U=\dim V$.

Let $X$ be a manifold with g-corners. For $x\in X$, choose a g-chart $(P,U,\phi)$ on the manifold $X$ with $\phi(u)=x$ for $u\in U$, and define the {\it depth\/} $\depth_Xx$ of $x$ in $X$ by $\depth_Xx=\depth_{X_P}u$. This is independent of the choice of $(P,U,\phi)$, by invariance of the depth stratification under diffeomorphisms. For each $l=0,\ldots,\dim X$, define the {\it depth\/ $l$ stratum\/} of $X$ to be
\begin{equation*}
S^l(X)=\bigl\{x\in X:\depth_Xx=l\bigr\}.
\end{equation*}
Then $X=\coprod_{l=0}^{\dim X}S^l(X)$. Each $S^l(X)$ is a manifold without boundary of dimension $\dim X-l$, with $S^0(X)=X^\ci$, and $\overline{S^l(X)}=\bigcup_{k=l}^{\dim X}S^k(X)$, since this holds for the stratifications of the local models $U\subseteq X_P$.
\label{gc3def9}
\end{dfn}

\begin{ex} Let $P=\N^k\t\Z^{n-k}$, and identify $X_P$ with $\R^n_k=[0,\iy)^k\t\R^{n-k}$ as in Example \ref{gc3ex3}(iii). Then faces $F$ of $P$ are in 1-1 correspondence with subsets $I\subseteq\{1,\ldots,k\}$, where the face $F_I$ corresponding to a subset $I$ is
\begin{equation*}
F_I=\bigl\{(a_1,\ldots,a_n)\in \N^k\t\Z^{n-k}:\text{$a_i=0$ for $i\in I$}\bigr\},
\end{equation*}
so that $\rank F_I=n-\md{I}$ and $\codim F_I=\md{I}$. We can show that
\begin{equation*}
X^P_{F_I}=\bigl\{(x_1,\ldots,x_n)\in\R^n_k:\text{$x_i=0$, $i\in I$, and $x_j\ne 0$, $j\in\{1,\ldots,k\}\sm I$}\bigr\},
\end{equation*}
so that $X^P_{F_I}\cong (0,\iy)^{k-\md{I}}\t\R^{n-k}\cong \R^{n-\md{I}}$. Thus, for $(x_1,\ldots,x_n)\in X_P\cong\R^n_k$, $\depth_{X_P}x$ in Definition \ref{gc3def9} is the number of $x_1,\ldots,x_k$ which are zero, and
\begin{equation*}
S^l(\R^n_k)=\bigl\{(x_1,\ldots,x_n)\in\R^n_k:\text{exactly $l$ out of $x_1,\ldots,x_k$ are zero}\bigr\}.
\end{equation*}

But this coincides with the definition of $\depth_{\R^n_k}x$ and $S^l(\R^n_k)$ in Definition \ref{gc2def3}. Therefore we deduce:
\label{gc3ex6}
\end{ex}

\begin{cor} Let\/ $X$ be a manifold with corners as in\/ {\rm\S\ref{gc2},} and regard\/ $X$ as a manifold with g-corners as in Definition\/ {\rm\ref{gc3def7}}. Then the two definitions of depth\/ $\depth_Xx$ for $x\in X,$ and of the depth stratification $X=\coprod_{l=0}^{\dim X}S^l(X),$ in Definitions\/ {\rm\ref{gc2def3}} and\/ {\rm\ref{gc3def9}} agree.
\label{gc3cor2}
\end{cor}

Following Definition \ref{gc2def4} closely, we define boundaries $\pd X$ and corners $C_k(X)$ of manifolds with g-corners.

\begin{dfn} Let $X$ be an $n$-manifold with g-corners, $x\in X$, and $k=0,1,\ldots,n$. A {\it local $k$-corner component\/ $\ga$ of\/ $X$ at\/} $x$ is a local choice of connected component of $S^k(X)$ near $x$. That is, for each sufficiently small open neighbourhood $V$ of $x$ in $X$, $\ga$ gives a choice of connected component $W$ of $V\cap S^k(X)$ with $x\in\overline W$, and any two such choices $V,W$ and $V',W'$ must be compatible in that~$x\in\overline{(W\cap W')}$. When $k=1$, we also call local 1-corner components {\it local boundary components of\/ $X$ at\/} $x$. 

As sets, define the {\it boundary\/} $\pd X$ and {\it k-corners\/} $C_k(X)$ for $k=0,1,\ldots,n$ by
\begin{align*}
\pd X&=\bigl\{(x,\be):\text{$x\!\in\! X$, $\be$ is a local boundary
component of $X$ at $x$}\bigr\},\\
C_k(X)&=\bigl\{(x,\ga):\text{$x\in X$, $\ga$ is a local $k$-corner 
component of $X$ at $x$}\bigr\},
\end{align*}
so that $\pd X=C_1(X)$. Since each $x\in X$ has a unique 0-boundary component $[X^\ci]$, we have $C_0(X)\cong X$. Define maps $i_X:\pd X\ra X$, $\Pi:C_k(X)\ra X$, $\io:X\ra C_0(X)$ by $i_X:(x,\be)\mapsto x$, $\Pi:(x,\ga)\mapsto x$ and $\io:x\mapsto(x,[X^\ci])$.

We will explain how to give $\pd X,C_k(X)$ the structure of manifolds with g-corners, so that $i_X,\Pi,\io$ are smooth maps, with $\io$ a diffeomorphism. Let $(P,U,\phi)$ be a g-chart on $X$, and $u\in U\subseteq X_P$ with $\phi(u)=x\in X$. Then \eq{gc3eq10} gives
\begin{equation*}
S^l(U)=\coprod\nolimits_{\text{faces $F$ of $P$: $\codim F=k$}}X^P_F\cap U
\end{equation*}
As $X^P_F\cong\R^{n-k}$ is connected, and furthermore locally connected in $X_P$, we see that local $k$-corner components of $U$ at $u$ are in 1-1 correspondence with faces $F$ of $P$ with $\codim F=k$, such that $u\in\overline{X^P_F}$. Hence by \eq{gc3eq9}, local $k$-corner components of $U$ at $u$ are in 1-1 correspondence with faces $F$ of $P$ with $\codim F=k$ such that $u\in i_F^P(X_F)$. Thus, we have natural 1-1 correspondences
\ea
C_k&(X)\supseteq \Pi^{-1}(\phi(U))
\nonumber\\
&=\bigl\{(x,\ga):\text{$x\in \phi(U)\subseteq X$, $\ga$ is a local $k$-corner component of $X$ at $x$}\bigr\}
\nonumber\\
&\cong\bigl\{(u,\ga'):\text{$u\in U$, $\ga'$ is a local $k$-corner component of $U$ at $u$}\bigr\}
\nonumber\\
&\cong \coprod\nolimits_{\text{faces $F$ of $P$: $\codim F=k$}}(i_F^P)^{-1}(U),
\label{gc3eq12}
\ea
where $(i_F^P)^{-1}(U)\subseteq X_F$ is an open set.

For each face $F$ of $P$ with $\codim F=k$, let $\phi_F^P:(i_F^P)^{-1}(U)\ra\Pi^{-1}(\phi(U))\subseteq C_k(X)$ be the map determined by \eq{gc3eq12}. Then $\bigl(F,(i_F^P)^{-1}(U),\phi_F^P\bigr)$ is a g-chart of dimension $n-k$ on $C_k(X)$, and the union of these over all $F$ covers $\Pi^{-1}(\phi(U))$. If $(P',U',\phi')$ is another g-chart on $X$ then $(P,U,\phi)$, $(P',U',\phi')$ are compatible. Using this one can show that the g-charts $\bigl(F,(i_F^P)^{-1}(U),\phi_F^P\bigr)$ on $C_k(X)$ from $(P,U,\phi)$ and $\bigl(F',(i_{F'}^{P'})^{-1}(U'),\phi_{F'}^{P'}\bigr)$ from $(P',U',\phi')$ are pairwise compatible. Hence the collection of all g-charts $\bigl(F,(i_F^P)^{-1}(U),\phi_F^P\bigr)$ on $C_k(X)$ from all g-charts $(P,U,\phi)$ on $X$ is a g-atlas, where $C_k(X)$ has a unique Hausdorff topology such that $\phi_F^P$ is a homeomorphism with an open set for all such g-charts, and the corresponding maximal g-atlas makes $C_k(X)$ into an $(n-k)$-manifold with g-corners, and $\pd X=C_1(X)$ into an $(n-1)$-manifold with g-corners.
\label{gc3def10}
\end{dfn}

\begin{ex} Let $P$ be a weakly toric monoid, and take $X=X_P$ in Definition \ref{gc3def10}, so that $X$ is covered by one g-chart $(P,X_P,\id_{X_P})$. Then taking $U=X_P$ in \eq{gc3eq12} gives a diffeomorphism
\e
C_k(X_P)\cong \coprod\nolimits_{\text{faces $F$ of $P$: $\codim F=k$}}X_F.
\label{gc3eq13}
\e

\label{gc3ex7}
\end{ex}

\begin{ex} Set $P=\bigl\{(a,b,c)\in \Z^3:a\ge 0,\; b\ge 0,\; a+b\ge c\ge 0\bigr\}$, as in Example \ref{gc3ex5}. Then $P$ has one face $F=P$ of codimension 0, four faces $F$ of codimension 1 all with $F\cong\N^2$, four faces $F$ of codimension 2 all with $F\cong\N$, and one face $F=\{0\}$ of codimension 3. Thus by \eq{gc3eq13} we have diffeomorphisms
\begin{gather*}
C_0(X_P)\cong X_P, \;\> C_1(X_P)\!=\!\pd X_P\!\cong\! [0,\iy)^2\!\amalg\! [0,\iy)^2\!\amalg \![0,\iy)^2\!\amalg\! [0,\iy)^2,\\
C_2(X_P)\cong [0,\iy)\amalg [0,\iy)\amalg [0,\iy)\amalg [0,\iy)\;\>\text{and}\;\> C_3(X_P)\cong *.
\end{gather*}
From these we deduce that
\begin{equation*}
\pd^2X_P=\text{8 copies of $[0,\iy)$,}\quad \pd^3X_P=\text{8 points.}
\end{equation*}

We use these to show that some results in \S\ref{gc22} for manifolds with corners are {\it false\/} for manifolds with g-corners. For a manifold with (ordinary) corners $X$, equations \eq{gc2eq5}, \eq{gc2eq7}, \eq{gc2eq8} and \eq{gc2eq9} say that
\ea
\begin{split}
C_k(X)&\cong\bigl\{(x,\{\be_1,\ldots,\be_k\}):\text{$x\in X,$
$\be_1,\ldots,\be_k$ are distinct}\\
&\qquad\qquad\text{local boundary components for $X$ at $x$}\bigr\},
\end{split}
\label{gc3eq14}\\
\begin{split}
\pd^kX&\cong\bigl\{(x,\be_1,\ldots,\be_k):\text{$x\in X,$
$\be_1,\ldots,\be_k$ are distinct}\\
&\qquad\qquad\text{local boundary components for $X$ at $x$}\bigr\},
\end{split}
\label{gc3eq15}\\
C_k(X)&\cong \pd^kX/S_k,
\label{gc3eq16}\\
\pd C_k(X)&\cong C_k(\pd X),
\label{gc3eq17}
\ea
using in \eq{gc3eq16} the natural free $S_k$-action on $\pd^kX$ permuting $\be_1,\ldots,\be_k$ in~\eq{gc3eq15}.

For the manifold with g-corners $X_P$, equation \eq{gc3eq14} is false for $k=2,3$, as over $x=\de_0$ there are 4 points on the l.h.s.\ and 6 points on the r.h.s.\ for $k=2$, and 1 point on the l.h.s.\ and 4 points on the r.h.s.\ for $k=3$. Similarly \eq{gc3eq15} is false when $k=2,3$. Equation \eq{gc3eq16} is true when $k=2$, but false when $k=3$, since $S_3$ cannot act freely on 8 points, and even for a non-free action, $\pd^3X_P/S_3$ would be at least two points. In \eq{gc3eq17} for $X_P$ when $k=2$, both sides are four points. However, the l.h.s.\ corresponds to the four edges in Figure \ref{gc3fig1}, and the r.h.s.\ to the four faces in Figure \ref{gc3fig1}. There is no natural 1-1 correspondence between these two four-point sets equivariant under automorphisms of $X_P$, so \eq{gc3eq17} is false for $X_P$, that is, there is no such canonical diffeomorphism. 
\label{gc3ex8}
\end{ex}

Example \ref{gc3ex8} shows that in general \eq{gc3eq14}--\eq{gc3eq17} are {\it false\/} for manifolds with g-corners $X$, at least for $k\ge 3$. But some modifications of them might be true, and we certainly expect some relation between $\pd^kX$ and $C_k(X)$. By considering local models $X_P$, and some simple properties of faces in weakly toric monoids, one can prove the following proposition. The moral is that for $k=2$, equations \eq{gc3eq14}--\eq{gc3eq16} have a good extension to manifolds with g-corners, but for $k\ge 3$ they do not generalize very well.

\begin{prop} Let\/ $X$ be a manifold with g-corners. Then:
\begin{itemize}
\setlength{\itemsep}{0pt}
\setlength{\parsep}{0pt}
\item[{\bf(a)}] There are natural identifications
\ea
\begin{split}
C_2(X)\!&\cong\!\bigl\{(x,\{\be_1,\be_2\}):\text{$x\!\in\! X,$
$\be_1,\be_2$ are distinct local boundary}\\
&\quad\text{components of $X$ at $x$ intersecting in codimension $2$}\bigr\},
\end{split}
\label{gc3eq18}\\
\begin{split}
\pd^2X\!&\cong\!\bigl\{(x,\be_1,\be_2):\text{$x\!\in\! X,$
$\be_1,\be_2$ are distinct local boundary}\\
&\quad\text{components of $X$ at $x$ intersecting in codimension $2$}\bigr\}.
\end{split}
\label{gc3eq19}
\ea
There is a natural, free action of\/ $S_2\cong\Z_2$ on $\pd^2X,$ exchanging $\be_1,\be_2$ in {\rm\eq{gc3eq19},} and a natural diffeomorphism\/ $C_2(X)\cong \pd^2X/S_2$.
\item[{\bf(b)}] For all\/ $k=0,1,\ldots,\dim X$ there are natural projections $\pi:\pd^kX\ra C_k(X)$ which are smooth, surjective, and \'etale (a local diffeomorphism).
\item[{\bf(c)}] The symmetric group $S_k$ for $k\ge 2$ is generated by the $k-1$ two-cycles $(12),(23),\cdots,(k-1\,k),$ satisfying relations. Thus, an $S_k$-action on a space is equivalent to $k-1$ actions of\/ $S_2\cong\Z_2,$ satisfying relations.

We can define $k-1$ actions of\/ $S_2$ on $\pd^kX$ as follows: for $j=0,\ldots,k-2,$ part\/ {\bf(a)} with $\pd^jX$ in place of $X$ gives an $S_2$-action on $\pd^{j+2}X,$ and applying $\pd^{k-j-2}$ induces an $S_2$-action on $\pd^kX$. If\/ $X$ has ordinary corners, these $k-1$ $S_2$-actions satisfy the relations required to define an $S_k$-action on $\pd^kX,$ but if\/ $k\ge 3$ and\/ $X$ has g-corners they may not satisfy the relations, and so generate an action of some group $G\not\cong S_k$ on $\pd^kX$.  
\end{itemize}
\label{gc3prop5}
\end{prop}

Here in \eq{gc3eq18}--\eq{gc3eq19}, distinct local boundary components $\be_1,\be_2$ of $X$ at $x$ may intersect in codimension $2,3,\ldots,\dim X$. For example, $X_P$ in Example \ref{gc3ex5} has four local boundary components $\be_{13},\be_{32},\be_{24},\be_{41},$ at $x=\de_0$, of which adjacent pairs $(\be_{13},\be_{32})$, $(\be_{32},\be_{24})$, $(\be_{24},\be_{41})$ and $(\be_{41},\be_{13})$ intersect in codimension 2, and opposite pairs $(\be_{13},\be_{24})$ and $(\be_{32},\be_{41})$ intersect in codimension~3.

Here is the analogue of Lemma~\ref{gc2lem1}.

\begin{lem} Let\/ $f:X\ra Y$ be a smooth map of manifolds with g-corners. Then $f$ \begin{bfseries}is compatible with the depth stratifications\end{bfseries} $X=\coprod_{k\ge 0}S^k(X),$ $Y=\coprod_{l\ge 0}S^l(Y)$ in Definition\/ {\rm\ref{gc3def9},} in the sense that if\/ $\es\ne W\subseteq S^k(X)$ is a connected subset for some $k\ge 0,$ then $f(W)\subseteq S^l(Y)$ for some unique $l\ge 0$.
\label{gc3lem3}
\end{lem}

\begin{proof} The lemma is a local property, so by restricting to single g-charts on $X,Y$, we see it is sufficient to prove that if $P,Q$ are weakly toric monoids, $U\subseteq X_P$, $V\subseteq X_Q$ are open, and $f:U\ra V$ is smooth in the sense of Definition \ref{gc3def5}, then $f$ preserves the stratifications $U=\coprod_{k\ge 0}S^k(U),$ $V=\coprod_{l\ge 0}S^l(V)$.

In Definition \ref{gc3def9}, $S^k(U)$ is a disjoint union of pieces $U\cap X^P_F$ for $\codim F=k$, where the subsets $X^P_F\subseteq X_P$ may be characterized as subsets where either $\la_p=0$ (if $p\notin F$) or $\la_p>0$ (if $p\in F$), for each $p\in P$. Thus we see that
\ea
&\bigl\{S\subseteq U:\text{for some $k\ge 0$, $S$ is a connected component of $S^k(U)$}\bigr\}
\nonumber\\
&=\bigl\{S\subseteq U:\text{for some $I\subseteq P$, $S$ is a connected component of} 
\label{gc3eq20}\\
&\qquad\qquad\text{$\{u\in U:\la_p(u)=0$ for $p\in I$, $\la_p>0$ for $p\in P\sm I\}$}\bigr\}.
\nonumber
\ea
The analogue also holds for $V$. Now as $f:U\ra V$ is smooth, Definition \ref{gc3def5} implies that for each $q\in Q$, locally on $U$ we may write $\la_q\ci f=h\cdot \la_p$ for some $p\in P$ and $h>0$, or $\la_q\ci f=0$. Hence locally on $U$, $f$ pulls back subsets $\{\la_q=0\}$ and $\{\la_q>0\}$ in $V$ for $q\in Q$ to subsets $\{\la_p=0\}$ and $\{\la_p>0\}$ for $p\in P$, or else $f$ pulls back $\{\la_q=0\}$ to $U$ and $\{\la_q>0\}$ to $\es$. This implies that $f$ maps each set in the r.h.s.\ of \eq{gc3eq20} for $U$ to a set in the r.h.s.\ of \eq{gc3eq20} for $V$. The lemma then follows by \eq{gc3eq20} for~$U,V$.
\end{proof}

Here is the analogue of Definition~\ref{gc2def5}.

\begin{dfn} Define the {\it corners\/} $C(X)$ of a manifold with g-corners $X$ by
\begin{align*}
&C(X)=\ts\coprod_{k=0}^{\dim X}C_k(X)\\
&=\bigl\{(x,\ga):\text{$x\in X$, $\ga$ is a local $k$-corner 
component of $X$ at $x$, $k\ge 0$}\bigr\},
\end{align*}
considered as an object of $\cMangc$ in Definition \ref{gc3def6}, a manifold with g-corners of mixed dimension. Define a smooth map $\Pi:C(X)\ra X$ by~$\Pi:(x,\ga)\mapsto x$.

Let $f:X\ra Y$ be a smooth map of manifolds with g-corners, and suppose $\ga$ is a local $k$-corner component of $X$ at $x\in X$. For each sufficiently small open neighbourhood $V$ of $x$ in $X$, $\ga$ gives a choice of connected component $W$ of $V\cap S^k(X)$ with $x\in\overline W$, so by Lemma \ref{gc3lem3} $f(W)\subseteq S^l(Y)$ for some $l\ge 0$. As $f$ is continuous, $f(W)$ is connected, and $f(x)\in\ov{f(W)}$. Thus there is a unique $l$-corner component $f_*(\ga)$ of $Y$ at $f(x)$, such that if $\ti V$ is a sufficiently small open neighbourhood of $f(x)$ in $Y$, then the connected component $\ti W$ of $\ti V\cap S^l(Y)$ given by $f_*(\ga)$ has $\ti W\cap f(W)\ne\es$. This $f_*(\ga)$ is independent of the choice of sufficiently small $V,\ti V$, so is well-defined.

Define a map $C(f):C(X)\ra C(Y)$ by $C(f):(x,\ga)\mapsto (f(x),f_*(\ga))$. A similar proof to Definition \ref{gc2def5} shows $C(f)$ is smooth, that is, a morphism in $\cMangc$. If $g:Y\ra Z$ is another smooth map of manifolds with corners, and $\ga$ is a local $k$-corner component of $X$ at $x$, then $(g\ci f)_*(\ga)=g_*\ci f_*(\ga)$ in local $m$-corner components of $Z$ at $g\ci f(x)$. Therefore $C(g\ci f)=C(g)\ci C(f):C(X)\ra C(Z)$. Clearly $C(\id_X)=\id_{C(X)}:C(X)\ra C(X)$. Hence $C:\Mangc\ra\cMangc$ is a functor, which we call the {\it corner functor}. We extend $C$ to $C:\cMangc\ra\cMangc$ by $C(\coprod_{m\ge 0}X_m)=\coprod_{m\ge 0}C(X_m)$. 

As in \S\ref{gc33} we have full subcategories $\Manc\subset\Mangc$, $\cManc\subset\cMangc$. Corollary \ref{gc3cor2} implies that the definitions of $C:\Mangc\ra\cMangc$, $C:\cMangc\ra\cMangc$ above restrict on $\Manc$, $\cManc$ to the corner functors $C:\Manc\ra\cManc$, $C:\cManc\ra\cManc$ defined in~\S\ref{gc22}.
\label{gc3def11}
\end{dfn}

We show corners are compatible with products.

\begin{ex} Let $X,Y$ be manifolds with g-corners, and consider the product $X\t Y$, with projections $\pi_X:X\t Y\ra X$, $\pi_Y:X\t Y\ra Y$. We form $C(\pi_X):\ab C(X\t Y)\ra C(X)$, $C(\pi_Y):C(X\t Y)\ra C(Y)$, and take the direct product
\e
(C(\pi_X),C(\pi_Y)):C(X\t Y)\longra C(X)\t C(Y).
\label{gc3eq21}
\e
Since $S^k(X\t Y)=\coprod_{i+j=k}S^i(X)\t S^j(Y)$, from Definition \ref{gc3def11} we can show that \eq{gc3eq21} is a diffeomorphism. Thus, as for \eq{gc2eq10}--\eq{gc2eq11} we have diffeomorphisms
\begin{align*}
\pd(X\t Y)&\cong (\pd X\t Y)\amalg (X\t\pd Y),\\
C_k(X\t Y)&\cong \ts\coprod_{i,j\ge 0,\; i+j=k}C_i(X)\t C_j(Y).
\end{align*}
The functor $C$ preserves products and direct products, as in Proposition~\ref{gc2prop1}(f).

\label{gc3ex9}
\end{ex}

Here is a partial analogue of Proposition \ref{gc2prop1}. The proof is straightforward, by considering local models.

\begin{prop} Let\/ $f\!:\!X\!\ra\!Y$ be a smooth map of manifolds with g-corners. 

\smallskip

\noindent{\bf(a)} $C(f):C(X)\ra C(Y)$ is an interior map of manifolds with g-corners of mixed dimension, so $C$ is a functor $C:\Mangc\ra\cMangcin$.
\smallskip

\noindent{\bf(b)} $f$ is interior if and only if\/ $C(f)$ maps $C_0(X)\ra C_0(Y),$ if and only if the following commutes:
\begin{equation*}
\xymatrix@C=95pt@R=13pt{ *+[r]{X} \ar[d]^\io \ar[r]_{f} &
*+[l]{Y} \ar[d]_\io \\ *+[r]{C(X)} \ar[r]^{C(f)} & *+[l]{C(Y).} }\qquad
\end{equation*}
Thus $\io:\Id\!\Ra\! C$ is a natural transformation on
$\Id,C\vert_{\Mangcin}:\Mangcin\!\ra\!\cMangcin$.

\noindent{\bf(c)} The following commutes:
\begin{equation*}
\xymatrix@C=85pt@R=13pt{ *+[r]{C(X)} \ar[d]^\Pi \ar[r]_{C(f)} &
*+[l]{C(Y)} \ar[d]_\Pi \\ *+[r]{X} \ar[r]^f & *+[l]{Y.} }
\end{equation*}
Thus $\Pi:C\Ra\Id$ is a natural transformation.
\label{gc3prop6}
\end{prop}

\subsection{\texorpdfstring{B-tangent bundles ${}^bTX$ of manifolds with g-corners}{B-tangent bundles of manifolds with g-corners}}
\label{gc35}

Here is the analogue of Definition~\ref{gc2def6}.

\begin{dfn} We define vector bundles over manifolds with g-corners exactly as for vector bundles over other classes of manifolds: a {\it vector bundle $E\ra X$ of rank\/} $\rank E=k$ is a smooth map $\pi:E\ra X$ of manifolds with g-corners, with a vector space structure on the each fibre $E_x=\pi^{-1}(x)$ for $x\in X$, which locally over $X$ admits a smooth identification with the projection $X\t\R^k\ra X$, preserving the vector space structures on each~$E_x$.

Sometimes we also consider {\it vector bundles of mixed rank\/}
$E\ra X$, in which we allow the rank $k$ to vary on different connected components of $X$. This happens often when working with objects
$X=\coprod_{m=0}^\iy X_m$ in $\cMangc$ from \S\ref{gc33}, for instance, the b-tangent bundle ${}^bTX$ has rank $m$ over $X_m$ for each~$m$.
\label{gc3def12}
\end{dfn}

In \S\ref{gc23} we defined tangent bundles $TX$ and b-tangent bundles ${}^bTX$ for a manifold with (ordinary) corners. The expressions \eq{gc2eq13} for $T_xX$, and \eq{gc2eq14} for ${}^bT_xX$, also make sense for manifolds with g-corners. The next example shows that for manifolds with g-corners $X$, `tangent bundles' $TX$ are not well-behaved.

\begin{ex} Let $X_P$ be the manifold with g-corners of Example \ref{gc3ex5}. Define $T_xX_P$ by \eq{gc2eq13} for all $x\in X_P$. As $X_P\sm\{\de_0\}$ is a manifold with corners of dimension 3, as in \S\ref{gc23} we have $\dim T_xX_P=3$ for all $\de_0\ne x\in X_P$. However, calculation shows that $T_{\de_0}X_P$ has dimension 4, with basis $v_1,v_2,v_3,v_4$ which act on the functions $\la_p:X_P\ra[0,\iy)$ for $p\in P$ by $v_i([\la_p])=1$ if $p=p_i$ and $v_i([\la_p])=0$ otherwise, where $p_1,p_2,p_3,p_4$ are the generators of $P$ in \eq{gc3eq6}. Thus, $\pi:TX_P\ra X_P$ {\it is not a vector bundle over\/} $X_P$, but something more like a coherent sheaf in algebraic geometry, in which the dimensions of the fibres are not locally constant, but only upper semicontinuous. Also $TX_P$ does not have the structure of a manifold with g-corners in a sensible way.
\label{gc3ex10}
\end{ex}

Because of this, we will not discuss tangent bundles for manifolds with corners, but only b-tangent bundles ${}^bTX$, which are well-behaved. First we define ${}^bTX$, and $\pi_X:{}^bTX\ra X$, ${}^bTf:{}^bTX\ra{}^bTY$ just as sets and maps.

\begin{dfn} Let $X$ be a manifold with g-corners, and $x\in X$. Define $C^\iy_x(X),\cI_x(X)$ and $\ev,\exp,\inc$ as in Definitions \ref{gc2def7} and \ref{gc2def8}. As in \eq{gc2eq14}, define a real vector space ${}^bT_xX$ by
\ea
{}^bT_xX=\bigl\{&(v,v'):\text{$v$ is a linear map $C^\iy_x(X)\ra\R$,}
\nonumber\\
&\text{$v'$ is a monoid morphism $\cI_x(X)\ra\R$,}
\nonumber\\
&\text{$v([a]\!\cdot\! [b])\!=\!v([a])\ev([b])\!+\!\ev([a])v([b])$, all $[a],[b]\!\in\!
C^\iy_x(X)$},
\nonumber\\
&\text{$v'\ci\exp([a])=v([a])$, all $[a]\in C^\iy_x(X)$, and}
\nonumber\\
&\text{$v\ci\inc([b])=\ev([b])v'\bigl([b]\bigr)$, all $[b]\in\cI_x(X)$}\bigr\}.
\label{gc3eq22}
\ea
The conditions in \eq{gc3eq22} are not all independent. As a set, define ${}^bTX=\bigl\{(x,v,v'):x\in X$, $(v,v')\in {}^bT_xX\bigr\}$, and define a projection $\pi_X:{}^bTX\ra X$ by $\pi_X:(x,v,v')\mapsto x$, so that $\pi_X^{-1}(x)\cong {}^bT_xX$.

If $f:X\ra Y$ is an interior map of manifolds with g-corners, define a map of sets ${}^bTf:{}^bTX\ra{}^bTY$ as in Definition \ref{gc2def8} by ${}^bTf:(x,v,v')\mapsto (y,w,w')$ for $y=f(x)$, $w=v\ci f$ and $w'=v'\ci f$, where composition with $f$ maps $\ci f:C^\iy_y(Y)\ra C^\iy_x(X)$, $\ci f:\cI_y(Y)\ra\cI_x(X)$, as $f$ is interior.

If $g:Y\ra Z$ is a second interior map of manifolds with g-corners, it is easy to see that ${}^bT(g\ci f)={}^bTg\ci {}^bTf:{}^bTX\ra{}^bTZ$, and ${}^bT(\id_X)=\id_{{}^bTX}:{}^bTX\ra{}^bTX$, so the assignment $X\mapsto {}^bTX$, $f\mapsto{}^bTf$ is functorial.
\label{gc3def13}
\end{dfn}

In Definition \ref{gc3def14} below we will give ${}^bTX$ the structure of a manifold with g-corners, such that $\pi_X:{}^bTX\ra X$ is smooth and makes ${}^bTX$ into a vector bundle over $X$, and ${}^bTf:{}^bTX\ra{}^bTY$ is smooth for all interior maps $f:X\ra Y$. First we explain this for the model spaces $X_P$. Equation \eq{gc3eq24} shows that for monoids, passing from $X_P$ to ${}^bTX_P$ corresponds to passing from $P$ to~$P\t P^\gp$.

\begin{prop} Let\/ $P$ be a weakly toric monoid, so that\/ $X_P$ is a manifold with g-corners as in Example\/ {\rm\ref{gc3ex4},} with b-tangent bundle ${}^bTX_P$. Then there are natural inverse bijections $\Phi_P,\Psi_P$ in the diagram
\e
\xymatrix@C=60pt{
{}^bTX_P \ar@<.5ex>[r]^(0.4){\Psi_P} & X_P\t\Hom(P^\gp,\R), \ar@<.5ex>[l]^(0.6){\Phi_P} }
\label{gc3eq23}
\e
where\/ $\Hom(P^\gp,\R)\cong\R^{\rank P},$ and\/ $\Phi_P,\Psi_P$ are compatible with the projections $\pi:{}^bTX_P\ra X_P,$ $X_P\t\Hom(P^\gp,\R)\ra X_P$. Also there are natural bijections
\e
X_P\t\Hom(P^\gp,\R)\cong X_P\t X_{P^\gp}\cong X_{P\t P^\gp}.
\label{gc3eq24}
\e

\label{gc3prop7}
\end{prop}

\begin{proof} As $P$ is weakly toric we have a natural inclusion $P\hookra P^\gp$, where $P^\gp\cong\Z^r$ for $r=\rank P$, so that $\Hom(P^\gp,\R)\cong\R^r$. There are obvious natural bijections $\Hom(P^\gp,\R)\cong X_{P^\gp}$ and $X_P\t X_Q\cong X_{P\t Q}$, so \eq{gc3eq24} follows. 

For $(x,y)\in X_P\t\Hom(P^\gp,\R)$ define a map $v_{x,y}:C^\iy_x(X_P)\ra\R$ by
\e
v_{x,y}:[a]\longmapsto \ts\sum_{i=1}^n\frac{\pd g}{\pd x_i}\bigl(x(r_1),\ldots,x(r_n)\bigr)\cdot x(r_i)\cdot y(r_i)
\label{gc3eq25}
\e
if $U$ is an open neighbourhood of $x$ in $X_P$, $a:U\ra\R$ is smooth, and as in Definition \ref{gc3def5} we write $a:x'\mapsto g\bigl(x'(r_1),\ldots,x'(r_n)\bigr)$ for $x'\in U$, where $r_1,\ldots,r_n\in P$ and $g:W\ra\R$ is smooth, for $W$ an open neighbourhood of $(\la_{r_1}\t\cdots\t\la_{r_n})(U)$ in $[0,\iy)^n$.

Similarly, define $v_{x,y}':\cI_x(X_P)\ra\R$ by
\e
v_{x,y}':[b]\longmapsto y(p)+\ts\sum_{i=1}^n\frac{\pd}{\pd x_i}(\log h)\bigl(x(r_1),\ldots,x(r_n)\bigr)\cdot x(r_i)\cdot y(r_i)
\label{gc3eq26}
\e
if $U$ is an open neighbourhood of $x$ in $X_P$, $b:U\ra[0,\iy)$ is interior, and as in Definition \ref{gc3def5} we write $b:x'\mapsto x'(p)\cdot h\bigl(x'(r_1),\ldots,x'(r_n)\bigr)$ for $x'\in U$, where $p,r_1,\ldots,r_n\in P$ and $h:W\ra (0,\iy)$ is smooth, for $W$ an open neighbourhood of $(\la_{r_1}\t\cdots\t\la_{r_n})(U)$ in $[0,\iy)^n$.

It is not difficult to show that $v_{x,y},v'_{x,y}$ are independent of the choices of presentations for $a,b$, and that they satisfy the conditions of \eq{gc3eq22}, so $(v_{x,y},v'_{x,y})\in {}^bT_xX_P$. Define
\begin{equation*}
\Phi_P:X_P\t\Hom(P^\gp,\R)\longra{}^bTX_P\;\>\text{by}\;\> \Phi_P:(x,y)\longmapsto(x,v_{x,y},v'_{x,y}).
\end{equation*}

Now let $(x,v,v')\in {}^bTX_P$, and consider the map $P\ra\R$ acting by $p\mapsto v'([\la_p])$. By \eq{gc3eq22} this is a monoid morphism $P\ra(\R,+)$, so it factors through a group morphism $P^\gp\ra\R$ as $(\R,+)$ is a group. Thus there exists a unique $y_{x,v'}\in\Hom(P^\gp,\R)$ with $v'([\la_p])=y_{x,v'}(p)$ for all $p\in P$. Define 
\e
\Psi_P:{}^bTX_P\longra X_P\t\Hom(P^\gp,\R)\quad\text{by}\quad \Psi_P:(x,v,v')\longmapsto(x,y_{x,v'}).
\label{gc3eq27}
\e

We will show that that $\Phi_P,\Psi_P$ are inverse maps. By definition $\Psi_P\ci\Phi_P$ maps $(x,y)\mapsto(x,y_{x,v'_{x,y}})$, where for $p\in P^\gp$ we have
\begin{equation*}
y_{x,v'_{x,y}}(p)=v'_{x,y}([\la_p])=y(p)+\log 1=y(p),
\end{equation*}
using \eq{gc3eq26} for $b=\la_p$ and $h=1$. Thus $y_{x,v'_{x,y}}=y$, and $\Psi_P\ci\Phi_P=\id$. Also $\Psi_P\ci\Phi_P$ maps $(x,v,v')\mapsto(x,v_{x,y_{x,v'}},v'_{x,y_{x,v'}})$, where if $x\in U\subseteq X_P$ is open, $a:U\ra\R$ is smooth, and as in Definition \ref{gc3def5} we write $a:x'\mapsto g\bigl(x'(r_1),\ldots,x'(r_n)\bigr)$ for $x'\in U$, where $r_1,\ldots,r_n\in P$ and $g:W\ra\R$ is smooth, then
\begin{align*}
v_{x,y_{x,v'}}([a])&=\ts\sum_{i=1}^n\frac{\pd g}{\pd x_i}\bigl(x(r_1),\ldots,x(r_n)\bigr)\cdot x(r_i)\cdot y_{x,v'}(r_i)\\
&=\ts\sum_{i=1}^n\frac{\pd g}{\pd x_i}\bigl(x(r_1),\ldots,x(r_n)\bigr)\cdot x(r_i)\cdot v'([\la_{r_i}])\\
&=\ts\sum_{i=1}^n\frac{\pd g}{\pd x_i}\bigl(x(r_1),\ldots,x(r_n)\bigr)\cdot  v([x(r_i)])=v([a]),
\end{align*}
using \eq{gc3eq25} in the first step, $v'([\la_p])=y_{x,v'}(p)$ in the second, and $v\ci\inc([b])=\ev([b])v'\bigl([b]\bigr)$ from \eq{gc3eq22} with $b=\la_{r_i}=x(r_i)$ in the third. So $v_{x,y_{x,v'}}=v$. 

Similarly, using \eq{gc3eq26} and $v'([\la_p])=y_{x,v'}(p)$ we find that $v'_{x,y_{x,v'}}=v'$, so that $\Phi_P\ci\Psi_P=\id$. Hence $\Phi_P,\Psi_P$ are inverse maps, and bijections. Clearly they are compatible with the projections $\pi:{}^bTX_P\ra X_P$, $X_P\t\Hom(P^\gp,\R)\ra X_P$. This completes the proof.
\end{proof}

\begin{ex} Let $P\!=\!\N^k\!\t\!\Z^{n-k}$, so that $P^\gp\!\cong\!\Z^n$, and identify $X_{\N^k\t\Z^{n-k}}\!\cong\![0,\iy)^k\t\R^{n-k}$ as in Example \ref{gc3ex3}(iii), and $\Hom(P^\gp,\R)\cong\R^n$ in the obvious way. Following through the definition of $\Phi_P$ in Proposition \ref{gc3prop7}, we find that if $\Phi_P\bigl((x_1,\ldots,x_n),(y_1,\ldots,y_n)\bigr)\!=\!\bigl((x_1,\ldots,x_n),v,v'\bigr)$, then $v:C^\iy_{\bs x}(X_P)\!\ra\!\R$ is
\begin{align*}
v:[a]\longmapsto\ts y_1x_1\frac{\pd }{\pd x_1}a(x_1,\ldots,x_n)+\cdots+y_kx_k\frac{\pd}{\pd x_k}a(x_1,\ldots,x_n)&\\
\ts{}+y_{k+1}\frac{\pd}{\pd x_{k+1}}a(x_1,\ldots,x_n)+\cdots+y_n\frac{\pd}{\pd x_n}a(x_1,\ldots,x_n)&.
\end{align*}
Thus, the identification ${}^bTX_P\cong X_P\t\R^n$ from \eq{gc3eq23} gives a basis of sections of ${}^bTX_P$ corresponding to $x_1\frac{\pd}{\pd x_1},\ldots, x_k\frac{\pd}{\pd x_k},\frac{\pd}{\pd x_{k+1}},\ldots,\frac{\pd}{\pd x_n}$, as ordinary vector fields on~$X_P\cong[0,\iy)^k\t\R^{n-k}$. 

But in Definition \ref{gc2def8}(a) we defined the b-tangent bundle ${}^bT([0,\iy)^k\!\t\!\R^{n-k})$ of $[0,\iy)^k\t\R^{n-k}$ as a manifold with corners to have basis of sections $x_1\frac{\pd}{\pd x_1},\ab\ldots,\ab x_k\frac{\pd}{\pd x_k},\ab\frac{\pd}{\pd x_{k+1}},\ab\ldots,\ab\frac{\pd}{\pd x_n}$. This shows the definitions of ${}^bT([0,\iy)^k\t\R^{n-k})$ in \S\ref{gc23}, and in Definition \ref{gc3def13} and Proposition \ref{gc3prop7} above, are equivalent.
\label{gc3ex11}
\end{ex}

\begin{lem} Let\/ $P,Q$ be weakly toric monoids, $U\subseteq X_P,$ $V\subseteq X_Q$ be open, and\/ $f:U\ra V$ be an interior map, in the sense of Definition\/ {\rm\ref{gc3def5}}. Then the composition of maps
\begin{equation*}
\xymatrix@C=30pt{
U\t \Hom(P^\gp,\R) \ar[r]^(0.65){\Phi_P\vert_{\cdots}}_(0.65)\cong & {}^bTU \ar[r]^{{}^bTf} & {}^bTV \ar[r]^(0.35){\Psi_Q\vert_{\cdots}}_(0.35)\cong & V\t \Hom(Q^\gp,\R) }
\end{equation*}
is an interior map of manifolds with g-corners in the sense of\/ {\rm\S\ref{gc32}--\S\ref{gc33},} where ${}^bTf$ is as in Definition\/ {\rm\ref{gc3def13}} and\/ $\Phi_P,\Psi_Q$ as in Proposition\/~{\rm\ref{gc3prop7}}.
\label{gc3lem4}
\end{lem}

\begin{proof} Use the notation of Proposition \ref{gc3prop3}(c). This gives a commutative diagram of interior maps of manifolds with g-corners
\e
\begin{gathered}
\xymatrix@R=17pt@C=120pt{
*+[l]{X_P \supseteq U} \ar[r]_{\la_{p_1}\t\cdots\t\la_{p_m}} \ar@<-.2ex>[d]^{f} & *+[r]{W \subseteq [0,\iy)^m} \ar@<.2ex>[d]_{g} \\
*+[l]{X_Q \supseteq V} \ar[r]^{\la_{q_1}\t\cdots\t\la_{q_n} } & *+[r]{[0,\iy)^n.} } 
\end{gathered}
\label{gc3eq28}
\e
Consider the diagram
\e
\begin{gathered}
\xymatrix@R=17pt@C=150pt{
*+[r]{U\t \Hom(P^\gp,\R)} \ar[d]^{\Phi_P\vert_{U\t\Hom(P^\gp,\R)}}
\ar[r]_{\begin{subarray}{l}(\la_{p_1}\t\cdots\t\la_{p_m})\t \\ ((\ci p_1)\t\cdots\t(\ci p_m))\end{subarray}} & *+[l]{W\t\R^m} \ar[d]_{\cong } 
\\
*+[r]{{}^bTU} \ar[r]_{{}^bT(\la_{p_1}\t\cdots\t\la_{p_m})} \ar[d]^{{}^bTf} & *+[l]{{}^bTW} \ar[d]_{{}^bTg} \\
*+[r]{{}^bTV} \ar[d]^{\Psi_Q\vert_{{}^bTV}} \ar[r]^{{}^bT(\la_{q_1}\t\cdots\t\la_{q_n}) } & *+[l]{{}^bT([0,\iy)^n)} \ar[d]_{\cong }
\\
*+[r]{V\t \Hom(Q^\gp,\R)} \ar[r]^{\begin{subarray}{l}(\la_{q_1}\t\cdots\t\la_{q_n})\t \\ ((\ci q_1)\t\cdots\t(\ci q_n))\end{subarray}} & *+[l]{[0,\iy)^n\t\R^n.\!\!{}} } 
\end{gathered}
\label{gc3eq29}
\e
The middle rectangle commutes by applying the functor ${}^bT$ of Definition \ref{gc3def13} to \eq{gc3eq28}, and the upper and lower rectangles commute by the definitions. 

The right hand column of \eq{gc3eq29} involves manifolds with corners $W\subseteq [0,\iy)^m$, $[0,\iy)^n$, and Example \ref{gc3ex11} showed that for these the definitions of ${}^bTX$ in \ref{gc23} and above, are equivalent. This equivalence is functorial, so the definitions of ${}^bTg$ in \S\ref{gc23} and Definition \ref{gc3def13} are also equivalent. But ${}^bTg$ in \S\ref{gc23} is an interior map of manifolds with corners. Hence the composition of the right hand column in \eq{gc3eq29} is an interior map of manifolds with corners. Regarding $U\t \Hom(P^\gp,\R)$ and $V\t \Hom(Q^\gp,\R)$ as open sets in $X_{P\t P^\gp},X_{Q\t Q^\gp}$ as in Proposition \ref{gc3prop7}, Proposition \ref{gc3prop3}(c) with $P\t P^\gp,Q\t Q^\gp$ in place of $P,Q$ now implies that the composition of the left hand column of \eq{gc3eq29} is an interior map of manifolds with g-corners.
\end{proof}

Note that \eq{gc3eq28}--\eq{gc3eq29} give a convenient way to compute the maps ${}^bTf:{}^bTX\ra{}^bTY$ in Definition \ref{gc3def13} locally. We can now give ${}^bTX$ the structure of a manifold with g-corners, a vector bundle over~$X$:

\begin{dfn} Let $X$ be a manifold with g-corners, so that ${}^bTX$ is defined as a set in Definition \ref{gc3def13}, with projection $\pi:{}^bTX\ra X$. Suppose $(P,U,\phi)$ is a g-chart on $X$. For $\Phi_P$ as in Proposition \ref{gc3prop7}, consider the composition
\begin{equation*}
\xymatrix@C=35pt{
U\t \Hom(P^\gp,\R) \ar[rr]^(0.62){\Phi_P\vert_{U\t \Hom(P^\gp,\R)}} && {}^bTU \ar[r]^{{}^bT\phi} & {}^bTX, }
\end{equation*}
which has image ${}^bT(\phi(U))\subseteq {}^bTX$. Here $U\t\Hom(P^\gp,\R)$ is open in $X_P\t\Hom(P^\gp,\R)\cong X_P\t X_{P^\gp}\cong X_{P\t P^\gp}$, so identifying $U\t \Hom(P^\gp,\R)$ with an open set in $X_{P\t P^\gp}$, we can regard
\e
\bigl(P\t P^\gp,U\t \Hom(P^\gp,\R),{}^bT\phi\ci \Phi_P\vert_{U\t \Hom(P^\gp,\R)}\bigr)
\label{gc3eq30}
\e
as a g-chart on~${}^bTX$.

We claim that ${}^bTX$ has the unique structure of a manifold with g-corners (including a topology), of dimension $2\dim X$, such that \eq{gc3eq30} is a g-chart on ${}^bTX$ for all g-charts $(P,U,\phi)$ on $X$, and that with this structure $\pi:{}^bTX\ra X$ is interior and makes ${}^bTX$ into a vector bundle over $X$. To see this, note that if $(P,U,\phi),(Q,V,\psi)$ are g-charts on $X$, then they are compatible, so the change of g-charts morphism $\psi^{-1}\ci\phi:\phi^{-1}\bigl(\phi(U)\cap\psi(V)\bigr)\ra\psi^{-1}\bigl(\phi(U)\cap\psi(V)\bigr)$ is a diffeomorphism between open subsets of $X_P,X_Q$. Applying Lemma \ref{gc3lem4} to $\psi^{-1}\ci\phi$ and its inverse implies that the change of charts morphism between the g-charts \eq{gc3eq30} from $(P,U,\phi),(Q,V,\psi)$ is also a diffeomorphism, so \eq{gc3eq30} and its analogue for $(Q,V,\psi)$ are compatible.

Thus, the g-charts \eq{gc3eq30} from g-charts $(P,U,\phi)$ on $X$ are all pairwise compatible. These g-charts \eq{gc3eq30} also cover ${}^bTX$, since the image of \eq{gc3eq30} is ${}^bT\phi(U)\subseteq {}^bTX$, and the $\phi(U)$ cover $X$. Since $X$ is Hausdorff and second countable, one can show that there is a unique Hausdorff, second countable topology on ${}^bTX$ such that for all $(P,U,\phi)$ as above, ${}^bT\phi(U)$ is open in ${}^bTX$, and ${}^bT\phi\ci \Phi_P\vert_{U\t \Hom(P^\gp,\R)}:U\t \Hom(P^\gp,\R)\ra {}^bT\phi(U)$ is a homeomorphism. Therefore the g-charts \eq{gc3eq30} form a g-atlas on ${}^bTX$ with this topology, which extends to a unique maximal g-atlas, making ${}^bTX$ into a manifold with g-corners. That $\pi:{}^bTX\ra X$ is interior and makes ${}^bTX$ into a rank $n$ vector bundle over $X$ follows from the local models.

Since $\pi:{}^bTX\ra X$ is a vector bundle, it has a dual vector bundle, which we call the {\it b-cotangent bundle\/} and write as $\pi:{}^bT^*X\ra X$

Now let $f:X\ra Y$ be an interior map of manifolds with g-corners. Then for all g-charts $(P,U,\phi)$ on $X$ and $(Q,V,\psi)$ on $Y$, the map $\psi^{-1}\ci f\ci\phi$ in \eq{gc3eq5} is an interior map between open subsets of $X_P,X_Q$. Applying Lemma \ref{gc3lem4} shows that the corresponding map for ${}^bTf:{}^bTX\ra{}^bTY$ and the g-charts \eq{gc3eq30} from $(P,U,\phi),(Q,V,\psi)$ is also interior. As these g-charts cover ${}^bTX,{}^bTY$, this proves that ${}^bTf:{}^bTX\ra{}^bTY$ is an interior map of manifolds with g-corners.

Clearly ${}^bTf:{}^bTX\ra{}^bTY$ satisfies $\pi\ci{}^bTf=f\ci\pi$ and is linear on the vector space fibres ${}^bT_xX,{}^bT_yY$. Thus, ${}^bTf$ induces a morphism of vector bundles on $X$, which we write as ${}^b\d f:{}^bTX\ra f^*({}^bTY)$, as in \S\ref{gc23}. Dually, we have a morphism of b-cotangent bundles, which we write as $({}^b\d f)^*:f^*({}^bT^*Y)\ra {}^bT^*X$.

If $g:Y\ra Z$ is another interior map of manifolds with g-corners, then ${}^bT(g\ci f)={}^bTg\ci{}^bTf$ implies that ${}^b\d(g\ci f)=f^*({}^b\d g)\ci{}^b\d f:{}^bTX\ra (g\ci f)^*({}^bTZ)$, and dually $({}^b\d(g\ci f))^*=({}^b\d f)^*\ci f^*(({}^b\d g)^*):(g\ci f)^*({}^bT^*Z)\ra{}^bT^*X$.

Define the {\it b-tangent functor\/} ${}^bT:\Mangcin\ra\Mangcin$ to map ${}^bT:X\mapsto{}^bTX$ on objects, and ${}^bT:f\mapsto{}^bTf$ on (interior) morphisms $f:X\ra Y$. Then ${}^bT$ is a functor, as in Definition \ref{gc3def13}. It extends naturally to ${}^bT:\cMangcin\ra\cMangcin$. The projections $\pi:{}^bTX\ra X$ and zero sections $0:X\ra{}^bTX$ induce natural transformations $\pi:{}^bT\Ra \Id$ and $0:\Id\Ra {}^bT$. On the subcategories $\Mancin\subset\Mangcin$, $\cMancin\subset\cMangcin$, these functors ${}^bT$ restrict to those defined in~\S\ref{gc23}.
\label{gc3def14}
\end{dfn}

We show b-tangent bundles are compatible with products.

\begin{ex} Let $X,Y$ be manifolds with g-corners, and consider the product $X\t Y$, with projections $\pi_X:X\t Y\ra X$, $\pi_Y:X\t Y\ra Y$. These are interior maps, so we may form ${}^bT\pi_X:{}^bT(X\t Y)\ra {}^bTX$, ${}^bT\pi_Y:{}^bT(X\t Y)\ra {}^bTY$, and take the direct product
\e
({}^bT\pi_X,{}^bT\pi_Y):{}^bT(X\t Y)\longra {}^bTX\t {}^bTY.
\label{gc3eq31}
\e
Considering local models as in Proposition \ref{gc3prop7}, it is easy to check that \eq{gc3eq31} is a diffeomorphism. We sometimes use \eq{gc3eq31} to identify ${}^bT(X\t Y)$ with ${}^bTX\t {}^bTY$, and ${}^bT_{(x,y)}(X\t Y)$ with ${}^bT_xX\op {}^bT_yY$. The functor ${}^bT$ preserves products and direct products, in the sense of Proposition~\ref{gc2prop1}(f). 
\label{gc3ex12}
\end{ex}

\subsection{\texorpdfstring{B-normal bundles of $C_k(X)$}{B-normal bundles of corners}}
\label{gc36}

In \S\ref{gc24}, if $X$ is a manifold with (ordinary) corners, and $\Pi:C_k(X)\ra X$ the projection, we constructed a canonical rank $k$ vector bundle $\pi:{}^bN_{C_k(X)}\ra C_k(X)$, the {\it b-normal bundle of\/ $C_k(X)$ in\/} $X$, fitting into an exact sequence
\e
\smash{\xymatrix@C=19pt{ 0 \ar[r] & {}^bN_{C_k(X)} \ar[rr]^(0.45){{}^bi_T} &&
\Pi^*({}^bTX) \ar[rr]^{{}^b\pi_T} && {}^bT(C_k(X)) \ar[r] & 0, }}
\label{gc3eq32}
\e
and a {\it monoid bundle\/} $M_{C_k(X)}\subseteq{}^bN_{C_k(X)}$, a submanifold of ${}^bN_{C_k(X)}$ such that $\pi:M_{C_k(X)}\ra C_k(X)$ is a locally constant family of toric monoids over $C_k(X)$. We showed that ${}^bN_{C(X)}=\coprod_{k\ge 0}{}^bN_{C_k(X)}$ and $M_{C(X)}=\coprod_{k\ge 0}M_{C_k(X)}$ are functorial over interior $f:X\ra Y$, as for the corner functor~$C:\Manc\ra\cManc$. 

We now generalize all this to manifolds with g-corners. As for ${}^bTX$ in \S\ref{gc35} we do this in stages: first we define ${}^bN_{C_k(X)},M_{C_k(X)}$ just as sets, and $\pi:{}^bN_{C_k(X)}\ra C_k(X)$, ${}^bN_{C(f)}:{}^bN_{C(X)}\ra {}^bN_{C(Y)}$, $M_{C(f)}:M_{C(X)}\ra M_{C(Y)}$ just as maps. Then after some calculations, in Definition \ref{gc3def16} we will give ${}^bN_{C_k(X)},M_{C_k(X)}$ the structure of manifolds with g-corners, such that $\pi,{}^bN_{C(f)},M_{C(f)}$ are smooth.

\begin{dfn} Let $X$ be a manifold with g-corners, and let $(x,\ga)\in C_k(X)$ for $k\ge 0$. As in Definition \ref{gc2def7} we have $\R$-algebras $C^\iy_x(X)$ of germs $[a]$ at $x$ of smooth functions $a:X\ra\R$, and $C^\iy_{(x,\ga)}\bigl(C_k(X)\bigr)$ of germs $[b]$ at $(x,\ga)$ of smooth functions $b:C_k(X)\ra\R$. Then composition with $\Pi$ defines a map
\e
\Pi^*:C^\iy_x(X)\longra C^\iy_{(x,\ga)}\bigl(C_k(X)\bigr),\quad \Pi^*:[a]\longmapsto [a\ci\Pi].
\label{gc3eq33}
\e
This is an $\R$-algebra morphism.

As in Definition \ref{gc2def8} we have monoids $\cI_x(X)$ of germs $[c]$ at $x$ of interior functions $c:X\ra[0,\iy)$, and $\cI_{(x,\ga)}(C_k(X))$ of germs $[d]$ at $(x,\ga)$ of interior functions $d:C_k(X)\ra[0,\iy)$. If $x\in U\subseteq X$ is open and $c:U\ra[0,\iy)$ is interior, setting $V=\Pi^{-1}(U)\subseteq C_k(X)$ and $d=c\ci\Pi:V\ra[0,\iy)$, then $(x,\ga)\in V\subseteq C_k(X)$ is open and either $d$ is interior near $(x,\ga)$, or $d=0$ near $(x,\ga)$. Thus composition with $\Pi$ defines a map
\e
\begin{split}
\Pi^*&:\cI_x(X)\longra \cI_{(x,\ga)}(C_k(X))\amalg \{0\}, \\ \Pi^*&:[c]\longmapsto[c\ci\Pi].
\end{split}
\label{gc3eq34}
\e
This is a monoid morphism, making $\cI_{(x,\ga)}(C_k(X))\amalg \{0\}$ into a monoid by setting $[d]\cdot 0=0$ for all $[d]\in\cI_{(x,\ga)}(C_k(X))$. (Note that $[1]\in\cI_{(x,\ga)}(C_k(X))$ is the monoid identity element, not 0.) Define
\ea
\begin{split}
{}^bN_{C_k(X)}\vert_{(x,\ga)}&=\bigl\{\al\in \Hom_\Mon\bigl(\cI_x(X),\R):\\
&\qquad\qquad\qquad \al\vert_{(\Pi^*)^{-1}[\cI_{(x,\ga)}(C_k(X))]}=0\bigr\}, 
\end{split}
\label{gc3eq35}\\
\begin{split}
M_{C_k(X)}\vert_{(x,\ga)}&=\bigl\{\al\in \Hom_\Mon\bigl(\cI_x(X),\N):\\
&\qquad\qquad\qquad \al\vert_{(\Pi^*)^{-1}[\cI_{(x,\ga)}(C_k(X))]}=0\bigr\}.
\end{split}
\label{gc3eq36}
\ea
Then ${}^bN_{C_k(X)}\vert_{(x,\ga)}$ is a real vector space, and $M_{C_k(X)}\vert_{(x,\ga)}$ is a monoid, and $M_{C_k(X)}\vert_{(x,\ga)}\subseteq {}^bN_{C_k(X)}\vert_{(x,\ga)}$ as $\N\subset\R$. In Example \ref{gc3ex13} we will show that ${}^bN_{C_k(X)}\vert_{(x,\ga)}\cong\R^k$, and $M_{C_k(X)}\vert_{(x,\ga)}$ is a toric monoid of rank $k$, with
\e
{}^bN_{C_k(X)}\vert_{(x,\ga)}\cong M_{C_k(X)}\vert_{(x,\ga)}\ot_\N\R.
\label{gc3eq37}
\e

Equation \eq{gc3eq22} defines ${}^bT_xX$ as a vector space of pairs $(v,v')$. We claim that if $\al\in {}^bN_{C_k(X)}\vert_{(x,\ga)}$, then $(0,\al)\in{}^bT_xX$. To see this, note that the first three conditions of \eq{gc3eq22} for $(0,\al)$ are immediate, and the final two follow from the fact that if $[c]\in\cI_x(X)$ with $c(x)\ne 0$ then $\Pi^*([c])\in \cI_{(x,\ga)}(C_k(X))$, so $\al([c])=0$. Thus we may define a linear map 
\e
{}^bi_T\vert_{(x,\ga)}:{}^bN_{C_k(X)}\vert_{(x,\ga)}\longra {}^bT_xX,
\quad {}^bi_T\vert_{(x,\ga)}:\al\longmapsto(0,\al).
\label{gc3eq38}
\e

Now let $(v,v')\in {}^bT_xX$. We will show in Example \ref{gc3ex13} that there is a unique $(w,w')\in {}^bT_{(x,\ga)}C_k(X)$ such that $w(\Pi^*([a]))=v([a])$ for all $[a]\in C^\iy_x(X)$ and $w'(\Pi^*([b]))=v'([b])$ for all $[b]\in \cI_x(X)$ with $\Pi^*([b])\ne 0$, where the $\Pi^*$ maps are as in \eq{gc3eq33}--\eq{gc3eq34}. Define a linear map ${}^b\pi_T\vert_{(x,\ga)}:{}^bT_xX\ra{}^bT_{(x,\ga)}C_k(X)$ by ${}^b\pi_T\vert_{(x,\ga)}:(v,v')\mapsto(w,w')$. So we have a sequence 
\e
\smash{\xymatrix@C=18pt{ 0 \ar[r] & {}^bN_{C_k(X)}\vert_{(x,\ga)} \ar[rr]^(0.55){{}^bi_T\vert_{(x,\ga)}} && {}^bT_xX \ar[rr]^(0.35){{}^b\pi_T\vert_{(x,\ga)}} && {}^bT_{(x,\ga)}(C_k(X)) \ar[r] & 0 }}
\label{gc3eq39}
\e
of real vector spaces, as in \eq{gc3eq32}. It follows from the definitions that \eq{gc3eq39} is a complex. We will show in Example \ref{gc3ex13} that \eq{gc3eq39} is exact.

Just as sets, define the {\it b-normal bundle\/} ${}^bN_{C_k(X)}$ and {\it monoid bundle\/} $M_{C_k(X)}$ of $C_k(X)$ in $X$ by
\begin{align*}
{}^bN_{C_k(X)}&=\bigl\{(x,\ga,\al):(x,\ga)\in C_k(X),\;\> \al\in {}^bN_{C_k(X)}\vert_{(x,\ga)}\bigr\},\\ 
M_{C_k(X)}&=\bigl\{(x,\ga,\al):(x,\ga)\in C_k(X),\;\> \al\in M_{C_k(X)}\vert_{(x,\ga)}\bigr\},
\end{align*}
so that $M_{C_k(X)}\subseteq {}^bN_{C_k(X)}$. Define projections $\pi:{}^bN_{C_k(X)}\ra C_k(X)$ and $\pi:M_{C_k(X)}\ra C_k(X)$ by $\pi:(x,\ga,\al)\mapsto(x,\ga)$. Define ${}^bi_T:{}^bN_{C_k(X)}\ra\Pi^*({}^bTX)$ by ${}^bi_T:(x,\ga,\al)\mapsto \bigl((x,\ga),{}^bi_T\vert_{(x,\ga)}(\al)\bigr)$ and ${}^b\pi_T:\Pi^*({}^bTX)\ra {}^bT(C_k(X))$ by ${}^b\pi_T:\bigl((x,\ga),(v,v')\bigr)\mapsto \bigl((x,\ga),{}^b\pi_T\vert_{(x,\ga)}(v,v')\bigr)$. In Definition \ref{gc3def16} we will make ${}^bN_{C_k(X)},M_{C_k(X)}$ into manifolds with g-corners, such that $\pi:{}^bN_{C_k(X)}\ra C_k(X)$ is smooth and makes ${}^bN_{C_k(X)}$ into a vector bundle over $C_k(X)$ of rank $k$, and $\pi:M_{C_k(X)}\ra C_k(X)$ is smooth and makes $M_{C_k(X)}$ into a bundle of toric monoids over $C_k(X)$, and \eq{gc3eq32} is an exact sequence of vector bundles.

Define ${}^bN_{C(X)}=\coprod_{k=0}^{\dim X}{}^bN_{C_k(X)}$, with projection $\pi:{}^bN_{C(X)}\ra C(X)=\coprod_{k=0}^{\dim X}C_k(X)$ given by $\pi\vert_{{}^bN_{C_k(X)}}=\pi:{}^bN_{C_k(X)}\ra C_k(X)$. Set $M_{C(X)}=\coprod_{k=0}^{\dim X}M_{C_k(X)}$, so that $M_{C(X)}\subseteq {}^bN_{C(X)}$, and define $\pi=\pi\vert_{M_{C(X)}}:M_{C(X)}\ra C(X)$. Later we will see that ${}^bN_{C(X)}$ is a manifold with g-corners of dimension $\dim X$, with $\pi:{}^bN_{C(X)}\ra C(X)$ is a vector bundle of mixed rank, and $M_{C(X)}$ is an object in $\cMangc$, with $\pi:M_{C(X)}\ra C(X)$ a bundle of toric monoids.

Next let $f:X\ra Y$ be an interior map of manifolds with g-corners, so that $C(f):C(X)\ra C(Y)$ is a morphism in $\cMangc$ as in \S\ref{gc34}. Define a map of sets ${}^bN_{C(f)}:{}^bN_{C(X)}\ra{}^bN_{C(Y)}$ as in Definition \ref{gc2def12} by ${}^bN_{C(f)}:(x,\ga,\al)\mapsto (f(x),f_*(\ga),\al\ci f^*)$, where $f^*:\cI_{f(x)}(Y)\ra\cI_x(X)$ maps $[c]\mapsto [c\ci f]$, and is well-defined as $f$ is interior. From \eq{gc3eq35} we can check that if $\al\in {}^bN_{C_k(X)}\vert_{(x,\ga)}$ then $\al\ci f^*\in{}^bN_{C_l(Y)}\vert_{(f(x),f^*(\ga))}$. As $C(f):C(X)\ra C(Y)$ maps $C(f):(x,\ga)\ra (f(x),f_*(\ga))$, we have $\pi\ci {}^bN_{C(f)}=C(f)\ci\pi:{}^bN_{C(X)}\ra C(Y)$. From the definitions of ${}^bTf$ in Definition \ref{gc3def13} and ${}^bi_T$ above, we see that the following commutes:
\e
\begin{gathered}
\xymatrix@C=90pt@R=13pt{*+[r]{{}^bN_{C(X)}}
\ar[d]^{{}^bN_{C(f)}} \ar[r]_{{}^bi_T} &
*+[l]{{}^bTX} \ar[d]_{{}^bTf} \\
*+[r]{{}^bN_{C(Y)}} \ar[r]^{{}^bi_T} & *+[l]{{}^bTY.} }
\end{gathered}
\label{gc3eq40}
\e
This characterizes ${}^bN_{C(f)}$, as ${}^bi_T$ in \eq{gc3eq38} is injective.

Now $M_{C(X)}$ is the subset of points $(x,\ga,\al)$ in ${}^bN_{C(X)}$ such that $\al$ maps to $\N\subset\R$. If $\al$ maps to $\N$ then $\al\ci f^*$ maps to $\N$, so ${}^bN_{C(f)}$ maps $M_{C(X)}\ra M_{C(Y)}$. Define $M_{C(f)}:M_{C(X)}\ra M_{C(Y)}$ by $M_{C(f)}={}^bN_{C(f)}\vert_{M_{C(X)}}$.

If $g:Y\ra Z$ is a second interior map of manifolds with g-corners, as $\al\ci f^*\ci g^*=\al\ci(g\ci f)^*$ we see that ${}^bN_{C(g\ci f)}={}^bN_{C(g)}\ci{}^bN_{C(f)}:{}^bN_{C(X)}\ra{}^bN_{C(Z)}$, which implies that $M_{C(g\ci f)}=M_{C(g)}\ci M_{C(f)}$. Also ${}^bN_{C(\id_X)}=\id_{{}^bN_{C(X)}}:{}^bN_{C(X)}\ra{}^bN_{C(X)}$, and $M_{C(\id_X)}=\id_{M_{C(X)}}$. Hence the assignments $X\mapsto {}^bN_{C(X)}$, $f\mapsto{}^bN_{C(f)}$ and $X\mapsto M_{C(X)}$, $f\mapsto M_{C(f)}$ are functorial.

Now let $X$ be a manifold with (ordinary) corners. Then \S\ref{gc24} defined a rank $k$ vector bundle ${}^bN_{C_k(X)}\ra C_k(X)$. Comparing the top row of \eq{gc2eq23} with \eq{gc3eq32}, and noting that the definitions of ${}^bTX$ agree in \S\ref{gc23} and \S\ref{gc35} agree for manifolds with corners, we see that there is a canonical identification between ${}^bN_{C_k(X)}$ defined in \S\ref{gc24}, and ${}^bN_{C_k(X)}$ defined above. One can show this identifies the subsets $M_{C_k(X)}\subset{}^bN_{C_k(X)}$ in \S\ref{gc24} and above. Comparing \eq{gc2eq24} and \eq{gc3eq40}, we see that for $f:X\ra Y$ an interior map of manifolds with corners, the maps ${}^bN_{C(f)}$ defined in \S\ref{gc24} and above coincide under these canonical identifications.

\label{gc3def15}
\end{dfn}

We work out the ideas of Definition \ref{gc3def15} explicitly when~$X=X_P$.

\begin{ex} Let $P$ be a weakly toric monoid, so that $X_P$ is a manifold with g-corners as in Example \ref{gc3ex4}. Example \ref{gc3ex7} gives a canonical diffeomorphism
\e
C_k(X_P)\cong \coprod\nolimits_{\text{faces $F$ of $P$: $\codim F=k$}}X_F.
\label{gc3eq41}
\e
Suppose $(x,\ga)\in C_k(X_P)$ is identified with $x'\in X_F$ by \eq{gc3eq41}, for some face $F$ of $P$. Let $[a]\in\cI_x(X_P)$. Then by Definition \ref{gc3def5}, there exist an open neighbourhood $U$ of $x$ in $X_P$, an element $p\in P$ and a smooth function $h:U\ra(0,\iy)$ such that $a=\la_p\vert_U\cdot h:U\ra[0,\iy)$. Then 
\begin{align*}
\Pi^*([a])\cong [(\la_p\vert_U)\cdot h\ci i_F^P]&\in \cI_{(x,\ga)}(C_k(X_P))\amalg\{0\}\\
&\cong \cI_{x'}(X_F)\amalg\{0\},
\end{align*}
where $i_F^P:X_F\hookra X_P$ is as in Definition \ref{gc3def9}. But
\begin{equation*}
\la_p\ci i_F^P=\begin{cases} \la_p', & p\in F, \\ 0, & p\notin F, \end{cases}
\end{equation*}
where $\la_p'$ means $\la_p$, but on $X_F$ rather than $X_P$. Therefore $\Pi^*([\la_p\cdot h])$ lies in $\cI_{(x,\ga)}(C_k(X_P))$ if and only if $p\in F$. So \eq{gc3eq35} becomes
\begin{align*}
{}^bN_{C_k(X_P)}\vert_{(x,\ga)}=\bigl\{\text{$\al:\{[\la_p\cdot h]:p\in P$, $h$ a germ of positive smooth functions}&\\
\text{near $x$ in $X_P\}\ra\R$ is a monoid morphism, and $\al([\la_p\cdot h])=0$ if $p\in F$}\bigr\}&.
\end{align*}
If $\al\in {}^bN_{C_k(X_P)}\vert_{(x,\ga)}$ then as $\al$ is a monoid morphism and $0\in F$
\begin{equation*}
\al\bigl([\la_p\cdot h]\bigr)=\al([\la_p])+\al([\la_0\cdot h])=\al([\la_p])+0=\al([\la_p]).
\end{equation*}
Thus we have canonical isomorphisms
\e
{}^bN_{C_k(X_P)}\vert_{(x,\ga)}\cong\bigl\{\be\in\Hom_\Mon(P,\R):\be\vert_F=0\bigr\}\cong\Hom(P^\gp/F^\gp,\R).
\label{gc3eq42}
\e
Here in the first step we identify $\al\in{}^bN_{C_k(X_P)}\vert_{(x,\ga)}$ with $\be:P\ra\R$ by if $\al([\la_p\cdot h])=\be(p)$ for all $p,h$. In the second step, such $\be:P\ra\R$ with $\be\vert_F=0$ factor through $\be':P^\gp/F^\gp\ra\R$ as $\R$ is a group. Similarly we have
\e
M_{C_k(X_P)}\vert_{(x,\ga)}\cong\bigl\{\be\in\Hom_\Mon(P,\N):\be\vert_F=0\bigr\}=F^\w,
\label{gc3eq43}
\e
where $F^\w$ is as in Proposition \ref{gc2prop2}(c). It is a toric monoid of rank $k$. We have ${}^bN_{C_k(X_P)}\vert_{(x,\ga)}\cong M_{C_k(X_P)}\vert_{(x,\ga)}\ot_\N\R$, proving~\eq{gc3eq37}.

Combining \eq{gc3eq41}, \eq{gc3eq42} and \eq{gc3eq43} gives identifications like $\Psi_P$ in~\eq{gc3eq27}:
\ea
\Psi_P'&:{}^bN_{C_k(X_P)}\longra\coprod_{\text{faces $F$ of $P$: $\codim F=k$}}X_F\t\Hom(P^\gp/F^\gp,\R),
\label{gc3eq44}\\
\Psi_P''&:M_{C_k(X_P)}\longra \coprod_{\text{faces $F$ of $P$: $\codim F=k$}}X_F\t F^\w.
\label{gc3eq45}
\ea
These give ${}^bN_{C_k(X_P)}$ and $M_{C_k(X_P)}$ the structure of manifolds with g-corners of dimensions $\rank P$ and $\rank P-k$, respectively. The projections $\pi:{}^bN_{C_k(X_P)}\ra C_k(X_P)$, $\pi:M_{C_k(X_P)}\ra C_k(X_P)$ are identified with the projections $X_F\t\Hom(P^\gp/F^\gp,\R)\ra X_F$, $X_F\t F^\w\ra X_F$ for each $F$, and so are smooth.

For the case $P=\N^k\t\Z^{n-k}$, so that $X_P\cong[0,\iy)^k\t\R^{n-k}$ is a manifold with (ordinary) corners, it is easy to check that the manifold with corner structures on ${}^bN_{C_k(X_P)}$ and $M_{C_k(X_P)}$ above coincide with those in~\S\ref{gc24}.

Continuing with the notation above for $(x,\ga)\in C_k(X_P)$ identified with $x'\in X_F$, Proposition \ref{gc3prop7} defined isomorphisms 
\begin{equation*}
{}^bT_xX_P\cong\Hom(P^\gp,\R),\quad\text{and}\quad
{}^bT_{(x,\ga)}(C_k(X_P))\cong {}^bT_{x'}X_F\cong \Hom(F^\gp,\R).
\end{equation*}
Under these isomorphisms and \eq{gc3eq42}, one can show that equation \eq{gc3eq39} is identified with the natural exact sequence
\begin{equation*}
\smash{\xymatrix@C=10pt{ 0 \ar[r] & \Hom(P^\gp/F^\gp,\R) \ar[rr]^(0.55){\ci\pi} && \Hom(P^\gp,\R) \ar[rr]^(0.5){\vert_{F^\gp}} && \Hom(F^\gp,\R) \ar[r] & 0, }}
\end{equation*}
where $\pi:P^\gp\ra P^\gp/F^\gp$ is the projection. Hence \eq{gc3eq39} is exact.
\label{gc3ex13}
\end{ex}

Here is an analogue of Lemma \ref{gc3lem4}. It can be proved by the same method, using the fact that ${}^bN_{C(g)}$ defined in \S\ref{gc24} for manifolds with (ordinary) corners is a smooth map, and agrees with Definition \ref{gc3def15} in this case.

\begin{lem} Let\/ $P,Q$ be weakly toric monoids, $U\subseteq X_P,$ $V\subseteq X_Q$ be open, and\/ $f:U\ra V$ be an interior map, in the sense of Definition\/ {\rm\ref{gc3def5}}. Then the composition of maps
\begin{equation*}
\xymatrix@C=160pt@R=15pt{
*+[r]{\coprod\limits_{\text{faces $F$ of $P$}}(i_F^P)^{-1}(U)\!\t\!\Hom(P^\gp/F^\gp,\R)} \ar@<2.8ex>@{.>}[d] \ar[r]_(0.75){\Psi_P'\vert_{\cdots}^{-1}}^(0.75)\cong & *+[l]{{}^bN_{C(U)}} \ar[d]_{{}^bN_{C(f)}} \\
*+[r]{\coprod\limits_{\text{faces $G$ of $Q$}}(i_G^Q)^{-1}(V)\!\t\!\Hom(Q^\gp/G^\gp,\R)} & *+[l]{{}^bN_{C(V)}} \ar[l]_(0.25){\Psi_Q'\vert_{\cdots}}^(0.25)\cong  }
\end{equation*}
is an interior map of manifolds with g-corners in the sense of\/ {\rm\S\ref{gc32}--\S\ref{gc33},} where ${}^bN_{C(f)}$ is as in Definition\/ {\rm\ref{gc3def15}} and\/ $\Psi'_P,\Psi'_Q$ as in Example\/~{\rm\ref{gc3ex13}}.
\label{gc3lem5}
\end{lem}

\begin{dfn} Let $X$ be a manifold with g-corners, so that ${}^bN_{C_k(X)},M_{C_k(X)}$ are defined as sets in Definition \ref{gc3def15}. Suppose $(P,U,\phi)$ is a g-chart on $X$. For each face $F$ of $P$ with $\codim F=k$, define a g-chart on ${}^bN_{C_k(X)}$ 
\e
\bigl(F\t P^\gp/F^\gp,(i_F^P)^{-1}(U)\t\Hom(P^\gp/F^\gp,\R),{}^bN_{C(\phi)}\ci (\Psi_P')^{-1}\vert_{\cdots}\bigr).
\label{gc3eq46}
\e
Here $(i_F^P)^{-1}(U)\t\Hom(P^\gp/F^\gp,\R)$ is open in $X_F\t\Hom(P^\gp/F^\gp,\R)\cong X_F\t X_{P^\gp/F^\gp}\cong X_{F\t P^\gp/F^\gp}$, we identify $(i_F^P)^{-1}(U)\t\Hom(P^\gp/F^\gp,\R)$ with an open set in $X_{F\t P^\gp/F^\gp}$, and $\Psi_P'$ is as in \eq{gc3eq44}. Similarly, for each $\al\in F^\w$, with $\Psi_P''$ as in \eq{gc3eq45}, define a g-chart on $M_{C_k(X)}$ 
\e
\bigl(F,(i_F^P)^{-1}(U),{}^bM_{C(\phi)}\ci (\Psi_P'')^{-1}\ci(\id\t\al)\bigr),
\label{gc3eq47}
\e
where $\id\t\al:(i_F^P)^{-1}(U)\ra X_F\t F^\w$ maps $\id\t\al:y\mapsto(y,\al)$.

We claim that ${}^bN_{C_k(X)},M_{C_k(X)}$ have unique structures of manifolds with g-corners (including a topology), of dimensions $\dim X$ and $\dim X-k$ respectively, such that \eq{gc3eq46}--\eq{gc3eq47} are g-charts on ${}^bN_{C_k(X)},M_{C_k(X)}$ for all g-charts $(P,U,\phi)$ on $X$, faces $F$ of $P$, and $\al\in F^\w$. To see this, note that if $(P,U,\phi),(Q,V,\psi)$ are g-charts on $X$, then they are compatible, so the change of g-charts morphism $\psi^{-1}\ci\phi:\phi^{-1}\bigl(\phi(U)\cap\psi(V)\bigr)\ra\psi^{-1}\bigl(\phi(U)\cap\psi(V)\bigr)$ is a diffeomorphism between open subsets of $X_P,X_Q$. Applying Lemma \ref{gc3lem5} to $\psi^{-1}\ci\phi$ and its inverse implies that the change of charts morphisms between the g-charts \eq{gc3eq46}--\eq{gc3eq47} from $(P,U,\phi),(Q,V,\psi)$ are also diffeomorphisms, so \eq{gc3eq46}--\eq{gc3eq47} and their analogues for $(Q,V,\psi)$ are compatible.

Thus, the g-charts \eq{gc3eq46} on ${}^bN_{C_k(X)}$ from g-charts $(P,U,\phi)$ on $X$ are all pairwise compatible. These g-charts also cover ${}^bN_{C_k(X)}$, since for fixed $(P,U,\phi)$ the union over all faces $F$ of image of \eq{gc3eq46} is ${}^bN_{C_k(\phi(U))}\subseteq {}^bN_{C_k(X)}$, and the $\phi(U)$ cover $X$, so the ${}^bN_{C_k(\phi(U))}$ cover ${}^bN_{C_k(X)}$. Since $X$ is Hausdorff and second countable, one can show that there is a unique Hausdorff, second countable topology on ${}^bN_{C_k(X)}$ such that for all g-charts \eq{gc3eq46}, ${}^bN_{C(\phi)}\ci (\Psi_P')^{-1}\vert_{\cdots}$ is a homeomorphism with an open set. Therefore the g-charts \eq{gc3eq46} form a g-atlas on ${}^bN_{C_k(X)}$ with this topology, which extends to a unique maximal g-atlas, making ${}^bN_{C_k(X)}$ into a manifold with g-corners. The same argument works for $M_{C_k(X)}$, using the g-charts~\eq{gc3eq47}.

Taking unions now shows that ${}^bN_{C(X)}=\coprod_{k\ge 0}{}^bN_{C_k(X)}$ is a manifold with g-corners of dimension $\dim X$, and $M_{C(X)}=\coprod_{k\ge 0}M_{C_k(X)}$ an object of~$\cMangc$.

Definition \ref{gc3def15} also defined an inclusion of sets $M_{C_k(X)}\hookra {}^bN_{C_k(X)}$, and maps of sets $\pi:{}^bN_{C_k(X)}\ra C_k(X)$, $\pi:M_{C_k(X)}\ra C_k(X)$, ${}^bi_T:{}^bN_{C_k(X)}\ra\Pi^*({}^bTX)$ and ${}^b\pi_T:\Pi^*({}^bTX)\ra {}^bT(C_k(X))$. Example \ref{gc3ex13} showed that in the local models $X_P$, these are smooth, interior maps, with $M_{C_k(X_P)}\hookra{}^bN_{C_k(X_P)}$ an embedded submanifold, $\pi:{}^bN_{C_k(X_P)}\ra C_k(X_P)$ a vector bundle of rank $k$, and $\pi:M_{C_k(X_P)}\ra C_k(X_P)$ a locally constant bundle of toric monoids, and ${}^bi_T,{}^b\pi_T$ bundle-linear and forming an exact sequence \eq{gc3eq32}. Thus, using the g-charts \eq{gc3eq46}--\eq{gc3eq47}, we see that for general manifolds with g-corners $X$, $M_{C_k(X)}$ is an embedded submanifold of ${}^bN_{C_k(X)}$, and $\pi:{}^bN_{C_k(X)}\ra C_k(X)$ is interior and makes ${}^bN_{C_k(X)}$ into a vector bundle of rank $k$, and $\pi:M_{C_k(X)}\ra C_k(X)$ is interior and a locally constant bundle of toric monoids, and ${}^bi_T,{}^b\pi_T$ are morphisms of vector bundles in an exact sequence~\eq{gc3eq32}.

Since $\pi:{}^bN_{C_k(X)}\ra C_k(X)$ is a vector bundle, it has a dual vector bundle, which we call the {\it b-conormal bundle\/} and write as $\pi:{}^bN^*_{C_k(X)}\ra C_k(X)$. Similarly, $\pi:M_{C_k(X)}\ra C_k(X)$ has a natural dual bundle $\pi:M_{C_k(X)}^\vee\ra C_k(X)$, the {\it comonoid bundle}, with fibres $M_{C_k(X)}^\vee\vert_{(x,\ga)}$ the dual toric monoids $M_{C_k(X)}\vert_{(x,\ga)}^\vee$. Equation \eq{gc3eq37} implies there is a natural inclusion $M_{C_k(X)}^\vee\hookra {}^bN^*_{C_k(X)}$ as an embedded submanifold.

Now let $f:X\ra Y$ be an interior map of manifolds with g-corners. Then for all g-charts $(P,U,\phi)$ on $X$ and $(Q,V,\psi)$ on $Y$, the map $\psi^{-1}\ci f\ci\phi$ in \eq{gc3eq5} is an interior map between open subsets of $X_P,X_Q$. Applying Lemma \ref{gc3lem5} shows that the corresponding maps for ${}^bN_{C(f)}:{}^bN_{C(X)}\ra{}^bN_{C(Y)}$ and $M_{C(f)}:M_{C(X)}\ra M_{C(Y)}$ and the g-charts \eq{gc3eq46}--\eq{gc3eq47} from $(P,U,\phi),(Q,V,\psi)$ are also interior. As these g-charts cover ${}^bN_{C(X)},{}^bN_{C(Y)},M_{C(X)},M_{C(Y)}$, this proves that ${}^bN_{C(f)}:{}^bN_{C(X)}\ra{}^bN_{C(Y)}$ and $M_{C(f)}:M_{C(X)}\ra M_{C(Y)}$ are interior morphisms in $\Mangc$ and $\cMangc$.

Since $\pi\ci {}^bN_{C(f)}=C(f)\ci\pi$ and ${}^bN_{C(f)}$ is bundle-linear, we may also regard ${}^bN_{C(f)}$ as a morphism ${}^bN_{C(f)}:{}^bN_{C(X)}\ra C(f)^*({}^bN_{C(Y)})$ of vector bundles of mixed rank over $C(X)$, with dual morphism ${}^bN^*_{C(f)}:C(f)^*({}^bN_{C(Y)}^*)\ra{}^bN^*_{C(X)}$. Similarly, we can regard $M_{C(f)}$ as a morphism $M_{C(f)}:M_{C(X)}\ra C(f)^*(M_{C(Y)})$ of toric monoid bundles over $C(X)$, with dual morphism $M_{C(f)}^\vee:C(f)^*(M^\vee_{C(Y)})\ra M^\vee_{C(X)}$.

Definition \ref{gc3def13} showed that the maps $N_{C(f)},M_{C(f)}$ are functorial. Thus
$X\mapsto {}^bN_{C(X)}$, $f\mapsto{}^bN_{C(f)}$ defines functors ${}^bN_C:\Mangcin\ra\Mangcin$ and ${}^bN_C:\cMangcin\ra\cMangcin$, which we call the {\it b-normal corner functors}. Similarly $X\mapsto M_{C(X)}$, $f\mapsto M_{C(f)}$ defines functors
$M_C:\Mangcin,\cMangcin\ra\cMangcin$, which we call the {\it monoid corner functors}.
\label{gc3def16}
\end{dfn}

We show ${}^bN_{C_k(X)},M_{C(X)}$ are compatible with products.

\begin{ex} Let $X,Y$ be manifolds with g-corners, and consider the product $X\t Y$. Then $C(X\t Y)\cong C(X)\t C(Y)$, as in \S\ref{gc34}. The projections $\pi_X:X\t Y\ra X$, $\pi_Y:X\t Y\ra Y$ are interior maps, so we may form ${}^bN_{C(\pi_X)},{}^bN_{C(\pi_Y)},M_{C(\pi_X)},M_{C(\pi_Y)}$, and take the direct products
\ea
({}^bN_{C(\pi_X)},{}^bN_{C(\pi_Y)})&:{}^bN_{C(X\t Y)}\longra {}^bN_{C(X)}\t {}^bN_{C(Y)},
\label{gc3eq48}\\
(M_{C(\pi_X)},M_{C(\pi_Y)})&:M_{C(X\t Y)}\longra M_{C(X)}\t M_{C(Y)}.
\label{gc3eq49}
\ea
Considering local models as in Example \ref{gc3ex13}, we find that \eq{gc3eq48}--\eq{gc3eq49} are diffeomorphisms. We sometimes use \eq{gc3eq48}--\eq{gc3eq49} to identify ${}^bN_{C(X\t Y)},\ab M_{C(X\t Y)}$ with ${}^bN_{C(X)}\t {}^bN_{C(Y)},M_{C(X)}\t M_{C(Y)}$. The functors ${}^bN_C,M_C$ preserve products and direct products, in the sense of Proposition~\ref{gc2prop1}(f). 
\label{gc3ex14}
\end{ex}

As for Proposition \ref{gc2prop5} we have:

\begin{prop} Definition\/ {\rm\ref{gc3def16}} defines functors
${}^bN_C:\Mangcin\ra\Mangcin,$ ${}^bN_C:\cMangcin\ra\cMangcin$ and\/
$M_C:\Mangcin,\cMangcin\ra\cMangcin,$ preserving (direct) products, with a commutative diagram of natural transformations:
\begin{equation*}
\xymatrix@C=60pt@R=2pt{ & M_C \ar@{=>}[dd]^{\text{inclusion}}
\ar@{=>}[dr]^\Pi \\ C \ar@{=>}[ur]^{\raisebox{2pt}{$\st\text{zero section $0$\qquad}$}}
\ar@{=>}[dr]_{\text{zero section $0$\qquad}} && C. \\ & {}^bN_C
\ar@{=>}[ur]_\Pi }
\end{equation*}
\label{gc3prop8}
\end{prop}

Here is the analogue of Definition \ref{gc2def13}.

\begin{dfn} Let $X$ be a manifold with g-corners. For $x\in S^k(X)\subseteq X$, as above we have a monoid $\cI_x(X)$ of germs $[c]$ at $x$ of interior functions $c:X\ra[0,\iy)$, and a submonoid $C^\iy_x(X,(0,\iy))\subseteq\cI_x(X)$ of $[c]$ with $c(x)>0$. In a similar way to \eq{gc3eq35}--\eq{gc3eq36}, define
\ea
\begin{split}
{}^b\ti N_xX&=\bigl\{\al\in \Hom_\Mon\bigl(\cI_x(X),\R\bigr):\al\vert_{C^\iy_x(X,(0,\iy))}=0\bigr\}, 
\end{split}
\label{gc3eq50}\\
\begin{split}
{}^b\ti N_x^{\sst\ge 0}X&=\bigl\{\al\in \Hom_\Mon\bigl(\cI_x(X),([0,\iy),+)\bigr):\\
&\qquad\qquad\qquad\qquad\qquad\qquad \al\vert_{C^\iy_x(X,(0,\iy))}=0\bigr\}, 
\end{split}
\label{gc3eq51}\\
\begin{split}
\ti M_xX&=\bigl\{\al\in \Hom_\Mon\bigl(\cI_x(X),\N\bigr):\al\vert_{C^\iy_x(X,(0,\iy))}=0\bigr\}, 
\end{split}
\label{gc3eq52}
\ea
so that $\ti M_xX\subseteq{}^b\ti N_x^{\sst\ge 0}X\subseteq {}^b\ti N_xX$. Here $([0,\iy),+)$ in \eq{gc3eq51} is $[0,\iy)$ with monoid operation addition, rather than multiplication as usual. As in Example \ref{gc3ex13}, one can show that $\ti M_xX\cong\N^k$ is a toric monoid of rank $k=\depth_Xx$, with ${}^b\ti N_xX=\ti M_xX\ot_\N\R\cong\R^k$ the corresponding real vector space, and ${}^b\ti N_x^{\sst\ge 0}X$ as the corresponding rational polyhedral cone in ${}^b\ti N_xX$, as in~\S\ref{gc314}.

Now let $f:X\ra Y$ be an interior map of manifolds with g-corners, and $x\in X$ with $f(x)=y\in Y$. As for ${}^bN_{C(f)}$ in Definition \ref{gc3def15}, define maps ${}^b\ti N_xf:{}^b\ti N_xX\ra {}^b\ti N_yY$, ${}^b\ti N^{\sst\ge 0}_xf:{}^b\ti N_x^{\sst\ge 0}X\ra {}^b\ti N_y^{\sst\ge 0}Y$ and $\ti M_xf:\ti M_xX\ra\ti M_yY$ to map $\al\mapsto \al\ci f^*$, where $f^*:\cI_y(Y)\ra \cI_x(X)$ maps $[c]\mapsto [c\ci f]$. Then ${}^b\ti N_xf$ is linear, and ${}^b\ti N^{\sst\ge 0}_xf,\ti M_xf$ are monoid morphisms. These ${}^b\ti N_xX,\ab {}^b\ti N^{\sst\ge 0}_xX,\ab\ti M_xX,\ab{}^b\ti N_xf,\ab{}^b\ti N^{\sst\ge 0}_xf,\ab\ti M_xf$ are functorial.

When $X,Y$ are manifolds with (ordinary) corners, these definitions of ${}^b\ti N_xX,\ab {}^b\ti N^{\sst\ge 0}_xX,\ab\ti M_xX,\ab{}^b\ti N_xf,\ab{}^b\ti N^{\sst\ge 0}_xf,\ab\ti M_xf$ are canonically isomorphic to those in~\S\ref{gc24}.

We could define ${}^b\ti NX=\bigl\{(x,v):x\in X$, $v\in {}^b\ti N_xX\bigr\}$ and ${}^b\ti Nf:{}^b\ti NX\ra{}^b\ti NY$ by ${}^b\ti Nf:(x,v)\mapsto (f(x),{}^b\ti N_xf(v))$, and similarly for ${}^b\ti N^{\sst\ge 0}X,{}^b\ti N^{\sst\ge 0}f$ and $\ti MX,\ti Mf$, and these would also be functorial. They are useful for stating conditions on interior $f:X\ra Y$. However, in contrast to ${}^bN_{C(X)}$ above, these ${}^b\ti NX,{}^b\ti N^{\sst\ge 0}X$ {\it would not be manifolds with g-corners}, as the dimensions of ${}^b\ti N_xX,\ab {}^b\ti N^{\sst\ge 0}_xX$ vary discontinuously with $x$ in $X$. The rational polyhedral cones ${}^b\ti N_x^{\sst\ge 0}X$ may not be manifolds with g-corners either. 
\label{gc3def17}
\end{dfn}

The relation between $M_{C_k(X)}\vert_{(x,\ga)}$ and $\ti M_xX$ in Definitions \ref{gc3def15} and \ref{gc3def17} is this: for each $x\in S^k(X)\subseteq X$, there is a unique local $k$-corner component $\ga$ to $X$ at $x$, and then $M_{C_k(X)}\vert_{(x,\ga)}\cong\ti M_xX$. More generally, if $\de$ is some local $l$-corner component of $X$ at $x$ for $l=0,\ldots,k$, then $M_{C_l(X)}\vert_{(x,\de)}\cong \ti M_xX/F$ for some face $F$ of $\ti M_xX$ with $\rank F=k-l$, and there is a 1-1 correspondence between such $\de$ and such $F$. Also, writing $P=\ti M_xX$, as a toric monoid, then $X$ near $x$ is locally modelled on $X_P\t\R^{\dim X-k}$ near $(\de_0,0)$. Since $X_P$ is a manifold with (ordinary) corners near $\de_0$ if and only if $P\cong\N^k$, we deduce:

\begin{lem} Let\/ $X$ be a manifold with g-corners. Then $X$ is a manifold with corners if and only if the following two equivalent conditions hold:
\begin{itemize}
\setlength{\itemsep}{0pt}
\setlength{\parsep}{0pt}
\item[{\bf(i)}] $M_{C_k(X)}\vert_{(x,\ga)}\cong\N^k$ for all\/ $(x,\ga)\in C_k(X)$ and\/~$k\ge 0$.
\item[{\bf(ii)}] $\ti M_xX\cong\N^k$ for all\/ $x\in X,$ for $k\ge 0$ depending on $x$.
\end{itemize}
\label{gc3lem6}
\end{lem}

\section[Differential geometry of manifolds with g-corners]{Differential geometry of manifolds with \\ g-corners}
\label{gc4}

We now extend parts of ordinary differential geometry to manifolds with g-corners: special classes of smooth maps; immersions, embeddings and submanifolds; transversality and fibre products in $\Mangc,\Mangcin$; and other topics. The proofs of Theorems \ref{gc4thm2}, \ref{gc4thm3}, \ref{gc4thm5}, \ref{gc4thm6} and \ref{gc4thm7} below are deferred to~\S\ref{gc5}.

\subsection{Special classes of smooth maps}
\label{gc41}

We define several classes of smooth maps of manifolds with g-corners.

\begin{dfn} Let $f:X\ra Y$ be a smooth map of manifolds with g-corners. We call $f$ {\it simple\/} if either (hence both) of the following two conditions hold: 
\begin{itemize}
\setlength{\itemsep}{0pt}
\setlength{\parsep}{0pt}
\item[(i)] $M_{C(f)}:M_{C(X)}\ra C(f)^*(M_{C(Y)})$ in \S\ref{gc36} is an isomorphism of monoid bundles over $C(X)$.
\item[(ii)] $f$ is interior and $\ti M_xf:\ti M_xX\ra\ti M_{f(x)}Y$ in Definition \ref{gc3def17} is an isomorphism of monoids for all $x\in X$.
\end{itemize}
It is easy to show that (i) and (ii) are equivalent. For manifolds with (ordinary) corners, this agrees with the definition of simple maps in~\S\ref{gc21}.

Clearly, compositions of simple morphisms, and identity morphisms, are simple. Thus, we may define subcategories $\Mangcsi\subset\Mangc$ and $\cMangcsi\subset\cMangc$ with all objects, and morphisms simple maps. Simple maps are closed under products (that is, if $f:W\ra Y$, $g:X\ra Z$ are simple then $f\t g:W\t X\ra Y\t Z$ is simple), but not under direct products (that is, if $f:X\ra Y$, $g:X\ra Z$ are simple then $(f,g):X\ra Y\t Z$ need not be simple).

Suppose $f:X\ra Y$ is a simple morphism in $\Mangc$. Then $C(f):C(X)\ra C(Y)$ is a simple morphism in $\cMangc$. If $(x,\ga)\in C_k(X)$ with $C(f)(x,\ga)=(y,\de)\in C_l(Y)$ then $M_{C(X)}\vert_{(x,\ga)}\cong M_{C(Y)}\vert_{(y,\de)}$ by (i). But $k=\rank M_{C(X)}\vert_{(x,\ga)}$, $l=\rank M_{C(Y)}\vert_{(y,\de)}$, so $k=l$, and $C(f)$ maps $C_k(X)\ra C_k(Y)$ for all $k=0,1,\ldots,$ and maps $\pd X\ra\pd Y$ when~$k=1$.

Thus, we may define a {\it boundary functor\/} $\pd:\Mangcsi\ra\Mangcsi$ mapping $X\mapsto\pd X$ on objects and\/ $f\mapsto \pd f:=C(f)\vert_{C_1(X)}:\pd X\ra\pd Y$ on (simple) morphisms $f:X\ra Y$, and for all\/ $k\ge 0$ a $k$-{\it corner functor\/} $C_k:\Mangcsi\ra\Mangcsi$ mapping $X\mapsto C_k(X)$ on objects and $f\mapsto C_k(f):=C(f)\vert_{C_k(X)}:C_k(X)\ra C_k(Y)$ on morphisms. They extend to~$\pd,C_k:\cMangcsi\ra\cMangcsi$.
\label{gc4def1}
\end{dfn}

Diffeomorphisms are simple maps. Simple maps are important in the definition of Kuranishi spaces with corners in \cite{Joyc5}. Next we define {\it b-normal\/} maps between manifolds with g-corners. For manifolds with (ordinary) corners, these were introduced by Melrose \cite{Melr1,Melr2,Melr3,Melr4}, and several equivalent definitions appear in the literature, two of which we extend to manifolds with g-corners. For manifolds with corners, part (i) below (translated into our notation) appears in Grieser \cite[Def.~3.9]{Grie}, and part (ii) in \cite[\S 2]{Melr2}, \cite[Def.~2.4.14]{Melr4}.

\begin{dfn} Let $f:X\ra Y$ be a smooth map of manifolds with g-corners. We call $f$ {\it b-normal\/} if either of the following two equivalent conditions hold: 
\begin{itemize}
\setlength{\itemsep}{0pt}
\setlength{\parsep}{0pt}
\item[(i)] $C(f):C(X)\ra C(Y)$ in \S\ref{gc34} maps $C_k(X)\ra \coprod_{j=0}^kC_j(Y)$ for all~$k\ge 0$.
\item[(ii)] $f$ is interior, and ${}^bN_{C(f)}:{}^bN_{C(X)}\ra C(f)^*\bigl({}^bN_{C(Y)}\bigr)$ in \S\ref{gc36} is a surjective morphism of vector bundles of mixed rank on $C(X)$.
\end{itemize}

For manifolds with (ordinary) corners, this agrees with the definition of b-normal maps in \S\ref{gc21}, by Proposition~\ref{gc2prop1}(c).

B-normal maps are closed under composition and include identities, so manifolds with g-corners with b-normal maps define a subcategory of $\Mangc$. B-normal maps are closed under products, but not under direct products, as Example \ref{gc2ex3}(a) shows.

The following notation is sometimes useful, for instance in describing boundaries of fibre products. If $f:X\ra Y$ is b-normal then $C(f)$ maps $C_1(X)\ra C_0(Y)\amalg C_1(Y)$, where $C_1(X)=\pd X$, $C_1(Y)=\pd Y$, and $\io:Y\ra C_0(Y)$ is a diffeomorphism. Define $\pd_+^fX=C(f)\vert_{C_1(X)}^{-1}(C_0(Y))$ and $\pd_-^fX=C(f)\vert_{C_1(X)}^{-1}(C_1(Y))$. Then $\pd_\pm^f(X)$ are open and closed in $\pd X$, with $\pd X=\pd^f_+X\amalg \pd^f_-X$. Define $f_+:\pd^f_+X\ra Y$ and $f_-:\pd^f_-X\ra\pd Y$ by $f_+=\io^{-1}\ci C(f)\vert_{\pd^f_+X}$ and~$f_-=C(f)\vert_{\pd^f_-X}$. 

Then $f_\pm$ are smooth maps of manifolds with g-corners. Also, $C(f_\pm)$ are related to $C(f)$ by an \'etale cover, so by (i) or (ii) we see that $f_+$ and $f_-$ are both b-normal. So we can iterate the process, and define $f_{+,-}:\pd^{f_+}_-\pd^f_+X\ra \pd Y$, and so on, where~$\pd^2X=\pd^{f_+}_+\pd^f_+X\amalg\pd^{f_+}_-\pd^f_+X\amalg\pd^{f_-}_+\pd^f_-X\amalg\pd^{f_-}_-\pd^f_-X$.
\label{gc4def2}
\end{dfn}

A smooth map $f:X\ra Y$ of manifolds without boundary is a {\it submersion\/} if $\d f:TX\ra f^*(TY)$ is a surjective morphism of vector bundles on $X$. For manifolds with corners, {\it b-submersions\/} and {\it b-fibrations\/} are two notions of submersions, as in Melrose \cite[\S I]{Melr1}, \cite[\S 2]{Melr2}, \cite[\S 2.4]{Melr4}. Both are important in Melrose's theory of analysis on manifolds with corners. We extend to g-corners.

\begin{dfn} Let $f:X\ra Y$ be an interior map of manifolds with g-corners. We call $f$ a {\it b-submersion\/} if ${}^b\d f:{}^bTX\ra f^*({}^bTY)$ is a surjective morphism of vector bundles on $X$. We call $f$ a {\it b-fibration\/} if $f$ is b-normal and a b-submersion.

B-submersions and b-fibrations are both closed under composition and contain identities, and so define subcategories of $\Mangc$. B-submersions and b-fibrations are closed under products, but not under direct products.

If $f$ is a b-submersion or b-fibration of manifolds with (ordinary) corners, so that $TX,TY$ are defined, then $\d f:TX\ra f^*(TY)$ need not be surjective.
\label{gc4def3}
\end{dfn}

\begin{ex}{\bf(i)} Any projection $\pi_X:X\t Y\ra X$ for $X,Y$ manifolds with g-corners is b-normal, a b-submersion, and a b-fibration.
\smallskip

\noindent{\bf(ii)} Define $f:[0,\iy)^2\ra[0,\iy)$ by $f(x,y)=xy$. Then ${}^b\d f$ is given by the matrix $\bigl(\begin{smallmatrix} 1 \\ 1 \end{smallmatrix}\bigr)$ with respect to the bases $\bigl(x\frac{\pd}{\pd x},y\frac{\pd}{\pd y}\bigr)$ for ${}^bT\bigl([0,\iy)^2\bigr)$ and $z\frac{\pd}{\pd z}$ for ${}^bT\bigl([0,\iy)\bigr)$, so ${}^b\d f$ is surjective, and $f$ is a b-submersion. Also $C(f)$ maps $C_0\bigl([0,\iy)^2\bigr)\!\ra\! C_0\bigl([0,\iy)\bigr)$, $C_1\bigl([0,\iy)^2\bigr)\!\ra\! C_1\bigl([0,\iy)\bigr)$, and $C_2\bigl([0,\iy)^2\bigr)\!\ra\! C_1\bigl([0,\iy)\bigr)$. Thus $f$ is b-normal by Definition \ref{gc4def2}(i), and a b-fibration.
\smallskip

\noindent{\bf(iii)} Define $g:[0,\iy)\t\R\ra[0,\iy)^2$ by $g(w,x)=(w,we^x)$. Then ${}^b\d g$ is given by the matrix $\bigl(\begin{smallmatrix} 1 & 0 \\ 1 & 1 \end{smallmatrix}\bigr)$ with respect to the bases $\bigl(w\frac{\pd}{\pd w},\frac{\pd}{\pd x}\bigr)$ for ${}^bT\bigl([0,\iy)\t\R\bigr)$ and $\bigl(y\frac{\pd}{\pd y},z\frac{\pd}{\pd z}\bigr)$ for ${}^bT\bigl([0,\iy)^2\bigr)$, so $g$ is a b-submersion. 

Also $C(g)$ maps $C_0\bigl([0,\iy)\t\R\bigr)\ra C_0\bigl([0,\iy)^2)$, but $C_1\bigl([0,\iy)\t\R\bigr)\ra C_2\bigl([0,\iy)^2\bigr)$. Thus $g$ is not b-normal, or a b-fibration, by Definition~\ref{gc4def2}(i).
\label{gc4ex1}
\end{ex}

\subsection{Immersions, embeddings, and submanifolds}
\label{gc42}

Recall some definitions and results for ordinary manifolds without boundary:

\begin{dfn} A smooth map $i:X\ra Y$ of manifolds without boundary $X,Y$ is an {\it immersion\/} if $\d i:TX\ra i^*(TY)$ is an injective morphism of vector bundles on $X$, and an {\it embedding\/} if also $i:X\ra i(X)$ is a homeomorphism, where $i(X)\subseteq Y$ is the image.

An {\it immersed\/} (or {\it embedded\/}) {\it submanifold\/} $X$ of $Y$ is an immersion (or embedding) $i:X\ra Y$, where usually we take $i$ to be implicitly given. For the case of embedded submanifolds, as in Remark \ref{gc4rem1}(A) below we often identify $X$ with the image $i(X)\subseteq Y$, and consider $X$ to be a subset of~$Y$.
\label{gc4def4}
\end{dfn}

\begin{thm} Let\/ $i:X\ra Y$ be an embedding of manifolds without boundary $X,Y$ of dimensions $m,n$. Then for each\/ $x\in X,$ there exist local coordinates $(y_1,\ldots,y_n)$ defined on an open neighbourhood\/ $V$ of\/ $i(x)$ in $Y,$ such that\/ $i(X)\cap V=\bigl\{(y_1,\ldots,y_m,0,\ldots,0)\in V\bigr\},$ and setting\/ $U=i^{-1}(V)\subseteq X$ and\/ $x_a=y_a\ci i:U\ra\R,$ then $(x_1,\ldots,x_m)$ are local coordinates on\/~$U\subseteq X$.
\label{gc4thm1}
\end{thm}

\begin{rem} Theorem \ref{gc4thm1} has two important consequences:
\begin{itemize}
\setlength{\itemsep}{0pt}
\setlength{\parsep}{0pt}
\item[(A)] We can give the image $i(X)$ the canonical structure of a manifold without boundary, depending only on the subset $i(X)\subseteq Y$. Then $i:X\ra i(X)$ is a diffeomorphism. Thus, we can regard embedded submanifolds $X\hookra Y$ as being special subsets $X\subseteq Y$, rather than special smooth maps $i:X\ra Y$.
\item[(B)] Locally in $Y$, we can describe embedded submanifolds $X\hookra Y$ in two complementary ways: {\it either\/} as the image of an embedding $i:X\ra Y$, {\it or\/} as the zeroes $y_{m+1}=\cdots=y_n=0$ of $\dim Y-\dim X$ local, transverse smooth functions~$y_{m+1},\ldots,y_n:Y\ra\R$. 
\end{itemize}

\label{gc4rem1}
\end{rem}

We now extend all this to manifolds with g-corners. Our aim is to give a definition of embedding $i:X\ra Y$ of manifolds with g-corners $X,Y$, which is as general as possible such that an analogue of Theorem \ref{gc4thm1} holds, and in particular, so that the manifold with g-corner structure on $X$ can be recovered up to canonical diffeomorphism from the subset $i(X)\subseteq Y$ and the manifold with g-corner structure on $Y$. This turns out to be quite complicated.

For interior maps $i:X\ra Y$ of manifolds with g-corners, the obvious way to define immersions would be to require ${}^b\d i:{}^bTX\ra i^*({}^bTY)$ to be injective. This is implied by the definition, but we also impose extra conditions on how $i$ acts on the monoids $\ti M_xX,\ti M_yY$ and tangent spaces to strata~$T_xS^k(X),T_yS^l(Y)$.

\begin{dfn} Let $i:X\ra Y$ be a smooth map of manifolds with g-corners, or more generally a morphism in $\cMangc$. We will define when $i$ is an immersion, first when $i$ is interior, and then in the general case.

If $i$ is interior, we call $i$ an {\it immersion\/} (or {\it interior immersion\/}) if whenever $x\in S^k(X)\subseteq X$ with $i(x)=y\in S^l(Y)\subseteq Y$, then: 
\begin{itemize}
\setlength{\itemsep}{0pt}
\setlength{\parsep}{0pt}
\item[(i)] $\d(i\vert_{S^k(X)})\vert_x:T_xS^k(X)\ra T_yS^l(Y)$ must be injective;
\item[(ii)] The monoid morphism $\ti M_xi:\ti M_xX\ra\ti M_yY$ (defined as $i$ is assumed interior) must be injective; and 
\item[(iii)] The quotient monoid $\ti M_yY\big/(\ti M_xi)[\ti M_xX]$ must be torsion-free.
\end{itemize}

To understand this, note that we have noncanonical splittings
\begin{equation*}
{}^bT_xX\cong (\ti M_xX\ot_\Z\R) \op T_xS^k(X),\quad {}^bT_yY\cong (\ti M_yY\ot_\Z\R) \op T_yS^l(Y),
\end{equation*}
and with respect to these we have
\e
{}^bT_xi=\begin{pmatrix} \ti M_xi\ot_\Z\R & * \\ 0 & \d (i\vert_{S^k(X)})\vert_x \end{pmatrix}: \begin{matrix} \ti M_xX\ot_\Z\R \\ {}\op T_xS^k(X) \end{matrix} \longra 
\begin{matrix} \ti M_yY\ot_\Z\R \\ {}\op T_yS^l(Y). \end{matrix}
\label{ag4eq1}
\e
Conditions (i),(ii) are equivalent to the diagonal terms in this matrix being injective, and so imply that ${}^bT_xi:{}^bT_xX\ra {}^bT_yY$ is injective. Conversely, ${}^bT_xi$ injective implies (ii), but not necessarily (i). So for $i$ to be an interior immersion implies that ${}^b\d i:{}^bTX\ra i^*({}^bTY)$ is an injective morphism of vector bundles, but is stronger than this.

If $i:X\ra Y$ is a general smooth map of manifolds with g-corners then $C(i):C(X)\ra C(Y)$ is an interior morphism in $\cMangc$, and we call $i$ an {\it immersion\/} if $C(i)$ is an interior immersion in the sense above. It is not difficult to show that $C(i)\vert_{C_0(X)}:C_0(X)\ra C(Y)$ an interior immersion implies $C(i)\vert_{C_k(X)}:C_k(X)\ra C(Y)$ is an interior immersion for $k>0$, so we could instead say $i$ is an immersion if $C(i)\vert_{C_0(X)}:C_0(X)\ra C(Y)$ is an interior immersion. If $i$ is interior then $C(i)\vert_{C_0(X)}$ maps $C_0(X)\ra C_0(Y)$ and is naturally identified with $i:X\ra Y$, so this yields the same definition of immersion as before.

We call $i:X\ra Y$ an {\it embedding\/} if it is an immersion, and $i:X\ra i(X)$ is a homeomorphism (so in particular, $i$ is injective). We call $i:X\ra Y$ an {\it s-immersion\/} (or {\it s-embedding\/}) if it is a simple immersion (or simple embedding).

An {\it immersed}, or {\it embedded}, or {\it s-immersed}, or {\it s-embedded submanifold\/} $X$ of $Y$ is an immersion, or embedding, or s-immersion, or s-embedding $i:X\ra Y$, respectively, where usually we take $i$ to be implicitly given. For the case of (s-)embedded submanifolds, we often identify $X$ with the image $i(X)\subseteq Y$, and consider $X$ to be a subset of~$Y$.
\label{gc4def5}
\end{dfn}

\begin{ex}{\bf(i)} Define $X=Y=[0,\iy)$, and $f:X\ra Y$ by $f(x)=x^2$. Then $f$ is interior, and ${}^b\d f:{}^bTX\ra f^*({}^bTY)$ maps $x\frac{\pd}{\pd x}\mapsto 2y\frac{\pd}{\pd y}$, and so is an isomorphism of vector bundles. However, $f$ is not an immersion or embedding, because $\ti M_0f:\ti M_0X\ra\ti M_0Y$ is the map $\N\ra\N$, $n\mapsto 2n$, so the quotient monoid $\ti M_0Y\big/(\ti M_0f)[\ti M_0X]$ is $\N/2\N=\Z_2$, which is not torsion-free.

We do not want $f$ to be an embedding, as Remark \ref{gc4rem1}(A) fails for $f$. As $f(X)=Y$, the only sensible manifold with g-corners structure on $f(X)$ depending only on $f(X)\subseteq Y$ and the manifold with g-corners structure on $Y$, is to give $f(X)$ the same manifold with g-corners structure as $Y$. But then $f:X\ra f(X)$ is not a diffeomorphism. The torsion-free condition in Definition \ref{gc4def5}(iii), which fails for $f$, will be needed to prove the analogue of Remark~\ref{gc4rem1}(A).
\smallskip

\noindent{\bf(ii)} Define $X=[0,\iy)$, $Y=[0,\iy)^2$, and $g:X\ra Y$ by $g(x)=(x^2,x^3)=(y,z)$. Then $g$ is interior, and ${}^b\d g:{}^bTX\ra g^*({}^bTY)$ maps $x\frac{\pd}{\pd x}\mapsto 2y\frac{\pd}{\pd y}+3z\frac{\pd}{\pd z}$, and so is an injective morphism of vector bundles. 

The monoid morphism $\ti M_0g:\ti M_0X\ra\ti M_0Y$ is the map $\N\ra\N^2$, $n\mapsto (2n,3n)$. The quotient monoid $\ti M_0Y\big/\ti M_0g[\ti M_0X]$ is $\Z$, which is torsion-free, with projection $\ti M_0Y\ra\ti M_0Y\big/\ti M_0g[\ti M_0X]$ the map $\N^2\ra\Z$ taking $(m,n)\mapsto 3m-2n$. So $g$ is an embedding. Here the torsion-free condition holds as the powers 2,3 in $g(x)=(x^2,x^3)$ have highest common factor~1.

Note that the embedded submanifold $g(X)\subset Y$ may be defined as the solutions of the equation $y^3=z^2$ in $Y$, in smooth maps~$Y\ra[0,\iy)$.

Note too that the smooth function $x:X\ra[0,\iy)$ cannot be written $h\ci g$ for any smooth function $h:Y\ra[0,\iy)$. So when we identify $X$ with the diffeomorphic embedded submanifold $g(X)\subset Y$, this does not imply that the smooth functions $X\ra\R$ or $X\ra[0,\iy)$ can be identified with the restrictions of smooth functions $Y\ra\R$ or $Y\ra[0,\iy)$ to $g(X)\subset Y$.

\smallskip

\noindent{\bf(iii)} Define $X=[0,\iy)\t\R$, $Y=[0,\iy)^2$, and $h:X\ra Y$ by $h(w,x)=(w,we^x)\ab=(y,z)$. Then $h$ is interior, and ${}^b\d h:{}^bTX\ra h^*({}^bTY)$ is given by the matrix $\bigl(\begin{smallmatrix} 1 & 0 \\ 1 & 1 \end{smallmatrix}\bigr)$ with respect to the bases $\bigl(w\frac{\pd}{\pd w},\frac{\pd}{\pd x}\bigr)$ for ${}^bTX$ and $\bigl(y\frac{\pd}{\pd y},z\frac{\pd}{\pd z}\bigr)$ for ${}^bTY$, so ${}^b\d h$ is an isomorphism. The monoid morphism $\ti M_{(0,x)}h:\ti M_{(0,x)}X\ra\ti M_{(0,0)}Y$ maps $\N\ra\N^2$, $n\mapsto (n,n)$, and the quotient monoid $\ti M_{(0,0)}Y\big/\ti M_{(0,x)}h[\ti M_{(0,x)}X]$ is $\Z$, which is torsion-free.

However, $h$ is not an immersion or embedding, as at $(0,x)\in S^1(X)$ with $h(0,x)=(0,0)\in S^2(Y)$, the map $\d (h\vert_{S^1(X)})\vert_{(0,x)}:T_{(0,x)}S^1(X)\ra T_{(0,0)}S^2(Y)$ in Definition \ref{gc4def5}(i) maps $\R\ra 0$, and is not injective.
\smallskip

\noindent{\bf(iv)} As in Proposition \ref{gc3prop3}, let $P$ be a weakly toric monoid, choose generators $p_1,\ldots,p_m$ for $P$ and a generating set of relations \eq{gc3eq3} for $p_1,\ldots,p_m$, and consider the interior map $\La=\la_{p_1}\t\cdots\t\la_{p_m}:X_P\ra[0,\iy)^m$, which has image $\La(X_P)=X_P'\subseteq[0,\iy)^m$ defined in \eq{gc3eq4} by equations in~$[0,\iy)^m$.

One can check that $\La:X_P\ra[0,\iy)^m$ is an embedding. In particular, ${}^b\d\La:{}^bTX_P\ra \La^*({}^bT[0,\iy)^m)$ is the injective morphism of trivial vector bundles $X_P\t\Hom(P,\R)\ra X_P\t\Hom(\N^m,\R)$ induced by the injective linear map $\Hom(P,\R)\ra\Hom(\N^m,\R)$ by composition with the surjective morphism $\pi:\N^m\ra P$ mapping $(a_1,\ldots,a_m)\mapsto a_1p_1+\cdots+a_mp_m$. Similarly, $\ti M_{\de_0}\La:\ti M_{\de_0}X_P\ra\ti M_0[0,\iy)^m$ is the map $\Hom(P,\N)\ra \Hom(\N^m,\N)$ by composition with $\pi$, and this is injective with torsion-free quotient as $\pi$ is surjective.

Thus, any $X_P$ is an embedded submanifold of some $[0,\iy)^m$, so locally any manifold with g-corners is an embedded submanifold of a manifold with corners.

\label{gc4ex2}
\end{ex}

Here are some local properties of immersions, proved in~\S\ref{gc51}.

\begin{thm} Suppose $Q,R$ are toric monoids, $V$ is an open neighbourhood of\/ $(\de_0,0)$ in $X_Q\t\R^m,$ and\/ $i:V\ra X_R\t\R^n$ is an interior immersion with\/ $i(\de_0,0)=(\de_0,0)$. Then:
\begin{itemize}
\setlength{\itemsep}{0pt}
\setlength{\parsep}{0pt}
\item[{\bf(i)}] $\rank Q\le\rank R$ and\/ $m\le n$. 
\item[{\bf(ii)}] There is an open neighbourhood\/ $\ti V$ of\/ $(\de_0,0)$ in $V$ such that\/ $i\vert_{\smash{\ti V}}:\ti V\ra i(\ti V)$ is a homeomorphism, that is, $i\vert_{\smash{\ti V}}:\ti V\ra X_R\t\R^n$ is an embedding.
\item[{\bf(iii)}] There is a natural identification of the monoid morphism
\e
\ti M_{(\de_0,0)}i:\ti M_{(\de_0,0)}(X_Q\t\R^m)\longra\ti M_{(\de_0,0)}(X_R\t\R^n)
\label{gc4eq2}
\e
with\/ $\al^\vee:Q^\vee\ra R^\vee,$ for $\al:R\ra Q$ a unique monoid morphism.

Then $Q,\al$ and\/ $m$ are determined uniquely, up to canonical isomorphisms of\/ $Q,$ by the subset\/ $i(\ti V)$ in $X_R\t\R^n$ near $(\de_0,0),$ for $\ti V$ as in {\bf(ii)}.
\item[{\bf(iv)}] Suppose $P$ is another toric monoid, $U$ is an open neighbourhood of\/ $(\de_0,0)$ in $X_P\t\R^l,$ and\/ $f:U\ra X_R\t\R^n$ is a smooth map with\/ $f(\de_0,0)=(\de_0,0)$ and\/ $f(U)\subseteq i(\ti V),$ for $\ti V$ as in {\bf(ii)}. Then there is an open neighbourhood\/ $\ti U$ of\/ $(\de_0,0)$ in $U$ and a unique smooth map $g:\ti U\ra \ti V$ with\/~$f\vert_{\smash{\ti U}}=i\ci g:\ti U\ra X_R\t\R^n$.
\item[{\bf(v)}] Now suppose $\al:R\ra Q$ in {\bf(iii)} is an isomorphism, and\/ $m=n$. Then there exist open neighbourhoods $\dot V$ of\/ $(\de,0)$ in $V$ and\/ $\dot W$ of\/ $(\de,0)$ in $X_R\t\R^n$ such that\/ $i\vert_{\smash{\dot V}}:\dot V\ra\dot W$ is a diffeomorphism.
\end{itemize}

\label{gc4thm2}
\end{thm}

We give three corollaries of Theorem \ref{gc4thm2}. The first is a factorization property of embeddings.

\begin{cor} Suppose $f:W\ra Y$ and\/ $i:X\ra Y$ are smooth maps of manifolds with g-corners, with\/ $i$ an embedding, and\/ $f(W)\subseteq i(X)$. Then there is a unique smooth map $g:W\ra X$ with\/~$f=i\ci g$. 

If also $f$ is an embedding, then $g$ is an embedding.
\label{gc4cor1}
\end{cor}

\begin{proof} First assume $i$ is interior. The fact that there is a unique continuous map $g:W\ra X$ with $f=i\ci g$ follows from $i:X\ra i(X)$ a homeomorphism. If $w\in W$ with $g(w)=x\in X$ and $f(w)=y\in Y$, then $W$ near $w$ is locally modelled on $X_P\t\R^l$ near $(\de_0,0)$, and $X$ near $x$ is locally modelled on $X_Q\t\R^m$ near $(\de_0,0)$, and $Y$ near $y$ is locally modelled on $X_R\t\R^n$ near $(\de_0,0)$, for some toric monoids $P,Q,R$ and $l,m,n\ge 0$. Using Theorem \ref{gc4thm2}(iv) we see that $g$ is smooth near $w$ in $W$, so $g$ is smooth.

If $i$ is not interior, then $C(i)\vert_{C_0(X)}:C_0(X)\cong X\ra C(Y)$ is an interior embedding, and we use basically the same proof with $C(i)\vert_{C_0(X)}$ in place of~$i$.

The final part is easy to check from Definition \ref{gc4def5}. For example, in Definition \ref{gc4def5}(i),(ii) $\d(f\vert_{S^k(W)})\vert_w,\ti M_xf$ injective imply $\d(g\vert_{S^k(W)})\vert_w,\ti M_xg$ injective, as $\d(f\vert_{S^k(W)})\vert_w,\ti M_xf$ factor via~$\d(g\vert_{S^k(W)})\vert_w,\ti M_xg$.
\end{proof}

The second is an analogue of Remark \ref{gc4rem1}(A). It means we can regard embedded submanifolds of manifolds with g-corners $Y$ as being special subsets $X\subseteq Y$, rather than special smooth maps~$i:X\ra Y$.

\begin{cor} Suppose $i:X\ra Y$ is an embedding of manifolds with g-corners. Then we can construct on the image $i(X)$ the canonical structure of a manifold with g-corners, depending only on the subset\/ $i(X)\subseteq Y$ and the manifold with g-corners structure on $Y$ and independent of\/ $i,X,$ and with this structure $i:X\ra i(X)$ is a diffeomorphism. 
\label{gc4cor2}
\end{cor}

\begin{proof} Since $i:X\ra i(X)$ is a homeomorphism by Definition \ref{gc4def5}, there is a unique manifold with g-corners structure on $i(X)$, such that $i:X\ra i(X)$ is a diffeomorphism. We have to prove this depends only on the subset $i(X)\subseteq Y$, and not on the choice of manifold with g-corners $X$ and embedding $i:X\ra Y$ with image $i(X)$. So suppose $i':X'\ra Y$ is another embedding of manifolds with g-corners with $i'(X')=i(X)$. Corollary \ref{gc4cor1} gives unique smooth maps $g:X\ra X'$ with $i=i'\ci g$, and $h:X'\ra X$ with $i'=i\ci h$. Then $i\ci h\ci g=i'\ci g=i$, so $h\ci g=\id_X$ as $i$ is injective, and similarly $g\ci h=\id_{\smash{X'}}$. 

Thus $g$ and $h$ are inverse, and $g:X\ra X'$ is a diffeomorphism. Hence the manifold with g-corners structure on $i(X)$ making $i:X\ra i(X)$ a diffeomorphism is the same as the manifold with g-corners structure on $X$ making $i':X'\ra i(X)$ a diffeomorphism, and is independent of the choice of~$X,i$.
\end{proof}

Here are analogues of Definition \ref{gc2def9} and Proposition~\ref{gc2prop3}.

\begin{dfn} A smooth map $f:X\ra Y$ of manifolds with g-corners is called {\it \'etale\/} if it is a local diffeomorphism. That is, $f$ is \'etale if and only if for all $x\in X$ there are open neighbourhoods $U$ of $x$ in $X$ and $V=f(U)$ of $f(x)$ in $Y$ such that $f\vert_U:U\ra V$ is a diffeomorphism (invertible with smooth inverse).

\label{gc4def6}
\end{dfn}

\begin{cor} A smooth map $f:X\ra Y$ of manifolds with g-corners is \'etale if and only if\/ $f$ is simple (hence interior) and\/ ${}^b\d f:{}^bTX\ra f^*({}^bTY)$ is an isomorphism of vector bundles on $X$.

If\/ $f$ is \'etale, then $f$ is a diffeomorphism if and only if it is a bijection.
\label{gc4cor3}
\end{cor}

\begin{proof} Suppose $f$ is \'etale. For $x\in X$ with $f(x)=y$, $f$ has a local inverse $g$ near $x$, so ${}^b\d f\vert_x:{}^bT_xX\ra {}^bT_yY$ is an isomorphism with inverse ${}^b\d g\vert_y:{}^bT_yY\ra {}^bT_xX$, and $\ti M_xf:\ti M_xX\ra\ti M_yY$ is an isomorphism with inverse $\ti M_yg:\ti M_yY\ra\ti M_xX$. As this holds for all $x\in X$, ${}^b\d f:{}^bTX\ra f^*({}^bTY)$ is an isomorphism, and $f$ is simple. This proves the `only if' part.

Next suppose $f$ is simple and ${}^b\d f$ is an isomorphism, and let $x\in X$ with $f(x)=y\in Y$. Then $X$ near $x$ is locally modelled on $X_Q\t\R^m$ near $(\de_0,0)$, and $Y$ near $y$ is locally modelled on $X_R\t\R^n$ near $(\de_0,0)$, for some toric monoids $Q,R$ and $m,n\ge 0$. Also $f$ is an immersion, so we can apply Theorem \ref{gc4thm2}. As $f$ is simple, $\al:R\ra Q$ identified with \eq{gc4eq2} is an isomorphism, and as ${}^b\d f$ is an isomorphism, $\dim X=\dim Y$, so $m=n$. Thus Theorem \ref{gc4thm2}(v) says there exist open neighbourhoods $x\in\dot V\subseteq X$ and $y\in\dot W\subseteq Y$ with $f\vert_{\smash{\dot V}}:\dot V\ra\dot W$ a diffeomorphism. Hence $f$ is \'etale, proving the `if' part.

For the final part, diffeomorphisms are \'etale bijections, and if $f:X\ra Y$ is an \'etale bijection, then it has an inverse map $f^{-1}:Y\ra X$, and the \'etale condition implies that $f^{-1}$ is smooth near each point $f(x)$ in $Y$, so $f^{-1}$ is smooth, and $f$ is a diffeomorphism.
\end{proof}

Next we investigate the analogue of Remark \ref{gc4rem1}(B): the question of whether embedded submanifolds $X\hookra Y$ can be described locally as the solutions of $\dim Y-\dim X$ transverse equations in $Y$, and conversely, whether the solution set of $k$ transverse equations in $Y$ is an embedded submanifold $X\hookra Y$ with $\dim Y-\dim X=k$. The answer turns out to be complicated. 

The next theorem, proved in \S\ref{gc52}, gives a special case in which a set of transverse equations can be used to define an embedded submanifold. It will be used to prove theorems in \S\ref{gc43} on existence of transverse fibre products.

\begin{thm} Suppose $Q$ is a toric monoid, $V$ is an open neighbourhood of\/ $(\de_0,0)$ in $X_Q\t\R^n,$ and\/ $f_i,g_i:V\ra[0,\iy)$ are interior maps for $i=1,\ldots,k$ with\/ $f_i(\de_0,0)=g_i(\de_0,0),$ and\/ $h_j:V\ra\R$ are smooth maps for $j=1,\ldots,l$ with\/ $h_j(\de_0,0)=0,$ such that\/ ${}^b\d f_1\vert_{(\de_0,0)}-{}^b\d g_1\vert_{(\de_0,0)},\ldots,{}^b\d f_k\vert_{(\de_0,0)}-{}^b\d g_k\vert_{(\de_0,0)},\d h_1\vert_{(\de_0,0)},\ldots,\d h_l\vert_{(\de_0,0)}$ are linearly independent in\/ ${}^bT_{(\de_0,0)}^*V$. Define
\e
X^\ci\!=\!\bigl\{v\!\in\! V^\ci:f_i(v)\!=\!g_i(v),\;\> i\!=\!1,\ldots,k,\;\> h_j(v)\!=\!0,\;\> j\!=\!1,\ldots,l\bigr\},
\label{gc4eq3}
\e
and let\/ $X=\ov{X^\ci}$ be the closure of\/ $X^\ci$ in $V$. Suppose $(\de_0,0)\in X$.

Here ${}^b\d f_i$ is a vector bundle morphism ${}^bTV\ra f_i^*({}^bT[0,\iy)),$ but we regard it as a morphism ${}^bTV\ra\R,$ and hence a section of\/ ${}^bT^*V,$ by identifying ${}^bT[0,\iy)\cong\R$ with\/ $x\frac{\pd}{\pd x}\cong 1,$ for $x$ the coordinate on $[0,\iy),$ and similarly for ${}^b\d g_i$. In effect we have\/ ${}^b\d f_i=f_i^{-1}\d f_i=\d\log f_i,$ so that\/ ${}^b\d f_i-{}^b\d g_i=\d\log(f_i/g_i),$ but\/ ${}^b\d f_i,{}^b\d g_i$ are still well-defined where $f_i=0$ and\/ $g_i=0$ in\/~$V$.

Then there exists a toric monoid\/ $P,$ an open neighbourhood\/ $U$ of\/ $(\de_0,0)$ in $X_P\t\R^m,$ where $\rank P+m=\rank Q+n-k-l,$ an interior embedding $\phi:U\hookra V$ with $\phi(\de_0,0)=(\de_0,0),$ and an open neighbourhood\/ $V'$ of\/ $(\de_0,0)$ in $V,$ such that\/~$\phi(U)=X\cap V'$.

Using the isomorphism 
\e
{}^bT_{(\de_0,0)}^*V={}^bT_{\de_0}^*X_Q\op T_0^*\R^n\cong (Q\ot_\N\R)\op\R^n,
\label{gc4eq4}
\e
write\/ ${}^b\d f_i\vert_{(\de_0,0)}-{}^b\d g_i\vert_{(\de_0,0)}=\be_i\op\ga_i$ for\/ $i=1,\ldots,k,$ where\/ $\be_i\in Q\ot_\N\Z\subseteq Q\ot_\N\R$ and\/ $\ga_i\in\R^n$. Then there is a natural isomorphism
\e
P^\vee\cong\bigl\{\rho\in Q^\vee:\text{$\rho(\be_i)=0$ for\/ $i=1,\ldots,k$}\bigr\},
\label{gc4eq5}
\e
which identifies the inclusion $P^\vee\hookra Q^\vee$ with\/~$\ti M_{(\de_0,0)}\phi:\ti M_{(\de_0,0)}U\ra\ti M_{(\de_0,0)}V$.

Let us now make the additional assumption that\/ $\be_1,\ldots,\be_k$ are linearly independent over\/ $\R$ in\/ $Q\ot_\N\R$. The r.h.s.\ of\/ \eq{gc4eq5} makes sense without supposing that\/ $(\de_0,0)\in X$. Under our additional assumption, $(\de_0,0)\in X$ is equivalent to the condition that the r.h.s.\ of\/ \eq{gc4eq5} does not lie in any proper face $F\subsetneq Q^\vee$ of the toric monoid\/~$Q^\vee$.
\label{gc4thm3}
\end{thm}

Theorem \ref{gc4thm3} is only a partial analogue of Remark \ref{gc4rem1}(B), as it proves that subsets locally defined as the zeroes of transverse equations (as in \eq{gc4eq3}) are embedded submanifolds, but it does {\it not\/} claim the converse, that embedded submanifolds are always locally defined as the zeroes of transverse equations. The next example shows that the converse of Theorem \ref{gc4thm3} is actually false.

\begin{ex} Define $\phi:[0,\iy)\ra[0,\iy)^2$ by $\phi(x)=(x^2+x^3,x^3)$. Then $\phi$ is an interior embedding. However, {\it there do not exist\/} interior $f,g:[0,\iy)^2\ra[0,\iy)$ with $f(0,0)=g(0,0)=0$ such that $\phi((0,\iy))=\bigl\{(y,z)\in(0,\iy)^2:f(y,z)=g(y,z)\bigr\}$ and ${}^b\d f-{}^b\d g$ is nonzero on $\phi([0,\iy))$, even only near $(0,0)$ in $[0,\iy)^2$. To see this, observe that we must have
\begin{equation*}
f(y,z)=D(y,z)y^az^b,\;\> g(y,z)=E(y,z)y^cz^d
\end{equation*}
for $D,E:[0,\iy)^2\ra(0,\iy)$ smooth and defined near $(0,0)$ and $a,b,c,d\in\N$. Then ${}^b\d f-{}^b\d g$ nonzero at $(0,0)$ implies that~$(a,b)\ne(c,d)$.

The equation $f(x^2+x^3,x^3)=g(x^2+x^3,x^3)$ is now equivalent to
\e
x^{2a+3b-2c-3d}(1+x)^{a-c}=E(x^2+x^3,x^3)/D(x^2+x^3,x^3). 
\label{gc4eq6}
\e
Putting $x=0$ and using $D,E>0$ gives $2a+3b-2c-3d=0$. Applying $\frac{d}{\d x}$ to \eq{gc4eq6} and setting $x=0$ then yields $a-c=0$, so $(a,b)=(c,d)$, a contradiction.

We can write $\phi([0,\iy))$ in the form $\bigl\{(y,z)\in[0,\iy)^2:h(y,z)=0\bigr\}$ for $h:[0,\iy)^2\ra\R$ smooth, e.g.\ with $h(y,z)=(y-z)^3-z^2$. But then $\d h\vert_{(0,0)}=0$ in both $T_{(0,0)}[0,\iy)^2$ and ${}^bT_{(0,0)}[0,\iy)^2$, so $h$ is not transverse. 

Note that if we had defined $\phi(x)=(x^2,x^3)$, we could write $\phi([0,\iy))=\bigl\{(y,z)\in[0,\iy)^2:f(y,z)=g(y,z)\bigr\}$ for $f(y,z)=y^3$ and $g(y,z)=z^2$. The problem is with the higher-order $x^3$ term in $x^2+x^3$ in~$\phi(x)=(x^2+x^3,x^3)$.
\label{gc4ex3}
\end{ex}

Using Theorem \ref{gc4thm3} we prove:

\begin{cor} Let\/ $Y$ be a manifold with g-corners, $f_i,g_i:Y\ra[0,\iy)$ be interior for $i=1,\ldots,k,$ and\/ $h_j:Y\ra\R$ be smooth for $j=1,\ldots,l,$ set
\begin{equation*}
X^\ci=\bigl\{x\in Y^\ci:f_i(x)=g_i(x),\;\> i=1,\ldots,k,\;\> h_j(x)=0,\;\> j=1,\ldots,l\bigr\},
\end{equation*}
and let\/ $X=\ov{X^\ci}$ be the closure of\/ $X^\ci$ in $Y$. Suppose that\/ ${}^b\d f_1\vert_x-{}^b\d g_1\vert_x,\ldots,\ab{}^b\d f_k\vert_x-{}^b\d g_k\vert_x,\d h_1\vert_x,\ldots,\d h_l\vert_x$ are linearly independent in\/ ${}^bT_x^*Y$ for each\/ $x\in X,$ interpreting ${}^b\d f_i-{}^b\d g_i$ as in Theorem\/ {\rm\ref{gc4thm3}}. Then $X$ has a unique structure of a manifold with g-corners with\/ $\dim X=\dim Y-k-l,$ such that the inclusion $X\hookra Y$ is an embedding.
\label{gc4cor4}
\end{cor}

\begin{proof} Let $x\in X\subseteq Y$. Then we can locally identify $Y$ near $x$ with $X_Q\t\R^n$ near $(\de_0,0)$, for some toric monoid $Q$ and $n\ge 0$. Theorem \ref{gc4thm3} proves that $X$ near $x$ is of the form $\phi(U)$ for $\phi:U\ra Y$ an embedding, and $U$ a manifold with g-corners of dimension $\rank P+m=\rank Q+n-k-l=\dim Y-k-l$. Corollary \ref{gc4cor2} now shows that $X$ near $x$ has a unique structure of a manifold with g-corners with $\dim X=\dim Y-k-l\ge 0,$ such that the inclusion $X\hookra Y$ near $x$ is an embedding. As this holds for all $x\in X$, the corollary follows.
\end{proof}

Note that the corollary is false for $X,Y$ manifolds with (ordinary) corners, as the next example shows.

\begin{ex} Let $Y=[0,\iy)^4$ with coordinates $(y_1,y_2,y_3,y_4)$, and define $f,g:Y\ra[0,\iy)^2$ by $f(y_1,y_2,y_3,y_4)=y_1y_2$ and $g(y_1,y_2,y_3,y_4)=y_3y_4$. Then ${}^b\d f-{}^b\d g$ is a nonvanishing section of ${}^bTY$, so Corollary \ref{gc4cor4} defines a manifold with g-corners $X$ embedded in $Y$, which is
\begin{equation*}
X=\bigl\{(y_1,y_2,y_3,y_4)\in[0,\iy)^4:y_1y_2=y_3y_4\bigr\}.
\end{equation*}
This is $X_P'$ in \eq{gc3eq7}, so $X$ is diffeomorphic to $X_P$ in Example \ref{gc3ex5}, which is our simplest example of a manifold with g-corners which is not a manifold with corners. Thus in Corollary \ref{gc4cor4}, if $Y$ is a manifold with corners, $X$ can still have g-corners rather than (ordinary) corners.
\label{gc4ex4}
\end{ex}

\subsection{Transversality and fibre products}
\label{gc43}

Here is a definition from category theory.

\begin{dfn} Let $\cC$ be a category, and $g:X\ra Z$, $h:Y\ra Z$ be morphisms in $\cC$. A {\it fibre product\/} of $g,h$ in $\cC$ is an object $W$ and morphisms $e:W\ra X$ and $f:W\ra Y$ in $\cC$, such that $g\ci e=h\ci f$, with the universal property that if $e':W'\ra X$ and $f':W'\ra Y$ are morphisms in $\cC$ with $g\ci e'=h\ci f'$ then there is a unique morphism $b:W'\ra W$ with $e'=e\ci b$ and $f'=f\ci b$. Then we write $W=X\t_{g,Z,h}Y$ or $W=X\t_ZY$. The diagram
\e
\begin{gathered}
\xymatrix@R=13pt@C=50pt{ W \ar[r]_f \ar[d]^e & Y
\ar[d]_h \\ X \ar[r]^g & Z}
\end{gathered}
\label{gc4eq7}
\e
is called a {\it Cartesian square}. Fibre products need not exist, but if they do exist they are unique up to canonical isomorphism in~$\cC$.
\label{gc4def7}
\end{dfn}

The next definition and theorem are well known.

\begin{dfn} Let $g:X\ra Z$ and $h:Y\ra Z$ be smooth maps of manifolds without boundary. We call $g,h$ {\it transverse\/} if $T_xg\op T_yh:T_xX\op T_yY\ra T_zZ$ is surjective for all $x\in X$ and $y\in Y$ with~$g(x)=h(y)=z\in Z$.
\label{gc4def8}
\end{dfn}

\begin{thm} Suppose $g:X\ra Z$ and\/ $h:Y\ra Z$ are transverse smooth maps of manifolds without boundary. Then a fibre product $W=X\t_{g,Z,h}Y$ exists in $\Man,$ with\/ $\dim W=\dim X+\dim Y-\dim Z$. We may write
\begin{equation*}
W=\bigl\{(x,y)\in X\t Y:\text{$g(x)=h(y)$ in $Z$}\bigr\}
\end{equation*}
as an embedded submanifold of\/ $X\t Y,$ where $e:W\ra X$ and\/ $f:W\ra Y$ act by $e:(x,y)\mapsto x$ and\/~$f:(x,y)\mapsto y$.
\label{gc4thm4}
\end{thm}

The goal of this section is to extend Definition \ref{gc4def8} and Theorem \ref{gc4thm4} to manifolds with g-corners. We will consider fibre products in both the category $\Mangc$, and in the subcategory $\Mangcin$ with morphisms interior maps. Remark \ref{gc4rem2} compares our results with others in the literature.

Writing $*$ for the point regarded as an object of $\Mangc$, for any manifold with g-corners, morphisms $e:*\ra X$ in $\Mangc$ correspond to points $x\in X$, and (interior) morphisms $e:*\ra X$ in $\Mangcin$ correspond to points $x\in X^\ci$. So applying the universal property in Definition \ref{gc4def7} with $W'=*$ yields:

\begin{lem} Suppose we are given a Cartesian square \eq{gc4eq7} in $\Mangc$. Then as in Theorem\/ {\rm\ref{gc4thm4}} there is a canonical identification \begin{bfseries}of sets only\end{bfseries}
\e
W\cong\bigl\{(x,y)\in X\t Y:\text{$g(x)=h(y)$ in $Z$}\bigr\},
\label{gc4eq8}
\e
identifying $e:W\ra X,$ $f:W\ra Y$ with\/ $e':(x,y)\mapsto x,$ $f':(x,y)\mapsto y$.

If instead\/ \eq{gc4eq7} is a Cartesian square in $\Mangcin,$ in the same way, for the interiors $W^\ci,X^\ci,Y^\ci,Z^\ci$ we have a canonical identification of sets
\e
W^\ci\cong\bigl\{(x,y)\in X^\ci\t Y^\ci:\text{$g(x)=h(y)$ in $Z^\ci$}\bigr\}.
\label{gc4eq9}
\e

\label{gc4lem}
\end{lem}

The next example shows the lemma may not hold at the level of topological spaces, or embedded submanifolds, even for manifolds without boundary.

\begin{ex} Take $X=Y=\R$ and $Z=\R^2$, and define $g:X\ra Z$, $h:Y\ra Z$ by $g(x)=(x,0)$, $h(y)=\bigl(y,e^{-1/y^2}\sin \frac{\pi}{y}\bigr)$ for $y\ne 0$, and $h(0)=(0,0)$. Then a fibre product $W$ exists in $\Man$. As in \eq{gc4eq8}, as sets we may write
\begin{equation*}
W=\bigl\{(x,y)\in X\t Y:\text{$g(x)=h(y)$ in $Z$}\bigr\}=
\bigl\{(\ts\frac{1}{n},\frac{1}{n}):0\ne n\in\Z\bigr\}\cup\bigl\{(0,0)\bigr\}.
\end{equation*}
However, $W$ is a 0-manifold, a set with the discrete topology, but the topology induced on $W$ by its inclusion in $X\t Y=\R^2$ is not discrete near $(0,0)$. Thus in this case \eq{gc4eq8} is not an isomorphism of topological spaces, and $W$ is not an embedded submanifold of $X\t Y$. This does not contradict Theorem \ref{gc4thm4}, as $g,h$ are not transverse at~$(0,0)$.
\label{gc4ex5}
\end{ex}

Here are two notions of transversality for manifolds with g-corners, generalizing Definition \ref{gc4def8}. We take $g,h$ interior so that ${}^bT_xg,{}^bT_yh$ are defined.

\begin{dfn} Let $g:X\ra Z$ and $h:Y\ra Z$ be interior maps of manifolds with g-corners, or more generally interior morphisms in $\cMangc$. Then:
\begin{itemize}
\setlength{\itemsep}{0pt}
\setlength{\parsep}{0pt}
\item[(a)] We call $g,h$ {\it b-transverse\/} if ${}^bT_xg\op{}^bT_yh:{}^bT_xX\op{}^bT_yY\ra{}^bT_zZ$ is surjective for all $x\in X$ and $y\in Y$ with~$g(x)=h(y)=z\in Z$.
\item[(b)] We call $g,h$ {\it c-transverse\/} if they are b-transverse, and for all $x\in X$ and $y\in Y$ with $g(x)=h(y)=z\in Z$, the linear map ${}^b\ti N_xg\op{}^b\ti N_yh:{}^b\ti N_xX\op{}^b\ti N_yY\ra{}^b\ti N_zZ$ is surjective, and the submonoid
\end{itemize}
\e
\bigl\{(\la,\mu)\!\in\! \ti M_xX\!\t\! \ti M_yY:\text{$\ti M_xg(\la)\!=\!\ti M_yh(\mu)$ in $\ti M_zZ$}\bigr\}\!\subseteq\! \ti M_xX\!\t \!\ti M_yY\label{gc4eq10}
\e
\begin{itemize}
\setlength{\itemsep}{0pt}
\setlength{\parsep}{0pt}
\item[]is not contained in any proper face $F\subsetneq\ti M_xX\t \ti M_yY$ of $\ti M_xX\t\ti M_yY$.
\end{itemize}
If $g$ (or $h$) is a b-submersion in the sense of \S\ref{gc41}, and $h$ (or $g$) is interior, then $g,h$ are b-transverse.
\label{gc4def9}
\end{dfn}

B-normal maps and b-fibrations from \S\ref{gc41} give conditions for c-transversality.

\begin{prop} Let\/ $g:X\ra Z$ and\/ $h:Y\ra Z$ be interior maps of manifolds with g-corners. Then $g,h$ are c-transverse if either
\begin{itemize}
\setlength{\itemsep}{0pt}
\setlength{\parsep}{0pt}
\item[{\bf(i)}] $g,h$ are b-transverse and\/ $g$ or\/ $h$ is b-normal; or
\item[{\bf(ii)}] $g$ or\/ $h$ is a b-fibration.
\end{itemize}
\label{gc4prop1}
\end{prop}

\begin{proof} For (i), suppose $g,h$ are b-transverse and $g$ is b-normal, and let $x\in X$ and $y\in Y$ with $g(x)=h(y)=z\in Z$. As $g$ is b-normal, one can show that ${}^b\ti N_xg:{}^b\ti N_xX\ra {}^b\ti N_zZ$ is surjective, which implies that ${}^b\ti N_xg\op{}^b\ti N_yh:{}^b\ti N_xX\op{}^b\ti N_yY\ra{}^b\ti N_zZ$ is surjective, as we want.

If $P$ is a toric monoid, write $P^\ci=P\sm\bigcup_{\text{faces $F\subsetneq P$}}F$ for the complement of all proper faces $F$ in $P$. Since $h$ is interior, $\ti M_xh$ maps $(\ti M_yY)^\ci\ra (\ti M_zZ)^\ci$. Let $\mu\in (\ti M_yY)^\ci$, and set $\nu=\ti M_xh(\mu)\in(\ti M_zZ)^\ci$. As $g$ is b-normal, one can show that $\ti M_xg:\ti M_xX\ra \ti M_zZ$ is surjective up to finite multiples: there exists $\la\in (\ti M_xX)^\ci$ with $\ti M_xg(\la)=n\cdot\nu$ for some $n>0$. Then $(\la,n\cdot\mu)$ lies in \eq{gc4eq10} and in $(\ti M_xX)^\ci\t(\ti M_yY)^\ci$. So \eq{gc4eq10} does not lie in any proper face of $\ti M_xX\t\ti M_yY$, and $g,h$ are c-transverse.

For (ii), $g$ a b-fibration means it is a b-normal b-submersion, and $g$ a b-submersion implies $g,h$ b-transverse, so (ii) follows from~(i).
\end{proof}

The following theorem is proved in~\S\ref{gc53}.

\begin{thm} Let\/ $g:X\ra Z$ and\/ $h:Y\ra Z$ be b-transverse (or c-transverse) interior maps of manifolds with g-corners. Then $C(g):C(X)\ra C(Z)$ and\/ $C(h):C(Y)\ra C(Z)$ are also b-transverse (or c-transverse, respectively) interior maps in $\cMangc$.
\label{gc4thm5}
\end{thm}

The next two theorems, perhaps the most important in the paper, proved in \S\ref{gc54} and \S\ref{gc55}, show that b-transversality is is a sufficient condition for existence of a fibre product $W=X\t_{g,Z,h}Y$ in $\Mangcin$, and c-transversality a sufficient condition for existence of a fibre product in $\Mangc$, and in the latter case we have $C(W)=C(X)\t_{C(g),C(Z),C(h)}C(Y)$ in $\cMangc$ and $\cMangcin$. The explicit expressions for $W^\ci,W$ in \eq{gc4eq11}--\eq{gc4eq12} come from Lemma~\ref{gc4lem}.

\begin{thm} Let\/ $g:X\ra Z$ and\/ $h:Y\ra Z$ be b-transverse interior maps of manifolds with g-corners. Then a fibre product\/ $W=X\t_{g,Z,h}Y$ exists in $\Mangcin,$ with\/ $\dim W=\dim X+\dim Y-\dim Z$. Explicitly, we may write
\e
W^\ci=\bigl\{(x,y)\in X^\ci\t Y^\ci:\text{$g(x)=h(y)$ in $Z^\ci$}\bigr\},
\label{gc4eq11}
\e
and take $W$ to be the closure $\ov{W^\ci}$ of\/ $W^\ci$ in $X\t Y,$ and then $W$ is an embedded submanifold of\/ $X\t Y$ in the sense of\/ {\rm\S\ref{gc42},} and\/ $e:W\ra X$ and\/ $f:W\ra Y$ act by $e:(x,y)\mapsto x$ and\/~$f:(x,y)\mapsto y$.
\label{gc4thm6}
\end{thm}

\begin{thm} Suppose $g:X\ra Z$ and\/ $h:Y\ra Z$ are c-transverse interior maps of manifolds with g-corners. Then a fibre product\/ $W=X\t_{g,Z,h}Y$ exists in $\Mangc,$ with\/ $\dim W=\dim X+\dim Y-\dim Z$. Explicitly, we may write
\e
W=\bigl\{(x,y)\in X\t Y:\text{$g(x)=h(y)$ in $Z$}\bigr\},
\label{gc4eq12}
\e
and then $W$ is an embedded submanifold of\/ $X\t Y$ in the sense of\/ {\rm\S\ref{gc42},} and\/ $e:W\ra X$ and\/ $f:W\ra Y$ act by\/ $e:(x,y)\mapsto x$ and\/ $f:(x,y)\mapsto y$. This $W$ is also a fibre product in $\Mangcin,$ and agrees with that in Theorem\/~{\rm\ref{gc4thm6}}.

Furthermore, the following is Cartesian in both\/ $\cMangc$ and\/ $\cMangcin\!:$
\e
\begin{gathered}
\xymatrix@R=13pt@C=90pt{ *+[r]{C(W)} \ar[r]_{C(f)} \ar[d]^{C(e)} & *+[l]{C(Y)}
\ar[d]_{C(h)} \\ *+[r]{C(X)} \ar[r]^{C(g)} & *+[l]{C(Z).\!\!} }
\end{gathered}
\label{gc4eq13}
\e
Equation \eq{gc4eq13} has a grading-preserving property, in that if\/ $(w,\be)\in C_i(W)$ with\/ $C(e)(w,\be)=(x,\ga)\in C_j(X),$ and\/ $C(f)(w,\be)=(y,\de)\in C_k(Y),$ and\/ $C(g)(x,\ga)\ab=C(h)(y,\de)=(z,\ep)\in C_l(Z),$ then\/ $i+l=j+k$. Hence
\e
C_i(W)\cong \ts\coprod_{j,k,l\ge 0: i=j+k-l} C_j^l(X)\t_{C(g)\vert_{\cdots},C_l(Z),C(h)\vert_{\cdots}}C_k^l(Y),
\label{gc4eq14}
\e
where $C_j^l(X)=C_j(X)\cap C(g)^{-1}(C_l(Z))$ and\/ $C_k^l(Y)=C_k(Y)\cap C(h)^{-1}(C_l(Z)),$ open and closed in $C_j(X),C_k(Y)$. When $i=1,$ this gives a formula for~$\pd W$.
\label{gc4thm7}
\end{thm}

\begin{rem} Here is how our work above relates to previous results in the literature. The author \cite[\S 6]{Joyc1} defined `transverse' and `strongly transverse' maps $g:X\ra Z$, $h:Y\ra Z$ in the category $\Mancst$ of manifolds with corners and strongly smooth maps,  similar to b- and c-transverse maps above, and proved an analogue of Theorem \ref{gc4thm7} for (strongly) transverse fibre products in~$\Mancst$.

Kottke and Melrose \cite[\S 11]{KoMe} studied fibre products in the category $\Mancin$ of manifolds with (ordinary) corners and interior smooth maps, in the notation of \S\ref{gc2}. They defined `b-transversal' maps $g:X\ra Z$, $h:Y\ra Z$ in $\Mancin$, which agree with our b-transverse maps when $X,Y,Z$ have ordinary corners. They prove an analogue of Theorem \ref{gc4thm6}, that if $g,h$ are b-transversal and satisfy an extra condition, then a fibre product $X\t_{g,Z,h}Y$ exists in $\Mancin$. Under further conditions including $g,h$ b-normal, they prove $X\t_{g,Z,h}Y$ is also a fibre product in $\Manc$, as in Theorem~\ref{gc4thm7}.

Kottke and Melrose's extra condition is equivalent to saying that the fibre product $W=X\t_{g,Z,h}Y$ in $\Mangcin$ given by Theorem \ref{gc4thm6} has ordinary corners rather than g-corners. Without this condition, they know that $W=X\t_{g,Z,h}Y$ exists as an `interior binomial variety', which is basically a manifold with g-corners $W$ embedded in a manifold with ordinary corners $X\t Y$. So they come close to proving our Theorem \ref{gc4thm6} when $X,Y,Z$ have ordinary corners and $W$ has g-corners. Their results were part of the motivation for this paper.
\label{gc4rem2}
\end{rem}

Combining Proposition \ref{gc4prop1}(ii) and Theorem \ref{gc4thm7} yields:

\begin{cor} Suppose\/ $g:X\ra Z$ and\/ $h:Y\ra Z$ are morphisms in\/ $\Mangc,$ with\/ $g$ a b-fibration and\/ $h$ interior. Then a fibre product\/ $W=X\t_{g,Z,h}Y$ with\/ $\dim W=\dim X+\dim Y-\dim Z$ exists in $\Mangc,$ which may be written\/ $W=\bigl\{(x,y)\in X\t Y:g(x)=h(y)\bigr\},$ as an embedded submanifold of\/~$X\t Y$. 
\label{gc4cor5}
\end{cor}

If we do not assume $h$ is interior, Corollary \ref{gc4cor5} is false:

\begin{ex} Define $X=[0,\iy)^2$, $Y=*$, $Z=[0,\iy)$ and smooth maps $g:X\ra Z$, $h:Y\ra Z$ by $g(x,y)=xy$ and $h:*\mapsto 0$. Then $g$ is a b-fibration, but $h$ is not interior. In this case no fibre product $W=X\t_{g,Z,h}Y$ exists in $\Mangc$, as by Lemma \ref{gc4lem} it would be given as a set by $W=\bigl\{(x,y)\in[0,\iy)^2:xy=0\bigr\}$, but no manifold with g-corners structure on $W$ near $(0,0)$ can satisfy all the required conditions. 
\label{gc4ex6}
\end{ex}

Here are examples of three phenomena which can occur with b-transverse but not c-transverse fibre products in $\Mangcin$ and~$\Mangc$:

\begin{ex} Let $X=[0,\iy)\t\R$, $Y=[0,\iy)$ and $Z=[0,\iy)^2$. Define $g:X\ra Z$ by $g(x_1,x_2)=(x_1,x_1e^{x_2})$ and $h:Y\ra Z$ by $h(y)=(y,y)$. Then $g,h$ are b-transverse, as $g$ is a b-submersion by Example \ref{gc4ex1}(iii). But $g,h$ are not c-transverse, since at $(0,x_2)\in X$ and $0\in Y$ with $g(0,x_2)=h(0)=(0,0)\in Z$, we may identify ${}^b\ti N_{(0,x_2)}g\op{}^b\ti N_0h:{}^b\ti N_{(0,x_2)}X\op {}^b\ti N_0Y\ra {}^b\ti N_{(0,0)}Z$ with the map $\R\op\R\ra\R^2$ taking $(\la,\mu)\mapsto(\la+\mu,\la+\mu)$, which is not surjective.

Theorem \ref{gc4thm6} gives a fibre product $W=X\t_{g,Z,h}Y$ in $\Mangcin$, where
\begin{equation*}
W=\bigl\{(w,0,w):w\in[0,\iy)\bigr\}\cong[0,\iy).
\end{equation*}
Lemma \ref{gc4lem} shows that if a fibre product $W'=X\t_{g,Z,h}Y$ exists in $\Mangc$, then as a set with projections $e:W'\ra X$, $f:W'\ra Y$ we have
\begin{equation*}
W'=\bigl\{(w,0,w):w\in[0,\iy)\bigr\}\cup\bigl\{(0,x,0):x\in\R\bigr\}\subset X\t Y.
\end{equation*}
This is the union of copies of $[0,\iy)$ and $\R$ intersecting in one point $(0,0,0)$. In this case {\it no fibre product\/ $X\t_ZY$ exists in\/} $\Mangc$, as no manifold with g-corners structure on $W'$ near $(0,0,0)$ can satisfy all the required conditions.

Theorem \ref{gc4thm5} shows $C(g),C(h)$ are b-transverse, so by Theorem \ref{gc4thm6} (generalized to $\cMangcin$) the fibre product $C(X)\t_{C(g),C(Z),C(h)}C(Y)$ in $\cMangcin$ exists. It is the disjoint union of $[0,\iy)$ from $C_0(X)\t_{C_0(Z)}C_0(Y)$ and $\R$ from $C_1(X)\t_{C_2(Z)}C_1(Y)$. But $C(W)=[0,\iy)\amalg\{0\}$, so $C(W)\not\cong C(X)\t_{C(Z)}C(Y)$. The
fibre product $C(X)\t_{C(g),C(Z),C(h)}C(Y)$ in $\cMangc$ does not exist.
\label{gc4ex7}
\end{ex}

\begin{ex} Let $X=Y=[0,\iy)$ and $Z=[0,\iy)^2$, and define $g:X\ra Z$, $h:Y\ra Z$ by $g(x)=(x,x)$, $h(y)=(y,y^2)$. Then $g,h$ are b-transverse. However, they are not c-transverse, as at $0\in X$ and $0\in Y$ with $g(0)=h(0)=(0,0)\in Z$, although ${}^b\ti N_{0}g\op{}^b\ti N_0h:{}^b\ti N_0X\op {}^b\ti N_0Y\ra {}^b\ti N_{(0,0)}Z$ is surjective, the submonoid \eq{gc4eq10} is zero, and so lies in a proper face of~$\ti M_0X\t\ti M_0Y\cong\N^2$.

The fibre product $W$ in $\Mangcin$ in \eq{gc4eq11} given by Theorem \ref{gc4thm6} is $W=\{(1,1)\}$, a single point. Although Theorem \ref{gc4thm7} does not apply, it is easy to show that $W'=\{(0,0),(1,1)\}$ in equation \eq{gc4eq12} is a fibre product in $\Mangc$. So {\it fibre products $X\t_ZY$ in $\Mangcin$ and $\Mangc$ exist but do not coincide}. 

Note that $W\subsetneq W'$. In general, if $g,h$ are b-transverse but not c-transverse, and fibre products $W=X\t_ZY$ in $\Mangcin$ and $W'=X\t_ZY$ in $\Mangc$ both exist, then $W$ is (diffeomorphic to) a proper, open and closed subset of~$W'$.

In this case a fibre product $C(X)\t_{C(Z)}C(Y)$ exists in $\cMangcin$ and is 2 points, so agrees with $C(W')$ but not with $C(W)$, and a fibre product $C(X)\t_{C(Z)}C(Y)$ exists in $\cMangc$ and is 3 points, so does not agree with either $C(W)$ or~$C(W')$.

\label{gc4ex8}
\end{ex}

\begin{ex} Let $X=Y=[0,1)^2$ and $Z=\bigl\{(z_1,z_2,z_3,z_4)\in[0,\iy)^2:z_1z_2=z_3z_4\bigr\}$, as in \eq{gc3eq7}, so that $Z\cong X_P$ for $P$ the toric monoid of Example \ref{gc3ex5}. Define $g:X\ra Z$, $h:Y\ra Z$ by $g(x_1,x_2)=(x_1,x_1x_2^2,x_2,x_1^2x_2)$ and $h(y_1,y_2)=(y_1y_2^2,y_1,y_1^2y_2,y_2)$. Then the only points $\bs x\in X$, $\bs y\in Y$, $\bs z\in Z$ with $g(\bs x)=h(\bs y)=\bs z$ are $\bs x=(0,0)$, $\bs y=(0,0)$, $\bs z=(0,0,0,0)$. These $g,h$ are b-transverse, but not c-transverse, as at $\bs x=\bs y=(0,0)$ the submonoid \eq{gc4eq10} is zero, and lies in a proper face of~$\ti M_{\bs x}X\t\ti M_{\bs y}Y\cong\N^4$.

In this case the fibre product $W=X\t_{g,Z,h}Y$ in $\Mangcin$ given by Theorem \ref{gc4thm6} is $W=\es$. A fibre product $W'=X\t_{g,Z,h}Y$ in $\Mangc$ exists, with $W'=\bigl\{\bigl((0,0),(0,0)\bigr)\bigr\}$. Note however that $\dim W'=0<1=\dim X+\dim Y-\dim Z$, so {\it the fibre product\/ $W'$ in $\Mangc$ has smaller than the expected dimension}. 

Again, a fibre product $C(X)\t_{C(Z)}C(Y)$ exists in $\cMangcin$ and agrees with $C(W')$ but not with $C(W)$, and a fibre product $C(X)\t_{C(Z)}C(Y)$ exists in $\cMangc$ and is 2 points, so does not agree with either $C(W)$ or~$C(W')$.
\label{gc4ex9}
\end{ex}

\begin{rem} One could also look for useful sufficient conditions for fibre products $X\t_{g,Z,h}Y$ to exist in $\Mangc$ when $g:X\ra Z$, $h:Y\ra Z$ are not both interior. Example \ref{gc4ex6} shows that $g$ a b-fibration and $h$ general is not a sufficient condition, but one can prove that $g$ a simple b-fibration and $h$ general is sufficient. A good approach may be to suppose that $C(g):C(X)\ra C(Z)$, $C(h):C(Y)\ra C(Z)$ are b-transverse (they are already interior), so that a fibre product $C(X)\t_{C(Z)}C(Y)$ exists in $\cMangcin$, and then seek extra discrete conditions ensuring that the highest-dimensional component of $C(X)\t_{C(Z)}C(Y)$ is a fibre product $X\t_{g,Z,h}Y$ in~$\Mangc$.
\label{gc4rem3}
\end{rem}

\subsection{(M-)Kuranishi spaces with g-corners}
\label{gc44}

`Kuranishi spaces' are a class of singular spaces generalizing manifolds and orbifolds, which first appeared in the work of Fukaya, Oh, Ohta and Ono \cite{FOOO,FuOn} as the geometric structure on moduli spaces of $J$-holomorphic curves in symplectic geometry. One can consider both Kuranishi spaces without boundary \cite{FuOn}, and with corners \cite{FOOO}. The definition of Kuranishi spaces has been controversial from the outset, and has changed several times.

Recently it has become clear \cite{Joyc5} that Kuranishi spaces should be understood as `derived smooth orbifolds' and are part of the subject of Derived Differential Geometry, the differential-geometric analogue of the Derived Algebraic Geometry of Jacob Lurie and To\"en--Vezzosi.

One version of Derived Differential Geometry is the author's 2-categories of `d-manifolds' $\dMan$ and `d-orbifolds' $\dOrb$ \cite{Joyc2,Joyc3,Joyc4}, which are defined as special classes of derived schemes and derived stacks over $C^\iy$-rings, using the tools of (derived) algebraic geometry.

In a second approach, the author \cite{Joyc5} gave a new definition of Kuranishi space, modifying \cite{FOOO,FuOn}. This yielded an ordinary category $\MKur$ of `M-Kuranishi spaces' $\MKur$, a kind of derived manifold, and a 2-category of `Kuranishi spaces' $\Kur$, a kind of derived orbifold. The definition involves an atlas of charts (`Kuranishi neighbourhoods' $(V,E,\Ga,s,\psi)$) and looks very different to that of d-manifolds and d-orbifolds, but there are equivalences of categories $\MKur\simeq\mathop{\rm Ho}(\dMan)$ and of 2-categories~$\Kur\simeq\dOrb$.

In \cite[\S 3 \& \S 5]{Joyc5} the author also defined (2-)categories $\MKurc$, $\Kurc$ of \hbox{(M-)Kuranishi} spaces with corners. The construction starts with a category $\Manc$ of manifolds with corners, as in \S\ref{gc2}, with the $V$ in Kuranishi neighbourhoods $(V,E,\Ga,s,\psi)$ objects in $\Manc$. The definition is not very sensitive to the details of the category $\Manc$ --- variations on $\Manc$ satisfying a list of basic properties we expect of manifolds with corners will do just as well. 

So, as explained in detail in \cite[\S 3.8 \& \S 5.6]{Joyc5}, by replacing $\Manc$ by $\Mangc$ in \cite[\S 3 \& \S 5]{Joyc5}, we can define a category $\MKurgc$ of {\it M-Kuranishi spaces with g-corners\/} containing $\MKurc,\ab\Manc,\ab\Mangc$ as full subcategories, and a 2-category $\Kurgc$ of {\it Kuranishi spaces with g-corners\/} containing $\Kurc$, $\Manc$, $\Mangc$ as full (2-)subcategories.

Fibre products in $\Kurgc$ exist under weaker conditions than in $\Kurc$, as the same holds for $\Mangc,\Manc$. For example, in \cite{Joyc7} we will prove analogues of Theorem \ref{gc4thm6} and Corollary~\ref{gc4cor5}:

\begin{thm}{\bf (a)} Suppose $\bX,\bY$ are Kuranishi spaces with g-corners, $Z$ is a manifold with g-corners, and\/ $\bs g:\bX\ra Z,$ $\bs h:\bY\ra Z$ are interior $1$-morphisms in $\Kurgc$. Then a fibre product\/ $\bW=\bX\t_{\bs g,Z,\bs h}\bY$ exists in the\/ $2$-category\/ $\Kurgcin$ of Kuranishi spaces with g-corners and interior\/ $1$-morphisms, with virtual dimension\/~$\vdim\bW=\vdim\bX+\vdim\bY-\dim Z$.
\smallskip

\noindent{\bf(b)} Suppose $\bs g:\bX\ra\bZ$ is a (weak) b-fibration and\/ $\bs h:\bY\ra\bZ$ an interior $1$-morphism in $\Kurgc$. Then a fibre product\/ $\bW=\bX\t_{\bs g,\bZ,\bs h}\bY$ exists in $\Kurgc,$ with\/~$\vdim\bW=\vdim\bX+\vdim\bY-\vdim\bZ$.
\label{gc4thm8}
\end{thm}

Neither part holds in $\Kurc$ rather than $\Kurgc$. Note that there is no transversality assumption in (a), or any discrete conditions on monoids.

Kuranishi spaces with g-corners will be important in future applications in symplectic geometry that the author is planning, for two reasons. Firstly, the author would like to develop an approach to moduli spaces of $J$-holomorphic curves using `representable 2-functors', modelled on Grothendieck's representable functors in algebraic geometry. It turns out that even if the moduli space is a Kuranishi space with (ordinary) corners, as in \cite{FOOO}, the definition of the moduli 2-functor near curves with boundary nodes involves fibre products which do not exist in $\Kurc$, and the moduli 2-functor cannot be defined unless Theorem \ref{gc4thm8}(b) holds. So we need $\Kurgc$ to define moduli spaces using this method.

Secondly, some kinds of moduli spaces of $J$-holomorphic curves should actually have g-corners rather than ordinary corners, in particular the moduli spaces of `pseudoholomorphic quilts' of Ma'u, Wehrheim and Woodward \cite{Mau,MaWo,WeWo1,WeWo2,WeWo3}, which are used to define actions of Lagrangian correspondences on Lagrangian Floer cohomology and Fukaya categories. 

Ma'u and Woodward \cite{MaWo} define moduli spaces $\oM_{n,1}$ of `stable $n$-marked quilted discs'. As in \cite[\S 6]{MaWo}, for $n\ge 4$ these are not ordinary manifolds with corners, but have an exotic corner structure; in the language of this paper, the $\oM_{n,1}$ are manifolds with g-corners. As in \cite[Ex.~6.3]{MaWo}, the first exotic example $\oM_{4,1}$ has a point locally modelled on $X_P$ near $\de_0$ in Example \ref{gc3ex5}. Ma'u and Woodward \cite[Th.~1.2]{MaWo} show the complexification $\oM_{n,1}^\C$ of $\oM_{n,1}$ is a complex projective variety with toric singularities, which fits with our discussion of complex toric varieties and the model spaces $X_P$ in \S\ref{gc316} and Remark~\ref{gc3rem1}.

More generally, if one omits the simplifying monotonicity and genericity assumptions in \cite{Mau,WeWo1,WeWo2,WeWo3}, the moduli spaces of marked quilted $J$-holomorphic discs discussed in \cite{Mau,WeWo1,WeWo2,WeWo3} should be Kuranishi spaces with g-corners (though we do not claim to prove this), just as moduli spaces of marked $J$-holomorphic discs in Fukaya et al. \cite{FOOO} are Kuranishi spaces with (ordinary) corners.

In another area of symplectic geometry, Pardon \cite{Pard} defines contact homology of Legendrian submanifolds using moduli spaces of $J$-holomorphic curves which are a topological version of Kuranishi spaces with g-corners.

\subsection{Other topics}
\label{gc45}

Sections \ref{gc42}--\ref{gc43} extended known results for manifolds without boundary or with corners to manifolds with g-corners, but the extensions were not obvious, did not always work, and required new proofs when they did. Quite a lot of other material in differential geometry does extend to manifolds with g-corners in an obvious way, and does not require new proofs. This section gives some examples.

\subsubsection{Orientations}
\label{gc451}

Orientations on manifolds with corners are discussed by the author \cite[\S 7]{Joyc1}, \cite[\S 5.8]{Joyc2} and Fukaya et al.\ \cite[\S 8.2]{FOOO}. We extend to manifolds with g-corners:

\begin{dfn} Let $X$ be a manifold with g-corners with $\dim X=n$. Then $\La^n({}^bT^*X)$ is a real line bundle on $X$. An {\it orientation\/} $o$ on $X$ is an equivalence class $[\om]$ of top-dimensional forms $\om\in C^\iy\bigl(\La^n({}^bT^*X)\bigr)$ with $\om\vert_x\ne 0$ for all $x\in X$, where two such $\om,\om'$ are equivalent if $\om'=c\cdot\om$ for $c:X\ra(0,\iy)$ smooth. The {\it opposite orientation\/} is $-o=[-\om]$. Then we call $(X,o)$ an {\it oriented manifold with g-corners}. Usually we suppress the orientation $o$, and
just refer to $X$ as an oriented manifold with g-corners. When $X$ is an oriented
manifold with g-corners, we write $-X$ for $X$ with the opposite orientation.
\label{gc4def10}
\end{dfn}

This is the same as one of the usual definitions of orientations on manifolds or manifolds with corners, except that we use ${}^bT^*X$ rather than $T^*X$. Since ${}^bT^*X$ and $T^*X$ coincide on $X^\ci$, the difference is not important.

As in conventional differential geometry, locally on $X$ there are two possible orientations. Globally orientations need not exist -- the obstruction to existence lies in $H^1(X,\Z_2)$ -- and if they do exist then the family of orientations on $X$ is a torsor for $H^0(X,\Z_2)$. 

As discussed in \cite[\S 7]{Joyc1}, \cite[\S 5.8]{Joyc2}, \cite[\S 8.2]{FOOO} for manifolds with corners, if $X$ is an oriented manifold with g-corners we can define a natural orientation on $\pd X$, and hence on $\pd^2X,\pd^3X,\ldots,\pd^{\dim X}X$, and if $X,Y,Z$ are oriented manifolds with g-corners and $g:X\ra Z$, $h:Y\ra Z$ are b-transverse interior maps then we can define a natural orientation on the fibre product $W=X\t_{g,Z,h}Y$ in $\Mangcin$ from Theorem \ref{gc4thm6}. To do these requires a choice of orientation convention. 

Orientations do not lift to corners $C_k(X)$ for $k\ge 2$. If $X$ is oriented then $\pd^2X$ is oriented, and the natural free $\Z_2$-action on $\pd^2X$ from Proposition \ref{gc3prop5}(a) is orientation-reversing, so that $C_2(X)\cong\pd^2X/\Z_2$ does not have a natural orientation, and $C_k(X)$ need not be orientable for $k\ge 2$, as in~\cite[Ex.~7.3]{Joyc1}.

In all of this, there are no new issues in working with orientations on manifolds with g-corners, except for using $\La^n({}^bT^*X)$ rather than $\La^nT^*X$, which is easy, and which one can already do for manifolds with ordinary corners.

\subsubsection{Partitions of unity}
\label{gc452}

Partitions of unity are often used in differential-geometric constructions, to glue together choices of local data.

\begin{dfn} Let $X$ be a manifold with g-corners and $\{U_i:i\in I\}$ an open cover of $X$, where $I$ is an indexing set. A {\it partition of unity on $X$ subordinate to\/} $\{U_i:i\in I\}$ is a family $\{\eta_i:i\in I\}$ of smooth functions $\eta_i:X\ra\R$ satisfying:
\begin{itemize}
\setlength{\itemsep}{0pt}
\setlength{\parsep}{0pt}
\item[(i)] $\eta_i(X)\subseteq[0,1]$ for all $i\in I$.
\item[(ii)] $\eta_i\vert_{X\sm U_i}=0$ for all $i\in I$.
\item[(iii)] Each $x\in X$ has an open neighbourhood $x\in V\subseteq X$ such that $\eta_i\vert_V=0$ for all except finitely many $i\in I$.
\item[(iv)] $\sum_{i\in I}\eta_i=1$, where the sum makes sense by (iii) as near any $x\in X$ there are only finitely many nonzero terms.
\end{itemize}
\label{gc4def11}
\end{dfn}

By the usual proof for manifolds, as in Lee \cite[Th.~2.23]{Lee}, one can show:

\begin{prop} Let\/ $X$ be a manifold with g-corners and\/ $\{U_i:i\in I\}$ an open cover of\/ $X$. Then there exists a partition of unity $\{\eta_i:i\in I\}$ on $X$ subordinate to~$\{U_i:i\in I\}$.
\label{gc4prop2}
\end{prop}

\subsubsection{Riemannian metrics}
\label{gc453}

Following Melrose \cite[\S 2]{Melr3}, \cite[\S 4]{Melr4} for manifolds with corners, we define:

\begin{dfn} Let $X$ be a manifold with g-corners. A {\it b-metric\/} $g$ on $X$ is a smooth section $g\in C^\iy\bigl(S^2({}^bT^*X)\bigr)$ which restricts to a positive definite quadratic form on ${}^bT_xX$ for all~$x\in X$.

\label{gc4def12}
\end{dfn}

This follows the usual definition of Riemannian metrics on manifolds without boundary, but using ${}^bTX,{}^bT^*X$ rather than $TX,T^*X$. By the usual proof for manifolds using partitions of unity (as in \S\ref{gc452}) one can show that any manifold with g-corners $X$ admits b-metrics~$g$.

On the interior $X^\ci$ we have ${}^bTX=TX$, ${}^bT^*X=T^*X$, so $g^\ci:=g\vert_{X^\ci}$ is an ordinary Riemannian metric on the manifold without boundary $X^\ci$. If $X$ is a compact manifold with g-corners, then $(X^\ci,g^\ci)$ is a complete, generally noncompact Riemannian manifold, with interesting asymptotic behaviour near infinity, determined by the boundary and corners of~$X$.

Melrose \cite{Melr1,Melr2,Melr3,Melr4} studies analysis of elliptic operators on $(X^\ci,g^\ci)$ for $X$ a compact manifold with corners (and also more general situations). It seems likely that his theory extends to $X$ a compact manifold with g-corners.

\subsubsection{Extension of smooth maps from boundaries}
\label{gc454}

Let $X$ be a manifold with corners. As in \eq{gc2eq7}, there is a natural identification
\e
\begin{split}
\pd^2X\cong\bigl\{(x,\be_1,\be_2):\,&\text{$x\in X,$
$\be_1,\be_2$ are distinct}\\
&\text{local boundary components for $X$ at $x$}\bigr\},
\end{split}
\label{gc4eq15}
\e
where $i_{\pd X}:\pd^2X\ra\pd X$ maps $(x,\be_1,\be_2)\mapsto(x,\be_1)$ and $\Pi:\pd^2X\ra X$ maps $(x,\be_1,\be_2)\mapsto x$. There is a natural, free action of $\Z_2=\{1,\si\}$ on $\pd^2X$ by diffeomorphisms, where $\si:\pd^2X\ra\pd^2X$ acts by $\si:(x,\be_1,\be_2)\mapsto(x,\be_2,\be_1)$, with $\Pi\ci\si=\Pi$. It is easy to show:

\begin{prop} Let\/ $X$ be a manifold with (ordinary) corners, and\/ $\si:\pd^2X\ra\pd^2X$ be as above. Then:
\begin{itemize}
\setlength{\itemsep}{0pt}
\setlength{\parsep}{0pt}
\item[{\bf(a)}] Suppose $g:\pd X\ra\R$ is a smooth function. Then there exists a smooth function $f:X\ra\R$ with\/ $f\vert_{\pd X}=g$ if and only if\/ $g\vert_{\pd^2X}:\pd^2X\ra\R$ satisfies $g\vert_{\pd^2X}=g\vert_{\pd^2X}\ci\si$.
\item[{\bf(b)}] Suppose $E\ra X$ is a vector bundle, and\/ $t\in C^\iy(E\vert_{\pd X})$. Then there exists\/ $s\in C^\iy(E)$ with\/ $s\vert_{\pd X}=t$ if and only if\/ $t\vert_{\pd^2X}\in C^\iy(E\vert_{\pd^2X})$ satisfies\/~$\si^*(t\vert_{\pd^2X})=t\vert_{\pd^2X}$.
\end{itemize}
\label{gc4prop3}
\end{prop}

Since local solutions $f$ or $s$ to the equations $f\vert_{\pd X}=g$, $s\vert_{\pd X}=t$ can be combined using a partition of unity (as in \S\ref{gc452}) to make global solutions, it is enough to prove Proposition \ref{gc4prop3} near 0 in~$X=\R^n_k=[0,\iy)^k\t\R^{n-k}$.

Note that the analogue of Proposition \ref{gc4prop3}(a) for smooth maps $X\ra[0,\iy)$ is false. For example, there is no smooth map $f:[0,\iy)^2\ra[0,\iy)$ with $f(x,0)=x$ and $f(0,y)=y$, as $f(x,y)=x+y$ is not a smooth map~$f:[0,\iy)^2\ra[0,\iy)$. 

Now let $X$ be a manifold with g-corners. By Proposition \ref{gc3prop5}(a), we have\begin{align*}
\pd^2X\cong\,&\bigl\{(x,\be_1,\be_2):\text{$x\!\in\! X,$
$\be_1,\be_2$ are distinct local boundary}\\
&\quad\text{components of $X$ at $x$ intersecting in codimension $2$}\bigr\},
\end{align*}
as in \eq{gc4eq15}, and a free action of $\Z_2=\{1,\si\}$ on $\pd^2X$ by diffeomorphisms, where $\si:\pd^2X\ra\pd^2X$ acts by $\si:(x,\be_1,\be_2)\mapsto(x,\be_2,\be_1)$. We can show:

\begin{prop} The analogue of Proposition\/ {\rm\ref{gc4prop3}} holds for $X$ a manifold with g-corners.
\label{gc4prop4}
\end{prop}

Again, since partitions of unity exist for manifolds with g-corners as in \S\ref{gc452}, it is enough to prove Proposition \ref{gc4prop4} near $(\de_0,0)$ in $X=X_P\t\R^n$ for $P$ a toric monoid, and we can do this by embedding $X_P\t\R^n$ in $[0,\iy)^N\t\R^n$ and using Proposition \ref{gc4prop3} for~$[0,\iy)^N\t\R^n$.

Results like Proposition \ref{gc4prop3} are important in constructing virtual chains for Kuranishi spaces with corners with prescribed values on the boundary, as in Fukaya et al.\ \cite{FOOO}, and Proposition \ref{gc4prop4} will be useful for applications of manifolds with g-corners and Kuranishi spaces with g-corners that the author plans in symplectic geometry.

A different generalization of manifolds with corners would be to consider spaces $X$ locally modelled on polyhedra in $\R^n$, with the obvious notion of smooth map. For such spaces, the analogue of Proposition \ref{gc4prop3} is false. For example, suppose $X$ near $x$ is modelled on the corner of an octahedron in $\R^3$, as in Figure \ref{gc3fig1}. Consider smooth $g:\pd X\ra\R$ with $g\vert_{\pd^2X}=g\vert_{\pd^2X}\ci\si$. The possible sets of derivatives $(\pd_1g,\pd_2g,\pd_3g,\pd_4g)$ of $g$ at $x$ along the four edges at $x$ span a space $\R^4$, but for $g=f\vert_{\pd X}$ with $f:X\ra\R$ smooth the derivatives $(\pd_1g,\pd_2g,\pd_3g,\pd_4g)$ lie in an $\R^3\cong T_x^*X$ in $\R^4$, so there are many smooth $g:\pd X\ra\R$ with $g\vert_{\pd^2X}=g\vert_{\pd^2X}\ci\si$ for which there exists no smooth $f:X\ra\R$ with~$g=f\vert_{\pd X}$.

\section{Proofs of theorems in \S\ref{gc4}}
\label{gc5}

Finally we prove Theorems \ref{gc4thm2}, \ref{gc4thm3}, \ref{gc4thm5}, \ref{gc4thm6}, and~\ref{gc4thm7}.

\subsection{Proof of Theorem \ref{gc4thm2}}
\label{gc51}

Let $Q,R,m,n,i$ be as in Theorem \ref{gc4thm2}. Using \S\ref{gc3} we can show there are canonical isomorphisms $\ti M_{(\de_0,0)}(X_Q\t\R^m)\cong Q^\vee$ and $\ti M_{(\de_0,0)}(X_R\t\R^n)\cong R^\vee$. So \eq{gc4eq2} is identified with a monoid morphism $Q^\vee\ra R^\vee$, which must be of the form $\al^\vee$ for unique $\al:R\ra Q$, as in (iii), since $Q\cong(Q^\vee)^\vee$, $R\cong(R^\vee)^\vee$ for the toric monoids~$Q,R$.

By Definition \ref{gc4def5}, $i$ being an immersion imposes strong conditions on the monoid morphism \eq{gc4eq2}, and hence on $\al^\vee:Q^\vee\ra R^\vee$ and $\al:R\ra Q$. So $\al^\vee$ is injective, which implies that $\rank Q\le \rank R$ as in (i). The dual morphism $\al:R\ra Q$ need not be surjective (e.g.\ in Example \ref{gc4ex2}(ii), $\al:\N^2\ra\N$ maps $\al:(a,b)\mapsto 2a+3b$, so $\al(\N^2)=\N\sm\{1\}$), but $\al$ is close to being surjective -- for example, $\al^\gp:R^\gp\ra P^\gp$ is surjective, and the map $C_\al:C_R\ra C_Q$ of the rational polyhedral cones $C_Q,C_R$ associated to $Q,R$ in \S\ref{gc314} is surjective. The surjectivity property we need, which can be proved from Definition \ref{gc4def5}, is that if $q\in Q$ then there exist $r\in R$ and $a=1,2,\ldots$ such that $\al(r)=a\cdot q$, that is, $\al$ is surjective up to positive integer multiples in~$Q$. 

Choose a set of generators $q_1,\ldots,q_M$ for $Q$. Then we can choose $r_1,\ldots,r_M\in R$ and $a_1,\ldots,a_M=1,2,\ldots$ with $\al(r_j)=a_j\cdot q_j$ for $j=1,\ldots,M$. Extend $r_1,\ldots,r_M$ to a set of generators $r_1,\ldots,r_N$ for $R$, for $N\ge M$. Then as in Proposition \ref{gc3prop3}(a), $\la_{q_1}\t\cdots\t\la_{q_M}:X_Q\ra[0,\iy)^M$ and $\la_{r_1}\t\cdots\t\la_{r_N}:X_R\ra[0,\iy)^N$ are homeomorphisms from $X_Q,X_R$ to closed subsets $X_Q'\subseteq [0,\iy)^M$, $X_R'\subseteq [0,\iy)^N$ defined in \eq{gc3eq4} using generating sets of relations for $q_1,\ldots,q_M$ in $Q$ and $r_1,\ldots,r_N$ in $R$. Hence $(\la_{q_1}\t\cdots\t\la_{q_M})\t\id_{\R^m}$ identifies $X_Q\t\R^m$ with $X_Q'\t\R^m\subseteq[0,\iy)^M\t\R^m$. Let $V'\subseteq X_Q'\t\R^m$ be the image of $V$. Similarly $(\la_{r_1}\t\cdots\t\la_{r_N})\t\id_{\R^n}$ identifies $X_R\t\R^n$ with~$X_R'\t\R^n\subseteq[0,\iy)^N\t\R^n$.

Then Proposition \ref{gc3prop3}(c) applied to $i:V\ra X_R\t\R^n$ shows that there exists an open neighbourhood $Y$ of $V'$ in $[0,\iy)^M\t\R^m$, and an interior map $h:Y\ra [0,\iy)^N\t\R^n$ of manifolds with (ordinary) corners, such that 
\e
[(\la_{r_1}\!\t\!\cdots\!\t\!\la_{r_N})\!\t\!\id_{\R^n}]\ci i\!=\!h\ci [(\la_{q_1}\!\t\!\cdots\!\t\!\la_{q_M})\!\t\!\id_{\R^m}]:U\ra [0,\iy)^N\t\R^n.
\label{gc5eq1}
\e
We have simplified things here, since Proposition \ref{gc3prop3}(c) does not allow for the factors $\R^m,\R^n$, but these can be included using embeddings $\R^m\ra[0,\iy)^{m+1}$, $\R^n\ra[0,\iy)^{n+1}$ coming from minimal sets of monoid generators of~$\Z^m,\Z^n$.

Write $(w_1,\ldots,w_M,x_1,\ldots,x_m)$ for the coordinates on $Y\subseteq[0,\iy)^M\t\R^m$ and $(y_1,\ldots,y_N,z_1,\ldots,z_n)$ for the coordinates on $[0,\iy)^N\t\R^n$, and write $h=(H_1,\ldots,H_N,h_1,\ldots,h_n)$ for $H_j=H_j(w_1,\ldots,x_m)$, $h_j=h_j(w_1,\ldots,x_m)$. Then near 0 in $Y$ we have $H_j=C_j(w_1,\ldots,x_m)\cdot \prod_{i=1}^Mw_i^{b_{i,j}}$ for $b_{i,j}\in\N$ and $C_j:Y\ra(0,\iy)$ smooth. Since the coordinates $w_1,\ldots,w_M$ correspond to the generators $q_1,\ldots,q_M\in Q$, and the coordinates $y_1,\ldots,y_N$ to $r_1,\ldots,r_N\in\R$, and $\al(r_j)=a_j\cdot q_j$ for $j=1,\ldots,M$, we see that we can choose $h$ such that
\e
H_j(w_1,\ldots,w_M,x_1,\ldots,x_m)=C_j(w_1,\ldots,x_m)\cdot w_j^{a_j},\quad j=1,\ldots,M.
\label{gc5eq2}
\e

We can now show that
\begin{align*}
{}^b\d i\vert_{(\de_0,0)}=\begin{pmatrix} {}\ci\al & \bigl(\sum_{c=1}^N\frac{\pd C_c}{\pd x_b}(0)\cdot \al(r_c)\bigr)_{b=1}^n \\ 0 & \bigl(\frac{\pd h_c}{\pd x_b}(0)\bigr)_{b=1,\ldots,m}^{c=1,\ldots,n} \end{pmatrix}&: \\
\Hom(Q,\R)\op\R^m \longra \Hom(R,\R)&\op\R^n.
\end{align*}
As $i$ is an immersion, Definition \ref{gc4def5}(i) implies that $\bigl(\frac{\pd h_c}{\pd x_b}(0)\bigr)_{b=1,\ldots,m}^{c=1,\ldots,n}$ is injective. Hence $m\le n$, completing part (i) of Theorem \ref{gc4thm2}. By applying a linear transformation to the coordinates $(z_1,\ldots,z_n)$ on $\R^n$, we can suppose that
\e
\frac{\pd h_c}{\pd x_b}(0,\ldots,0)=\begin{cases} 1, & b=c=1,\ldots,m, \\ 0, & \text{otherwise.}\end{cases}
\label{gc5eq3}
\e

Define a continuous, non-smooth map $\Pi:[0,\iy)^N\t\R^n\ra [0,\iy)^M\t\R^m$ by
\e
\Pi:(y_1,\ldots,y_N,z_1,\ldots,z_n)\longmapsto (y_1^{1/a_1},\ldots,y_M^{1/a_M},z_1,\ldots,z_m).
\label{gc5eq4}
\e
By \eq{gc5eq2}, the composition $\Pi\ci h:V\ra [0,\iy)^M\t\R^m$ is given by
\e
\begin{split}
\Pi&\ci h(w_1,\ldots,w_M,x_1,\ldots,x_m)=\bigl(C_1(w_1,\ldots,x_m)^{1/a_1}\cdot w_1,\ldots,\\
&C_M(w_1,\ldots,x_m)^{1/a_M}\cdot w_m,h_1(w_1,\ldots,x_m),\ldots,h_m(w_1,\ldots,w_m)\bigr),
\end{split}
\label{gc5eq5}
\e
which is smooth (as $C_c>0$), although $\Pi$ is not. By \eq{gc5eq3} and \eq{gc5eq5}, the derivative of $\Pi\ci h$ at $(0,\ldots,0)\in V$ is the $(M+m)\t(M+m)$ matrix
\begin{equation*}
\d(\Pi\ci h)(0)=\begin{pmatrix} \mathop{\rm diag}(C_1(0)^{1/a_1},\ldots,C_M(0)^{1/a_M}) & 0 \\ \bigl(\frac{\pd h_c}{\pd w_b}(0)\bigr)_{b=1,\ldots,M}^{c=1,\ldots,m} & \id_{m\t m}\end{pmatrix},
\end{equation*}
which is invertible. Also \eq{gc5eq5} implies that $\Pi\ci h$ is simple near $0$. Therefore Proposition \ref{gc2prop3} shows $\Pi\ci h$ is \'etale near 0. So there exists an open $0\in\ti Y\subseteq Y$ such that $\Pi\ci h\vert_{\ti Y}$ is a diffeomorphism from $\ti Y$ to its image. 

Set $\ti V=[(\la_{q_1}\t\cdots\t\la_{q_M})\t\id_{\R^m}]^{-1}(\ti Y)$. Then by \eq{gc5eq1} we have
\begin{align*}
(\Pi\ci &h)\vert_{\ti Y}^{-1}\ci
\Pi\ci[(\la_{r_1}\!\t\!\cdots\!\t\!\la_{r_N})\!\t\!\id_{\R^n}]\ci i\vert_{\ti V}\\
&=(\Pi\ci h)\vert_{\ti Y}^{-1}\ci(\Pi\ci h)\vert_{\ti Y}\ci [(\la_{q_1}\!\t\!\cdots\!\t\!\la_{q_M})\!\t\!\id_{\R^m}]\vert_{\ti V}\\
&=[(\la_{q_1}\!\t\!\cdots\!\t\!\la_{q_M})\!\t\!\id_{\R^m}]\vert_{\ti V}.\end{align*}
Since $(\la_{q_1}\t\cdots\t\la_{q_M})\t\id_{\R^m}$ is a homeomorphism with its image, $i\vert_{\ti V}$ is a homeomorphism with its image, proving part (ii) of Theorem~\ref{gc4thm2}.

We have already proved the first part of (iii). For the second part, consider
\e
\begin{split}
S=\bigl\{&(u_1,\ldots,u_N)\in (-\iy,0)^N:\text{there exist sequences $(\bs y_a,\bs z_a)_{a=1}^\iy$}\\
&\text{in $i(\ti V)\cap (X_R^\ci\t\R^n)$ and $(\mu_a)_{a=0}^\iy$ in $(0,\iy)$ such that}\\
&\text{as $a\ra\iy$ we have $(\bs y_a,\bs z_a)\ra (\de_0,0)$ in $X_R\t\R^n$,}\\
&\text{$\mu_a\ra 0$ in $\R$, and $\mu_a\cdot\log[\la_{r_j}(\bs y_a)]\ra u_j$ in $\R$ for $j=1,\ldots,N$}\bigr\}.
\end{split}
\label{gc5eq6}
\e
If $(\bs y_a,\bs z_a)\in i(\ti V)\cap (X_R^\ci\t\R^n)$ is close to $(\de_0,0)$ in $X_R\t\R^n$, then $(\bs y_a,\bs z_a)=i(\bs w_a,\bs x_a)$ for $(\bs w_a,\bs x_a)\in \ti V\subseteq X_Q^\ci\t\R^n$, and $(\bs w_a,\bs x_a)$ is close to $(\de_0,0)$ in $X_Q\t\R^m$ as $i\vert_{\ti V}$ is a homeomorphism with its image.

The definition of smooth maps in \S\ref{gc32} now gives $\la_{r_j}(\bs y_a)=D_j(\bs w_a,\bs z_a)\cdot\la_{\al(r_j)}(\bs w_a)$, for some smooth $D_j:\ti V\ra(0,\iy)$. Hence
\e
\mu_a\cdot \log[\la_{r_j}(\bs y_a)]=\mu_a\cdot\log[\la_{\al(r_j)}(\bs w_a)]+
\mu_a\cdot \log D_j(\bs w_a,\bs z_a).
\label{gc5eq7}
\e
As $a\ra\iy$ we have $\log D_j(\bs w_a,\bs z_a)\ra \log D_j(\de_0,0)$, and $\mu_a\ra 0$, so the final term in \eq{gc5eq7} tends to zero. Thus we may rewrite \eq{gc5eq6} as
\begin{align*}
S=\bigl\{&(u_1,\ldots,u_N)\in (-\iy,0)^N:\text{there exist sequences $(\bs w_a)_{a=1}^\iy$ in $X_Q^\ci$}\\
&\text{and $(\mu_a)_{a=0}^\iy$ in $(0,\iy)$ such that as $a\ra\iy$ we have $\bs w_a\ra \de_0$ in $X_Q$,}\\
&\text{$\mu_a\ra 0$ in $\R$, and $\mu_a\cdot\log[\la_{\al(r_j)}(\bs w_a)]\ra u_j$ in $\R$ for $j=1,\ldots,N$}\bigr\}.
\end{align*}
It is now easy to see that $S$ is the intersection of $(-\iy,0)^N$ with the image of the composition of linear maps
\e
\xymatrix@C=70pt{ 
\Hom(Q,\R) \ar[r]^{\ci \al} & \Hom(R,\R) \ar[r]^{(r_1,\ldots,r_N)} & \R^N. }
\label{gc5eq8}
\e
Thus the subset $i(\ti V)\subseteq X_R\t\R^n$ near $(\de_0,0)$ determines $S$, which determines the image of \eq{gc5eq8}. As $r_1,\ldots,r_N$ generate $R$, the second map in \eq{gc5eq8} is injective, so $i(\ti V)$ near $(\de_0,0)$ determines the image of~$\ci \al:\Hom(Q,\R)\ra\Hom(R,\R)$.

We have a commutative diagram
\begin{equation*}
\xymatrix@C=120pt@R=15pt{
*+[r]{Q^\vee=\Hom(Q,\N)} \ar[d]^{\rm inc} \ar[r]_{\al^\vee=\ci \al} & *+[l]{R^\vee=\Hom(R,\N)} \ar[d]_{\rm inc}\\
*+[r]{\Hom(Q,\R)} \ar[r]^{\ci \al} & *+[l]{\Hom(R,\R).\!} }
\end{equation*}
Since \eq{gc4eq2} is identified with $\al^\vee$, Definition \ref{gc4def5}(ii),(iii) say that $\al^\vee$ is injective and $R^\vee/\al^\vee(Q^\vee)$ is torsion-free. The torsion-freeness implies that $\al^\vee[Q^\vee]=R^\vee\cap (\ci\al)[\Hom(Q,\R)]$. Therefore $i(\ti V)$ near $(\de_0,0)$ determines the image $\al^\vee(Q^\vee)$ in $R^\vee$, where $\al^\vee(Q^\vee)\cong Q^\vee$. The inclusion $\al^\vee(Q^\vee)\hookra R$ is dual to $\al:R\ra Q$, up to $[\al^\vee(Q^\vee)]^\vee\cong Q$. Hence $Q,\al$ are determined uniquely, up to canonical isomorphisms of $Q$, by $i(\ti V)$ near $(\de_0,0)$. Also $i(\ti V)\cap (X_R^\ci\t\R^n)$ is a manifold of dimension $\rank Q+m$, so $m$ is determined. This completes part~(iii).

Let $P,U,l,f$ be as in (iv). Since $i\vert_{\ti V}:\ti V\ra i(\ti V)$ is a homeomorphism and $f(U)\subseteq i(\ti V)$, there is a unique continuous map $g:U\ra \ti V$ with $f=i\ci g$. We must show that $g$ is smooth near $(\de_0,0)\in U$. It is sufficient to show 
$[(\la_{q_1}\t\cdots\t\la_{q_M})\t\id_{\R^m}]\ci g:U\ra [0,\iy)^M\t\R^m$ is smooth near $(\de_0,0)$. But
\e
\begin{split}
[(\la_{q_1}&\t\cdots\t\la_{q_M})\t\id_{\R^m}]\ci g\\
&=(\Pi\ci h\vert_{\ti Y})^{-1}\ci\Pi\ci h\vert_{\ti Y}\ci[(\la_{q_1}\t\cdots\t\la_{q_M})\t\id_{\R^m}]\ci g\\
&=(\Pi\ci h\vert_{\ti Y})^{-1}\ci\Pi\ci [(\la_{r_1}\!\t\!\cdots\!\t\!\la_{r_N})\!\t\!\id_{\R^n}]\ci i\ci g\\
&=(\Pi\ci h\vert_{\ti Y})^{-1}\ci\Pi\ci [(\la_{r_1}\!\t\!\cdots\!\t\!\la_{r_N})\!\t\!\id_{\R^n}]\ci f,
\end{split}
\label{gc5eq9}
\e
where the first step uses $[(\la_{q_1}\t\cdots\t\la_{q_M})\t\id_{\R^m}]\ci g(U)\subseteq\ti Y$ and $\Pi\ci h\vert_{\ti Y}$ has a smooth inverse, the second \eq{gc5eq1}, and the third $f=i\ci g$. In the last line of \eq{gc5eq9}, each term is smooth except $\Pi$ in \eq{gc5eq4}, which involves functions~$y_j^{1/a_j}$.

As in part (iii), we can identify $\ti M_{(\de_0,0)}f$ with $\be^\vee:P^\vee\ra R^\vee$, for some monoid morphism $\be:R\ra P$. Since $f(U)\subseteq i(\ti V)$, using the argument of the proof of (iii) we see that $(\ci\be)[\Hom(P,\R)]\subseteq(\ci\al)[\Hom(Q,\R)]\subseteq\Hom(R,\R)$, and hence that $\be^\vee(P^\vee)\subseteq\al^\vee(Q^\vee)\subseteq R^\vee$. Since $\al^\vee$ is injective, it follows that $\be^\vee:P^\vee\ra R^\vee$ factors through $\al^\vee:Q^\vee\ra R^\vee$. That is, there exists a monoid morphism $\ga^\vee:P^\vee\ra Q^\vee$ with $\be^\vee=\al^\vee\ci\ga^\vee$. Then $\ga:Q\ra P$ is a monoid morphism with $\be=\ga\ci\al$. 

Hence as $f$ is smooth, for $j=1,\ldots,M$, near $(\de_0,0)$ in $U$ we may write 
\begin{equation*}
\la_{r_j}\ci f=E_j\cdot \la_{\be(r_j)}=E_j\cdot \la_{\ga\ci\al(r_j)}=\la_{\ga(a_j\cdot q_j)}=E_j\cdot \la_{\ga(q_j)}^{a_j}:U\ra[0,\iy),
\end{equation*}
where $E_j:U\ra(0,\iy)$ is smooth, as $\al(r_j)=a_j\cdot q_j$. Thus $(\la_{r_j}\ci f)^{1/a_j}=E_j^{1/a_j}\cdot\la_{\ga(q_j)}:U\ra[0,\iy)$ near $(\de_0,0)$ in $U$, which is smooth. But by \eq{gc5eq4}, the only potentially non-smooth functions in the factor $\Pi$ in the last line of \eq{gc5eq9} are $(\la_{r_j}\ci f)^{1/a_j}$ for $j=1,\ldots,M$. So by \eq{gc5eq9}, $[(\la_{q_1}\t\cdots\t\la_{q_M})\t\id_{\R^m}]\ci g$ is smooth on an open neighbourhood $\ti U$ of $(\de_0,0)$ in $U$, and therefore $g$ is smooth on $\ti U$. This completes part~(iv).

Finally suppose $\al:R\ra Q$ is an isomorphism, and $m=n$. Then in the proof above, after choosing generators $q_1,\ldots,q_M$ for $Q$, we can take $r_j=\al^{-1}(q_j)$ for $j=1,\ldots,M$, so that $\al(r_j)=q_j$ with $a_j=1$, and then $r_1,\ldots,r_M$ are already a set of generators for $R\cong Q$, so we take $N=M$. Then $\Pi$ in \eq{gc5eq4} is the identity, and $\Pi\ci h=h$, so the proof above shows that $h$ is \'etale near 0, and we choose open $0\in \ti Y\subseteq Y$ with $0\in h(\ti Y)\subseteq [0,\iy)^M\t\R^m$ open, and $h\vert_{\smash{\ti Y}}:\ti Y\ra h(\ti Y)$ a diffeomorphism.

We have $X_Q'=X_R'\subseteq [0,\iy)^M$, and $h$ maps the closed set $\ti Y\cap (X_Q'\t\R^m)\subseteq\ti Y$ into a closed subset of $h(\ti Y)\cap (X_R'\t\R^m)\subseteq h(\ti Y)$. On the interior $(0,\iy)^M$, $h$ maps $\ti Y\cap (X_Q^{\prime\ci}\t\R^m)$ to an open subset of $h(\ti Y)\cap (X_R^{\prime\ci}\t\R^n)$, as it is a local diffeomorphism of manifolds without boundary. Hence $h[\ti Y\cap (X_Q^{\prime\ci}\t\R^m)]$ is open and closed in $h(\ti Y)\cap (X_R^{\prime\ci}\t\R^n)$. As $h(\ti Y)\cap (X_R^{\prime\ci}\t\R^n)$ is connected near $(\de_0,0)$, making $\ti Y$ smaller we can suppose $h[\ti Y\cap (X_Q^{\prime\ci}\t\R^m)]=h(\ti Y)\cap (X_R^{\prime\ci}\t\R^n)$, so taking closures gives~$h[\ti Y\cap (X_Q'\t\R^m)]=h(\ti Y)\cap (X_R'\t\R^n)$.

Thus, $h^{-1}:h(\ti Y)\ra\ti Y$ maps $h(\ti Y)\cap (X_R'\t\R^n)\ra \ti Y\cap (X_Q'\t\R^m)$. Setting $\dot V\!=\![(\la_{q_1}\!\t\!\cdots\!\t\!\la_{q_M})\!\t\!\id_{\R^m}]^{-1}(\ti Y)\!\subseteq\! V$ and $\dot W\!=\![(\la_{r_1}\!\t\!\cdots\!\t\!\la_{r_M})\!\t\!\id_{\R^n}]^{-1}(h(\ti Y))\ab\!\subseteq\! X_R\t\R^n$, we see that $i\vert_{\smash{\dot V}}:\dot V\ra \dot W$ has a smooth inverse $i\vert_{\smash{\dot V}}^{-1}$ with
\begin{equation*}
[(\la_{q_1}\!\t\!\cdots\!\t\!\la_{q_M})\!\t\!\id_{\R^m}]\ci i\vert_{\smash{\dot V}}^{-1}=\!h^{-1}\ci [(\la_{r_1}\!\t\!\cdots\!\t\!\la_{r_M})\!\t\!\id_{\R^n}]:\dot W\!\ra\! [0,\iy)^M\!\t\!\R^m,
\end{equation*}
as in \eq{gc5eq1}, so $i\vert_{\smash{\dot V}}$ is a diffeomorphism, as in (v). This completes the proof.

\subsection{Proof of Theorem \ref{gc4thm3}}
\label{gc52}

Let $Q,n,V,f_i,g_i,h_j,\be_i,\ga_i,X^\ci$ and $X\ni(\de_0,0)$ be as in Theorem \ref{gc4thm3}. From \S\ref{gc32}, on an open neighbourhood $V'$ of $(\de_0,0)$ in $V$ we can write
\e
f_i(\bs y,\bs z)\!=\!D_i(\bs y,\bs z)\cdot \la_{s_i}(\bs y),\;\>
g_i(\bs y,\bs z)\!=\!E_i(\bs y,\bs z)\cdot \la_{t_i}(\bs y),\;\>
i\!=\!1,\ldots,k,
\label{gc5eq10}
\e
where $(\bs y,\bs z)\in V'$, $\bs y\in X_Q$, $\bs z=(z_1,\ldots,z_n)\in\R^n$, and $s_i,t_i\in Q$, $D_i,E_i:V'\ra(0,\iy)$ are smooth, for $i=1,\ldots,k$. Under the isomorphism \eq{gc4eq4}, the components of ${}^b\d f_i\vert_{(\de_0,0)},{}^b\d g_i\vert_{(\de_0,0)}$ in $Q\ot_\N\R\supseteq Q$ are $s_i,t_i$, so the component $\be_i$ of ${}^b\d f_i\vert_{(\de_0,0)}-{}^b\d g_i\vert_{(\de_0,0)}$ in $Q\ot_\N\R$ is~$\be_i=s_i-t_i$.

Now $\be_1,\ldots,\be_k$ are elements of $Q\ot_\N\Z\subseteq Q\ot_\N\R$. We will first show that if $\be_1,\ldots,\be_k$ are not linearly independent over $\R$ in $Q\ot_\N\R$ then we can replace $f_i,g_i,s_i,t_i,\be_i,h_j$ by $f'_i,g'_i,s'_i,t'_i,\be'_i$ for $i=1,\ldots,k'$ and $h_j'$ for $j=1,\ldots,l'$, such that $k'<k$, $l'>l$ with $k'+l'=k+l$, and $\be'_1,\ldots,\be'_{\smash{k'}}$ are linearly independent over $\R$, and $X^{\prime\ci}$ defined in \eq{gc4eq3} using $f'_i,g'_i,h_j'$ for $i=1,\ldots,k'$, $j=1,\ldots,l'$ agrees near $(\de_0,0)$ with $X^\ci$ defined using $f_i,g_i,h_j$ for $i=1,\ldots,k$, $j=1,\ldots,l$. 

Since $k+l=k'+l'$, this substitution does not change the equation $\rank P+m=\rank Q+n-k-l$ in the theorem. Also the substitution does not change $\langle\be_1,\ldots,\be_k\rangle_\R$, and so does not change the expression for $P^\vee$ in \eq{gc4eq5}. Note that in the last part of Theorem \ref{gc4thm3} we assume that $\be_1,\ldots,\be_k$ are linearly independent over $\R$, so the substitution is unnecessary for the last part.

To do this, permute the indices $i=1,\ldots,k$ in $f_i,g_i,s_i,t_i,\be_i$ if necessary such that $\be_1,\ldots,\be_{\smash{k'}}$ are linearly independent over $\R$, where $k'=\dim_\R\langle \be_1,\ldots,\be_k\rangle_\R$, and for $i=k'+1,\ldots,k$ we have
\e
\be_i=\ts\sum_{i'=1}^{k'}C_{ii'}\be_{i'}
\label{gc5eq11}
\e
for unique $C_{ii'}\in\R$. Then define $l'=l+k-k'$, and $f'_i=f_i$, $g'_i=g_i$, $s'_i=s_i$, $t'_i=t_i$, $\be'_i=\be_i$ for $i=1,\ldots,k$, and $h'_j=h_j$ for $j=1,\ldots,l$, and define $h'_j$ for $j=l+1,\ldots,l'$ by
\e
h'_j=\log D_{j+k'-l}-\log E_{j+k'-l}-\ts\sum_{i'=1}^{k'}C_{(j+k'-l)i'}(\log D_{i'}-\log E_{i'}).
\label{gc5eq12}
\e
The point of this equation is that by \eq{gc5eq10}--\eq{gc5eq12}, on $V^\ci$ we have
\begin{align*}
&\frac{f_{j+k'-l}(\bs y,\bs z)}{g_{j+k'-l}(\bs y,\bs z)}\cdot \prod_{i'=1}^{k'}\frac{g_{i'}(\bs y,\bs z)^{C_{(j+k'-l)i'}}}{f_{i'}(\bs y,\bs z)^{C_{(j+k'-l)i'}}}\\
&=\frac{D_{j+k'-l}(\bs y,\bs z)\la_{s_{j+k'-l}}(\bs y)}{E_{j+k'-l}(\bs y,\bs z)\la_{t_{j+k'-l}}(\bs y)}\cdot \prod_{i'=1}^{k'}\frac{E_{i'}(\bs y,\bs z)^{C_{(j+k'-l)i'}}\la_{t_{i'}}(\bs y)^{C_{(j+k'-l)i'}}}{D_{i'}(\bs y,\bs z)^{C_{(j+k'-l)i'}}\la_{s_{i'}}(\bs y)^{C_{(j+k'-l)i'}}}\\
&=\frac{D_{j+k'-l}(\bs y,\bs z)}{E_{j+k'-l}(\bs y,\bs z)}\cdot \prod_{i'=1}^{k'}\frac{E_{i'}(\bs y,\bs z)^{C_{(j+k'-l)i'}}}{D_{i'}(\bs y,\bs z)^{C_{(j+k'-l)i'}}}=\exp\bigl(h_j'(\bs y,\bs z)\bigr).
\end{align*}
Thus, if we assume $f'_i=g'_i$ for $i=1,\ldots,k'$, which gives $f_i=g_i$ for $i=1,\ldots,k'$, then $f_{j+k'-l}=g_{j+k'-l}$ is equivalent to $\exp(h_j')=1$ is equivalent to $h_j'=0$ on $V^\ci$ for~$j=l+1,\ldots,l'$. 

That is, replacing $f_{j+k'-l}=g_{j+k'-l}$ by $h_j'=0$ for $j=l+1,\ldots,l'$ does not change $X^\ci$ in \eq{gc4eq3}, at least in $V'$ where \eq{gc5eq10} holds. The $(k+l)$-tuples ${}^b\d f_1\vert_{(\de_0,0)}-{}^b\d g_1\vert_{(\de_0,0)},\ab\ldots,{}^b\d f_k\vert_{(\de_0,0)}-{}^b\d g_k\vert_{(\de_0,0)},\d h_1\vert_{(\de_0,0)},\ldots,\d h_l\vert_{(\de_0,0)}$ and ${}^b\d f'_1\vert_{(\de_0,0)}-{}^b\d g'_1\vert_{(\de_0,0)},\ldots,\ab {}^b\d f'_{\smash{k'}}\vert_{(\de_0,0)}-{}^b\d g'_{\smash{k'}}\vert_{(\de_0,0)},\d h'_1\vert_{(\de_0,0)},\ldots,\d h'_{\smash{l'}}\vert_{(\de_0,0)}$ in ${}^bT_{(\de_0,0)}^*V$ differ by an invertible $(k+l)\t(k+l)$ matrix, so ${}^b\d f'_1\vert_{(\de_0,0)}-{}^b\d g'_1\vert_{(\de_0,0)},\ab\ldots,\ab{}^b\d f'_{\smash{k'}}\vert_{(\de_0,0)}-{}^b\d g'_{\smash{k'}}\vert_{(\de_0,0)},\d h'_1\vert_{(\de_0,0)},\ldots,\d h'_{\smash{l'}}\vert_{(\de_0,0)}$ are linearly independent. 

Note that $f_{j+k'-l}(\de_0,0)=g_{j+k'-l}(\de_0,0)$ does not imply that $h_j'(\de_0,0)=0$. Instead, we can deduce $h_j'(\de_0,0)=0$ from the assumption that $(\de_0,0)\in X$, since $h_j'$ is continuous and $(\de_0,0)$ is the limit of points $v\in X^\ci$ in \eq{gc4eq3} with~$h_j'(v)=0$.

We will suppose for the next part of the proof that $f_i,g_i,s_i,t_i,D_i,E_i$ for $i=1,\ldots,k$ and $h_j$ for $j=1,\ldots,l$ are as above, and $\be_1,\ldots,\be_k$ are linearly independent over $\R$ in $Q\ot_\N\R$. Now $\d h_1\vert_{(\de_0,0)},\ldots,\d h_l\vert_{(\de_0,0)}$ are linearly independent in ${}^bT_{(\de_0,0)}^*V={}^bT_{\de_0}^*X_Q\op T_0^*\R^n$, and the components in ${}^bT_{\de_0}^*X_Q$ are zero, so the components in $T_0^*\R^n$ are linearly independent. Hence $l\le n$, and by a linear change of variables $(z_1,\ldots,z_n)$ in $\R^n$ we can suppose that
\e
\frac{\pd h_j}{\pd z_p}(\de_0,0)=\begin{cases} 1, & j=p=1,\ldots, l, \\
0, & j=1,\ldots,l,\;\> p=1,\ldots,n,\;\> j\ne p.
\end{cases}
\label{gc5eq13}
\e

Choose a set of generators $q_1,\ldots,q_N$ for $Q$. Writing $r=\rank Q$, as in \eq{gc3eq3} choose relations for $q_1,\ldots,q_N$ in $Q$ of the form
\e
a^1_iq_1+\cdots+a^N_iq_N=b^1_iq_1+\cdots+b^N_iq_N\quad\text{for $i=1,\ldots,N-r,$}
\label{gc5eq14}
\e
where $a^j_i,b^j_i\in\N$ for $1\le i\le N-r$, $1\le j\le N$, such that the relations \eq{gc5eq14} form a basis over $\R$ for $\Ker\bigl((\N^N)^\vee\ra Q^\vee\bigr)\ot_\N\R$. Then following the proof of Proposition \ref{gc3prop3}(a), we can show that $\la_{q_1}\t\cdots\t\la_{q_N}:X_Q^\ci\ra(0,\iy)^N$ is a homeomorphism from $X_Q^\ci$ to
\e
X_Q^{\prime\ci}\!=\!\bigl\{(x_1,\ldots,x_N)\!\in\!(0,\iy)^N\!:\!x_1^{a^1_i}\cdots x_N^{a^N_i}\!\!=\!x_1^{b^1_i}\cdots x_N^{b^N_i},\; i\!=\!1,\ldots,N\!-\!r\bigr\}.\!
\label{gc5eq15}
\e
Here we restrict to interiors $X_Q^\ci,X_Q^{\prime\ci},(0,\iy)^N$ as we don't assume that the relations \eq{gc5eq14} define $Q$ as a quotient monoid of $\N^N$, but only the weaker condition that they span $\Ker\bigl((\N^N)^\vee\ra Q^\vee\bigr)\ot_\N\R$ over $\R$.

By Proposition \ref{gc3prop3}(b) (slightly generalized as in the proof of Theorem \ref{gc4thm2} in \S\ref{gc51}), there exists an open neighbourhood $W$ of $[(\la_{q_1}\t\cdots\t\la_{q_N})\t\id_{\R^n}](V)$ in $[0,\iy)^N\t\R^n$ such that the interior functions $f_i,g_i:V\ra[0,\iy)$ and $h_j:V\ra\R$ are compositions of $(\la_{q_1}\t\cdots\t\la_{q_N})\t\id_{\R^n}:V\ra [0,\iy)^N\t\R^n$ with interior functions $\ti f_i,\ti g_i:W\ra[0,\iy)$ and $\ti h_j:W\ra\R$, for $i=1,\ldots,N$ and $j=1,\ldots,n$. As in \eq{gc5eq10}, on an open neighbourhood $W'$ of $(0,\ldots,0)$ in $W$ with 
\e
[(\la_{q_1}\t\cdots\t\la_{q_N})\t\id_{\R^n}](V')=W'\cap (X_Q'\t\R^n),
\label{gc5eq16}
\e
we can write
\e
\begin{split}
\ti f_i(\bs x,\bs z)=\ti D_i(\bs x,\bs z)\cdot x_1^{s_i^1}\cdots x_N^{s_i^N},\;\>
\ti g_i(\bs x,\bs z)=\ti E_i(\bs x,\bs z)\cdot x_1^{t_i^1}\cdots x_N^{t_i^N},
\end{split}
\label{gc5eq17}
\e
for $i=1,\ldots,k$, where $\bs x=(x_1,\ldots,x_N)\in[0,\iy)^N$ and $\bs z=(z_1,\ldots,z_n)\in\R^n$ with $(\bs x,\bs z)\in W'\subseteq W\subseteq [0,\iy)^N\t\R^n$, and $\ti D_i,\ti E_i:W\ra(0,\iy)$ are smooth, and $s_i^j,t_i^j\in\N$ with $s_i^1q_1+\cdots+s_i^Nq_N=s_i$, $t_i^1q_1+\cdots+t_i^Nq_N=t_i$ in $Q$. From equation \eq{gc5eq13} it follows that
\e
\frac{\pd\ti h_j}{\pd z_p}(\bs 0,\bs 0)=\begin{cases} 1, & j=p=1,\ldots, l, \\
0, & j=1,\ldots,l,\;\> p=1,\ldots,n,\;\> j\ne p.
\end{cases}
\label{gc5eq18}
\e

Consider the $(N-r+k)\t N$ matrix
\e
\begin{pmatrix}
a^1_1-b^1_1 & a^2_1-b^2_1 & \cdots & a^N_1-b^N_1 \\
a^1_2-b^1_2 & a^2_2-b^2_2 & \cdots & a^N_2-b^N_2 \\
\vdots & \vdots && \vdots \\
a^1_{N-r}-b^1_{N-r} & a^2_{N-r}-b^2_{N-r} & \cdots & a^N_{N-r}-b^N_{N-r} \\[4pt]
s^1_1-t^1_1 & s^2_1-t^2_1 & \cdots & s^N_1-t^N_1 \\
s^1_2-t^1_2 & s^2_2-t^2_2 & \cdots & s^N_2-t^N_2 \\
\vdots & \vdots && \vdots \\
s^1_k-t^1_k & s^2_k-t^2_k & \cdots & s^N_k-t^N_k 
\end{pmatrix}.
\label{gc5eq19}
\e
By definition of the $a^j_i,b^j_i$, the first $N-r$ rows are linearly independent over $\R$. But the last $k$ rows are lifts of $s_1-t_1,\ldots,s_k-t_k$, which are linearly independent over $\R$ in $Q\ot_\N\R$, and $Q\ot_\N\R$ is the quotient of $\R^N$ by the span of the first $N-r$ rows. It follows that all $N-r+k$ rows of \eq{gc5eq19} are linearly independent over $\R$, and the matrix \eq{gc5eq19} has rank~$N-r+k\le N$.

By elementary linear algebra, $N-r+k$ of the columns of \eq{gc5eq19} are linearly independent over $\R$. By permuting $q_1,\ldots,q_N$ we can suppose the first $N-r+k$ columns are linearly independent, so that the first $N-r+k$ columns form an invertible $(N-r+k)\t(N-r+k)$ matrix. Write the inverse matrix as
\begin{equation*}
\begin{pmatrix}
c_1^1 & c_1^2 & \cdots & c_1^{N-r} & d_1^1 & d_1^2 & \cdots & d_1^k \\
c_2^1 & c_2^2 & \cdots & c_2^{N-r} & d_2^1 & d_2^2 & \cdots & d_2^k \\
\vdots & \vdots && \vdots & \vdots & \vdots && \vdots \\
c_{N-r+k}^1 & c_{N-r+k}^2 & \cdots & c_{N-r+k}^{N-r} & d_{N-r+k}^1 & d_{N-r+k}^2 & \cdots & d_{N-r+k}^k \end{pmatrix}.
\end{equation*}
Part of the condition of being inverse matrices is
\ea
\sum_{p=1}^{N-r+k}d^j_p(a_i^p-b_i^p)&=0,\;\> i=1,\ldots,N-r, \;\> j=1,\ldots,k,
\label{gc5eq20}\\
\sum_{p=1}^{N-r+k}d^j_p(s_i^p-t_i^p)&=\begin{cases} 1, & i=j=1,\ldots,k, \\
0, & i,j=1,\ldots,k,\;\> i\ne j.
\end{cases}
\label{gc5eq21}
\ea

Define interior functions $\hat x_1,\ldots,\hat x_N:W'\ra[0,\iy)$ and smooth functions $\hat z_1,\ldots,\hat z_n:W'\ra\R$ by
\ea
\hat x_p(\bs x,\bs z)&=\begin{cases} x_p\cdot \prod_{i=1}^k\frac{\ti D_i(\bs x,\bs z)^{d^i_p}}{\ti E_i(\bs x,\bs z)^{d^i_p}}, & p=1,\ldots,N-r+k,\\
x_p, & p=N-r+k+1,\ldots,x_N,
\end{cases}
\label{gc5eq22}\\
\hat z_j(\bs x,\bs z)&=\begin{cases} \ti h_j(\bs x,\bs z), & j=1,\ldots,l, \\
z_j, & j=l+1,\ldots,n.
\end{cases}
\label{gc5eq23}
\ea
Then \eq{gc5eq17} and \eq{gc5eq20}--\eq{gc5eq23} imply that for all $(\bs x,\bs z)\in W'$ we have
\ea
x_1^{a^1_i}\cdots x_N^{a^N_i}\!\!&=\!x_1^{b^1_i}\cdots x_N^{b^N_i}
&&\!\!\!\!\!\!\!\Longleftrightarrow\!\!\!\! &
\hat x_1^{a^1_i}\cdots \hat x_N^{a^N_i}&=\hat x_1^{b^1_i}\cdots \hat x_N^{b^N_i}, && \!\!\!i\!=\!1,\ldots,N\!-\!r,
\label{gc5eq24}\\
\ti f_i(\bs x,\bs z)\!&=\!\ti g_i(\bs x,\bs z)
&&\!\!\!\!\!\!\!\Longleftrightarrow\!\!\!\! &
\hat x_1^{s_i^1}\cdots\hat x_N^{s_i^N}&=\hat x_1^{t_i^1}\cdots\hat x_N^{t_i^N}, && \!\!\!i\!=\!1,\ldots,k,
\label{gc5eq25}\\
\ti h_j(\bs x,\bs z)\!&=\!0
&&\!\!\!\!\!\!\!\Longleftrightarrow\!\!\!\! &
\hat z_j&=0, && \!\!\!j\!=\!1,\ldots,l.
\label{gc5eq26}
\ea

Define a smooth function $\Psi:W'\ra[0,\iy)^N\t\R^n$ by
\e
\Psi(\bs x,\bs z)=\bigl(\hat x_1(\bs x,\bs z),\ldots,\hat x_N(\bs x,\bs z),\hat z_1(\bs x,\bs z),\ldots,\hat z_n(\bs x,\bs z)\bigr).
\label{gc5eq27}
\e
Then \eq{gc5eq18} and \eq{gc5eq22}--\eq{gc5eq23} imply that $\Psi$ is simple with $\Psi(0)=0$, and ${}^b\d\Psi\vert_0=\id:\R^{N+n}\ra\R^{N+n}$. Thus Proposition \ref{gc2prop3} says $\Psi$ is \'etale near $0$ in $W'$. So by making $V',W'$ smaller, we can suppose that $W'':=\Im\Psi$ is an open neighbourhood of $0$ in $[0,\iy)^N\t\R^n$, and $\Psi:W'\ra W''$ is a diffeomorphism. Equations \eq{gc4eq3}, \eq{gc5eq15}--\eq{gc5eq16} and \eq{gc5eq24}--\eq{gc5eq27} now imply that
\ea
\Psi\ci[(\la_{q_1}&\!\t\!\cdots\!\t\!\la_{q_N})\!\t\!\id_{\R^n}](X^\ci\cap V')\!=\!\bigl\{(x_1,\ldots,x_N,z_1,\ldots,z_n)\!\in \!W^{\prime\prime\ci}:
\nonumber\\
&x_1^{a^1_i}\cdots x_N^{a^N_i}=x_1^{b^1_i}\cdots x_N^{b^N_i},\;\> i=1,\ldots,N-r,
\label{gc5eq28}\\
&x_1^{s_i^1}\cdots x_N^{s_i^N}=x_1^{t_i^1}\cdots x_N^{t_i^N}, \;\> i=1,\ldots,k,\;\> z_j=0, \;\> j=1,\ldots,l\bigr\}.
\nonumber
\ea

As in equation \eq{gc4eq5}, define
\e
P^\vee=\bigl\{\rho\in Q^\vee:\rho(\be_i)=0,\;\> i=1,\ldots,k\bigr\}.
\label{gc5eq29}
\e
Then $P^\vee$ is a toric monoid, a submonoid of $Q^\vee$. Equivalently, we have
\e
\begin{split}
P^\vee\cong\bigl\{(c_1,\ldots,c_N)\in\N^N:\,&\ts\sum_{j=1}^N(a^j_i-b^j_i)c_j=0,\;\> i=1,\ldots,N-r,\\
&\ts\sum_{j=1}^N(s^j_i-t^j_i)c_j=0,\;\> i=1,\ldots,k\bigr\}.
\end{split}
\label{gc5eq30}
\e
Write $\al^\vee:P^\vee\ra Q^\vee$ for the inclusion morphism. Taking duals gives a toric monoid $P$ with a monoid morphism~$\al:Q\ra P$. 

We expect $P^\vee$ and $P$ to have rank $r-k=N-(N-r)-k$, since $P^\vee$ is defined by $k$ linearly independent equations in $Q^\vee$ of rank $r$ in \eq{gc5eq29}, or by $N-r+k$ linearly independent equations in $\N^N$ of rank $N$ in \eq{gc5eq30}. This is not immediate, as for monoids the rank could be lower than expected --- consider for instance $\bigl\{(c_1,c_2)\in\N^2:c_1+c_2=0\bigr\}=\bigl\{(0,0)\bigr\}$, defined by 1 equation in a monoid $\N^2$ of rank 2, but which has rank~$0<2-1$.

To see that $P^\vee,P$ do have the expected rank $r-k$, note that as $(\de_0,0)\in X$ by assumption, $(0,\ldots,0)$ lies in the closure of the r.h.s.\ of \eq{gc5eq28}, so we can find solutions $(x_1,\ldots,x_N,0,\ldots,0)$ to the equations of \eq{gc5eq28} with $x_1,\ldots,x_N>0$ arbitrarily small. Setting $c_j=-\log x_j$, we see $(\de_0,0)\in X$ implies that there exist solutions $(c_1,\ldots,c_N)$ to the equations in \eq{gc5eq30} with $c_1,\ldots,c_N\gg 0$ large in $\R$, and so also with $c_1,\ldots,c_N\gg 0$ large in $\N$, as $a_i^j,b_i^j,s_i^j,t_i^j\in\N$. The only way that $P^\vee$ could have smaller than the expected rank is if all solutions $(c_1,\ldots,c_N)$ in \eq{gc5eq30} lay in some boundary face of $\N^N$, but as there are solutions $(c_1,\ldots,c_N)$ with $c_j\gg 0$ for all $j$, this does not happen. So $P^\vee,P$ have rank~$r-k$.

Set $m=n-l$, so that $\rank P+m=\rank Q+n-k-l$, as in the theorem. Define $\Xi:X_P\t\R^m\ra [0,\iy)^N\t\R^n$ by
\begin{equation*}
\smash{\Xi\bigl(\bs v,(w_1,\ldots,w_m)\bigr)=\bigl(\la_{\al(q_1)}(\bs v),\ldots,\la_{\al(q_N)}(\bs v),{\buildrel{\ulcorner\,\,\,\,\,\, l \,\,\,\,\,\,\urcorner} \over
{\vphantom{i}\smash{0,\ldots,0}}},w_1,\ldots,w_m\bigr).}
\end{equation*}
It is easy to see that $\Xi$ is an embedding, and a similar proof to Proposition \ref{gc3prop3}(a) shows the image in the interior of $[0,\iy)^N\t\R^n$ is
\ea
(\Im\Xi)\cap {}&[(0,\iy)^N\t\R^n]=\bigl\{(x_1,\ldots,x_N,z_1,\ldots,z_n)\!\in (0,\iy)^N\t\R^n:
\nonumber\\
&x_1^{a^1_i}\cdots x_N^{a^N_i}=x_1^{b^1_i}\cdots x_N^{b^N_i},\;\> i=1,\ldots,N-r,
\label{gc5eq31}\\
&x_1^{s_i^1}\cdots x_N^{s_i^N}=x_1^{t_i^1}\cdots x_N^{t_i^N}, \;\> i=1,\ldots,k,\;\> z_j=0, \;\> j=1,\ldots,l\bigr\}.
\nonumber
\ea

Define $U=\Xi^{-1}(W'')$, an open neighbourhood of $(\de_0,0)$ in $X_P\t\R^m$. Then comparing \eq{gc5eq28} and \eq{gc5eq31} shows that
\begin{equation*}
\Xi(U^\ci)=\Psi\ci[(\la_{q_1}\t\cdots\t\la_{q_N})\t\id_{\R^n}](X^\ci\cap V'),
\end{equation*}
so composing with $\Psi^{-1}:W''\ra W'$ and taking closures in $U,V',W'$ shows that
\e
\Psi^{-1}\ci\Xi(U)=[(\la_{q_1}\t\cdots\t\la_{q_N})\t\id_{\R^n}](X\cap V').
\label{gc5eq32}
\e

As $[(\la_{q_1}\t\cdots\t\la_{q_N})\t\id_{\R^n}]\vert_{V'}:V'\hookra W'$ and $\Psi^{-1}\ci\Xi$ are both embeddings, Corollary \ref{gc4cor1} shows that there is a unique embedding $\phi:U\ra V'$ with
\begin{equation*}
[(\la_{q_1}\t\cdots\t\la_{q_N})\t\id_{\R^n}]\ci\phi=\Psi^{-1}\ci\Xi,
\end{equation*}
which is interior as $\Psi^{-1}\ci\Xi$ is. Then \eq{gc5eq32} gives $\phi(U)=X\cap V'$ as $(\la_{q_1}\t\cdots\t\la_{q_N})\t\id_{\R^n}$ is injective, and $\phi(\de_0,0)=(\de_0,0)$ as $\Psi^{-1}\ci\Xi(\de_0,0)=[(\la_{q_1}\t\cdots\t\la_{q_N})\t\id_{\R^n}](\de_0,0)=0$. The monoid morphism $\ti M_{(\de_0,0)}\phi:\ti M_{(\de_0,0)}U\ra\ti M_{(\de_0,0)}V$ is naturally identified with the inclusion $P^\vee\hookra Q^\vee$ from \eq{gc5eq29}. This proves the first two parts of Theorem~\ref{gc4thm3}.

At the beginning of the proof, if $\be_1,\ldots,\be_k$ were not linearly independent over $\R$ then we replaced $f_i,g_i,s_i,t_i,\be_i,h_j$ by $f'_i,g'_i,s'_i,t'_i,\be'_i$ for $i=1,\ldots,k'$ and $h_j'$ for $j=1,\ldots,l'$, with $\be'_1,\ldots,\be'_{k'}$ linearly independent over $\R$. For the last part of Theorem \ref{gc4thm3}, this replacement would cause problems, as if $(\de_0,0)\notin X$ we can have $h_j'(\de_0,0)\ne 0$ for $h_j'$ as in \eq{gc5eq12}. Therefore, as in the last part of the theorem, we now assume that $\be_1,\ldots,\be_k$ from the theorem are linearly independent over $\R$, and take $f_i,g_i,s_i,t_i,\be_i,h_j$ to be as in the theorem, without replacement. We also drop the standing assumption that~$(\de_0,0)\in X$.

The analysis above shows that $(\de_0,0)\in X$ if and only if $0$ lies in the closure of the r.h.s.\ of \eq{gc5eq28}, if and only if there are solutions $(x_1,\ldots,x_N,0,\ldots,0)$ to the equations of \eq{gc5eq28} with $x_1,\ldots,x_N>0$ arbitrarily small. Setting $c_j=-\log x_j$, we see $(\de_0,0)\in X$ if and only if there exist solutions $(c_1,\ldots,c_N)$ to the equations in \eq{gc5eq30} with $c_1,\ldots,c_N\gg 0$ large in $\R$, and so also with $c_1,\ldots,c_N\gg 0$ large in $\N$, as $a_i^j,b_i^j,s_i^j,t_i^j\in\N$. Such solutions $(c_1,\ldots,c_N)\in P^\vee$ cannot lie in any boundary face of $\N^N$, and so not in any boundary face of~$Q^\vee$.

Conversely, if $P^\vee$ in \eq{gc5eq29} does not lie in any boundary face of $Q^\vee$, then the r.h.s.\ of \eq{gc5eq30} does not lie in any boundary face of $\N^N$, and so contains solutions $(c_1,\ldots,c_N)$ with $c_j>0$ for $j=1,\ldots,N$. Then $(x_1,\ldots,x_N,z_1,\ldots,z_n) =(e^{-tc_1},\ldots,e^{-tc_N},0,\ldots,0)$ satisfies the equations of \eq{gc5eq30} for $t>0$, and taking $t\ra\iy$ shows that $(\de_0,0)\in X$.
Thus, $(\de_0,0)\in X$ is equivalent to the condition that the r.h.s.\ of \eq{gc4eq5} (i.e. equation \eq{gc5eq29}) does not lie in any proper face $F\subsetneq Q^\vee$ of the toric monoid $Q^\vee$. This completes the proof of Theorem~\ref{gc4thm3}.

\subsection{Proof of Theorem \ref{gc4thm5}}
\label{gc53}

Let $g:X\ra Z$ and $h:Y\ra Z$ be interior maps of manifolds with g-corners. Suppose $(x,\ga)\in C(X)$ and $(y,\de)\in C(Y)$ with $C(g)[(x,\ga)]=C(h)[(y,\de)]=(z,\ep)$ in $C(Z)$. Then we have a commutative diagram with exact rows \eq{gc3eq32}
\begin{equation*}
\xymatrix@C=18pt@R=16pt{ 0 \ar[r] & {\begin{subarray}{l}\ts {}^bN_{C(X)}\vert_{(x,\ga)}\op{} \\
\ts {}^bN_{C(Y)}\vert_{(y,\de)} \end{subarray}} \ar@<-4ex>[d]^{{}^bN_{C(g)}\vert_{(x,\ga)}\op {}^bN_{C(h)}\vert_{(y,\de)}}\ar[rr]_(0.55){{}^bi_T\op{}^bi_T} &&
{\begin{subarray}{l}\ts {}^bT_xX\op{} \\ \ts {}^bT_yY \end{subarray}} \ar[rr]_(0.4){{}^b\pi_T \op {}^b\pi_T} \ar[d]^{{}^bT_xg\op{}^bT_yh} && {\begin{subarray}{l} \ts {}^bT_{(x,\ga)}(C(X))\op{} \\ \ts {}^bT_{(y,\de)}(C(Y))\end{subarray}} \ar@<-4ex>[d]^{{}^bT_{(x,\ga)}C(g)\op{}^bT_{(y,\de)}C(h)} \ar[r] & 0 \\
 0 \ar[r] & {}^bN_{C(Z)}\vert_{(z,\ep)} \ar[rr]^(0.45){{}^bi_T} &&
{}^bT_zZ \ar[rr]^{{}^b\pi_T} && {}^bT_{(z,\ep)}(C(Z)) \ar[r] & 0.\!\! }
\end{equation*}
If $g,h$ are b-transverse, the central column is surjective, so the right hand column is surjective, and $C(g),C(h)$ are b-transverse, as we have to prove.

Now suppose $g,h$ are c-transverse. Then they are b-transverse, so $C(g),C(h)$ are b-transverse from above, which is the first condition for $C(g),C(h)$ to be c-transverse. We have a commutative diagram with exact rows
\begin{equation*}
\xymatrix@C=18pt@R=14pt{ 0 \ar[r] & {\begin{subarray}{l}\ts {}^bN_{C(X)}\vert_{(x,\ga)}\op{} \\
\ts {}^bN_{C(Y)}\vert_{(y,\de)} \end{subarray}} \ar@<-4ex>[d]^(0.6){{}^bN_{C(g)}\vert_{(x,\ga)}\op {}^bN_{C(h)}\vert_{(y,\de)}} \ar[rr] &&
{\begin{subarray}{l}\ts {}^b\ti N_xX\op{} \\ \ts {}^b\ti N_yY \end{subarray}} \ar[rr] \ar[d]^(0.6){{}^b\ti N_xg\op{}^b\ti N_yh} && {\begin{subarray}{l} \ts {}^b\ti N_{(x,\ga)}(C(X))\op{} \\ \ts {}^b\ti N_{(y,\de)}(C(Y))\end{subarray}} \ar@<-6ex>[d]^(0.6){{}^b\ti N_{(x,\ga)}C(g)\op{}^b\ti N_{(y,\de)}C(h)} \ar[r] & 0 \\
 0 \ar[r] & {}^bN_{C(Z)}\vert_{(z,\ep)} \ar[rr] &&
{}^b\ti N_zZ \ar[rr] && {}^b\ti N_{(z,\ep)}(C(Z)) \ar[r] & 0.\!\! }
\end{equation*}
As $g,h$ are c-transverse, the central column is surjective, so the right hand column is surjective, the second condition for $C(g),C(h)$ to be c-transverse.

We have a commutative diagram of monoids with surjective columns
\e
\begin{gathered}
\xymatrix@C=80pt@R=16pt{
*+[r]{\ti M_xX} \ar[r]_(0.55){\ti M_xg} \ar@{->>}[d] & \ti M_zZ \ar@{->>}[d] & *+[l]{\ti M_yY} \ar[l]^(0.55){\ti M_yh} \ar@{->>}[d]\\
*+[r]{\ti M_{(x,\ga)}C(X)} \ar[r]^(0.55){\ti M_{(x,\ga)}C(g)} & \ti M_{(z,\ep)}C(Z) & *+[l]{\ti M_{(y,\de)}C(Y).\!\!} \ar[l]_(0.55){\ti M_{(y,\de)}C(h)} }
\end{gathered}
\label{gc5eq33}
\e
Equation \eq{gc4eq10} for $g,h$ at $x,y$ is constructed from the top line of \eq{gc5eq33}, and \eq{gc4eq10} for $C(g),C(h)$ at $(x,\ga),(y,\de)$ from the bottom line of \eq{gc5eq33}. Thus the columns of \eq{gc5eq33} induce a morphism from \eq{gc4eq10} for $g,h$ at $x,y$ to \eq{gc4eq10} for $C(g),C(h)$ at $(x,\ga),(y,\de)$. As $g,h$ are c-transverse, \eq{gc4eq10} for $g,h$ at $x,y$ does not lie in a proper face of $\ti M_xX\t\ti M_yY$, so surjectivity of the columns of \eq{gc5eq33} implies that its image in $\ti M_{(x,\ga)}C(X)\t\ti M_{(y,\de)}C(Y)$ does not lie in a proper face of $\ti M_{(x,\ga)}C(X)\t\ti M_{(y,\de)}C(Y)$. Thus \eq{gc4eq10} for $C(g),C(h)$ at $(x,\ga),(y,\de)$ does not lie in a proper face of $\ti M_{(x,\ga)}C(X)\t\ti M_{(y,\de)}C(Y)$, the final  condition for $C(g),C(h)$ to be c-transverse. This completes the proof.

\subsection{Proof of Theorem \ref{gc4thm6}}
\label{gc54}

Suppose $X,Y,Z,g,h,W^\ci$ and $W=\ov{W^\ci}$ are as in Theorem \ref{gc4thm6}. We first prove that $W$ is an embedded submanifold of $X\t Y$, with $\dim W=\dim X+\dim Y-\dim Z$. Suppose $(x,y)\in W$. Then $g(x)=h(y)=z\in Z$, since this holds for all $(x',y')\in W^\ci$ and extends to $W=\ov{W^\ci}$ by continuity of $g,h$. Thus ${}^bT_xg\op{}^bT_yh:{}^bT_xX\op{}^bT_yY\ra{}^bT_zZ$ is surjective by b-transversality.

Let $X,Y,Z$ near $x,y,z$ be modelled on $X_Q\t\R^m,X_R\t\R^n,X_S\t\R^q$ near $(\de_0,0)$ respectively, for toric monoids $Q,R,S$ and $m,n,q\ge 0$, and write points of $X,Y,Z$ near $x,y,z$ as $(\bs u,\bs x),(\bs v,\bs y),(\bs w,\bs z)$ for $\bs u\in X_Q$, $\bs x=(x_1,\ldots,x_m)\in\R^m$, $\bs v\in X_R$, $\bs y=(y_1,\ldots,y_n)\in\R^n$, $\bs w\in X_S$, $\bs z=(z_1,\ldots,z_q)\in\R^q$. Then write $g,h$ near $x,y$ as $g(\bs u,\bs x)=\bigl(G(\bs u,\bs x),(g_1(\bs u,\bs x),\ldots,g_q(\bs u,\bs x))\bigr)=(\bs w,\bs z)$ and $h(\bs v,\bs y)=\bigl(H(\bs v,\bs y),(h_1(\bs v,\bs y),\ldots,h_q(\bs v,\bs y))\bigr)=(\bs w,\bs z)$.

Set $p=\rank S$. Choose $s_1,\ldots,s_p\in S$ which are a basis over $\R$ of $S\ot_\N\R$. Then from the definitions in \S\ref{gc32} one can show that
\e
\bigl\{(\si,\si):\si\!\in\! X_S^\ci\bigr\}\!=\!\bigl\{(\si_1,\si_2)\!\in\! X_S^\ci\!\t\! X_S^\ci:\la_{s_i}(\si_1)\!=\!\la_{s_i}(\si_2),\;i\!=\!1,\ldots,p\bigr\},
\label{gc5eq34}
\e
although the analogue with $X_S$ in place of $X_S^\ci$ need not hold, as $s_1,\ldots,s_p$ may not generate $S$ as a monoid. From \eq{gc4eq11} and \eq{gc5eq34} it follows that for open neighbourhoods $U$ of $(x,y)$ in $X\t Y$ and $V$ of $(\de_0,\de_0,0,0)$ in $X_Q\t X_R\t\R^m\t\R^n$, we have an identification
\begin{align*}
W^\ci\cap U\cong\bigl\{(\bs u,\bs v,\bs x,\bs y)\in V^\ci:\,
&\la_{s_i}\ci G(\bs u,\bs x)=\la_{s_i}\ci H(\bs v,\bs y),\;\> i=1,\ldots,p,\\
&g_j(\bs u,\bs x)-h_j(\bs v,\bs y)=0,\;\> j=1,\ldots,q\bigr\}.
\end{align*}
 
We now apply Theorem \ref{gc4thm3} with $Q\t R,m+n,p,q,\la_{s_i}\ci G(\bs u,\bs x),\la_{s_i}\ci H(\bs v,\bs y),\ab g_j(\bs u,\bs x)-h_j(\bs v,\bs y)$ in place of $Q,n,k,l,f_i,g_i,h_j$, respectively, noting that $X_Q\t X_R\cong X_{Q\t R}$. The fact that ${}^bT_xg\op{}^bT_yh:{}^bT_xX\op{}^bT_yY\ra{}^bT_zZ$ is surjective and $s_1,\ldots,s_p$ are linearly independent in $S\ot_\N\R$ implies that
\begin{align*}
&{}^b\d[\la_{s_i}\ci G(\bs u,\bs x)]\vert_{(\de_0,\de_0,0,0)}-{}^b\d[\la_{s_i}\ci H(\bs v,\bs y)]\vert_{(\de_0,\de_0,0,0)},&& i=1,\ldots,p,\\
&\d[g_j(\bs u,\bs x)-h_j(\bs v,\bs y)]\vert_{(\de_0,\de_0,0,0)}, &&j=1,\ldots,q,
\end{align*}
are linearly independent in ${}^bT^*_{(\de_0,\de_0,0,0)}(X_Q\t X_R\t\R^m\t\R^n)$. So Theorem \ref{gc4thm3} implies that in an open neighbourhood $U'$ of $(x,y)$ in $U\subseteq X\t Y$, $W=\ov{W^\ci}$ is an embedded submanifold of $U$, of dimension $\rank Q+\rank R+m+n-p-q=\dim X+\dim Y-\dim Z$. As this holds for all $(x,y)\in W$, $W$ is an embedded submanifold of $X\t Y$, with~$\dim W=\dim X+\dim Y-\dim Z$.

Write $e:W\ra X$ and $f:W\ra Y$ for the compositions of the inclusion $W\hookra X\t Y$ with the projections to $X,Y$. Then $e,f$ are smooth, and interior as $W^\ci\subseteq X^\ci\t Y^\ci$ so that $e(W^\ci)\subseteq X^\ci$, $f(W^\ci)\subseteq Y^\ci$, and $g\ci e=h\ci f$ as $g(x)=h(y)$ for all $(x,y)\in W$. We claim that \eq{gc4eq7} is a Cartesian square in $\Mangcin$. To prove this, suppose $e':W'\ra X$, $f':W'\ra Y$ are interior morphisms of manifolds with g-corners, with $g\ci e'=h\ci f'$. Consider the direct product $(e',f'):W'\ra X\t Y$. As $e',f'$ are interior with $g\ci e'=h\ci f'$ we see from \eq{gc4eq9} that $(e',f')[W^{\prime\ci}]\subseteq W^\ci\subseteq X^\ci\t Y^\ci$. So taking closures implies that~$(e',f')[W']\subseteq \ov{W^\ci}=W\subseteq X\t Y$.

As the inclusion $W\hookra X\t Y$ is an embedding, Corollary \ref{gc4cor1} implies that $b=(e',f'):W'\ra W$ is smooth, and in fact interior, and is unique with $e'=e\ci b$ and $f'=f\ci b$. This proves the universal property for \eq{gc4eq7} to be Cartesian in $\Mangcin$, so $W=X\t_{g,Z,h}Y$ is a fibre product in~$\Mangcin$. 

\subsection{Proof of Theorem \ref{gc4thm7}}
\label{gc55}

Suppose $g:X\ra Z$ and $h:Y\ra Z$ are c-transverse morphisms in $\Mangcin$. Then $g,h$ are b-transverse, so Theorem \ref{gc4thm6}, proved in \S\ref{gc54}, shows that a fibre product $W=X\t_{g,Z,h}Y$ exists in $\Mangcin$, where as an embedded submanifold of $X\t Y$ we have $W=\ov{W^\ci}$ for $W^\ci$ given by \eq{gc4eq11}, with $\dim W=\dim X+\dim Y-\dim Z$, and projections $e:W\ra X$, $f:W\ra Y$ mapping $e:(x,y)\mapsto x$, $f:(x,y)\mapsto y$.

We first show that as $g,h$ are c-transverse, $W\subseteq X\t Y$ has the simpler expression $W=\bigl\{(x,y)\in X\t Y:g(x)=h(y)\bigr\}$, as in \eq{gc4eq12}. Clearly $W\subseteq\bigl\{(x,y)\in X\t Y:g(x)=h(y)\bigr\}$, since $W^\ci\subseteq\bigl\{(x,y)\in X\t Y:g(x)=h(y)\bigr\}$ by \eq{gc4eq11}, $W=\ov{W^\ci}$, and $g,h$ are continuous. 

Suppose $x\in X$ and $y\in Y$ with $g(x)=h(y)=z\in Z$, but do not assume $(x,y)\in W$. Follow the proof of Theorem \ref{gc4thm6} in \S\ref{gc54} up to the point where we apply Theorem \ref{gc4thm3}. As $g,h$ are c-transverse, ${}^b\ti N_xg\op{}^b\ti N_yh:{}^b\ti N_xX\op{}^b\ti N_yY\ra{}^b\ti N_zZ$ is surjective. In the notation of Theorem \ref{gc4thm3} we can identify ${}^b\ti N_xg\op{}^b\ti N_yh$ with $\be_1\op\cdots\op\be_k:\Hom(Q,\R)\ra\R^k$, so ${}^b\ti N_xg\op{}^b\ti N_yh$ surjective is equivalent to $\be_1,\ldots,\be_k$ linearly independent over $\R$ in $Q\ot_\N\R$, which is a hypothesis of the last part of Theorem \ref{gc4thm3}. 

Now $W,(x,y),\eq{gc4eq10},\ti M_xX\t\ti M_yY$ above are identified with
$X,(\de_0,0),\eq{gc4eq5}$ and $Q^\vee$ in Theorem \ref{gc4thm3}, respectively. Thus the last part of Theorem \ref{gc4thm3} says that $(x,y)\in W$ if and only if the submonoid \eq{gc4eq10} is not contained in any proper face $F\subsetneq\ti M_xX\t \ti M_yY$ of $\ti M_xX\t\ti M_yY$. The latter holds by Definition \ref{gc4def9} as $g,h$ are c-transverse, so $(x,y)\in W$. Therefore $\bigl\{(x,y)\in X\t Y:g(x)=h(y)\bigr\}\subseteq W$, so $W=\bigl\{(x,y)\in X\t Y:g(x)=h(y)\bigr\}$, proving~\eq{gc4eq12}.

We can now show $W$ is also a fibre product $X\t_{g,Z,h}Y$ in $\Mangc$ using Corollary \ref{gc4cor1}, following the proof for $\Mangcin$ in \S\ref{gc54}, but without supposing $e',f'$ are interior. This proves the first part of Theorem~\ref{gc4thm7}.

For the second part, $C(g)$ and $C(h)$ are c-transverse in $\cMangc$ by Theorem \ref{gc4thm5}, so by the first part (extended to $\cMangc$ in the obvious way), setting
\begin{equation*}
\check W=\bigl\{\bigl((x,\ga),(y,\de)\bigr)\in C(X)\t C(Y):C(g)[(x,\ga)]=C(h)[(y,\de)]\bigr\},
\end{equation*}
then $\check W$ is a submanifold of mixed dimension of $C(X)\t C(Y)$, and is a fibre product $\check W=C(X)\t_{C(g),C(Z),C(h)}C(Y)$ in both $\cMangc$ and $\cMangcin$. Applying the universal property of the fibre product to \eq{gc4eq13} gives a unique map $\check b:C(W)\ra\check W$, which is just the direct product $(C(e),C(f)):C(W)\ra C(X)\t C(Y)\supseteq \check W$. We must show $\check b$ is a diffeomorphism.

From the construction of $W$ in \S\ref{gc54}, we see that the strata $S^i(W)$ consist locally of those points $(x,y)\in X\t Y$ with $x\in S^j(X)$, $y\in S^k(Y)$ and $g(x)=h(y)=z\in S^l(Z)$ for some fixed strata $S^j(X),S^k(Y),S^l(Z)$ of $X,Y,Z$. That is, locally $S^i(W)\cong S^j(X)\t_{S^l(Z)}S^k(Y)$. As this is a local transverse fibre product of manifolds without boundary, it has dimension $\dim W-i=(\dim X-j)+(\dim Y-k)-(\dim Z-l)$, which forces $i=j+k-l$. This shows that
\e
S^i(W)=\coprod_{j,k,l\ge 0:i=j+k-l} S^{j,l}(X)\t_{g\vert_{S^{j,l}(X)},S^l(Z),h\vert_{S^{k,l}(Y)}}S^{k,l}(Y),
\label{gc5eq35}
\e
where $S^{j,l}(X)=S^j(X)\cap g^{-1}(S^l(Z))$ and $S^{k,l}(Y)=S^k(Y)\cap h^{-1}(S^l(Z))$, and the fibre products in \eq{gc5eq35} are transverse fibre products of manifolds.

Since $\check W\in\cMangc$ it is a disjoint union of manifolds with g-corners of different dimensions, which range from 0 to $\dim W$. Write $\check W^i$ for the component of $\check W$ of dimension $\dim W-i$, so that $\check W=\coprod_{i=0}^{\dim W}\check W^i$. Then
\e
\check W^i=\coprod_{j,k,l\ge 0:i=j+k-l} C_j^l(X)\t_{C(g)\vert_{C_j^l(X)},C_l(Z),C(h)\vert_{C_k^l(Y)}}C_k^l(Y),
\label{gc5eq36}
\e
where $C_j^l(X)=C_j(X)\cap C(g)^{-1}(C_l(Z))$ and $C_k^l(Y)=C_k(Y)\cap C(h)^{-1}(C_l(Z))$, and the fibre products in \eq{gc5eq36} are b-transverse fibre products in $\Mangcin$. Restricting to interiors gives
\e
(\check W^i)^\ci=\coprod_{j,k,l\ge 0:i=j+k-l} C_j^l(X)^\ci\t_{C(g)\vert_{C_j^l(X)^\ci},C_l(Z)^\ci,C(h)\vert_{C_k^l(Y)^\ci}}C_k^l(Y)^\ci,
\label{gc5eq37}
\e
where the fibre products in \eq{gc5eq37} are transverse fibre products of manifolds.

Mapping $(x,\ga)\mapsto x$ gives a diffeomorphism $C_j(X)^\ci\ra S^j(X)$, which identifies $C_j^l(X)^\ci\cong S^{j,l}(X)$, and similarly $C_k(Y)^\ci\cong S^k(Y)$, $C_k^l(Y)^\ci\cong S^{k,l}(Y)$, and $C_l(Z)^\ci\cong S^l(Z)$. So comparing \eq{gc5eq35} and \eq{gc5eq37} shows we have a canonical diffeomorphism $S^i(W)\cong (\check W^i)^\ci$. But $S^i(W)\cong C_i(W)^\ci$, so $C_i(W)^\ci\cong(\check W^i)^\ci$. One can check that this diffeomorphism $C_i(W)^\ci\ra(\check W^i)^\ci$ is the restriction to $C_i(W)^\ci$ of $\check b:C(W)\ra\check W$. Therefore $\check b\vert_{C(W)^\ci}:C(W)^\ci\ra\check W^\ci$ is a diffeomorphism of the interiors~$C(W)^\ci,\check W^\ci$.

There are natural projections $\Pi_1:C(W)\ra X\t Y$ by composing $\Pi:C(W)\ra W$ with $W\hookra X\t Y$, and $\Pi_2:\check W\ra X\t Y$ by composing $\Pi\t\Pi:C(X)\t C(Y)\ra X\t Y$ with $\check W\hookra C(X)\t C(Y)$. Both $\Pi_1,\Pi_2$ are proper immersions, and $\Pi_1=\Pi_2\ci\check b$. One can prove using Corollary \ref{gc4cor1} that $\check b:C(W)\ra\check W$ smooth with $\check b\vert_{C(W)^\ci}:C(W)^\ci\ra\check W^\ci$ a diffeomorphism and $\Pi_1,\Pi_2$ proper immersions with $\Pi_1=\Pi_2\ci\check b$ together imply that $\check b$ is a diffeomorphism. Therefore \eq{gc4eq13} is Cartesian in both $\cMangc$ and $\cMangcin$, as we have to prove. For the last part, the grading-preserving property \eq{gc4eq14} holds on the interior $C(W)^\ci$ by \eq{gc5eq35}--\eq{gc5eq37}, and so extends to $C(W)$ by continuity. This completes the proof of Theorem~\ref{gc4thm7}.

\medskip

\noindent{\small\sc The Mathematical Institute, Radcliffe Observatory Quarter, Woodstock Road, Oxford, OX2 6GG, U.K.}

\noindent{\small\sc E-mail: \tt joyce@maths.ox.ac.uk}

\end{document}